\documentclass[12pt]{amsart}
\usepackage{amsfonts,graphicx,
stmaryrd,amscd,latexsym,amsbsy}
\usepackage{pst-node}
\usepackage{tikz-cd} 
\usepackage{subcaption}
\usepackage[colorlinks=true]{hyperref}

\usepackage{pst-node}
\usepackage{pst-intersect}

\usepackage{amssymb}
\usepackage{amsmath,amsthm}
\usepackage{xypic}

\newcommand\cc{c}

\newcommand\Pc{\mathcal{P}^{(c)}}

\usepackage[tmargin=1in]{geometry}

\numberwithin{equation}{section}

\newtheorem{theorem}{Theorem}[section]

\newtheorem{lemma}[theorem]{Lemma}
\newtheorem{definition}[theorem]{Definition}
\newtheorem{proposition}[theorem]{Proposition}
\newtheorem{corollary}[theorem]{Corollary}

\newtheorem{example}[theorem]{Example}

\theoremstyle{remark}
\newtheorem{remark}[theorem]{Remark}

\newcommand{\R}{{\mathbb R}}

\newcommand{\Z}{{\mathbb Z}}

\begin{document}
\title{Quasi-polynomial representations of double affine Hecke algebras}
\author{Siddhartha Sahi, Jasper Stokman, Vidya Venkateswaran}
\address{S. Sahi, Department of Mathematics, Rutgers University, 110 Frelinghhuysen Rd,
Piscataway, NJ 08854-8019, USA.}
\address{J. Stokman, KdV Institute for Mathematics, University of Amsterdam, Science
Park 105-107, 1098 XG Amsterdam, The Netherlands.}
\address{V. Venkateswaran, Center for Communications Research, 805 Bunn Dr, Princeton,
NJ 08540, USA. }
\begin{abstract}
We introduce an explicit family of representations of the double affine Hecke algebra $\mathbb{H}$ acting on spaces of quasi-polynomials, defined in terms of truncated Demazure-Lusztig type operators. We show that these quasi-polynomial representations provide concrete realisations of a natural family of cyclic $Y$-parabo\-li\-cally induced 
$\mathbb{H}$-representations.  We recover Cherednik's well-known polynomial representation as a special case.

The quasi-polynomial representation gives rise to a family of commuting operators acting on spaces of quasi-polynomials.  These generalize the Cherednik operators, which are fundamental in the study of Macdonald polynomials.  We provide a detailed study of their joint eigenfunctions, which may be regarded as quasi-polynomial, multi-parametric generalisations of nonsymmetric Macdonald polynomials.  We also introduce generalizations of symmetric Macdonald polynomials, which are invariant under a multi-parametric generalization of the standard Weyl group action.

We connect our results to the representation theory of metaplectic covers of reductive groups over non-archimedean local fields.  We introduce root system generalizations of the metaplectic polynomials from our previous work by taking a suitable restriction and reparametrization of the quasi-polynomial generalizations of Macdonald polynomials.  We show that metaplectic Iwahori-Whittaker functions can be recovered by taking the Whittaker limit of these metaplectic polynomials. 
\end{abstract}
\maketitle

\tableofcontents

\noindent
 \section{Introduction}\label{Intro}

Nonsymmetric Macdonald polynomials \cite{Ma,Ch,Sa,Ha,StAB} are a remarkable family of functions that depend on several parameters -- usually two or three, but as many as six for the Koornwinder setting. For suitable choices of parameters, they specialise to many important families of special functions arising in the representation theory of reductive groups, including spherical functions and Iwahori-Whittaker functions. In turn, Macdonald polynomials can be understood in terms of  the representation theory of Cherednik's double affine Hecke algebra $\mathbb{H}$. More precisely, they are simultaneous eigenfunctions of a certain commutative subalgebra $\mathcal{P}_Y\subset \mathbb{H}$ in Cherednik's polynomial representation of $\mathbb{H}$ on the group algebra $\mathcal{P}$ of a coroot lattice $Q^\vee$.  

In this paper we discuss a generalisation of Macdonald polynomials, which can be understood via a family of representations $\pi_{\cc,\mathfrak{t}}$ of $\mathbb{H}$ on 
the subspace $\mathcal{P}^{(\cc)}$ of the group algebra of $Q^\vee\otimes_\Z \R$ generated by 
an affine Weyl group orbit $\mathcal{O}_\cc$ in $Q^\vee\otimes_{\mathbb{Z}}\mathbb{R}$. Here $\cc$ is the representative of the orbit that lies in the fundamental alcove, and $\mathfrak{t}$ represents a number of additional representation parameters.

We refer to elements in the representation space as \emph{quasi-polynomials}, and to $\pi_{\cc,\mathfrak{t}}$ as the quasi-polynomial representation. The simultaneous eigenfunctions of the commuting operators $\pi_{\cc,\mathfrak{t}}(h)$ ($h\in\mathcal{P}_Y$) 
provide our quasi-polynomial generalisations of Macdonald polynomials. They depend on the usual Macdonald parameters as well as on the additional representation parameters $\mathfrak{t}$. 

We establish an explicit connection
to the representation theory of metaplectic covers of reductive groups over non-archimedean local fields 
and the associated metaplectic Whittaker functions 
\cite{CO,McN,CGP,PP,PP2}.  
 In our previous work \cite{SSV}, we introduced a metaplectic \textit{affine} Hecke algebra action via explicit formulas and provided a direct construction as an induced representation.  We showed that the Chinta-Gunnells \cite{CG07,CG} action of the Weyl group, 
originally used to construct Weyl group multiple Dirichlet series, 
 can be recovered from this affine Hecke algebra action through localization. We also related this affine Hecke algebra action to  the metaplectic Demazure-Lusztig operators from \cite{PP,CGP,PP2}.  This is related to the context of the present paper as follows.   If the orbit $\mathcal{O}_\cc$ is contained in the rational vector space $Q^\vee\otimes_{\mathbb{Z}}\mathbb{Q}$, then the quasi-polynomials in $\mathcal{P}^{(c)}$ can be regarded as \textit{polynomials} under an appropriate reparametrization.  Moreover, under a suitable transformation, the restriction of the representation $\pi_{c, \frak{t}}$ to the affine Hecke algebra generated by $T_1, \cdots, T_r$ and $\mathcal{P}$ recovers the metaplectic affine Hecke algebra representation of \cite{SSV}.  Under these identifications, the quasi-polynomial generalizations of Macdonald polynomials extend the metaplectic polynomials introduced in \cite{SSV}, and studied further in \cite{Sai}, from type $A$ to an arbitrary root system.
 We also recover the metaplectic Iwahori-Whittaker and spherical Whittaker functions \cite{PP, McN} as Whittaker limits of these metaplectic polynomials and their antisymmetric versions.

In the following six subsections we provide an overview of our main results on the quasi-polynomial representations and associated generalisations of the Macdonald polynomials relative to adjoint root datum.

\subsection{Double affine Weyl group actions on quasi-polynomials}\label{daWgInstroSub}
To set the stage it is instructive to first introduce the quasi-polynomial representations of the double affine Weyl group.

Let $\Phi_0$ be a reduced irreducible root system, realised within the linear dual of a Euclidean space $E$, and normalised so that long roots have squared length $2/m^2$ for some
natural number $m\in\mathbb{Z}_{>0}$. 
Let $\Phi=\Phi_0\times\mathbb{Z}$ be the associated irreducible affine root system of untwisted type.
The coroot lattice $\widehat{Q}^\vee$ of $\Phi$ is isomorphic to $Q^\vee+m^2\Z K$, where $Q^\vee$ is the coroot lattice of $\Phi_0$ and $K$ is the central element in the associated untwisted affine Lie algebra. The affine Weyl group $W$ of $\Phi$ is a semidirect product $W=W_0\ltimes Q^\vee$, where $W_0$ is the Weyl group of $\Phi_0$. 
The double affine Weyl group is the semidirect product 
\begin{equation}\label{PBWdaWg}
\mathbb{W}=W\ltimes \widehat{Q}^\vee=(W_0\ltimes Q^\vee)\ltimes \widehat{Q}^\vee,
\end{equation}
where the $W$-action $(w,\widehat{\mu})\mapsto w\cdot\widehat{\mu}$ on $\widehat{Q}^\vee$ is determined by the pairing of $Q^\vee$ into the center $m^2\Z K$ of $\mathbb{W}$ induced by the inner product $\langle\cdot,\cdot\rangle$ on $E$.

We fix simple roots $\alpha_0,\alpha_1,\ldots, \alpha_r$ for $\Phi$ such that $\alpha_1, \ldots, \alpha_r$ are simple roots for $\Phi_0$. Let $\alpha_j^\vee$ be the corresponding simple coroots. Then $\alpha_0=(-\varphi,1)$, where $\varphi$ is the highest root for $\Phi_0$, and $\alpha_0^\vee=m^2K-\varphi^\vee$. The simple reflection associated to $\alpha_j$ is denoted by $s_j\in W$. The affine root hyperplane configuration in $E$ provides a decomposition of $E$ in closed alcoves, each closed alcove being a fundamental domain for the action of $W$ on $E$ by reflections and translations. The choice of simple roots singles out the fundamental open alcove $C_+$. We write $\mathcal{O}_\cc$ for the $W$-orbit in $E$ containing $\cc\in\overline{C}_+$.
 
The face decomposition of $\overline{C}_+$ is the disjoint union
\[
\overline{C}_+=\bigsqcup_{J\subsetneq \{0,\ldots,r\}}C^J
\]
with $C^J$ 
the vectors $\cc$ in $\overline{C}_+$ satisfying $\alpha_j(\cc)=0$ if and only if $j\in J$. 

Let $\mathbf{F}$ be a field of characteristic zero and fix $q\in\mathbf{F}^\times$ not a root of unity. Let $T:=\textup{Hom}(Q^\vee,\mathbf{F}^\times)$ be the $\mathbf{F}$-torus consisting of multiplicative characters $\mathfrak{t}: Q^\vee\rightarrow\mathbf{F}^\times$.  We will denote the value of $\mathfrak{t}$ at $\mu\in Q^\vee$ by $\mathfrak{t}^\mu$. Consider the associated affine subtorus
\[
T_J:=\{\mathfrak{t}\in T \,\, | \,\, \mathfrak{t}^{\alpha_j^\vee}=1 \,\,\, \forall\, j\in J\}
\]
of dimension $r-\#J$,  where $\mathfrak{t}^{\mu+\ell K}:=q^\ell\mathfrak{t}^\mu$. The field $\mathbf{F}$ should contain a sufficiently large root of $q$ to ensure that $T_J\not=\emptyset$ for subsets $J$ containing $0$. 
 
 Denote by 
\[
\mathbf{F}[E]=\bigoplus_{y\in E}\mathbf{F}x^y
\] 
the group algebra of $E$, viewed as additive group. It splits up as the direct sum of the $\mathcal{P}$-submodules 
$\mathcal{P}^{(\cc)}:=\bigoplus_{y\in\mathcal{O}_\cc}\mathbf{F}x^y$ ($\cc\in\overline{C}_+$).

For $\cc\in C^J$ the space $\mathcal{P}^{(\cc)}$ of quasi-polynomials admits a natural family 
of $\mathbb{W}$-actions $\cdot_{\mathfrak{t}}$ depending on $q$ and $\mathfrak{t}\in T_J$.
The action of $W_0\ltimes\widehat{Q}^\vee$ is by reflections and translations of the exponents of the quasi-monomials $x^y$, with $m^2K$ acting as multiplication by $q^{m^2}$.
In particular, this part of the action is independent of $\mathfrak{t}$. 
The action of the standard abelian subgroup $\{\tau(\mu)\}_{\mu\in Q^\vee}$ in $W=W_0\ltimes Q^\vee$ is defined by
\begin{equation} \label{Xrep} 
\tau(\mu)_{\mathfrak{t}}x^y:=\mathfrak{t}_y^{-\mu}x^y\qquad\quad (\mu\in Q^\vee,\, y\in\mathcal{O}_\cc),
\end{equation}
where $\mathfrak{t}_y\in T$ is the multiplicative character $\mu\mapsto\mathfrak{t}^{w_y^{-1}\cdot\mu}$ with
$w_y\in W$ the element of shortest length such that $w_y\cc=y$ (actually, one may take here any $w\in W$ satisfying $w\cc=y$). Note that
$s_{j,\mathfrak{t}}x^y:=(s_j)_{\mathfrak{t}}x^y$ for $y\in\mathcal{O}_\cc$ is explicitly given by
\[
s_{i,\mathfrak{t}}x^y=x^{s_iy}\quad (1\leq i\leq r),\qquad\quad s_{0,\mathfrak{t}}x^y=\mathfrak{t}_y^{\varphi^\vee}x^{s_\varphi y}
\]
with $s_\varphi\in W_0$ the reflection associated to the highest root $\varphi\in\Phi_0$.

In case $J=\{1,\ldots,r\}$, $\cc=0$ and $\mathfrak{t}=1_T$ we reobtain the standard action of $\mathbb{W}$ on $\mathcal{P}$ by polynomial $q$-difference reflection operators, since
\[
\tau(\mu)_{1_T}(x^\nu)=q^{-\langle\mu,\nu\rangle}x^\nu\qquad\quad (\mu,\nu\in Q^\vee).
\]
In this case we write $w(x^\mu)$ for $w_{1_T}x^\mu$ ($w\in W$).

In this paper we deform the $\mathbb{W}$-action $\cdot_{\mathfrak{t}}$ on $\mathcal{P}^{(\cc)}$ to an action of the double affine Hecke algebra $\mathbb{H}$. The resulting quasi-polynomial representation of $\mathbb{H}$ plays a central role in this paper.

\subsection{The quasi-polynomial representation}\label{DAHAInstroSub}
Cherednik's \cite{Ch} double affine Hecke algebra (DAHA) is a certain flat deformation $\mathbb{H}=\mathbb{H}(\mathbf{k},q)$ of $\mathbf{F}[\mathbb{W}]/(m^2K-q^{m^2})\simeq W\ltimes\mathcal{P}$ depending on a multiplicity function $\mathbf{k}=(k_a)_{a\in\Phi}$.
We consider in this paper the case that $\mathbf{k}$ is invariant for the extended affine Weyl group. 
This means that $k_a\in\mathbf{F}^\times$ only depends on the length $\|\alpha\|$ of the gradient $\alpha\in\Phi_0$ of the affine root $a=(\alpha,\ell)\in\Phi$. In particular, there are at most two distinct values of the $k_a$. We write $k_j=k_{\alpha_j}$. 

The double affine Hecke algebra $\mathbb{H}$ has an explicit definition in terms of generators and relations (see Definition \ref{defDAHA}). The generators are $T_j$ ($0\leq j\leq r$) and $x^\mu$ ($\mu\in Q^\vee$).
The subalgebra generated by $T_j$ ($0\leq j\leq r$) is the affine Hecke algebra $H=H(\mathbf{k})$ of $W$. Its defining relations are the $(W,\{s_0,\ldots,s_r\})$-braid relations and the Hecke relations $(T_j-k_j)(T_j+k_j^{-1})=0$. The subalgebra generated by $x^\mu$ ($\mu\in Q^\vee$) is $\mathcal{P}$. The commutation relations between the $T_j$ and $x^\mu$ are $\mathbf{k}$-deformations of $s_jx^\mu=s_{j}(x^\mu)s_j$ in $W\ltimes\mathcal{P}$, given explicitly by 
\begin{equation}\label{crossintro}
T_jx^\mu-s_{j}(x^\mu)T_j=(k_j-k_j^{-1})\Bigl(\frac{x^\mu-s_{j}(x^\mu)}{1-x^{\alpha_j^\vee}}\Bigr).
\end{equation}
Note here that $x^{\alpha_0^\vee}=x^{m^2K^2-\varphi^\vee}=q^{m^2}x^{-\varphi^\vee}$.
The right hand side can be alternatively written as
\[
(k_j-k_j^{-1})\Bigl(\frac{1-x^{-D\alpha_j(\mu)\alpha_j^\vee}}{1-x^{\alpha_j^\vee}}\Bigr)x^\mu
\]
with $D\alpha_j\in\Phi_0$ the gradient of $\alpha_j$. In particular \eqref{crossintro} makes sense in $\mathbb{H}$, since $D\alpha_j(\mu)\in\mathbb{Z}$.
Together with the well known fact that $H\otimes\mathcal{P}\simeq\mathbb{H}$ as vector spaces by multiplication (Cherednik's \cite{Ch} PBW theorem for $\mathbb{H}$), this gives an explicit generalisation of the first isomorphism in \eqref{PBWdaWg}.

Also the decomposition $W\simeq W_0\ltimes Q^\vee$ has a natural analog in $H$, known as the Bernstein-Zelevinsky decomposition of $H$ (see \cite{Lu}). It provides a subalgebra $\mathcal{P}_Y=\bigoplus_{\mu\in Q^\vee}\mathbf{F}Y^\mu$ of $H$, isomorphic to $\mathcal{P}$, such that $H_0\otimes\mathcal{P}_Y\simeq H$ as vector spaces, with $H_0$ the subalgebra of $\mathbb{H}$ generated by $T_i$ ($1\leq i\leq r$). 
The {\it duality anti-involution} $\delta$ of $\mathbb{H}$ interchanges the two copies of $\mathcal{P}$ inside $\mathbb{H}$. Concretely, it satisfies
$\delta(Y^\mu)=x^{-\mu}$ and $\delta(T_i)=T_i$ for $\mu\in Q^\vee$ and $1\leq i\leq r$.

Let $\chi_{B}: \mathbb{R}\rightarrow\{0,1\}$ for a subset $B\subset\mathbb{R}$ be the characteristic function of $B$,
and let $\lfloor\cdot\rfloor: \mathbb{R}\rightarrow\mathbb{Z}$ be the floor-function, i.e., $\lfloor d\rfloor$ for $d\in\mathbb{R}$ is the largest integer $\leq d$.
For $0\leq j\leq r$ the truncated divided-difference operator $\nabla_j$ is the linear operator on $\mathbf{F}[E]$ defined by
\[
\nabla_j(x^y):=\Bigl(\frac{1-x^{-\lfloor D\alpha_j(y)\rfloor\alpha_j^\vee}}{1-x^{\alpha_j^\vee}}\Bigr)x^y\qquad\quad (y\in E).
\]
Its restriction to $\mathcal{P}$ is the usual divided-difference operator.

The first main result of this paper is the following DAHA analog of the $\mathbb{W}$-module $(\mathcal{P}^{(\cc)},\cdot_\mathfrak{t})$ ($\cc\in C^J$, $\mathfrak{t}\in T_J$), in which $T_j$ acts by a Demazure-Lusztig type operator involving the truncated divided-difference operator $\nabla_j$. We call $\pi_{\cc,\mathfrak{t}}$ the {\it quasi-polynomial representation} of $\mathbb{H}$.

\begin{theorem}\label{introTHM1}
Let $J\subsetneq\{0,\ldots,r\}$, $\cc\in C^J$ and $\mathfrak{t}\in T_J$. The formulas
\begin{equation*}
\begin{split}
\pi_{\cc,\mathfrak{t}}(T_j)x^y&:=k_j^{\chi_{\mathbb{Z}}(\alpha_j(y))}s_{j,\mathfrak{t}}x^y+(k_j-k_j^{-1})
\nabla_j(x^y),\\
\pi_{\cc,\mathfrak{t}}(x^\mu)x^y&:=x^{y+\mu}
\end{split}
\end{equation*}
for $0\leq j\leq r$, $\mu\in Q^\vee$ and $y\in\mathcal{O}_\cc$ define a representation $\pi_{\cc,\mathfrak{t}}: \mathbb{H}\rightarrow\textup{End}(\mathcal{P}^{(\cc)})$.
\end{theorem}

The representation $\pi:=\pi_{0,1_T}: \mathbb{H}\rightarrow\textup{End}(\mathcal{P})$ is Cherednik's \cite{Ch} polynomial representation of $\mathbb{H}$. In this case $J=\{1,\ldots,r\}$, and the commuting operators $\pi(Y^\mu)$ ($\mu\in Q^\vee$) are the Cherednik operators. The polynomial representation $\pi$ can alternatively be 
obtained by inducing the trivial one-dimensional $H$-representation to $\mathbb{H}$. As a second main result of this paper we extend this result to the quasi-polynomial representation $\pi_{\cc,\mathfrak{t}}$ ($\cc\in C^J$) by inducing a one-dimensional representation of the $Y$-parabolic subalgebra $\mathbb{H}_J^Y$ of $\mathbb{H}$ generated by $\delta(T_j)$ ($j\in J$) and $\mathcal{P}_Y$. 

The relevant one-dimensional $\mathbb{H}_J^Y$-representations are parametrised by the affine subtorus 
\[
L_J:=\{\mathfrak{t}\in T \,\, | \,\, \mathfrak{t}^{\alpha_j^\vee}=k_j^{-2}\quad \forall\, j\in J\}
\]
of $T$. The one-dimensional $\mathbb{H}_J^Y$-representation associated to $t\in L_J$ is then defined by
\[
\chi_{J,t}(Y^\mu):=t^{-\mu}\quad (\mu\in Q^\vee),\qquad \chi_{J,t}(\delta(T_j)):=k_j\quad (j\in J).
\]
We write $\mathbb{M}_{J,t}$
for the resulting induced $\mathbb{H}$-module, and $m_{J,t}$ for its canonical cyclic vector.

Consider
\[
\mathfrak{s}_y:=\prod_{\alpha\in\Phi_0^+}k_\alpha^{\eta(\alpha(y))\alpha}\in T\qquad (y\in E),
\]
with $\eta:=\chi_{\mathbb{Z}_{>0}}-\chi_{\mathbb{Z}_{\leq 0}}: \mathbb{R}\rightarrow\{-1,0,1\}$ (a discrete analog of the Heaviside function). 
The map $y\mapsto\mathfrak{s}_y$ is constant on faces, and we write $\mathfrak{s}_J\in T$ for its value on $C^J$. We then have $L_J=\mathfrak{s}_JT_J$ and
\begin{theorem}\label{introTHM2}
Let $\cc\in C^J$ and $\mathfrak{t}\in T_J$. Then $\mathbb{M}_{J,\mathfrak{s}_J\mathfrak{t}}\simeq (\mathcal{P}^{(\cc)},\pi_{\cc,\mathfrak{t}})$ as $\mathbb{H}$-modules, with the isomorphism determined by $m_{J,\mathfrak{s}_J\mathfrak{t}}\mapsto x^\cc$.
\end{theorem}
In particular, $\pi_{\cc,\mathfrak{t}}\simeq\pi_{\cc^\prime,\mathfrak{t}}$ if $\cc$ and $\cc^\prime$ lie in the same face of $\overline{C}_+$.

Quite remarkably, for all $\mu\in Q^\vee$ and $v\in W_0$ the vector $x^\mu T_vm_{J,\mathfrak{s}_J\mathfrak{t}}\in\mathbb{M}_{J,\mathfrak{s}_J\mathfrak{t}}$ is mapped to
an explicit constant multiple of the quasi-monomial $x^{\mu+v\cc}$ (see Theorem \ref{gbr}(3)). Actually, it is this formula which forms the starting point of the proof of both Theorem \ref{introTHM1} and Theorem \ref{introTHM2}. 

\subsection{Quasi-polynomial eigenfunctions}
The quasi-polynomial representations $\pi_{\cc,\mathfrak{t}}$ give rise to a family of quasi-polynomials, which generalize the nonsymmetric Macdonald polynomials.  These quasi-polynomials enjoy many properties paralleling the classical case, for example they are joint eigenfunctions of $\pi_{\cc,\mathfrak{t}}(Y^{\mu})$ for $\mu \in Q^{\vee}$, satisfy a triangularity condition, and can be constructed explicitly via intertwiners. 
Their definition is as follows. 

Let $\cc\in C^J$. Define a partial order on 
$\mathcal{O}_\cc$ by
\[
y\leq y^\prime\quad \Leftrightarrow \quad 
w_y\leq_B w_{y^\prime},
\]
with $\leq_B$ the Bruhat order on $(W,\{s_0,\ldots,s_r\})$. Denote by $T_J^\prime\subseteq T_J$ the set of characters $\mathfrak{t}\in T_J$ for which the $\mathfrak{s}_y\mathfrak{t}_y$ \textup{(}$y\in\mathcal{O}_\cc$\textup{)} are pairwise different in $T$ (it does not depend on the choice of $\cc\in C^J$).

\begin{theorem}\label{introTHM3}
Let $\cc\in C^J$ and $\mathfrak{t}\in T_J^\prime$.
For each $y\in\mathcal{O}_\cc$ there exists a unique joint eigenfunction $E_y^J(x;\mathfrak{t})\in\mathcal{P}^{(\cc)}$ of the commuting operators $\pi_{\cc,\mathfrak{t}}(Y^\mu)$
\textup{(}$\mu\in Q^\vee$\textup{)} satisfying
\begin{equation}\label{expansionIntro}
E_y^J(x;\mathfrak{t})=x^y+\sum_{y^\prime<y}e_{y,y^\prime;\mathfrak{t}}^Jx^{y^\prime}\qquad\quad (e_{y,y^\prime;\mathfrak{t}}^J\in\mathbf{F}).
\end{equation}
Furthermore, 
\begin{equation}\label{evIntro}
\pi_{\cc,\mathfrak{t}}(Y^\mu)E_y^J(x;\mathfrak{t})=(\mathfrak{s}_y\mathfrak{t}_y)^{-\mu}E_y^J(x;\mathfrak{t})\qquad\quad \forall\, \mu\in Q^\vee.
\end{equation}
\end{theorem}
In particular,
$\{E_y^J(x;\mathfrak{t})\}_{y\in\mathcal{O}_\cc}$ is a basis of $\mathcal{P}^{(\cc)}$ consisting of joint eigenfunctions of the commuting operators
$\pi_{\cc,\mathfrak{t}}(Y^\mu)$ ($\mu\in Q^\vee$). Furthermore, the quasi-polynomial eigenfunction of smallest degree is trivial, 
\[
E_\cc^J(x;\mathfrak{t})=x^\cc.
\] 

The choice of the vector $\cc\in C^J$ only influences the quasi-exponents in the quasi-monomial expansion of the $E_y^J(x;\mathfrak{t})$. In other words, the coefficients
\[
e_{w,w^\prime;\mathfrak{t}}^J:=e_{w\cc,w^\prime\cc;\mathfrak{t}}^J\qquad\quad (w,w^\prime\in W)
\]
are independent of $\cc\in C^J$ (see Theorem \ref{Edef}). Furthermore, if $J^\prime\supseteq J$ 
then $C^{J^\prime}$ lies in the closure of $C^J$ and the quasi-polynomial $E^{J^\prime}_{y^\prime}(x;\mathfrak{t})$ for 
$\mathfrak{t}\in T_{J^\prime}^\prime$ can be recovered from $E^J_y(x;\mathfrak{s}_J^{-1}\mathfrak{s}_{J^\prime}\mathfrak{t})$ (see Proposition \ref{projectionprop}). 

\begin{example}\hfill
Consider $\mathbf{F}=\mathbb{Q}(q,\mathbf{k},z_1,\ldots,z_r)$ 
with indeterminates $q,k_a,z_1,\ldots,z_r$, where $k_a=k_b$ if $\|Da\|=\|Db\|$.
For $I\subseteq \{1,\ldots,r\}$ let $I^{\textup{co}}$ be its complement.
Then 
\[
\prod_{i\in I^{\textup{co}}}z_i^{\varpi_i}\in T_I^\prime,
\]
where $\varpi_i$ denotes the $i^{\textup{th}}$ fundamental weight of $\Phi_0$. Then
\[
E_y^I(x;z_1,\ldots,z_r):=E_y^I\bigl(x;\prod_{i \in I^{\textup{co}}}z_i^{\varpi_i}\bigr)\qquad\quad (y\in\mathcal{O}_\cc)
\]
depend on the indeterminates $q, \mathbf{k}$ and $z_i$ \textup{(}$i\in I^{\textup{co}}$\textup{)}.
Quasi-polynomial eigenfunctions over $\mathbb{Q}(q,\mathbf{k})$ can be obtained by specialisation of the $z_i$'s. 

The situation is a bit more intricate when $0\in J$ since the $z_i$'s will no longer be independent, see Subsection \ref{genSection}.
\end{example}

If $1_T\in T_{\{1,\ldots,r\}}^\prime$ \textup{(}which imposes generic conditions on $q$ and $\mathbf{k}$\textup{)}, then 
\[
E_\mu(x):=E_\mu^{\{1,\ldots,r\}}(x;1_T)\in\mathcal{P}
\]
 is the monic nonsymmetric Macdonald polynomial of degree $\mu\in Q^\vee$.
The Macdonald polynomials and their normalised and (anti)symmetric variants have remarkable properties. They satisfy orthogonality relations with explicit quadratic norm formulas
\cite{Ma0,Cm,Ch,Ma}, have explicit evaluation formulas, and are self-dual (see, e.g., \cite{Ma,Ch} and references therein). They admit combinatorial formulas \cite{RY,HHL}, arise as spherical functions on quantum groups 
(see, e.g., \cite{N,L,EK}), admit interpretations as
partition functions on integrable lattice models \cite{BW}, and degenerate to Heckman-Opdam polynomials \cite{Mac}, spherical functions on p-adic reductive groups \cite{Mac}, and Iwahori-Whittaker functions \cite{I0,ChW1}. In this paper we discuss generalisations of a number of these properties to the quasi-polynomial setup:
\begin{enumerate}
\item
We define a normalized version $P_y^J(x;\mathfrak{t})$ of $E_y^J(x;\mathfrak{t})$ using the action of normalised $Y$-intertwiners on the lowest degree quasi-polynomial eigenfunction $x^\cc$. It leads to an explicit expression of $E_y^J(x;\mathfrak{t})$ in terms of $Y$-intertwiners (Theorem \ref{prop:I_and_E}). This entails a quasi-polynomial analog of the evaluation formula for Macdonald polynomials. 
\item We obtain a weak form of duality for $P_y^J(x;\mathfrak{t})$, called pseudo-duality, by showing that the $H$-action
on $P_y^J(x;\mathfrak{t})$ can be alternatively described in terms of an $H$-action on $y\in\mathcal{O}_\cc$ by
discrete Demazure-Lusztig operators (Theorem \ref{17}).
\item We establish orthogonality relations and quadratic norm formulas for the quasi-polynomial eigenfunctions $E_y^J(x;\mathfrak{t})$ and $P_y^J(x;\mathfrak{t})$ (Theorem \ref{unitaritytheorem}).
\item We construct (anti)symmetric variants $E_y^{J,\pm}(x;\mathfrak{t})$ of the quasi-polynomial eigenfunctions $E_y^J(x;\mathfrak{t})$ by (anti)symme\-tri\-sa\-tion with respect to the trivial (resp., sign) idempotent of $H_0$. They generalize the (anti)symmetric Macdonald polynomials.  The symmetric variants $E_{y}^{J, +}(x; \mathfrak{t})$ are invariant under the $W_0$-action obtained by localizing $\pi_{c, \mathfrak{t}}$ (see Theorem \ref{aCG} and Lemma \ref{symm_space}).  
We obtain explicit expansion formulas of $E_y^{J,\pm}(x;\mathfrak{t})$ in terms of the quasi-polynomials $\{E_{y^\prime}^J(x;\mathfrak{t})\}_{y^\prime\in W_0y}$ (Corollary \ref{corplus} and Proposition \ref{Wformulaminus}).
\item We establish the existence of the Whittaker limit $q\rightarrow\infty$ of $E_y^J(x;\mathfrak{t})$ and $E_y^{J,-}(x;\mathfrak{t})$ (Proposition \ref{prop:qlimit} and
Corollary \ref{linkmWcor}). For $y$ in the closure of the negative Weyl chamber and $y^\prime\in W_0y$, we obtain an explicit formula for the limit $q \rightarrow \infty$ of
$E_{y^\prime}^J(x;\mathfrak{t})$ in terms of the
$\pi_{\cc,\mathfrak{t}}\vert_{H_0}$-action on $x^y$ (Theorem \ref{qlim_formula}). It generalises a formula due to Patnaik-Puskas \cite[Cor. 5.4]{PP}, which expresses 
the metaplectic Iwahori-Whittaker function in terms of metaplectic Demazure-Lusztig operators. We will say more about the link with metaplectic representation theory in Subsection \ref{mintro}.
\end{enumerate}

We also consider the theory for reductive root data, when both
$\mathcal{P}$ and $\mathcal{P}_Y$ in $\mathbb{H}$ are extended to group algebras of a finitely generated abelian subgroup $\Lambda\subset E$ satisfying 
\begin{equation}\label{lcondIntro}
Q^\vee\subseteq\Lambda\,\, \hbox{ and }\,\, 
\alpha(\Lambda)\subseteq\mathbb{Z}\,\,\,\, \forall\alpha\in\Phi_0.
\end{equation}
When only $\mathcal{P}_Y$ is extended to the group algebra of such a lattice $\Lambda^\prime$, then the associated $Y$-extended double affine Hecke algebra $\mathbb{H}_{Q^\vee,\Lambda^\prime}$ is a smashed product algebra $\Omega_{\Lambda^\prime}\ltimes\mathbb{H}$, with $\Omega_{\Lambda^\prime}\simeq\Lambda^\prime/Q^\vee$ the subgroup of the extended affine Weyl group $W_{\Lambda^\prime}:=W_0\ltimes\Lambda^\prime$ consisting of elements of length zero. The associated extended quasi-polynomial representation $\pi_{\cc,\mathfrak{t}}^{Q^\vee,\Lambda^\prime}:\mathbb{H}_{Q^\vee,\Lambda^\prime}\rightarrow\textup{End}(\mathcal{P}^{(\cc)})$ now depends on parameters $\mathfrak{t}\in T_{\Lambda^\prime,J}$, where $T_{\Lambda^\prime,J}$ is the affine subtorus of $T_{\Lambda^\prime}=\textup{Hom}(\Lambda^\prime,\mathbf{F}^\times)$
consisting of the elements $\mathfrak{t}\in T_{\Lambda^\prime}$ satisfying $\mathfrak{t}^{\alpha_j^\vee}=1$ for all $j\in J$. 

The extra representation parameters in $\pi_{\cc,\mathfrak{t}}^{Q^\vee,\Lambda^\prime}$ compared to $\pi_{\cc,\mathfrak{t}\vert_{Q^\vee}}$ are captured by the kernel $T_{\Lambda^\prime,\{1,\ldots,r\}}$ of the restriction map $T_{\Lambda^\prime}\rightarrow T$. It identifies with the character group of $\Omega_{\Lambda^\prime}$, and naturally embeds into the automorphism group of $\mathbb{H}_{Q^\vee,\Lambda^\prime}\simeq\Omega_{\Lambda^\prime}\ltimes\mathbb{H}$
fixing $\mathbb{H}$ point-wise. We call these extra parameters {\it twist parameters}, because $\pi_{\cc,\mathfrak{t}\mathfrak{t}^\prime}^{Q^\vee,\Lambda^\prime}$ with $\mathfrak{t}^\prime\in T_{\Lambda^\prime,\{1,\ldots,r\}}$ coincides with $\pi_{\cc,\mathfrak{t}}^{Q^\vee,\Lambda^\prime}$ twisted by the automorphism of 
$\mathbb{H}_{Q^\vee,\Lambda^\prime}$ associated to $\mathfrak{t}^\prime$.
Since these automorphisms rescale the $Y^\mu$ ($\mu\in Q^\vee$), the quasi-polynomials $E_y^J(x;\mathfrak{t}\vert_{Q^\vee})$ ($y\in\mathcal{O}_c$) are eigenfunctions of all the commuting operators $\pi_{\cc,\mathfrak{t}}^{Q^\vee,\Lambda^\prime}(Y^\mu)$ ($\mu\in\Lambda^\prime$) under suitable generic conditions on the parameters (Theorem \ref{propEVext0}). 

The full double affine Hecke algebra $\mathbb{H}_{\Lambda,\Lambda^\prime}$ depends on two lattices $\Lambda,\Lambda^\prime$.
The extension of the quasi-polynomial representation $\pi_{\cc,\mathfrak{t}}^{Q^\vee,\Lambda^\prime}$ to $\mathbb{H}_{\Lambda,\Lambda^\prime}$ requires adding $(\cc,\Lambda)$-dependent constraints on $\mathfrak{t}\in T_{\Lambda^\prime,J}$, as well as enlarging the representation space. It will be discussed in detail in Section \ref{Extsection}.
In case of the $\textup{GL}_{r+1}$ root datum the quasi-polynomial eigenfunctions give rise to the metaplectic polynomials introduced in our previous work
\cite[\S 5.4]{SSV}. We will say more about the root system generalisations of these metaplectic polynomials in Subsection \ref{mintro}.
 
\subsection{The uniform quasi-polynomial representation}
The quasi-polynomial representations defined in Theorem \ref{introTHM1} can be naturally combined into a family of $\mathbb{H}$-representations on $\mathbf{F}[E]$. The parametrisation of this family of uniform quasi-polynomial representations
 requires that $q\in\mathbf{F}^\times$ has a $(2h)^{th}$ root $q^{\frac{1}{2h}}$ in $\mathbf{F}$, where $h$ is the Coxeter number of $W_0$.

To motivate the parametrisation it is instructive to consider an important special case first, which requires the stronger assumption that $q\in\mathbf{F}^\times$ is part of an injective group homomorphism 
\begin{equation}\label{1par}
\mathbb{R}\rightarrow\mathbf{F}^\times,\qquad d\mapsto q^d.
\end{equation}
Then
\begin{equation}\label{uqpbasic}
\bigoplus_{\cc\in\overline{C}_+}\pi_{\cc,q^\cc}: \mathbb{H}\rightarrow\textup{End}(\mathbf{F}[E])
\end{equation}
is the prototypical example of a uniform quasi-polynomial representation, where $q^y\in T$ \textup{(}$y\in E$\textup{)} is the character $\mu\mapsto q^{\langle y,\mu\rangle}$ of $Q^\vee$. 
Note that for the underlying $\mathbb{W}$-action on $\mathbf{F}[E]$ \textup{(}given by \eqref{uqpbasic} for $\mathbf{k}\equiv 1$\textup{)}, the translations $\tau(\mu)\in W$ 
are simply acting by $q$-dilations $x^y\mapsto q^{-\langle\mu,y\rangle}x^y$ \textup{(}$y\in E$\textup{)}. 
Furthermore, for $\cc^\prime\in C^J$ and $\mathfrak{t}\in T_J$ we have $\pi_{\cc^\prime,\mathfrak{t}}\simeq \pi_{\cc,q^\cc}$ for some $\cc\in C^J$ if $\mathfrak{t}$ takes values in 
$\{q^d\}_{d\in\mathbb{R}}$. The new set of parameters that we will assign to
the uniform quasi-polynomial representation \eqref{uqpbasic} is
the tuple $\{\widehat{p}_\alpha\}_{\alpha\in\Phi_0}$ with $\widehat{p}_\alpha: \mathbb{R}\rightarrow \mathbf{F}^\times$ given by $\widehat{p}_\alpha(d):=q_\alpha^{d/h}$, where  $q_\alpha:=q^{2/\|\alpha\|^2}$.  
The link with the original representation parameters of the quasi-polynomial representation is through the formula
\[
q^\cc=\prod_{\alpha\in\Phi_0^+} \widehat{p}_{\alpha}(\alpha(c))^{\alpha}
\]
for the character $q^\cc\in T$.

 In general we will consider the group $\mathcal{G}^{\textup{amb}}$
 of tuples $\mathbf{f}=(f_\alpha)_{\alpha\in\Phi_0}$ consisting of $\mathbf{F}^\times$-valued functions $f_\alpha$ on $\mathbb{R}$. Let $\mathcal{G}$ be the subgroup of $\mathcal{G}^{\textup{amb}}$ consisting of the tuples $\mathbf{g}=(g_\alpha)_{\alpha\in\Phi_0}$ satisfying
 \begin{equation}\label{gparinvcond}
 g_{v\alpha}(d+\ell)=g_\alpha(d)\qquad (v\in W_0,\,\ell\in\mathbb{Z}),
 \end{equation}
 $g_\alpha(-d)=g_\alpha(d)^{-1}$ and $g_\alpha(0)=1$. If we replace \eqref{gparinvcond} by the quasi-invariance condition
 \[
 \widehat{g}_{v\alpha}(d+\ell)=q_\alpha^{\ell/h}\widehat{g}_\alpha(d)\qquad (v\in W_0,\,\ell\in\mathbb{Z}),
 \]
 we obtain a $\mathcal{G}$-coset in $\mathcal{G}^{\textup{amb}}$, which we will denote by $\widehat{\mathcal{G}}$.
 
For a tuple $\mathbf{f}=(f_\alpha)_{\alpha\in\Phi_0}\in\mathcal{G}^{\textup{amb}}$  and $y\in E$
we set
\begin{equation}\label{reparTJ}
\mathfrak{t}_y(\mathbf{f}):=\prod_{\alpha\in\Phi_0^+}f_\alpha(\alpha(y))^\alpha\in T.
\end{equation}
Then $\mathfrak{t}_\cc(\widehat{\mathbf{g}})\in T_J$ for $\cc\in C^J$ when $\widehat{\mathbf{g}}\in\widehat{\mathcal{G}}$.
\begin{theorem}\label{introTHM4}
Let $\widehat{\mathbf{p}}\in\widehat{\mathcal{G}}$ and $\mathbf{g}\in\mathcal{G}$.
The formulas 
\begin{equation*}
\begin{split}
\pi_{\mathbf{g},\widehat{\mathbf{p}}}(T_i)x^y&:=k_i^{\chi_{\mathbb{Z}}(\alpha_i(y))}g_{\alpha_i}(\alpha_i(y))^{-1}x^{s_iy}+(k_i-k_i^{-1})\nabla_i(x^y)
\qquad (1\leq i\leq r),\\
\pi_{\mathbf{g},\widehat{\mathbf{p}}}(T_0)x^y&:=k_0^{\chi_{\mathbb{Z}}(\alpha_0(y))}g_{\varphi}(\varphi(y))\mathfrak{t}_y(\widehat{\mathbf{p}})^{\varphi^\vee}x^{s_\varphi y}
+(k_0-k_0^{-1})\nabla_0(x^y)
\end{split}
\end{equation*}
and $\pi_{\mathbf{g},\widehat{\mathbf{p}}}(x^\lambda)x^y:=x^{y+\lambda}$ for $\lambda\in Q^\vee$ define a $\mathbb{H}$-representation
on $\mathbf{F}[E]$, which is isomorphic to
$\bigoplus_{\cc\in\overline{C}_+}\pi_{\cc,\mathfrak{t}_\cc(\widehat{\mathbf{p}}\cdot\mathbf{g})}$.
\end{theorem}

In the main text we prove a version of this theorem for the extended double affine Hecke algebra $\mathbb{H}_{\Lambda,\Lambda^\prime}$, see Theorem \ref{gthm}. In that case the extra twisting parameters are added
by replacing $\mathfrak{t}_y(\widehat{\mathbf{p}})$, viewed as a multiplicative character of $\Lambda^\prime$, by
\[
\mathfrak{t}_y(\widehat{\mathbf{p}},\mathfrak{c}):=\mathfrak{t}_y(\widehat{\mathbf{p}})\mathfrak{c}(\textup{pr}_{E_{\textup{co}}}(y)).
\]  
Here $\textup{pr}_{E_{\textup{co}}}$ is the orthogonal projection of $E$ onto the orthocomplement $E_{\textup{co}}$ of $\mathbb{R}\Phi_0^\vee$ in $E$, and
$\mathfrak{c}: E_{\textup{co}}\rightarrow T_{\Lambda^\prime,[1,r]}$ satisfies $\mathfrak{c}(y+\mu)=q^\mu\mathfrak{c}(y)$ for $\mu\in\textup{pr}_{E_{\textup{co}}}(\Lambda^\prime)$.

The isomorphism $\pi_{\mathbf{g},\widehat{\mathbf{p}}}\simeq \bigoplus_{\cc\in\overline{C}_+}\pi_{\cc,\mathfrak{t}_\cc(\widehat{\mathbf{p}}\cdot\mathbf{g})}$ in (the extended version of) Theorem \ref{introTHM4} will be realised by 
an explicit $\mathbf{g}$-dependent
linear automorphism of $\mathbf{F}[E]$, which is diagonalised by the quasi-monomial basis $\{x^y\}_{y\in E}$ of $\mathbf{F}[E]$.
As a consequence, we obtain simultaneous quasi-polynomial eigenfunctions $\mathcal{E}_y(x;\mathbf{g},\widehat{\mathbf{p}})$ 
of $\pi_{\mathbf{g},\widehat{\mathbf{p}}}(Y^\mu)$ \textup{(}$\mu\in Q^\vee$\textup{)} by appropriately rescaling the quasi-monomials in the quasi-monomial expansion 
\eqref{expansionIntro} of  $E_y^J(x;\mathfrak{t}_\cc(\widehat{\mathbf{p}}\cdot\mathbf{g}))$ (see Corollary \ref{relEuncor}). 

If $L\subset E$ is a $W_0$-invariant, finitely generated abelian subgroup containing $\Lambda$,
then $\mathcal{P}_L:=\bigoplus_{\lambda\in L}\mathbf{F}x^\lambda$ is a 
$\mathbb{H}$-subrepresentation of $\pi_{\mathbf{g},\widehat{\mathbf{p}}}$ containing the quasi-polynomial eigenfunctions $\mathcal{E}_\lambda(x;\mathbf{g},\widehat{\mathbf{p}})$ of degree $\lambda\in L$.  Since $\mathcal{P}_L$ may be viewed as the space of polynomials on the torus $T_L:=\textup{Hom}(L,\mathbf{F}^\times)$, we thus obtain a
polynomial representation of $\mathbb{H}$ with $\pi_{\mathbf{g},\widehat{\mathbf{p}}}(T_j)\vert_{\mathcal{P}_L}$ acting by truncated Demazure-Lusztig operators
(these operators are still truncated since the roots $\alpha\in\Phi_0$ are allowed to take non-integral values on $L$). The decomposition of $L$ into $W$-orbits provides a decomposition of $\pi_{\mathbf{g},\widehat{\mathbf{p}}}(\cdot)\vert_{\mathcal{P}_L}$ as a direct sum of quasi-polynomial representations, with the number of free parameters in $\pi_{\mathbf{g},\widehat{\mathbf{p}}}(\cdot)\vert_{\mathcal{P}_L}$ depending on $L$. This representation plays an important role in establishing the connection to metaplectic representation theory, which we discuss in the next subsection.

\subsection{The metaplectic theory}\label{mintro}

The metaplectic root system $\Phi_0^m$ and the associated metaplectic double affine Hecke algebra $\mathbb{H}^m$ are defined in terms of a metaplectic datum. A metaplectic datum is a pair $(n,\mathbf{Q})$ consisting of a positive integer $n$ and a rational $W_0$-invariant quadratic form $\mathbf{Q}$ on $\Lambda$ taking integral values on $Q^\vee$. The associated symmetric bilinear form on $\Lambda$ will be denoted by $\mathbf{B}$.
The metaplectic root system $\Phi_0^m:=\{\alpha^m\}_{\alpha\in\Phi_0}$ consists of the rescaled roots $\alpha^m=m(\alpha)^{-1}\alpha$ where
\[
m(\alpha):=\frac{n}{\textup{gcd}(n,\mathbf{Q}(\alpha^\vee))}.
\]
This is a root system with basis $\{\alpha_1^m,\ldots,\alpha_r^m\}$, which is isomorphic to either $\Phi_0$ or $\Phi_0^\vee$. We denote by $\vartheta\in\Phi_0^+$ the root such that
$\vartheta^m$ is the long root of $\Phi_0^m$. The simple affine root for the associated metaplectic affine root system $\Phi^m$ is $\alpha_0^m=(-\vartheta^m,1)$.
The co-root lattice $Q^{m\vee}$ of $\Phi_0^m$ is contained in $\Lambda$, and $\alpha^m(\Lambda)\subseteq m(\alpha)^{-1}\mathbb{Z}$ for $\alpha\in\Phi_0$.
The {\it metaplectic double affine Hecke algebra} $\mathbb{H}^m$ is the double affine Hecke algebra relative to the metaplectic root system $\Phi_0^m$. It depends on a choice 
$\mathbf{k}=(k_{a^m})_{a^m\in\Phi^m}$ of a multiplicity function of $\Phi^m$.

To connect the uniform quasi-polynomial representation and the associated quasi-polynomial eigenfunctions with the metaplectic representation theory and the associated metaplectic polynomials from \cite{SSV}, we consider the polynomial representation
$\pi_{\mathbf{g},\widehat{\mathbf{p}}}(\cdot)\vert_{\mathcal{P}_L}$ with $\Phi_0$ and $\mathbb{H}$ replaced by their metaplectic variants $\Phi_0^m$ and $\mathbb{H}^m$, with $L$ a lattice $\Lambda$ satisfying \eqref{lcondIntro} relative to the original, nonmetaplectic root system $\Phi_0$, and with base-point $\widehat{\mathbf{p}}$ satisfying
$\mathfrak{t}_\lambda(\widehat{\mathbf{p}})=q^\lambda$ for all $\lambda\in L$. The remaining parameter dependencies are captured by
the set $\mathcal{M}_{(n,\mathbf{Q})}$ of tuples $\underline{h}=(h_s(\alpha))_{s\in\mathbf{Q}(\alpha^\vee)\mathbb{Z},\alpha\in\Phi_0}$
consisting of $W_0$-invariant functions $h_s: \Phi_0\rightarrow\mathbf{F}^\times$ ($s\in\mathbf{Q}(\alpha^\vee)\mathbb{Z}$) satisfying the additional requirements that
$s\mapsto h_s(\alpha)$ is a $\textup{lcm}(n,\mathbf{Q}(\alpha^\vee))$-periodic function, 
$h_0(\alpha)=-1$, and
\[ h_s(\alpha)h_{-s}(\alpha)=k_{\alpha^m}^{-2}\qquad\qquad \forall\, s\in\mathbf{Q}(\alpha^\vee)\mathbb{Z}\setminus\textup{lcm}(n,\mathbf{Q}(\alpha^\vee))\mathbb{Z}.
\]
The explicit relation with the parameter set $\mathcal{G}$ for the uniform quasi-polynomial representation is described in Lemma \ref{metgpar}. 

This specialisation leads to the following metaplectic basic representation of $\mathbb{H}^m$. 
Let $r_\ell(s)\in\{0,\ldots,\ell-1\}$ be the remainder of $s\in\mathbb{Z}$ modulo $\ell\in\mathbb{Z}_{>0}$, and write $k_j=k_{\alpha_j^m}$ for $0\leq j\leq r$.
\begin{theorem}\label{introTHM5}
Let $\underline{h}\in\mathcal{M}_{(n,\mathbf{Q})}$. The formulas
\begin{equation*}
\begin{split}
\pi_\Lambda^m(T_0)x^\lambda&=-k_0h_{\mathbf{B}(\lambda,\vartheta^\vee)}(\vartheta)q_\vartheta^{m(\vartheta)\vartheta(\lambda)}x^{s_{\vartheta}\lambda}\\
&+
(k_0-k_0^{-1})\Bigl(\frac{1-(q_\vartheta^{m(\vartheta)}x^{-\vartheta^\vee})^{(r_{m(\vartheta)}(-\vartheta(\lambda))+\vartheta(\lambda))}}{1-(q_\vartheta^{m(\vartheta)}x^{-\vartheta^\vee})^{m(\vartheta)}}\Bigr)x^\lambda,\\
\pi_\Lambda^m(T_i)x^\lambda&=-k_ih_{-\mathbf{B}(\lambda,\alpha_i^\vee)}(\alpha_i)x^{s_i\lambda}+(k_i-k_i^{-1})\Bigl(\frac{1-x^{(r_{m(\alpha_i)}(\alpha_i(\lambda))-\alpha_i(\lambda))\alpha_i^\vee}}{1-x^{m(\alpha_i)\alpha_i^\vee}}\Bigr)x^\lambda,\\
\pi_\Lambda^m(x^\mu)x^\lambda&=x^{\lambda+\mu}
\end{split}
\end{equation*}
for $1\leq i\leq r$, $\mu\in Q^{m\vee}$ and $\lambda\in\Lambda$ define a representation $\pi_\Lambda^m: \mathbb{H}^m\rightarrow\textup{End}(\mathcal{P}_\Lambda)$.
\end{theorem}
The metaplectic basic representation $\pi_\Lambda^m$ naturally extends to the metaplectic extended double affine Hecke algebra, see Theorem \ref{met_basic_rep}. 

At this stage we can make the link with the (announced) results in \cite{SSV}.  The representation $\pi_\Lambda^m$, restricted to the metaplectic affine Hecke algebra generated by $\mathcal{P}$ and $T_1,\ldots,T_r$, is the metaplectic affine Hecke algebra representation from \cite[Thm. 3.7]{SSV}. When the multiplicity functions $\mathbf{k}$ and $h_s$ are constant it recovers, after localisation, Chinta's and Gunnells' \cite{CG07, CG} $W_0$-action on a suitable subspace of rational functions on $T_\Lambda$, which plays an important role in the construction of the local parts of Weyl group multiple Dirichlet series. By $X$-localisation of the double affine Hecke algebra, Theorem \ref{introTHM5} gives rise to families of $W$-actions which generalise the Chinta-Gunnells action, see Corollary \ref{WhitCGcor} (compare also with \cite[\S 3.3]{PP2} and \cite{LZ} where the Chinta-Gunnells action is extended to an arbitrary Coxeter group). 
These $W$-actions actually extend to actions of the double affine Weyl group $\mathbb{W}$, with the second coroot lattice $Q^\vee$ acting by translations of the exponents of the monomials. 
The localisation procedure can also be applied to the (uniform) quasi-polynomial representations $\pi_{\cc,\mathfrak{t}}$ and $\pi_{\mathbf{g},\widehat{\mathbf{p}}}$, 
see Theorem \ref{aCG} and Theorem \ref{CGaffineextended} for the resulting $\mathbb{W}$-actions. 

Rescaling the quasi-monomials in the expansion \eqref{expansionIntro} gives rise to simultaneous eigenfunctions $E_\lambda^m(x)\in\mathcal{P}_\Lambda$ of the commuting operators
$\pi_{\Lambda}^m(Y^\mu)$ ($\mu\in Q^{m\vee}$) satisfying 
\[
E_\lambda^m(x)=x^\lambda+\sum_{\lambda^\prime<\lambda}e_{\lambda,\lambda^\prime}^m x^{\lambda^\prime}\qquad (e_{\lambda,\lambda^\prime}^m\in\mathbf{F})
\]
under suitable generic conditions on the metaplectic parameters $\underline{h}\in\mathcal{M}_{(n,\mathbf{Q})}$, 
see Theorem \ref{mptheorem}. An antisymmetric version of the metaplectic polynomial $E_\lambda^m(x)$, denoted $E_{\lambda}^{m, -}(x)$, is obtained by acting by the sign-idempotent of the finite Hecke algebra in $\mathbb{H}^m$ on 
$E_\lambda^m(x)$. For the root datum of type $\textup{GL}_{r+1}$, the representation $\pi_\Lambda^m$ of the extended metaplectic double affine Hecke algebra recovers the representation introduced in \cite[Thm. 5.4]{SSV}, and the $E_\lambda^m(x)$ are the metaplectic polynomials introduced in \cite[Thm. 5.7]{SSV}.

We consider the $q\rightarrow\infty$ limit of the metaplectic polynomials, for which we need to assume that $\mathbf{F}=\mathbf{K}(q^{\frac{1}{2h}})$ and that the remaining parameters $\mathbf{k}$ and $\underline{h}$ take values in the multiplicative group $\mathbf{K}^\times$ of the field $\mathbf{K}$ of characteristic zero. The Whittaker limit
\[
\overline{E}_\lambda^m(x):=E_\lambda^m(x)\vert_{q^{-\frac{1}{2h}}=0}\qquad (\lambda\in\Lambda)
\]
of $E_\lambda^m(x)$ then defines a Laurent polynomial in $\mathcal{P}_\Lambda$ with coefficients in $\mathbf{K}$. They are related to metaplectic Iwahori-Whittaker functions \cite{PP} as follows.

Suppose that $\mathbf{K}$ is a non-archimedean local field containing the $n^{\textup{th}}$ roots of unity, with the cardinality $\upsilon$ of its residue field congruent $1$ module $2n$. In \cite{PP}, Patnaik and Puskas introduced Iwahori-Whittaker functions for metaplectic covers of reductive groups over $\mathbf{K}$. They expressed them in terms of certain metaplectic Demazure-Lusztig operators which depend on certain Gauss sums $\mathbf{g}_s$, as well as on $\upsilon$, see \cite[Cor. 5.4]{PP}. Comparing this with the 
analogous formula for $\overline{E}_{\lambda}^{m}(x)$ in terms of truncated Demazure-Lusztig operators (Theorem \ref{qlim_formula}), we establish in Subsection \ref{MfinalSection} the following link between $\overline{E}_\lambda^m(x)$ and metaplectic Iwahori-Whittaker functions.
Let $I$ the linear automorphism of $\mathcal{P}_\Lambda$ satisfying $I(x^\mu)=x^{-\mu}$ for $\mu\in\Lambda$, let $\overline{E}_\lambda^{m,-}(x)$ be the Whittaker limit of $E_\lambda^{m,-}(x)$, and denote by $\rho^\vee$ the half-sum of positive co-roots.

\begin{theorem}\label{introTHMWhittaker}
Suppose that the multiplicity function $\mathbf{k}$ is constantly equal to $\upsilon^{-2}$ and the multiplicity functions $h_s$ of the metaplectic parameters $\underline{h}$ are constantly equal to $\mathbf{g}_s$.
For $v\in W_0$ and for $\lambda\in\Lambda$ lying in the closure of the positive Weyl chamber,
\[
x^{-\rho^\vee}I\bigl(\overline{E}^m_{-v(\lambda+\rho^\vee)}(x)\bigr)\quad \hbox{ and }\quad x^{-\rho^\vee}I\bigl(\overline{E}^{m,-}_{-\lambda-\rho^\vee}(x)\bigr)
\]
are constant multiples of metaplectic Iwahori-Whittaker functions and metaplectic spherical Whittaker functions, respectively.
\end{theorem}
In Subsection \ref{MfinalSection} we also recover McNamara's \cite{McN} Casselman-Shalika type formula for the metaplectic spherical Whittaker function as a consequence of the explicit expansion formula of $E_y^{J,-}(x;\mathfrak{t})$ as a linear combination of the quasi-polynomial eigenfunctions $E_{y^\prime}^J(x;\mathfrak{t})$ ($y^\prime\in W_0 y$), derived earlier in Subsection \ref{S65}.

 \subsection{Structure of the paper}
In Section \ref{sec:prelim} we collect basic properties of the affine Weyl group $W$ and the double affine Hecke algebra $\mathbb{H}$. 

In Section \ref{modulesection} we introduce the class of $Y$-parabolically induced cyclic $\mathbb{H}$-modules $\mathbb{M}_t^J$.

In Section \ref{polfct} we introduce the quasi-polynomial representation $\pi_{\cc,\mathfrak{t}}$. 
We state the main theorem that the quasi-polynomial representation is well defined and isomorphic to $\mathbb{M}_{\mathfrak{s}_J\mathfrak{t}}^J$ in Subsection \ref{gbrsectionStatement}. We discuss the dependence of $\pi_{\cc,\mathfrak{t}}$ on $\cc\in C^J$ and $J$ in Subsection \ref{FaceSection}, and localise it to obtain a non-trivial $W$-action on quasi-rational functions in Subsection \ref{aWaSection}. 

Section \ref{mainproof} is devoted to the proof of the main theorem.
As part of the proof we establish in Subsection \ref{POsection2} triangularity properties
of the commuting operators $\pi_{\cc,\mathfrak{t}}(Y^\mu)$ ($\mu\in Q^\vee$) with respect to the basis of quasi-monomials, ordered by the parabolic Bruhat order on their exponents. 

In Section \ref{QuasiSection} we define the quasi-polynomial generalisations $E_y^J(x;\mathfrak{t})$ of the monic Macdonald polynomials for adjoint root data. 
We show how they can be explicitly created from the quasi-monomial $x^\cc$ by acting by $Y$-intertwiners in Subsection \ref{CreationSection}.
Their dependence on $J\subsetneq\{0,\ldots,r\}$ is discussed in Subsection \ref{ClosureSection}. 
Pseudo-duality for the normalised versions $P_y^J(x;\mathfrak{t})$ of the quasi-polynomials $E_y^J(x;\mathfrak{t})$ is derived in Subsection \ref{S64},
quasi-polynomial generalisations $E_y^{J,\pm}(x;\mathfrak{t})$ of the (anti)symmetric Macdonald polynomials are discussed in Subsection \ref{S65}, and orthogonality relations and quadratic norm formulas are obtained in Subsection \ref{unitaritySection}. The Whittaker limit $q\rightarrow\infty$ of $E_y^J(x;\mathfrak{t})$ and $E_y^{J,-}(x;\mathfrak{t})$ is considered in Subsection \ref{rationalsection}.

In Section \ref{Extsection} we generalise the results of Sections \ref{polfct}-\ref{QuasiSection} to reductive root data (including the $\textup{GL}_{r+1}$ root datum as a special case). 
We reinterpret the role of the extra parameters in terms of twisting by automorphisms of the extended double affine Hecke algebra in Subsection \ref{Tp}. In Subsection \ref{S72} we introduce the analogs of the quasi-polynomial eigenfunctions $E_y^J(x;\mathfrak{t})$ for reductive root datum. We furthermore show that, under suitable generic conditions on the parameters,
they are independent of the twist parameters and reduce to the quasi-polynomial eigenfunctions for the underlying adjoint root datum.

In Section \ref{UnifMainSection} we introduce the uniform quasi-polynomial representation (Subsection \ref{unifSection}), the associated uniform quasi-polynomial eigenfunctions
(Subsection \ref{ExtUnifSection}), and we explicitly relate the latter to the quasi-polynomial eigenfunctions $E_y^J(x;\mathfrak{t})$, using an automorphism that rescales the quasi-monomials.

In Section \ref{MetaplecticSection} we attach to a metaplectic datum a \textit{polynomial} representation of the metaplectic double affine Hecke algebra, called the metaplectic basic representation (Subsection \ref{msec}), and we show that it is isomorphic
to a suitable sub-representation of the uniform quasi-polynomial representation. We obtain an affine version of the Chinta-Gunnells' $W_0$-action on rational functions by $X$-localising the metaplectic basic representation in Subsection \ref{CGsubsection}. We introduce metaplectic polynomials, which are simultaneous eigenfunctions for the action of $\mathcal{P}_Y$ under the metaplectic basic representation, and relate them to the quasi-polynomial eigenfunctions $E_y^J(x;\mathfrak{t})$ in Subsection \ref{mpsec}. 

In Section \ref{MfinalSection} we show the existence of the Whittaker limit 
of the metaplectic polynomial. In the context of representation theory of metaplectic covers of reductive groups over non-archimedean local fields, we prove that the Whittaker limit of the metaplectic polynomial is a metaplectic Iwahori-Whittaker function up to some elementary twists. We rederive McNamara's \cite{McN} Casselman-Shalika type formula for the metaplectic spherical Whittaker function from the Whittaker limit of an explicit expansion formula for the anti-symmetrised version of the quasi-polynomial eigenfunction derived in Subsection \ref{S65}.

 \subsection{Conventions}\label{ConvSection}
 We take $\mathbf{F}$ to be a field of characteristic zero, and we fix $q\in\mathbf{F}^\times$ not a root of unity (it contains the case of formal variable $q$ by taking $\mathbf{F}=\mathbf{K}(q)$ for some field $\mathbf{K}$ of characteristic zero). For a commutative ring $R$ and an abelian group $A$ we write
$R[A]=
\bigoplus_{y\in A}Rx^y$
for the group algebra of $A$ over $R$ (i.e., $x^yx^{y^\prime}=x^{y+y^\prime}$ for $y,y^\prime\in A$ and $x^0=1$). We will often apply this to the case that $A$ is a Euclidean space. 

For groups $G$ and $H$ we write $\textup{Hom}(G,H)$ for the class of group homomorphisms $G\rightarrow H$. For a finitely generated free abelian group $\Lambda$ we write $T_\Lambda:=\textup{Hom}(\Lambda,\mathbf{F}^\times)$ for the $\mathbf{F}$-torus of multiplicative characters 
$\Lambda\rightarrow\mathbf{F}^\times$, and $\mathcal{P}_\Lambda=\mathbf{F}[\Lambda]$ for the group algebra of $\Lambda$ over $\mathbf{F}$. Alternatively, it is the algebra of regular $\mathbf{F}$-valued functions on $T_\Lambda$, where the monomial $x^\lambda$ maps $t\in T_\Lambda$ to its value $t^\lambda$ at $\lambda\in\Lambda$. Finally, for $\ell,m\in\mathbb{Z}$ with $\ell\leq m$ we write 
$[\ell,m]:=\{\ell,\ell+1,\ldots,m\}$.\\
\vspace{.2cm}\\
 \noindent
{\it Acknowledgements:} JS thanks Ivan Cherednik, Oleg Chalykh and Valentin Buciumas, and VV thanks Alexei Borodin and Eric Rains, for interesting discussions. We thank the referees for their valuable comments and pointing out several typos. The research of SS was partially supported by NSF grants DMS-1939600 and 2001537, and Simons Foundation grants 509766 and 00006698.

\section{Preliminaries on double affine Weyl and Hecke algebras}\label{sec:prelim}
In this section we fix the notations for root systems, reflection groups and double affine Weyl and Hecke algebras. For ease of reference we also list well known properties that will be regularly used in subsequent sections. 
For further details see, for instance, \cite{Hu,Ma} for root systems and reflection groups, and \cite{Ch,Ma} for double affine Weyl and Hecke algebras.

\subsection{Root systems and Weyl groups}\label{subsect_root}
Let $E$ be an Euclidean space with scalar product $\langle\cdot,\cdot\rangle$ and corresponding
norm $\|\cdot\|$. Let $E^*$ be its linear dual. We turn $E^*$ into an Euclidean space by transporting the scalar product of $E$ through the linear isomorphism $E\overset{\sim}{\longrightarrow}
E^*$, $y\mapsto \langle y,\cdot\rangle$. The resulting scalar product and norm on $E^*$ are denoted by $\langle\cdot,\cdot\rangle$ and $\|\cdot\|$ again. 

Let $\Phi_0\subset E^*$ be a reduced irreducible root system in $\textup{span}_{\mathbb{R}}\{\alpha\, | \, \alpha\in\Phi_0\}$. We normalise $\Phi_0$ in such a way that long roots have squared length equal to $2/m^2$ for some $m\in\mathbb{Z}_{>0}$. Set
$E_{\textup{co}}:=\cap_{\alpha\in\Phi_0}\textup{Ker}(\alpha)$, 
and write $E^\prime\subseteq E$ for the orthogonal complement of $E_{\textup{co}}$ in $E$,
so that $E=E^\prime\oplus E_{\textup{co}}$ (orthogonal direct sum). Write $\textup{pr}_{E^\prime}$ and $\textup{pr}_{E_{\textup{co}}}$ for the corresponding projections onto
$E^\prime$ and $E_{\textup{co}}$, respectively.

For $\alpha\in\Phi_0$ write $s_\alpha\in\textup{GL}(E^*)$ for the reflection
\[
s_\alpha(\xi):=\xi-\xi(\alpha^\vee)\alpha\qquad (\xi\in E^*),
\]
where $\alpha^\vee\in E$ is the unique vector such that $\langle y,\alpha^\vee\rangle=
2\alpha(y)/\|\alpha\|^2$ for all $y\in E$.
The Weyl group $W_0$ of $\Phi_0$ is the subgroup of $\textup{GL}(E^*)$ generated by $s_\alpha$ ($\alpha\in\Phi_0$). The root system $\Phi_0\subset E^*$ is $W_0$-invariant. We denote by
$Q\subset E^*$ the root lattice of $\Phi_0$.

Write $\Phi_0^\vee=
\{\alpha^\vee\}_{\alpha\in\Phi_0}\subset E$ for the associated coroot system.  Note that $\Phi_0^\vee\subset E^\prime$,
and $E^\prime=\textup{span}_{\mathbb{R}}\{\alpha^\vee \,\, | \,\, \alpha\in\Phi_0\}$.
The coroot lattice and co-weight lattice of $\Phi_0$ are
\[
Q^\vee:=\mathbb{Z}\Phi_0^\vee,\qquad P^\vee:=\{\lambda\in E^\prime \,\, | \,\, \alpha(\lambda)\in\mathbb{Z}\quad \forall\, \alpha\in\Phi_0\}.
\]
They are full sub-lattices of $E^\prime$, and $Q^\vee\subseteq P^\vee$. By the choice of normalisation of the roots in $\Phi_0$ we have $Q^\vee\subseteq m^2Q$ upon identifying $E^*\simeq E$ via the scalar product $\langle\cdot,\cdot\rangle$ on $E$.

Consider the action of $W_0$ on $E$ with $s_\alpha$ ($\alpha\in\Phi_0$) acting as the orthogonal reflection in the hyperplane $\alpha^{-1}(0)$,
\[
s_\alpha(y):=y-\alpha(y)\alpha^\vee\qquad (y\in E).
\]
Then $(v\xi)(y)=\xi(v^{-1}y)$ for $v\in W_0$, $\xi\in E^*$ and $y\in E$, hence 
\[
v(\alpha^\vee)=(v\alpha)^\vee\qquad\quad (v\in W_0,\,\alpha\in\Phi_0).
\]
In particular, $\Phi_0^\vee\subset E$ is $W_0$-invariant, and hence so are the lattices $Q^\vee$ and $P^\vee$. 

Fix a set $\Delta_0:=\{\alpha_1,\ldots,\alpha_r\}$ of simple roots and write $\Phi_0=\Phi^+_0\cup\Phi^-_0$ for the resulting natural division of the root system in positive and negative roots. Let $\Delta_0^\vee:=\{\alpha_1^\vee,\ldots,\alpha_r^\vee\}$ be the associated set of simple coroots for $\Phi_0^\vee$ and $\Phi_0^{\vee,\pm}$ the associated sets of positive and negative coroots.
We have $P^\vee=\bigoplus_{i=1}^r\mathbb{Z}\varpi_i^\vee$ with
$\varpi_i^\vee\in E^\prime$ ($1\leq i\leq r$) the fundamental weights with respect to $\Delta_0$,  characterised by $\alpha_i(\varpi_j^\vee)=\delta_{i,j}$ ($1\leq i,j\leq r$). Set 
$P^{\vee,\pm}:=\pm\sum_{i=1}^r\mathbb{Z}_{\geq 0}\varpi_i^\vee$ for the cones of dominant and anti-dominant co-weights in $P^\vee$, respectively.

The Weyl group $W_0$ is a Coxeter group, with Coxeter generators the simple reflections $s_i:=s_{\alpha_i}$ ($1\leq i\leq r$). We write $w_0\in W_0$ for the longest Weyl group element, and $\varphi\in\Phi_0^+$ for the highest root.
Note that $E^\prime$ is $W_0$-invariant, and $W_0$ acts trivially on $E_{\textup{co}}$. 

We write 
\begin{equation}\label{Ereg}
E^{\textup{reg}}:=\{y\in E \,\,\, | \,\,\, \alpha(y)\not=0 \quad \forall\, \alpha\in\Phi_0\}
\end{equation}
for the regular elements in $E$, and $E_+\subset E^{\textup{reg}}$ for the fundamental
Weyl chamber with respect to $\Delta_0$,
\[
E_+=\{y\in E \,\, | \,\, \alpha(y)>0\quad \forall\, \alpha\in\Phi_0^+\}.
\]
The Weyl chamber opposite to $E_+$ will be denoted by $E_-$.
Note that $E_{\pm}$ are $E_{\textup{co}}$-translation invariant, and $E^{\prime}_{\pm}:=E^\prime\cap E_{\pm}$ form complete sets of representatives of the $E_{\textup{co}}$-orbits. Furthermore,
$\overline{E}_{\pm}$ are fundamental domains for the $W_0$-action on $E$. 
For $y\in E$ we denote by $y_\pm\in \overline{E}_{\pm}$ the unique element such that $W_0y\cap\overline{E}_{\pm}=\{y_{\pm}\}$. Note that 
$P^{\vee,\pm}=P^\vee\cap\overline{E}_{\pm}$.

\subsection{Affine root systems and affine Weyl groups}\label{S22}
The affine Weyl group is the semi-direct product group $W:=W_0\ltimes Q^\vee$. 
We extend the $W_0$-action on $E$ to a faithful $W$-action on $E$ by affine linear transformations, with $\mu\in Q^\vee$ acting by
\[
\tau(\mu)y:=y+\mu\qquad (y\in E).
\]
We will regularly identify $W$ with its realisation as a subgroup of the group of affine linear automorphisms of $E$. 

Note that $W_0\subset W$ is the subgroup of elements $w\in W$ fixing the origin $0\in E$. More generally we write
\begin{equation}\label{Wygroup}
W_y:=\{w\in W \,\,\, | \,\,\, wy=y\}\qquad (y\in E),
\end{equation}
which is always a finite subgroup of $W$.

The affine Weyl group $W$ contains, besides the orthogonal reflections $s_\alpha$ ($\alpha\in\Phi_0$), orthogonal reflections in affine hyperplanes. The affine hyperplanes can be described by an affine root system $\Phi$ of untwisted type in the following manner. 

We identify $E^*\oplus\mathbb{R}$ with the space of affine linear functionals on $E$ by interpreting  $(\xi,d)\in E^*\oplus\mathbb{R}$ as the affine linear functional $y\mapsto \xi(y)+d$ ($y\in E$).
The affine root system $\Phi$ is then defined by
\[
\Phi:=\{(\alpha,\ell)\,\,\, | \,\,\, \alpha\in\Phi_0\,\, \& \,\, \ell\in\mathbb{Z} \}\subset E^*\oplus\mathbb{R}.
\]
 For an affine root $a=(\alpha,\ell)\in\Phi$, we denote by $s_a$ the orthogonal reflection in the affine hyperplane $a^{-1}(0)\subset E$. We then have for $a\in\Phi$,
\[
s_a(y)=y-a(y)\alpha^\vee.
\]
In particular
\[
s_a=\tau(-\ell\alpha^\vee)s_\alpha\in W\qquad\quad (a=(\alpha,\ell)\in\Phi),
\]
showing that the affine Weyl group $W$ is generated by the reflections $s_a$ ($a\in\Phi$).

The formulas
\begin{equation}\label{linearaction}
v(\xi,d)=(v\xi,d),\qquad \tau(\mu)(\xi,d)=(\xi,d-\xi(\mu))
\end{equation}
for $v\in W_0$, $\mu\in Q^\vee$ and $(\xi,d)\in E^*\oplus\mathbb{R}$ define a linear $W$-action on the space $E^*\oplus\mathbb{R}$ of affine linear functionals on $E$.
It is contragredient to the action of $W$ on $E$ by reflections and translations,
\[
(wf)(y)=f(w^{-1}y)\qquad\quad (w\in W,\, f\in E^*\oplus\mathbb{R},\, y\in E).
\]
Note that $\Phi\subset E^*\oplus\mathbb{R}$ is $W$-stable. 

We identify $E^*$ with the subspace $\{(\xi,0)\,\, | \,\, \xi\in E^*\}$ of $E^*\oplus\mathbb{R}$ via the embedding $\xi\mapsto (\xi,0)$ ($\xi\in E^*$). This is an embedding of $W_0$-modules. The second formula of \eqref{linearaction} then gives
\[
\tau(\mu)\xi=(\xi,-\xi(\mu))\qquad (\xi\in E^*,\, \mu\in Q^\vee).
\]

The set $\Delta:=\{\alpha_0,\ldots,\alpha_r\}$, with the simple affine root $\alpha_0$ defined by 
\begin{equation}\label{alpha0}
\alpha_0:=(-\varphi,1)\in\Phi,
\end{equation}
forms a basis of the affine root system $\Phi$. 
It gives rise to the division $\Phi=\Phi^+\cup\Phi^-$of $\Phi$ in positive and negative roots, with 
\begin{equation}\label{Phiplusexplicit}
\Phi^+:=\mathbb{Z}_{\geq 0}\Delta\cap\Phi=\Phi^+_0\cup\{(\alpha,\ell) \,\,\, | \,\,\, (\alpha,\ell)\in\Phi_0\times\mathbb{Z}_{>0}\}
\end{equation}
and $\Phi^-:=-\Phi^+$.
The reflection
\begin{equation}\label{s0}
s_0:=s_{\alpha_0}=\tau(\varphi^\vee)s_{\varphi}\in W
\end{equation}
is called the simple affine reflection of $W$.
The affine Weyl group $W$ is a Coxeter group with Coxeter generators $s_0,\ldots,s_r$.

Let $D: E^*\oplus\mathbb{R}\rightarrow E^*$ be the projection on the first component. Note that $Df$
is the gradient of $f\in E^*\oplus\mathbb{R}$, viewed as affine linear functional on $E$. It induces a group epimorphism $W\twoheadrightarrow W_0$, which we also denote by $D$, by requiring that
\[
D(wf)=(Dw)Df\qquad (w\in W,\, f\in E^*\oplus\mathbb{R}).
\]
Concretely, $D(v\tau(\mu))=v$ for $v\in W_0$ and $\mu\in Q^\vee$.

The action of $W$ on $E$ preserves  
\begin{equation}\label{Eareg}
E^{a,\textup{reg}}:=\{y\in E \,\, | \,\, 
a(y)\not=0\quad \forall a\in\Phi \}.
\end{equation}
The connected components of $E^{a,\textup{reg}}$ are called alcoves. The resulting $W$-action on the set of alcoves is simply transitive. The alcove
\begin{equation}\label{Cplus}
\begin{split}
C_+:=&\{ y\in E \,\, | \,\, a(y)>0\quad \forall\, a\in\Phi^+ \}\\
=&\{ y\in E \,\, | \,\, 0<\alpha(y)<1\quad \forall\, \alpha\in\Phi_0^+ \}
\end{split}
\end{equation}
is the fundamental alcove in $E$. For $y\in E$ we write 
\begin{equation}\label{O}
\mathcal{O}_y:=Wy
\end{equation}
for the $W$-orbit of $y$ in $E$. We thus have
\begin{equation}\label{WorbitE}
E^{a,\textup{reg}}=\bigsqcup_{\cc\in C_+}\mathcal{O}_\cc,\qquad\quad E=\bigsqcup_{\cc\in\overline{C}_+}\mathcal{O}_\cc
\end{equation}
(disjoint unions).

\subsection{Length function and parabolic subgroups}\label{cosetsection}

We recall some properties of the length function
\begin{equation}\label{length}
\ell(w):=\textup{Card}\bigl(\Pi(w)\bigr),\qquad \Pi(w):=\Phi^+\cap w^{-1}\Phi^-
\end{equation}
that we will frequently use.
\begin{lemma}\label{lengthadd}
Let $0\leq j\leq r$ and $w\in W$. 
\begin{enumerate}
\item $|\ell(s_jw)-\ell(w)|=1$.
\item We have
\[
\ell(s_jw)=\ell(w)+1 \quad\Leftrightarrow \quad w^{-1}\alpha_j\in \Phi^+,
\]
and then $\Pi(s_jw)=\{w^{-1}\alpha_j\}\cup \Pi(w)$
\textup{(}disjoint union\textup{)}.
\end{enumerate}
\end{lemma}
\begin{proof}
See, e.g., \cite[(2.2.8)]{Ma}.
\end{proof}
The length $\ell(w)$ of $w\in W$ is the smallest nonnegative integer $\ell$ for which there exists an expression $w=s_{j_1}\cdots s_{j_\ell}$ of $w$ as product of $\ell$ simple reflections 
($0\leq j_i\leq r$).
Such shortest length expressions are called reduced expressions. If
$w=s_{j_1}\cdots s_{j_\ell}$ is a reduced expression then 
\begin{equation}\label{rootsdescription}
\Pi(w)=
\{b_1,\ldots,b_\ell\}\,\,\, \hbox{ with }\, b_m:=s_{j_\ell}\cdots s_{j_{m+1}}(\alpha_{j_m})\,\,\, (1\leq m<\ell),\quad b_\ell:=\alpha_{j_\ell}.
\end{equation}
Restricting $\ell$ to $W_0$  gives the length function of $W_0$ relative to $\{s_1,\ldots,s_r\}$. 
The following important length identity 
\begin{equation}\label{fundlength}
\ell(\tau(\mu)v)=\sum_{\alpha\in\Phi_0^+\cap v\Phi_0^-}|\alpha(\mu)-1|+\sum_{\alpha\in\Phi_0^+\cap v\Phi_0^+}|\alpha(\mu)|\qquad
(\mu\in Q^\vee,\, v\in W_0),
\end{equation}
is obtained by explicitly describing the affine roots in $\Pi(\tau(\mu)v)$ (see, e.g., \cite[(2.4.1)]{Ma}).

Fix a subset $J\subseteq [0,r]$. The associated parabolic subgroup $W_J\subseteq W$ is the subgroup of $W$ generated by 
$s_j$ ($j\in J$). Write $W^{J}$ for the set of elements $w\in W$ such that 
\[
\ell(ww^\prime)=\ell(w)+\ell(w^\prime)\qquad \forall\, w^\prime\in W_{J}.
\]
Write $\Phi_{J}:=\Phi\cap\textup{span}_{\mathbb{Z}}\{\alpha_j\}_{j\in J}$ and $\Phi_J^{\pm}:=\Phi^{\pm}\cap\Phi_J$, then 
\begin{equation}\label{Phipos}
W^{J}=\{w\in W \,\, | \,\, w\Phi_J^+\subseteq\Phi^+ \}.
\end{equation}
\begin{lemma}\label{cosetcomb}
Fix a subset $J\subseteq [0,r]$.
\begin{enumerate}
\item $W^J$ is a complete
set of representatives of $W/W_J$ \textup{(}the elements in $W^J$ are called the minimal coset representatives of $W/W_J$\textup{)}.
\item Let $w\in W^J$ and $0\leq j\leq r$. Then 
\[
s_jw\not\in W^J\,\,\,\Leftrightarrow\,\,\, s_jwW_J=wW_J,
\]
and then $s_jw=ws_{j^\prime}$ for some $j^\prime\in J$ \textup{(}in particular, 
$\ell(s_jw)=\ell(w)+1$\textup{)}.
\end{enumerate}
\end{lemma}
\begin{proof}
(1) See, e.g., \cite{B}.\\
(2) This follows from \cite[Lem. 3.1 \& Lem. 3.2]{D}. 
\end{proof}
The stabiliser subgroup $W_\cc$ of $\cc\in\overline{C}_+$ is parabolic, 
\[
W_\cc=W_{\mathbf{J}(\cc)}
\]
with $\mathbf{J}(\cc)\subseteq [0,r]$ given by
\begin{equation}\label{Jc}
\mathbf{J}(\cc):=\{j\in [0,r] \,\, | \,\, s_j\cc=\cc\}.
\end{equation}
\begin{definition}\label{wydef}
For $y\in E$ we write
\begin{enumerate}
\item $\cc_y$ for the unique element in $\overline{C}_+$ such that $y\in W\cc_y$,
\item $w_y$ for the unique element in $W^{\mathbf{J}(\cc_y)}$ such that $w_y\cc_y=y$.
\item $\mu_y:=w_y(0)\in Q^\vee$ and $v_y:=w_y^{-1}\tau(\mu_y)\in W_0$. 
\end{enumerate}
\end{definition}

Fix $\nu\in Q^\vee$. Then we have $\cc_\nu=0$ and $\mu_\nu=\nu$. In particular, $w_\nu$ is the unique element in $\tau(\nu)W_0$ of minimal length and
\begin{equation}\label{lengthwmu}
\ell(w_\nu v)=\ell(w_\nu)+\ell(v)\qquad \forall\, v\in W_0.
\end{equation}
By \cite[\S 2.4]{Ma}, $v_\nu\in W_0$ is the unique element of minimal length in $W_0$ such that $v_\nu\nu=\nu^-$,
and by \eqref{fundlength} we have
\begin{equation}\label{wmuneg}
w_\nu=\tau(\nu)\qquad \forall\, \nu\in Q^{\vee}\cap\overline{E}_-.
\end{equation}
For $\nu=\varphi^\vee$ we have 
\[
w_{\varphi^\vee}=s_0,\qquad\quad v_{\varphi^\vee}=s_\varphi,
\]
and hence $\ell(\tau(\varphi^\vee))=1+\ell(v_{\varphi^\vee})$. In addition we have
\begin{equation}\label{thetalength}
\begin{split}
&\Pi(s_{\varphi})=\{\alpha\in\Phi^+_0 \,\, | \,\, \alpha(\varphi^\vee)=1\}\cup\{\varphi\},\\
&\Phi^+_0\setminus\Pi(s_\varphi)=\{\alpha\in\Phi^+_0 \,\, | \,\, \alpha(\varphi^\vee)=0\}.
\end{split}
\end{equation}
For details see, e.g., \cite[\S2.4]{Ma}.

The following proposition reduces to \cite[Thm. 1.4]{C} when $y\in Q^\vee$.
\begin{proposition}\label{Pidescription}
For $y\in E$ we have
\[
\Pi(w_y^{-1})=\{a\in\Phi^+ \,\,\, | \,\,\, a(y)<0\}.
\]
\end{proposition}
\begin{proof}
Let $a\in\Phi^+$ with $a(y)<0$. Then $(w_y^{-1}a)(\cc_y)=a(y)<0$ and $\cc_y\in\overline{C}_+$, 
so $w_y^{-1}a\in\Phi^-$. In particular, $a\in\Pi(w_y^{-1})$.

Conversely,
if $a\in\Pi(w_y^{-1})$ then $w_y^{-1}a\in\Phi^-$, hence $a(y)=(w_y^{-1}a)(\cc_y)\leq 0$. If $a(y)=0$ then $s_ay=y$, hence $s_aw_y\in w_yW_{\cc_y}$. But $w_y$ is the element of minimal length in 
$w_yW_{\cc_y}$, so $\ell(s_aw_y)>\ell(w_y)$ and hence $\ell(w_y^{-1}s_a)>\ell(w_y^{-1})$. 
This implies $w_y^{-1}a\in\Phi^+$ by \cite[Prop. 5.7]{Hu}, which contradicts 
the fact that $a\in\Pi(w_y^{-1})$. Hence $a(y)<0$, which completes the proof.
\end{proof}

\subsection{The Coxeter complex of $W$}\label{S24}
Let $\Sigma(\Phi)$ be the poset consisting of faces of the affine hyperplane arrangement
 $\{a^{-1}(0)\,\, | \,\, a\in\Phi\}$ in $E$. Concretely, 
  \[
 \Sigma(\Phi)=\{wC^J\,\, | \,\, J\subsetneq [0,r]\,\, \& \,\, w\in W^J\}
 \]
 with 
 \[
 C^J:=\{\cc\in \overline{C}_+ \,\, | \,\, \alpha_j(\cc)=0\,\,\, (j\in J)\,\, \& \,\, \alpha_j(\cc)>0\,\,\,
 (j\in [0,r]\setminus J)\}.
 \]
 Note that the faces are $E_{\textup{co}}$-translation invariant, $C^{[1,r]}=E_{\textup{co}}$ and 
 \begin{equation}\label{decompdis}
 E=\bigsqcup_{F\in\Sigma(\Phi)}F
 \end{equation}
 (disjoint union). 
 The poset structure on $\Sigma(\Phi)$ is defined in terms of inclusion of closures of faces: $F\leq F^\prime$ if $\overline{F}\subseteq\overline{F}^\prime$.
 The closure $C_J$ of $C^J$ is given by
 \begin{equation}\label{closureCJ}
  C_J=\bigsqcup_{J\subseteq K\subsetneq [0,r]}C^K
 \end{equation}
 (disjoint union),  so that $C^{J^\prime}\leq C^J$ iff $J\subseteq J^\prime\subsetneq [0,r]$. 
 
 We obtain an abstract simplicial complex by adding $\emptyset$ to $\Sigma(\Phi)$. It is isomorphic to the Coxeter complex 
 \[
 \{wW_J\,\, | \,\, J\subseteq [0,r]\,\, \& \,\, w\in W\},
 \]
 partially ordered by opposite inclusion.
 
 For a set $X$ write $\mathcal{F}_\Sigma(E,X)$ for the set of functions $E\rightarrow X$ which 
 are constant on faces. It canonically identifies
 with the space $\mathcal{F}(\Sigma(\Phi),X)$ of functions $\Sigma(\Phi)\rightarrow X$.
 The following lemma provides a large class of examples when $X=\mathbb{R}$.
\begin{lemma}\label{faceconstantlem}
Suppose that $g:\mathbb{R}\rightarrow\mathbb{R}$ is a function such that $g|_{(m,m+1)}$ is constant for all $m\in\mathbb{Z}$. Let $a\in\Phi$. Then
\begin{equation}\label{thefunction}
y\mapsto g(a(y))\qquad (y\in E)
\end{equation}
lies in $\mathcal{F}_\Sigma(E,\mathbb{R})$.
\end{lemma}
\begin{proof}
Let $a\in\Phi$ with $Da\in\Phi_0^+$. It suffices to check that \eqref{thefunction} is constant on $C^J$ ($J\subsetneq [0,r]$). 
Fix $y\in C^J\subseteq\overline{C}_+$. Note that $0\leq Da(y)\leq 1$ by \eqref{Cplus}. 

Suppose that $a(y)\in\mathbb{Z}$. If $Da(y)=0$ then $Da\in\Phi_0^+\cap \Phi_J$, hence $Da\in\textup{span}_{\mathbb{R}}\{D\alpha_j\}_{j\in J}$. If $Da(y)=1$ then necessarily $\varphi(y)=1$, i.e., $0\in J$. It again follows that $Da\in\textup{span}\{D\alpha_j\}_{j\in J}$. In particular, in both cases $a\vert_{C^J}$ is constant.

If $a(y)\not\in\mathbb{Z}$ then $Da\not\in\textup{span}\{D\alpha_j\}_{j\in J}$ and hence $a(y^\prime)\not\in\mathbb{Z}$ for all $y^\prime\in C^J$. Since $C^J$ is convex and $a$ is an affine linear function on $E$, this forces $a\vert_{C^J}$ to take values in an interval $(m,m+1)$ for some $m\in\mathbb{Z}$.
\end{proof}
 In case $X=W$ we have the following example.
 \begin{lemma}\label{facefunction}
 The function $y\mapsto w_y$ lies in $\mathcal{F}_\Sigma(E,W)$.
 \end{lemma}
 \begin{proof}
Fix a face $F=wC^J$ ($J\subsetneq [0,r]$, $w\in W^J$). For $y\in F$ we have $\cc_y\in C^J$, hence $w_y\in W^J$. Furthermore, $F=w_yC^J$ since $y=w_y\cc_y$. Consequently $w_y=w$.
\end{proof}
\subsection{The double affine Weyl group}\label{S25}
The dual affine root system is 
\[
\Phi^\vee:=\{a^\vee\,\, | \,\, a\in\Phi\}\subset E\oplus\mathbb{R},
\]
with $a^\vee$ the affine coroot
\begin{equation}\label{acoroot}
a^\vee:=(\alpha^\vee,2\ell/\|\alpha\|^2)\qquad\quad (a=(\alpha,\ell)\in\Phi).
\end{equation}
Note that $2/\|\alpha\|^2\in\mathbb{Z}_{>0}$ ($\alpha\in\Phi_0$) by our convention on the length of the roots, and 
\[
\alpha_0^\vee=
m^2K-(\varphi^\vee,0)
\] 
with $K:=(0,1)\in E\oplus\mathbb{Z}$, in view of the convention that long roots have squared norm $2/m^2$. Hence $\Phi^\vee$ generates the affine coroot lattice
\[
\widehat{Q}^\vee:=Q^\vee\oplus m^2\mathbb{Z}K\subset E\oplus\mathbb{R},
\]
and the simple coroots $\alpha_j^\vee$ ($0\leq j\leq r$) form a basis of $\widehat{Q}^\vee$. 

The affine Weyl group $W$ acts linearly on $E\oplus\mathbb{R}$ by
\begin{equation}\label{linearactionE}
v\cdot(y,d):=(vy,d),\qquad \tau(\mu)\cdot (y,d)=(y,d-\langle \mu,y\rangle)
\end{equation}
for $v\in W_0$, $\mu\in Q^\vee$ and $(y,d)\in E\oplus\mathbb{R}$. We denote the action \eqref{linearactionE} with a dot to avoid confusion with the $W$-action on $E$ by translations and reflections. Note that the linear isomorphism $E\oplus\mathbb{R}\overset{\sim}{\longrightarrow} E^*\oplus\mathbb{R}$, $(y,d)\mapsto (\langle y,\cdot\rangle,d)$,
intertwines the $W$-actions \eqref{linearaction} and \eqref{linearactionE}.  

For $a\in\Phi$ and $w\in W$ we have
\begin{equation}
\label{compatibilities}
s_a\cdot(y,d)=(y,d)-Da(y)a^\vee,\qquad\quad w\cdot a^\vee=(wa)^\vee.
\end{equation}
In particular, $\Phi^\vee$ and $\widehat{Q}^\vee$ are $W$-invariant. Explicitly, $w\cdot K=K$ for all $w\in W$ and 
\[
v\cdot\mu=v\mu,\qquad \tau(\nu)\cdot \mu=\mu-\langle\nu,\mu\rangle K
\]
for $v\in W_0$ and $\mu,\nu\in Q^\vee$.
\begin{definition} 
The double affine Weyl group is the semidirect product 
\begin{equation}\label{dawg}
\mathbb{W}:=W\ltimes\widehat{Q}^\vee.
\end{equation}
\end{definition}
The action \eqref{linearactionE} of $W$ on $E\oplus\mathbb{R}$ then extends to a $\mathbb{W}$-action by
\begin{equation}\label{laext}
(\mu+\ell K)\cdot (y,d):=(y+\mu,d+\ell)
\end{equation}
for $\mu\in Q^\vee$, $\ell\in m^2\mathbb{Z}$, $y\in E$ and $d\in\mathbb{R}$ (which no longer is a linear action). Note that $m^2K$ generates the center of $\mathbb{W}$. 

We next consider, for a special class of finitely generated abelian subgroups $\widehat{\Lambda}$ in $E\oplus\mathbb{R}$ containing $\widehat{Q}^\vee$, the $\mathbf{F}$-linear extension of the $\mathbb{W}$-action \eqref{linearactionE} and \eqref{laext} on $\widehat{\Lambda}$ to the group algebra $\mathcal{P}_{\widehat{\Lambda}}=\bigoplus_{\widehat{\lambda}\in\widehat{\Lambda}}\mathbf{F}x^{\widehat{\lambda}}$ of $\widehat{\Lambda}$ over $\mathbf{F}$ (see Subsection \ref{ConvSection}).
\begin{definition}\label{latticeconddef}
Denote by $\mathcal{L}$ the set of finitely generated abelian subgroups $\Lambda\subset E$ satisfying
\begin{equation}\label{latticecond}
Q^\vee\subseteq\Lambda\quad \&\quad \alpha(\Lambda)\subseteq\mathbb{Z}\quad \forall\, \alpha\in\Phi_0.
\end{equation}
\end{definition}
Note that if $\Lambda\in\mathcal{L}$ then $L$ is $W_0$-invariant. Furthermore $\Lambda\cap E^\prime,\textup{pr}_{E^\prime}(\Lambda)\in\mathcal{L}$ and 
\[
Q^\vee\subseteq \Lambda\cap E^\prime\subseteq\textup{pr}_{E^\prime}(\Lambda)\subseteq P^\vee.
\]
As a consequence, we have 
\begin{corollary}\label{ilacor}
For $\Lambda\in\mathcal{L}$ we have a group homomorphism
\begin{equation}\label{invlimitalt}
j_\Lambda: T_{P^\vee}\rightarrow T_\Lambda,\qquad j_\Lambda(t)^\lambda:=t^{\textup{pr}_{E^\prime}(\lambda)}\quad (\lambda\in\Lambda).
\end{equation}
\end{corollary}
The following proposition is useful for writing co-weights as $\mathbb{Q}$-linear combination of co-roots.  
Recall that $h$ is the Coxeter number of $W$. 
\begin{proposition}\label{latticecompatibilityprop}
For all $y\in E$ we have
\begin{equation}\label{relationinvform}
\frac{1}{h}\sum_{\alpha\in\Phi_0^+}\alpha(y)\alpha^\vee=\textup{pr}_{E^\prime}(y).
\end{equation}
In particular, $\textup{pr}_{E^\prime}(\Lambda)\subseteq P^\vee\subseteq \frac{1}{h}Q^\vee$ for all $\Lambda\in\mathcal{L}$.
\end{proposition}
\begin{proof}
It suffices to prove \eqref{relationinvform} for $y\in E^\prime$. It then follows by applying Schur's lemma to the irreducible $W_0$-module $E^\prime$ and using the fact that $h=\#\Phi_0/r$, see \cite[Thm. 1.1 \& Lem. 7.2]{Ste}. The second statement is immediate.
\end{proof}

For $\Lambda\in\mathcal{L}$ consider the lattice
\[
\widehat{\Lambda}:=\Lambda\oplus\mathbb{Z}K\subset E\oplus\mathbb{R}
\]
containing $\widehat{Q}^\vee$. Note that $\widehat{\Lambda}$ is $\mathbb{W}$-invariant.
Set
\[
\mathbf{q}:=x^K\in\mathcal{P}_{\widehat{\Lambda}},
\]
then $\mathcal{P}_{\widehat{\Lambda}}$ can be alternatively viewed as the group algebra of $\Lambda$ over the Laurent polynomial ring $\mathbf{F}[\mathbf{q}^{\pm 1}]$,
\[
\mathcal{P}_{\widehat{\Lambda}}=\bigoplus_{\lambda\in\Lambda}\mathbf{F}[\mathbf{q}^{\pm 1}]x^\lambda.
\]
We now immediately obtain the following result.
\begin{proposition}\label{dWprop}
Let $\Lambda\in\mathcal{L}$. The $\mathbf{F}$-linear extension of the $\mathbb{W}$-action \eqref{linearactionE} and \eqref{laext} on $\widehat{\Lambda}$ to $\mathcal{P}_{\widehat{\Lambda}}$
is $\mathbf{F}[\mathbf{q}^{\pm 1}]$-linear. Furthermore, for $\lambda\in\Lambda$ we have
\begin{equation}\label{doubleWaction}
\begin{split}
\ell K(x^\lambda)&:=\mathbf{q}^\ell x^\lambda\,\,\, (\ell\in m^2\mathbb{Z}),\qquad\quad
\mu(x^\lambda):=x^{\lambda+\mu}\qquad\,\, (\mu\in Q^\vee),\\
v(x^\lambda)&:=x^{v\lambda}\quad (v\in W_0),\qquad
\tau(\nu)(x^\lambda):=\mathbf{q}^{-\langle\nu,\lambda\rangle}x^{\lambda}\quad (\nu\in Q^\vee),
\end{split}
\end{equation}
and for 
$w\in W$, $\widehat{\lambda}\in\widehat{\Lambda}$ and $a\in\Phi$ we have
\begin{equation}\label{actx}
w(x^{\widehat{\lambda}})=x^{w\cdot\widehat{\lambda}},\qquad\quad w(x^{a^\vee})=x^{(wa)^\vee}.
\end{equation}
\end{proposition}

Without much loss of generality one may absorb $\mathbf{q}$ into the ground field $\mathbf{F}$ (by replacing $\mathbf{F}$ by $\mathbf{F}(\mathbf{q})$). This is convenient for our purposes, since it allows to interpret
\eqref{doubleWaction} as an action by $q$-difference reflection operators on the $\mathbf{F}$-torus $T_\Lambda$. We discuss it in detail in the following subsection.

\subsection{Algebras of $q$-difference reflection operators}\label{ap}
Recall that we fixed a parameter $q\in\mathbf{F}^\times$ from the outset, which is assumed not to be a root of unity (see Subsection \ref{ConvSection}).  

Let $\Lambda\in\mathcal{L}$. The $\mathbb{W}$-action on $\mathcal{P}_{\widehat{\Lambda}}$ from Proposition \ref{dWprop} descends to a $\mathbb{W}$-action on the polynomial algebra $\mathcal{P}_\Lambda$ via the specialisation map
\[
\mathcal{P}_{\widehat{\Lambda}}\twoheadrightarrow\mathcal{P}_\Lambda, \qquad \sum_{\lambda\in\Lambda}d_\lambda(\mathbf{q})x^\lambda\mapsto
\sum_{\lambda\in\Lambda}d_\lambda(q)x^\lambda\qquad\quad (d_\lambda(\mathbf{q})\in\mathbf{F}[\mathbf{q}^{\pm 1}]).
\]
The resulting action of the affine Weyl group $W$ on $\mathcal{P}_\Lambda$ is by algebra homomorphisms,
\begin{equation}\label{xyW}
v(x^{\lambda}):=x^{v\lambda},\qquad \tau(\mu)(x^{\lambda}):=q^{-\langle\mu,\lambda\rangle}x^\lambda\qquad (v\in W_0,\, \mu\in Q^\vee),
\end{equation}
while $\widehat{Q}^\vee$ acts by $(\mu+\ell K)(x^\lambda)=q^\ell q^{\lambda+\mu}$ for $\mu\in Q^\vee$ and $\ell\in m^2\mathbb{Z}$.
The resulting action of the group algebra $\mathbf{F}[\mathbb{W}]$ of $\mathbb{W}$ over $\mathbf{F}$ on $\mathcal{P}_\Lambda$ descends to an action of
the quotient algebra $\mathbf{F}[\mathbb{W}]/(m^2K-q^{m^2})$.

Denote the smashed product algebra relative to the $W$-action \eqref{xyW} on $\mathcal{P}_\Lambda$ by $W\ltimes\mathcal{P}_\Lambda$, and write
$\mathcal{P}:=\mathcal{P}_{Q^\vee}$.
Then
\begin{equation}\label{isodW}
\mathbf{F}[\mathbb{W}]/(m^2K-q^{m^2})\overset{\sim}{\longrightarrow}W\ltimes\mathcal{P}
\end{equation}
as $\mathbf{F}$-algebras by $w\mapsto w$ ($w\in W$) and $\mu\mapsto x^\mu$ ($\mu\in Q^\vee$). Hence $\mathcal{P}_\Lambda$ becomes a $W\ltimes\mathcal{P}$-module.

Viewing $\mathcal{P}_\Lambda$ as the regular $\mathbf{F}$-valued functions on $T_\Lambda$, the action \eqref{xyW} can be interpreted as an action by $q$-difference reflection operators, i.e., it is the contragredient action of a $W$-action on $T_\Lambda$ by $q$-dilations and reflections. We will now describe this $W$-action on $T_\Lambda$ in detail, and extend it to an appropriate inverse system of tori. 

Let $\Lambda\in\mathcal{L}$. For $\xi\in Q$ and $z\in\mathbf{F}^\times$ we denote by $z^\xi\in T_{\Lambda}$ the multiplicative character
\begin{equation}\label{zxi}
\lambda\mapsto z^{\xi(\lambda)}\qquad (\lambda\in\Lambda).
\end{equation}
Note that $z^\xi\vert_\Lambda=j_\Lambda(z^\xi\vert_{P^\vee})$
since $\xi(\lambda)=\xi(\textup{pr}_{E^\prime}(\lambda))$ for $\xi\in Q$.

We have the following special case of this construction. Our assumption on the normalisation of the root system $\Phi_0$ allows to view $Q^\vee$ as a sub-lattice of $Q$ via the $W_0$-equivariant isomorphism 
\[
E\overset{\sim}{\longrightarrow} E^*,\qquad y\mapsto \langle y,\cdot\rangle.
\]
More precisely, since $\|\varphi\|^2=2/m^2$ ($m\in\mathbb{Z}_{>0}$) the co-root lattice $Q^\vee$ identifies with a sub-lattice of $mQ$.
Then we obtain the multiplicative characters $z^\mu\in T_{\Lambda}$ ($\mu\in Q^\vee$) satisfying $\lambda\mapsto z^{\langle\lambda,\mu\rangle}$
($\lambda\in\Lambda$). In particular, for $\mu\in Q^\vee$ we have the multiplicative character $q^\mu\in T_{\Lambda}$ mapping $\lambda\in\Lambda$ to $q^{\langle\mu,\lambda\rangle}$. Note that $vz^\xi=z^{v\xi}$ and $vz^\mu=z^{v\mu}$ for $v\in W_0$, $\xi\in Q$ and $\mu\in Q^\vee$. 

Fix $\Lambda\in\mathcal{L}$. 
Since $\Lambda$ is $W_0$-invariant, $W_0$ acts by group automorphisms on $T_\Lambda$ via the contragredient action on $T_\Lambda$. Extend the $W_0$-action on $T_\Lambda$ to a $W$-action by
\begin{equation}\label{tauaction}
\tau(\mu)t:=q^\mu t\qquad (t\in T_\Lambda,\, \mu\in Q^\vee).
\end{equation}
Note that the map \eqref{invlimitalt} is $W$-equivariant, and
\[
(w(p))(t)=p(w^{-1}t)\qquad\quad (w\in W,\, p\in\mathcal{P}_\Lambda,\, t\in T_\Lambda),
\]
with the $W$-action on $\mathcal{P}_\Lambda$ given by \eqref{xyW}.

Write 
\[
q_\alpha:=q^{\frac{2}{\|\alpha\|^2}}\qquad\quad (\alpha\in\Phi_0),
\]
which are integral powers of $q$ by the convention that $\|\varphi\|^2=2/m^2$ for some $m\in\mathbb{Z}_{>0}$.
For an affine root $a=(\alpha,\ell)\in\Phi$ and $\lambda\in\Lambda$ we have
\begin{equation}\label{saxy}
s_a(x^\lambda)=q_\alpha^{-\ell\alpha(\lambda)}x^{s_\alpha \lambda}.
\end{equation}

{}From now on we will be working with specialised $\mathbf{q}$, unless explicitly stated otherwise. In particular, 
\begin{equation*}
x^{\lambda+\ell K}=q^\ell x^\lambda\in\mathcal{P}_\Lambda\qquad\quad (\lambda\in\Lambda,\, m\in\mathbb{Z}),
\end{equation*}
and \eqref{actx} is interpreted  accordingly as identity in $\mathcal{P}_\Lambda$.
Note that for $a\in\Phi$ and $t\in T_\Lambda$ we have
\begin{equation}\label{actiononT}
s_at=(t^{-a^\vee})^{Da}t
\end{equation}
where $t^{-a^\vee}\in \mathbf{F}^\times$ is the evaluation of $x^{-a^\vee}\in\mathcal{P}_\Lambda$ at $t$,
and $(t^{-a^\vee})^{Da}\in T_\Lambda$ is the resulting multiplicative character of $\Lambda$.
 
Let $\mathcal{Q}_\Lambda$ be the quotient field of $\mathcal{P}_\Lambda$, which identifies with the field of rational functions on $T_\Lambda$, and set $\mathcal{Q}:=\mathcal{Q}_{Q^\vee}$. We extend the $W$-action on $\mathcal{P}_\Lambda$ to a $W$-action on $\mathcal{Q}_\Lambda$ by field automorphisms, and write $W\ltimes\mathcal{Q}_\Lambda$ for the associated smashed product algebra over $\mathbf{F}$. 

\subsection{The double affine Hecke algebra}\label{DAHA}
The double affine Hecke algebra $\mathbb{H}=\mathbb{H}(\mathbf{k},q)$ is a flat deformation of $\mathbf{F}[\mathbb{W}]/(m^2K-q^{m^2})\simeq W\ltimes\mathcal{P}$, with its deformation parameters encoded by a multiplicity function $\mathbf{k}: \Phi\rightarrow\mathbf{F}^\times$. In this paper we restrict attention to multiplicity functions that only depend on the length of the gradient of the affine root. 

Write ${}^{\textup{sh}}\Phi_0$ and ${}^{\textup{lg}}\Phi_0$ for the short and long roots in $\Phi_0$, respectively. Fix 
\[
{}^{\textup{sh}}k,{}^{\textup{lg}}k\in\mathbf{F}^\times,
\]
and assume that both are admitting square roots in $\mathbf{F}$, which we fix once and for all. If all the roots have the same length, then we will view the roots as long roots, and we will denote ${}^{\textup{lg}}k$ by $k$. 
We write for $a\in\Phi$, 
\[
k_a:={}^sk \quad \hbox{ if }\quad Da\in {}^s\Phi_0\qquad (s\in\{\textup{sh},\textup{lg}\}).
\]
Then $\mathbf{k}: \Phi\rightarrow \mathbf{F}^\times$, $a\mapsto k_a$ is $W$-invariant map (it is in fact invariant for the action of the extended affine Weyl group, see Section \ref{Extsection}). We write $k_j:=k_{\alpha_j}$ for $0\leq j\leq r$. Then $k_j=k_{j^\prime}$ iff $Da_{j^\prime}$ lies in the $W_0$-orbit of $Da_j$. In particular, $k_0=k_\varphi$. 

Cherednik's \cite{Ch} equal lattice double affine Hecke algebra is the following deformation of 
$W\ltimes\mathcal{P}$.

\begin{definition}\label{defDAHA}
The double affine Hecke algebra $\mathbb{H}=\mathbb{H}(\mathbf{k},q)$ is the unital associative $\mathbf{F}$-algebra generated by $T_0,\ldots,T_r$ and 
$x^\mu$ \textup{(}$\mu\in Q^\vee$\textup{)}, subject to the following relations:
\begin{enumerate}
\item[{\bf a.}] The braid relations
\begin{equation}\label{br}
T_jT_{j^\prime}T_j\ldots=T_{j^\prime}T_jT_{j^\prime}\cdots \qquad \textup{(}
m_{jj^\prime} \textup{ factors on each side}\textup{)}
\end{equation}
for $0\leq j\not=j^\prime\leq r$, with $m_{jj^\prime}$ the order of $s_js_{j^\prime}$ in $W$,
\item[{\bf b.}] The Hecke relations 
\begin{equation}\label{Hr}
(T_j-k_j)(T_j+k_j^{-1})=0\qquad\quad (0\leq j\leq r),
\end{equation}
\item[{\bf c.}] $x^\lambda x^\mu=x^{\lambda+\mu}$ \textup{(}$\lambda,\mu\in Q^\vee$\textup{)} and $x^0=1$,
\item[{\bf d.}] The cross relations
\begin{equation}\label{crossX}
T_jx^\mu-s_j(x^\mu)T_j=(k_j-k_j^{-1})\Bigl(\frac{x^\mu-s_j(x^\mu)}{1-x^{\alpha_j^\vee}}\Bigr)
\end{equation}
for $0\leq j\leq r$ and $\mu\in Q^\vee$, with the $W$-action \eqref{xyW} on $\mathcal{P}$.
\end{enumerate}
\end{definition}
\noindent
Note that $T_j\in\mathbb{H}^\times$ with inverse $T_j^{-1}=T_j-k_j+k_j^{-1}$ by \eqref{Hr}. Furthermore, note that
\begin{equation}\label{isodaWgDAHA}
\mathbb{H}(\mathbf{1},q)\overset{\sim}{\longrightarrow}W\ltimes\mathcal{P}
\end{equation}
as algebras, with $\mathbf{1}$ the multiplicity function identically equal to $1$ and the isomorphism given by $T_w\mapsto w$ ($w\in W$) and $x^\mu\mapsto x^\mu$ ($\mu\in Q^\vee$).
\begin{remark}
The cross relations \eqref{crossX} in $\mathbb{H}$ can alternatively be written as
\begin{equation}\label{crossexplicit}
\begin{split}
T_ix^\mu-x^{s_i\mu}T_i&=(k_i-k_i^{-1})\Bigl(\frac{1-x^{-\alpha_i(\mu)\alpha_i^\vee}}{1-x^{\alpha_i^\vee}}\Bigr)x^\mu\qquad (1\leq i\leq r),\\
T_0x^\mu-q_{\varphi}^{\varphi(\mu)}x^{s_\varphi\mu}T_0&=(k_0-k_0^{-1})\Bigl(\frac{1-(q_\varphi x^{-\varphi^\vee})^{\varphi(\mu)}}{1-q_\varphi x^{-\varphi^\vee}}\Bigr)x^\mu
\end{split}
\end{equation}
in view of \eqref{saxy}.
\end{remark}

For $w\in W$ with reduced expression $w=s_{j_1}\cdots s_{j_\ell}$
($0\leq j_i\leq \ell$) we write 
\begin{equation*}
T_w:=T_{j_1}\cdots T_{j_\ell}\in\mathbb{H},
\end{equation*}
which is independent of the choice of reduced expression, due to the braid relations \eqref{br}.
It follows from \eqref{fundlength} that there exists a unique group homomorphism 
\[
Q^\vee\rightarrow \mathbb{H}^\times,\qquad \mu\mapsto Y^\mu
\]
such that
\begin{equation}\label{YT}
Y^\mu=T_{\tau(\mu)}\qquad \forall\,\mu\in Q^{\vee}\cap\overline{E}_+.
\end{equation}
Let 
\[
\chi_\pm: \Phi_0\rightarrow\{0,1\}
\] 
be the characteristic function of $\Phi_0^\pm$, then we have
\begin{equation}\label{YT0}
Y^{v^{-1}\varphi^\vee}=T_v^{-1}T_0^{\chi(v^{-1}\varphi)}T_{s_\varphi v}\qquad (v\in W_0),
\end{equation}
where 
\begin{equation}\label{chidef}
\chi:=\chi_+-\chi_-
\end{equation}
(see, e.g., \cite[(3.3.6)]{Ma}). In particular, 
\[
Y^{\varphi^\vee}=T_0T_{s_\varphi}.
\]

The Poincar{\'e}-Birkhoff-Witt (PBW) theorem \cite[Thm. 3.2.1]{Ch} for $\mathbb{H}$ states that 
\[
\{x^\mu T_vY^\nu\,\, | \,\,
v\in W_0\,\,\&\,\, \mu,\nu\in Q^\vee\}
\]
as well as $\{x^\mu T_w\,\, | \,\, \mu\in Q^\vee, w\in W\}$ are linear $\mathbf{F}$-linear bases of $\mathbb{H}$. In particular, the subalgebra of $\mathbb{H}$ spanned by $x^\mu$ ($\mu\in Q^\vee$) is isomorphic to $\mathcal{P}$,
as is the subalgebra $\mathcal{P}_Y$ of $\mathbb{H}$ spanned by $Y^\mu$ ($\mu\in Q^\vee$). For $p=\sum_{\mu\in Q^\vee}d_\mu x^\mu\in\mathcal{P}$ we will write 
\[
p(Y^{\pm 1}):=\sum_{\mu\in Q^\vee}d_\mu Y^{\pm\mu}\in\mathcal{P}_Y.
\]

The {\it finite Hecke algebra} $H_0=H_0(\mathbf{k})$ is the subalgebra of $\mathbb{H}$ generated by
$T_1,\ldots,T_r$. It has $\{T_v\,\, | \,\, v\in W_0\}$ as a $\mathbf{F}$-linear basis. The defining relations in terms of $T_1,\ldots,T_r$ are the Hecke relations \eqref{Hr} and the braid relations \eqref{br} for $1\leq j\not=j^\prime\leq r$.

The {\it affine Hecke algebra} $H=H(\mathbf{k})$ is the subalgebra of $\mathbb{H}$ generated
by $T_0,\ldots,T_r$. It has $\{T_w\,\, | \,\, w\in W\}$ as a $\mathbf{F}$-linear basis. The defining relations in terms of the generators $T_0,\ldots,T_r$ are the Hecke relations and the braid relations.
Another $\mathbf{F}$-linear basis of $H$ is $\{T_vY^\mu\,\, | \,\, v\in W_0\,\, \&\,\, \mu\in Q^\vee\}$.
The defining relations of $H$ in terms of the subalgebra $H_0$ and $\mathcal{P}_Y$ are the {\it Bernstein-Zelevinsky cross relations} 
\begin{equation}\label{crossrelation}
Y^\mu T_i-T_iY^{s_i\mu}=(k_i-k_i^{-1})\Bigl(\frac{Y^\mu-Y^{s_i\mu}}
{1-Y^{-\alpha_i^\vee}}\Bigr)
\end{equation}
for $i=1,\ldots,r$ and $\mu\in Q^\vee$, see \cite{Lu}.

The following theorem is due to Cherednik, see, e.g., \cite[\S 3.3.2]{Ch}.
\begin{theorem}\label{deltatheorem}
There exists a unique anti-algebra involution $\delta: \mathbb{H}\rightarrow \mathbb{H}$ satisfying
$\delta(Y^\mu)=x^{-\mu}$ \textup{(}$\mu\in Q^\vee$\textup{)} and $\delta(T_i)=T_i$ \textup{(}$1\leq i\leq r$\textup{)}. 
\end{theorem}
We call $\delta$ the {\it duality anti-involution} of $\mathbb{H}$.  Note that 
\begin{equation}\label{delta0}
\delta(T_0)=T_{s_\varphi}^{-1}x^{-\varphi^\vee}.
\end{equation}
It is sometimes convenient to rewrite $\delta(T_0)$ in terms of the element
\begin{equation}\label{Uzero}
U_0:=q_\varphi^{-1}x^{\varphi^\vee}T_0^{-1}=q_\varphi^{-1}x^{\varphi^\vee} T_{s_\varphi}Y^{-\varphi^\vee}\in\mathbb{H},
\end{equation}
which satisfies 
\begin{equation}\label{deltaUzero}
\delta(T_0^{-1})=q_\varphi U_0Y^{\varphi^\vee},\qquad \delta(T_0)=
Y^{-\varphi^\vee}T_0x^{-\varphi^\vee}
\end{equation}
and the Hecke relation $(U_0-k_0)(U_0+k_0^{-1})=0$. The outer automorphism of $\mathbb{H}$, defined as conjugation by the Gaussian in the polynomial realisation of $\mathbb{H}$, maps $T_0$ to a scalar multiple of $U_0^{-1}$ (see, e.g., \cite{ChInt}).

\subsection{Intertwiners}\label{intSection}
We recall in this subsection well known facts about Cheredniki's $X$-intertwiners and $Y$-intertwi\-ners for the double affine Hecke algebra $\mathbb{H}$, see, e.g., \cite[\S 3.3.3]{Ch}.
For $0\leq j\leq r$ write
\begin{equation}\label{buildingXintertwiners}
S_j^X:=\bigl(x^{\alpha_j^\vee}-1\bigr)T_j+k_j-k_j^{-1}\in\mathbb{H}.
\end{equation}
The following theorem introduces the $X$-intertwiners $S_w^X\in\mathbb{H}$ ($w\in W$).
\begin{theorem}\label{intertwinerTHM}
For $w\in W$ set
\[
S_w^X:=S_{j_1}^X\cdots S_{j_\ell}^X\in\mathbb{H},
\] 
where $w=s_{j_1}\cdots s_{j_\ell}$ \textup{(}$0\leq j_i\leq r$\textup{)} is a reduced expression. Then $S_w^X$ is well defined and
\begin{equation}\label{Iprop}
\begin{split}
S_w^Xp&=w(p)S_w^X,\\
S_{w^{-1}}^XS_{w}^X&=\prod_{a\in\Pi(w)}\bigl(k_a^{-1}-k_ax^{a^\vee}\bigr)\bigl(k_a^{-1}-k_ax^{-a^\vee}\bigr)
\end{split}
\end{equation}
in $\mathbb{H}$ for $p\in\mathcal{P}$ and $w\in W$.
\end{theorem}
Let $\mathbb{H}^{X-\textup{loc}}$ be the ring of fractions of $\mathbb{H}$ with respect to the multiplicative set $\mathcal{P}^\times$. The canonical algebra embedding $\mathcal{P}\hookrightarrow\mathbb{H}^{X-\textup{loc}}$ extends to an algebra embedding $\mathcal{Q}\hookrightarrow
\mathbb{H}^{X-\textup{loc}}$. Multiplication defines a $\mathbf{F}$-linear isomorphism
\begin{equation}\label{factorloc}
\mathcal{Q}\otimes_{\mathbf{F}}H\overset{\sim}{\longrightarrow}\mathbb{H}^{X-\textup{loc}}.
\end{equation}
The cross relations in $\mathbb{H}^{X-\textup{loc}}$ are 
\begin{equation}\label{crossloc}
T_jf-s_j(f)T_j=(k_j-k_j^{-1})\Bigl(\frac{f-s_j(f)}{1-x^{\alpha_j^\vee}}\Bigr)
\end{equation}
for $0\leq j\leq r$ and $f\in\mathcal{Q}$. Set 
\begin{equation}\label{normS}
\widetilde{S}_j^X:=\frac{1}{k_jx^{\alpha_j^\vee}-k_j^{-1}}S_j^X=
\Bigl(\frac{x^{\alpha_j^\vee}-1}{k_jx^{\alpha_j^\vee}-k_j^{-1}}\Bigr)T_j+
\Bigl(\frac{k_j-k_j^{-1}}{k_jx^{\alpha_j^\vee}-k_j^{-1}}\Bigr)\in\mathbb{H}^{X-\textup{loc}}
\end{equation}
for $0\leq j\leq r$. Then $(S_j^X)^2=1$ for $j=0,\ldots,r$.
The normalised  $X$-intertwiners $\widetilde{S}_w^X\in\mathbb{H}^{X-\textup{loc}}$ ($w\in W$)
are defined by
\[
\widetilde{S}_w^X:=\widetilde{S}_{j_1}^X\cdots\widetilde{S}_{j_\ell}^X,
\]
where $w=s_{j_1}\cdots s_{j_\ell}$ may now be any expression of $w\in W$ as product of simple reflections. 
Note that 
\begin{equation}\label{SwX}
\widetilde{S}_w^X=S_w^X\prod_{a\in\Pi(w)}\Bigl(\frac{1}{k_ax^{-a^\vee}-k_a^{-1}}\Bigr).
\end{equation}
The following result is from \cite[\S 3.1]{S0}.
\begin{theorem}\label{locTHM}
The assignment 
\[
w\otimes f\mapsto\widetilde{S}_w^Xf\qquad (w\in W, f\in\mathcal{Q})
\] 
defines an isomorphism $\ss: W\ltimes\mathcal{Q}\overset{\sim}{\longrightarrow}\mathbb{H}^{X-\textup{loc}}$ of algebras.
\end{theorem}
Write $\mathbb{H}^{Y-\textup{loc}}$ for the ring of fractions of $\mathbb{H}$ with respect
to the multiplicative set $\mathcal{P}_Y^\times$. The duality anti-involution $\delta$ extends to an algebra anti-isomorphism $\mathbb{H}^{X-\textup{loc}}\rightarrow\mathbb{H}^{Y-\textup{loc}}$, which we again denote by $\delta$. The $Y$-intertwiners are defined by
\[
S_j^Y:=\delta(S_j^X),\qquad S_{w}^Y:=\delta(S_{w^{-1}}^X),
\]
and the normalised versions by
$\widetilde{S}_j^Y:=\delta(\widetilde{S}_j^X)$ and $\widetilde{S}_w^Y:=\delta(\widetilde{S}_{w^{-1}}^X)$ for $0\leq j\leq r$ and $w\in W$. Note that $S_w^Y\in\mathbb{H}$ and $\widetilde{S}_w^Y\in\mathbb{H}^{Y-\textup{loc}}$. Furthermore
\[
S_w^Y=S_{j_1}^YS_{j_2}^Y\cdots  S_{j_\ell}^Y
\]
if $w=s_{j_1}\cdots s_{j_\ell}$ is a reduced expression of $w\in W$. The same holds true for $\widetilde{S}_w^Y$ (in this case it holds true for any expression $w=s_{j_1}\cdots s_{j_\ell}$ as product of simple reflections).
For later reference we write out explicitly some of the key formulas for the $Y$-intertwiners.
 
For $0\leq j\leq r$ we have
\begin{equation}\label{buildingYintertwiners}
S_j^Y=\delta(T_j)\bigl((Y^{-1})^{\alpha_j^\vee}-1\bigr)+k_j-k_j^{-1}.
\end{equation}
Note here that for $j=0$ we are evaluating the polynomial $x^{\alpha_0^\vee}=q_\varphi x^{-\varphi^\vee}$ in $Y^{-1}$, so that $(Y^{-1})^{\alpha_0^\vee}=q_\varphi Y^{\varphi^\vee}$ (which is {\it not} equal to $Y^{-\alpha_0^\vee}= q_\varphi^{-1}Y^{\varphi^\vee}$). The Hecke relations \eqref{Hr} provide the
alternative expression 
\begin{equation}\label{buildingYintertwinersinverse}
S_j^Y=\delta(T_j^{-1})\bigl((Y^{-1})^{\alpha_j^\vee}-1\bigr)+(k_j-k_j^{-1})(Y^{-1})^{\alpha_j^\vee}.
\end{equation}
An expression for $S_0$ in terms of $U_0$ is
\begin{equation}\label{buildingYintertwinersTzero}
S_0^Y=\bigl(U_0(q_\varphi Y^{\varphi^\vee}-1)+k_0-k_0^{-1}\bigr)q_\varphi Y^{\varphi^\vee}.
\end{equation}
For $p\in\mathcal{P}$ and $w\in W$ we have
\begin{equation}\label{IpropY}
\begin{split}
S_w^Yp(Y^{-1})&=w(p)(Y^{-1})S_w^Y,\\
S_{w^{-1}}^YS_{w}^Y&=\prod_{a\in\Pi(w)}\bigl(k_a^{-1}-k_a(Y^{-1})^{a^\vee}\bigr)\bigl(k_a^{-1}-k_a(Y^{-1})^{-a^\vee}\bigr)
\end{split}
\end{equation}
in $\mathbb{H}$. Finally, we have
\begin{equation}\label{RelToUnnormY}
\widetilde{S}_w^Y=S_w^Y\prod_{a\in\Pi(w)}\Bigl(\frac{1}{k_a(Y^{-1})^{a^\vee}-k_a^{-1}}\Bigr)
\end{equation}
for $w\in W$.

\section{$Y$-parabolically induced cyclic $\mathbb{H}$-modules}\label{modulesection}

Throughout this section we fix a lattice $\Lambda\in\mathcal{L}$. As before (see Subsection \ref{ap}), we will omit $\Lambda$ from the notations when $\Lambda=Q^\vee$.
\subsection{Induction parameters}\label{IndPar}
\begin{definition}
For $J\subsetneq [0,r]$ write 
\begin{equation}\label{TJ}
\begin{split}
T_{\Lambda,J}&:=\{\mathfrak{t}\in T_\Lambda \,\, | \,\, \mathfrak{t}^{\alpha_j^\vee}=1\quad\,\,\, \forall\, j\in J\},\\
T^{\textup{red}}_{\Lambda,J}&:=\{\mathfrak{t}\in T_\Lambda\,\, | \,\, \mathfrak{t}^{D\alpha_j^\vee}=1\quad \forall\, j\in J\}.
\end{split}
\end{equation}
\end{definition}
Note that $T^{\textup{red}}_{\Lambda,J}\subseteq T_\Lambda$ is a subtorus.
Furthermore, $T_{\Lambda,I}=T^{\textup{red}}_{\Lambda,I}$ if $I\subseteq [1,r]$, and $T_{[1,r]}=T_{[1,r]}^{\textup{red}}=\{1_T\}$ with $1_T$ the neutral element of $T=T_{Q^\vee}$. 
If $T_{\Lambda,J}\not=\emptyset$ then $T_{\Lambda,J}$ is a $T^{\textup{red}}_{\Lambda,J}$-coset since
\[
(tt^\prime)^{a^\vee}=t^{a^\vee}t^\prime{}^{Da^\vee}
\]
for $t,t^\prime\in T_\Lambda$ and  $a\in\Phi$. 
In the following lemma we provide a natural, mild condition on $q\in\mathbf{F}$ ensuring that $T_J\not=\emptyset$ for all $J\subsetneq [0,r]$.

Let $h$ be the Coxeter number of $W_0$, i.e., $h$ is the order of $s_1s_2\cdots s_r$ in $W_0$ (see, e.g., \cite[Chpt. 3]{Hu} and \cite{Ste}).
\begin{lemma}\label{nonzeroLem}
$T_J\not=\emptyset$ for all $J\subsetneq [0,r]$ if $q$ has a $(2h)^{th}$ root in $\mathbf{F}$.
\end{lemma}
\begin{proof}
It $I\subseteq [1,r]$ then $T_I\not=\emptyset$ without restrictions on $q$ since $T_I$ contains $1_T$. 

Fix now a subset $J\subsetneq [0,r]$ containing $0$. Write $J_0:=[1,r]\cap J$, which is a proper subset of $[1,r]$. For $\mathfrak{t}\in T$ write $c_i:=\mathfrak{t}^{\alpha_i^\vee}$ ($i\in [1,r]$). Then $\mathfrak{t}\in T_J$
iff $c_i=1$ for $i\in J_0$ and 
\begin{equation}\label{nt}
\prod_{i\in [1,r]\setminus J_0}c_i^{m_i}=q_\varphi,
\end{equation}
where $m_i\in\mathbb{Z}_{>0}$ are the coefficients in the expansion $\varphi^\vee=\sum_{i=1}^rm_i\alpha_i^\vee$ of $\varphi^\vee$ in simple co-roots. It thus remains to show that there exists $c_i\in\mathbf{F}$ ($i\in [1,r]\setminus J_0$) satisfying \eqref{nt} when $q$ has a $(2h)^{th}$ root in $\mathbf{F}$.

It is clear from \eqref{nt} that such $c_i$ exists if 
$q_\varphi$ has a $m_\ell^{th}$ root in $\mathbf{F}$ for all $\ell\in [1,r]$. Since $q_\varphi$ is an integer power of $q$ (due to our choice of normalisation of the root lengths, cf. Subsection \ref{subsect_root}), this is the case when $q$ has a $m(\varphi)^{th}$ root in $\mathbf{F}$, where
\[
m(\varphi):=\textup{lcm}(m_i).
\] 
By direct inspection using table 2 of \cite[\S 12]{Hu0}, we have $m(\varphi)=1$ for types $\textup{A}_r$ and $\textup{B}_r$, $m(\varphi)=2$ for types $\textup{C}_r$, $\textup{D}_r$ and $\textup{G}_2$, $m(\varphi)=6$ for types $\textup{E}_6$ and $\textup{F}_4$, $m(\varphi)=12$ for type $\textup{E}_7$ and $m(\varphi)=60$ for type $\textup{E}_8$. 
Comparing with 
 the Coxeter numbers, which are explicitly listed in \cite[\S 3.18, Table 2]{Hu},
 one checks by direct inspection that
\begin{equation}\label{mh}
m(\varphi)\vert 2h.
\end{equation}
Hence $T_J\not=\emptyset$ if
 $q$ has a $(2h)^{th}$ root in $\mathbf{F}$.
\end{proof}

Let  $W^{\textup{red}}_J\subseteq W_0$ be the image of the parabolic subgroup $W_J\subseteq W$ under the group epimorphism $D: W\twoheadrightarrow W_0$. We then have
\begin{equation}\label{inclusions}
T_{\Lambda,J}\subseteq T_\Lambda^{W_J},\qquad T^{\textup{red}}_{\Lambda,J}\subseteq T_\Lambda^{W^{\textup{red}}_J}.
\end{equation}
Here $T^G_\Lambda\subseteq T_\Lambda$, for a subgroup $G\subseteq W$, denotes the set of $G$-fixed elements in $T_\Lambda$.
By \eqref{actiononT}, the inclusions \eqref{inclusions} are equalities if $D\alpha_j(\Lambda)=\mathbb{Z}$ for all $j\in J$.

Note that $T_{\Lambda,J}$ is a $T_{\Lambda,[1,r]}$-subset of $T_\Lambda$ for all $J\subsetneq [0,r]$, since $T_{\Lambda,[1,r]}=T_{\Lambda,[1,r]}^{\textup{red}}$. Write
\begin{equation}\label{widetildeTJ}
\widetilde{T}_\Lambda:=\{\widetilde{s}\in T_{\Lambda} \,\, | \,\, \widetilde{s}\vert_{\Lambda\cap E_{\textup{co}}}\equiv 1\},\qquad
\widetilde{T}_{\Lambda,J}:=T_{\Lambda,J}\cap\widetilde{T}_\Lambda,
\end{equation}
which are $\widetilde{T}_{\Lambda,[1,r]}$-subsets of $T_\Lambda$. The following lemma, which explains how characters $t\in T_J$ can be lifted to $T_\Lambda$,
will be useful when we discuss the quasi-polynomial representation for extended double affine Hecke algebras in Section \ref{Extsection}.

\begin{lemma}\label{resLambdaQ} Let $J\subsetneq [0,r]$.
\begin{enumerate}
\item We have injective maps
\[
T_{\Lambda,J}/T_{\Lambda,[1,r]}\hookrightarrow T_J,\qquad \widetilde{T}_{\Lambda,J}/\widetilde{T}_{\Lambda,[1,r]}\hookrightarrow T_J
\]
defined by $sT_{\Lambda,[1,r]}\mapsto s\vert_{Q^\vee}$ and $s\widetilde{T}_{\Lambda,[1,r]}\mapsto s\vert_{Q^\vee}$, respectively.
\item The maps in \textup{(1)} are bijective if the restriction map $T_{\textup{pr}_{E^\prime}(\Lambda),J}\rightarrow T_J$, $s\mapsto s\vert_{Q^\vee}$ is surjective.
\end{enumerate}
\end{lemma}
The condition in (2) is 
for instance met when 
$\mathbf{F}$ is algebraically closed.
\begin{proof}
(1) This is clear.\\
(2) For $s\in T_{\textup{pr}_{E^\prime}(\Lambda)}$ define $\widetilde{s}\in T_\Lambda$ by 
\begin{equation}\label{LambdatoQ}
\widetilde{s}^{\,\lambda}:=s^{\,\textup{pr}_{E^\prime}(\lambda)}\qquad\quad (\lambda\in\Lambda).
\end{equation}
It gives rise to a bijective map
\begin{equation}\label{maptoQ}
T_{\textup{pr}_{E^\prime}(\Lambda),J}\overset{\sim}{\longrightarrow}\widetilde{T}_{\Lambda,J},\qquad s\mapsto \widetilde{s}
\end{equation}
satisfying $s\vert_{Q^\vee}=\widetilde{s}\vert_{Q^\vee}$
for all $s\in T_{\textup{pr}_{E^\prime}(\Lambda),J}$, 
cf. Corollary \ref{ilacor}. Combined with the assumption that the restriction map $T_{\textup{pr}_{E^\prime}(\Lambda),J}\rightarrow T_J$, $s\mapsto s\vert_{Q^\vee}$ is surjective, we conclude that the maps in (1) are surjective.
\end{proof}

The following subsets of $T_\Lambda$ will serve as induction parameters for a class of cyclic double affine Hecke algebra modules, to be defined shortly.
\begin{definition}
For $J\subsetneq [0,r]$ we set
\begin{equation}\label{LJ}
L_{\Lambda,J}:=\{t\in T_\Lambda \,\, | \,\, t^{\alpha_j^\vee}=k_j^{-2}\quad\,\,\,\, \forall\, j\in J\}.
\end{equation}
\end{definition}
Note that the condition $t^{\alpha_0^\vee}=k_0^{-2}$ reduces to $t^{\varphi^\vee}=q_\varphi k_\varphi^2$. 

We next construct an explicit base element $\mathfrak{s}_J\in T_{P^\vee}$ such that 
\[
L_{\Lambda,J}=\mathfrak{s}_{J} T_{\Lambda,J}
\]
in $T_\Lambda$, where $\mathfrak{s}_{J}$ is regarded as multiplicative character of $\Lambda$ via Corollary \ref{ilacor}.

The image $\alpha(C^J)$ of the restriction of the root $\alpha\in\Phi_0^+$ to the face $C^J$ is either $\{0\}$, $\{1\}$, or it is contained in the open interval $(0,1)$ (see the proof of Lemma \ref{faceconstantlem}). We record these three cases by the indicator function
\begin{equation}\label{etaJ}
\eta_J(\alpha):=
\begin{cases}
-1\qquad&\hbox{if }\, \alpha(C^J)=\{0\},\\
0\qquad&\hbox{if }\, \alpha(C^J)\subseteq (0,1),\\
1\qquad&\hbox{if }\, \alpha(C^J)=\{1\}. 
\end{cases}
\end{equation}
It can be alternatively described as follows.
Set
\begin{equation}\label{J0}
J_0:=[1,r]\cap J,
\end{equation}
and denote by $\Phi_{0,J_0}\subseteq \Phi_0$ the parabolic root subsystem with basis $\{\alpha_i\}_{i\in J_0}$. Write 
$\Phi_{0,J_0}^+:=\Phi_0^+\cap\Phi_{0,J_0}$.  For $\alpha\in\Phi_0^+$ denote by $n_i(\alpha)$ \textup{(}$1\leq i\leq r$\textup{)} 
the nonnegative integers such that $\alpha=\sum_{i=1}^rn_i(\alpha)\alpha_i$. 
\begin{lemma}\label{etaJdef}
For $J\subsetneq [0,r]$ we have
\begin{equation*}
\eta_J(\alpha)=
\begin{cases}
-1\qquad &\hbox{if }\, \alpha\in\Phi_{0,J_0}^+,\\
0\qquad &\hbox{if }\, \alpha\in\Phi_0^+\setminus (\Phi_0^{+,J}\cup\Phi_{0,J_0}^+),\\
1\qquad &\hbox{if }\, \alpha\in\Phi_0^{+,J},
\end{cases}
\end{equation*}
where
\begin{equation}\label{setB}
\Phi_0^{+,J}:=
\begin{cases} 
\{\alpha\in\Phi_0^+ \,\, | \,\, n_i(\alpha)=n_i(\varphi)\,\,\,\forall\, i\in [1,r]\setminus J_0\}
\qquad &\hbox{ if }\, 0\in J,\\
\emptyset \qquad &\hbox{ if }\, 0\not\in J.
\end{cases}
\end{equation}
\end{lemma}
\begin{proof}
This is similar to the proof of Lemma \ref{faceconstantlem}.  
\end{proof}
\begin{definition}
For $\Lambda\in\mathcal{L}$ and $J\subsetneq [0,r]$ define the base-point $\mathfrak{s}_J\in T_{\Lambda}$ by 
\begin{equation}\label{fraksprime}
\mathfrak{s}_J:=\prod_{\alpha\in\Phi_0^+}k_\alpha^{\eta_J(\alpha)\alpha},
\end{equation}
i.e., $\mathfrak{s}_J$ is the multiplicative character mapping $\lambda\in\Lambda$ to $\prod_{\alpha\in\Phi_0^+}k_\alpha^{\eta_J(\alpha)\alpha(\lambda)}$.
\end{definition}
Special cases are 
\[
\mathfrak{s}_{[1,r]}=\prod_{\alpha\in\Phi_0^+}k_\alpha^{-\alpha},\qquad\quad \mathfrak{s}_{\{0\}}=k_\varphi^{\varphi},\qquad\quad \mathfrak{s}_\emptyset=1_{T_{\Lambda}},
\] 
with $1_{T_{\Lambda}}$ the neutral element of $T_{\Lambda}$. Note that 
$\mathfrak{s}_J\vert_\Lambda=j_\Lambda(\mathfrak{s}_J\vert_{P^\vee})\in T_\Lambda$ since $\mathfrak{s}_J$ maps $\textup{pr}_{E_{\textup{co}}}(\Lambda)$ to $1$.

\begin{proposition}\label{lemmabasepoint}
For $J\subsetneq [0,r]$ we have
\[
\mathfrak{s}_J^{D\alpha_j^\vee}=k_j^{-2}\qquad \forall\, j\in J.
\]
In particular, $L_{\Lambda,J}=\mathfrak{s}_{J}T_{\Lambda,J}$ in $T_\Lambda$ for all $\Lambda\in\mathcal{L}$.
\end{proposition}
 \begin{proof}
 Let $i\in J_0$. Since $s_iC^J=C^J$ we have $\eta_J(s_i\alpha)=\eta_J(\alpha)$ for all $\alpha\in\Phi_0^+\setminus\{\alpha_i\}$. Furthermore, $\eta_J(\alpha_i)=-1$, and hence
 \[
 s_i\mathfrak{s}_J=k_i^{-\eta_J(\alpha_i)\alpha_i}\prod_{\alpha\in\Phi_0^+\setminus\{\alpha_i\}}k_\alpha^{\eta_J(\alpha)s_i\alpha}=
 k_i^{2\alpha_i}\mathfrak{s}_J.
 \]
 By \eqref{actiononT} we conclude that $\mathfrak{s}_J^{\alpha_i^\vee}=k_i^{-2}$.
 
 Assume now that $0\in J$. It then remains to show that $\mathfrak{s}_J^{\varphi^\vee}=k_\varphi^2$. By \eqref{thetalength} we have
 \[
 s_\varphi\mathfrak{s}_J=k_\varphi^{-2\eta_J(\varphi)\varphi}\Bigl(\prod_{\alpha\in\Phi_0^+:\, \alpha(\varphi^\vee)=1}k_\alpha^{-\eta_J(\alpha)}\Bigr)^\varphi\mathfrak{s}_J.
 \]
 Clearly $\eta_J(\varphi)=1$, and the $2$-group $\langle -s_\varphi\rangle$ acts freely on $\{\alpha\in\Phi_0^+\,\, | \,\, \alpha(\varphi^\vee)=1\}$ since
 $\Phi_0$ is reduced. Let $Z$ be a complete set of representatives of the $\langle -s_\varphi\rangle$-orbits in $\{\alpha\in\Phi_0^+\,\, | \,\, \alpha(\varphi^\vee)=1\}$, then we conclude that
 \[
 s_\varphi\mathfrak{s}_J=k_\varphi^{-2\varphi}\Bigl(\prod_{\alpha\in Z}k_\alpha^{-(\eta_J(\alpha)+\eta_J(-\alpha+\varphi))}\Bigr)^\varphi\mathfrak{s}_J.
 \]
 But for $\alpha\in Z$ we have $\eta_J(-\alpha+\varphi)=-\eta_J(\alpha)$ (cf. Lemma \ref{etaJdef}), hence $s_\varphi\mathfrak{s}_J=k_\varphi^{-2\varphi}\mathfrak{s}_J$.
 By \eqref{actiononT} we then have $\mathfrak{s}_J^{\varphi^\vee}=k_\varphi^2$, which completes the proof of the proposition. 
 \end{proof}
 
\subsection{The $\mathbb{H}$-module $\mathbb{M}^J_t$}\label{SectionMt}
We introduce a family of $\mathbb{H}$-modules, obtained by induction from one-dimensional representations of  
$Y$-parabolic subalgebras.

\begin{definition}
Let $J\subsetneq [0,r]$. 
\begin{enumerate}
\item The $X$-parabolic subalgebra $\mathbb{H}_J^X\subseteq\mathbb{H}$ is the unital subalgebra generated
by $T_j$ \textup{(}$j\in J$\textup{)} and $\mathcal{P}$.
\item The $Y$-parabolic subalgebra $\mathbb{H}_J^Y\subseteq\mathbb{H}$ is the image of $\mathbb{H}_J^X$
under the duality anti-involution $\delta$. 
\end{enumerate} 
\end{definition}
Concretely, $\mathbb{H}_{J}^Y\subseteq\mathbb{H}$ is the unital subalgebra generated by $\delta(T_j)$ \textup{(}$j\in J$\textup{)} and 
$\mathcal{P}_Y$.
For $i\in J_0$ the generator $\delta(T_i)$ is simply $T_i$. If $0\in J$ then the algebraic generator $\delta(T_0)=Y^{-\varphi^\vee}T_0x^{-\varphi^\vee}$ of $\mathbb{H}_J^Y$
may be replaced by $U_0$, since $Y^{\varphi^\vee}\in\mathbb{H}_J^Y$.

\begin{lemma}\label{chitJ}
For $t\in L_{J}$
there exists a unique algebra homomorphism $\chi_{J,t}^X: \mathbb{H}_{J}^X\rightarrow\mathbf{F}$
satisfying  $\chi_{J,t}^X(p)=p(t)$ \textup{(}$p\in\mathcal{P}$\textup{)} and $\chi_{J,t}^X(T_j)=k_j$ \textup{(}$j\in J$\textup{)}.
\end{lemma}
\begin{proof}
Let $H_J\subseteq H$ be the parabolic sub-algebra generated by $T_j$ ($j\in J$).
A presentation of $\mathbb{H}^X_J$ in terms of $\mathcal{P}$ and $H_J$ is given by Definition \ref{defDAHA}, with the indices $j$ and $j^\prime$ now taken from the subset $J$. 

The assignment  $T_j\mapsto k_j$ ($j\in J$) defines the trivial one-dimensional representation of $H_J$. To show that it extends to a well defined algebra map $\chi_{J,t}^X: \mathbb{H}_J^X\rightarrow\mathbf{F}$ satisfying $\chi_{J,t}^X(p):=p(t)$ ($p\in\mathcal{P}$) it suffices to show that the cross relations \eqref{crossX} for $p\in\mathcal{P}$ and $j\in J$ are respected by $\chi_{J,t}^X$. 
 
 If $j\in J$ and $k_j^2=1$ then $t\in L_{J}$ implies
 $t^{\alpha_j^\vee}=1$, hence $s_jt=t$. In this case the cross relation \eqref{crossX} reduces to $T_jp=s_j(p)T_j$, which is clearly respected by $\chi_{J,t}^X$.

If $j\in J$ and $k_j^2\not=1$ then $t^{\alpha_j^\vee}=k_j^{-2}\not=1$, and a direct computation shows that
$\chi_{J,t}^X$ respects the cross relation \eqref{crossX}.
\end{proof}
\begin{corollary}
For $t\in L_{J}$ there exists a unique algebra map $\chi_{J,t}: \mathbb{H}_J^Y\rightarrow\mathbf{F}$ satisfying
\begin{equation}\label{expchitJ}
\begin{split}
\chi_{J,t}(p(Y))&=p(t^{-1})\qquad\,\,\,\, (p\in\mathcal{P}),\\
\chi_{J,t}(T_i)&=k_i\qquad\qquad\,\,\, (i\in J_0),\\
\chi_{J,t}(T_0x^{-\varphi^\vee})&=k_\varphi t^{-\varphi^\vee}\qquad \hbox{ if }\,\,\, 0\in J.
\end{split}
\end{equation}
\end{corollary}
\begin{proof}
Composing the anti-algebra isomorphism $\delta\vert_{\mathbb{H}_J^Y}: \mathbb{H}_J^Y\overset{\sim}{\longrightarrow}\mathbb{H}_J^X$ with the algebra map $\chi_{J,t}^X: \mathbb{H}_J^X\rightarrow\mathbf{F}$ gives a well defined algebra map 
\begin{equation}\label{chit}
\chi_{J,t}:=\chi_{J,t}^X\circ\delta\vert_{\mathbb{H}_J^Y}: \mathbb{H}_J^Y\rightarrow\mathbf{F}.
\end{equation}
A direct check shows that $\chi_{J,t}$ satisfies \eqref{expchitJ}. 
\end{proof}
Write 
\[
\mathbf{F}_{J,t}=\mathbf{F}1_{J,t}
\]
for the one-dimensional 
$\mathbb{H}_{J}^Y$-module with representation map 
$\chi_{J,t}$.
\begin{definition}\label{defMM}
We call
\[
\mathbb{M}^J_t:=\mathbb{H}\otimes_{\mathbb{H}_{J}^Y}\mathbf{F}_{J,t}
\]
the $Y$-parabolically induced cyclic $\mathbb{H}$-module relative to $J\subsetneq [0,r]$ and $t\in L_J$. 
\end{definition}
Write
\[
m^J_t:=1\otimes_{\mathbb{H}_{J}^Y}1_{J,t}\in\mathbb{M}^J_t.
\]
It is a cyclic vector of $\mathbb{M}^J_t$, and 
\[
hm^J_t=\chi_{J,t}(h)m^J_t\qquad\forall\, h\in\mathbb{H}_J^Y.
\]
In particular, 
$p(Y)m^J_t=p(t^{-1})m^J_t$ for all $p\in\mathcal{P}$. In Cherednik's terminology \cite[\S 3.6]{Ch},
$\mathbb{M}^J_t$ is a $Y$-cyclic $\mathbb{H}$-module, and the $\mathbb{H}$-module $\mathbb{M}^{\emptyset}_t$ ($t\in T$) is the universal $Y$-cyclic $\mathbb{H}$-module $\mathcal{I}_Y[t^{-1}]$ from \cite[\S 3.6.1]{Ch}. See \cite{V} for a geometric approach to the representation theory of $\mathbb{H}$.
\begin{remark}\label{specNJt}\hfill 
\begin{enumerate}
\item For $t\in T$ write
\begin{equation}\label{Jt}
J(t):=\{j\in [0,r] \,\, | \,\,  t^{\alpha_j^\vee}=k_j^{-2}\}.
\end{equation}
Then $L_J=\bigsqcup_{J^\prime\supseteq J}L^{J^\prime}$ (disjoint union) with 
\[
L^{J^\prime}:=\{ t\in T \,\ | \,\, J(t)=J^\prime\}.
\]
If $t^\prime\in L^{J^\prime}$ and $J^\prime\supseteq J$, then $M^{J^\prime}_{t^\prime}$ is a quotient of $M^J_{t^\prime}$. In Theorem \ref{irredTHM} we will show that $M^J_t$ is a simple $\mathbb{H}$-module for generic $t\in L^J$.
\item The $\mathbb{H}$-representations $\mathbb{M}^I_t$ ($I\subseteq [1,r]$ \& $t\in L^I$) are analogues of the so-called standard modules $M(\lambda,\textup{Triv})$ of the
rational Cherednik algebra, as defined in \cite{BEG}. Here $\lambda\in\mathfrak{h}^*$ with $\mathfrak{h}:=E\otimes_{\mathbb{R}}\mathbb{C}$, and $\textup{Triv}$ is the trivial representation of $W_{0,\lambda}$, with $W_{0,\lambda}\subseteq W_0$ the stabiliser subgroup of $\lambda$.
\end{enumerate}
\end{remark}
The``Fourier dual'' of $\mathbb{M}^J_t$ ($t\in L_J$) is the $X$-parabolically induced cyclic $\mathbb{H}$-module
\[
{}^X\mathbb{M}^J_t:=\mathbb{H}\otimes_{\mathbb{H}_J^X}\mathbf{F}^X_{J,t},
\]
with $\mathbf{F}^X_{J,t}$ the one-dimensional $\mathbb{H}_J^X$-module with representation map $\chi_{J,t}^X$. 

\begin{remark}\label{realIND}
The following explicit realisations of the
 $Y$-parabolically and $X$-para\-bo\-li\-cally induced $\mathbb{H}$-modules are known:
\begin{enumerate}
\item Cherednik's \cite{Ch} polynomial representation $\pi: \mathbb{H}\rightarrow\textup{End}(\mathcal{P})$, defined in terms of Demazure-Lusztig operators (see \eqref{Chpolrep}). It is isomorphic to $\mathbb{M}^{[1,r]}_{t_{\textup{sph}}}$
with 
\begin{equation}\label{tsph}
t_{\textup{sph}}:=\mathfrak{s}_{[1,r]}=\prod_{\alpha\in\Phi_0^+}k_\alpha^{-\alpha}\in L_{[1,r]}.
\end{equation}
Under suitable generic conditions on $q$ and $\mathbf{k}$, nonsymmetric Macdonald polynomials provide the simultaneous eigenfunctions of the commuting Cherednik operators $\pi(Y^\mu)$ ($\mu\in Q^\vee$).
\item Analytic realisation of the universal $Y$-cyclic $\mathbb{H}$-module $\mathbb{M}_t^\emptyset$ for generic $t\in T$. In this case $\mathbf{F}=\mathbb{C}$ and $0<k_a,q<1$,
and $\mathbb{M}_t^\emptyset$ is realised as subrepresentation of the space $\mathcal{M}(T)$ of meromorphic functions on the complex torus $T$, with the action again defined by Demazure-Lusztig operators. The cyclic vectors
$m_t^\emptyset$ are $\mathcal{E}(\cdot,w_0t)$ with $\mathcal{E}$ Cherednik's \cite{ChMM} global spherical function $\mathcal{E}\in\mathcal{M}(T\times T)$, also known as the basic hypergeometric function \cite[Thm. 2.13]{StAoM}. 
The simultaneous eigenfunctions for the action of the commuting Cherednik operators on $\mathcal{M}(T)$ are of the form $\mathcal{E}(\cdot,t^\prime)$ for appropriate $t^\prime\in T$.
\item Cherednik's realisation of ${}^X\mathbb{M}^J_t$ for generic $t\in L_J$ on the space of finitely supported $\mathbf{F}$-valued functions on $W/W_J$, with the action defined in terms of discrete Demazure-Lusztig type operators \cite[\S 3.4.2]{Ch}. The delta-functions are the simultaneous eigenfunctions for the action of the multiplication operators $x^\mu$ ($\mu\in Q^\vee$).
\end{enumerate}
\end{remark}

In the following section we realise the $Y$-parabolically induced $\mathbb{H}$-modules $\mathbb{M}^J_t$ ($t\in L_J$) on {\it spaces of quasi-polynomials}. The resulting simultaneous eigenfunctions for the action of $Y^\mu$ ($\mu\in Q^\vee$) become quasi-polynomial generalisations of the nonsymmetric Macdonald polynomials.

\begin{lemma}\label{mbasis}
For $t\in L_J$ set
\begin{equation}\label{mw}
m^J_{w;t}:=\delta(T_{w^{-1}})\,m^J_t\in\mathbb{M}^J_t\qquad (w\in W^J).
\end{equation}
Then $\{m^J_{w;t}\,\, | \,\, w\in W^J\}$ 
 is a  basis of $\mathbb{M}^J_t$.
\end{lemma}
\begin{proof}
The Poincar{\'e}-Birkhoff-Witt Theorem for $\mathbb{H}$ implies that $\mathbb{H}$ is a free left $\mathbb{H}_{J}^{X}$-module with basis $\{T_{w^{-1}}\,\, | \,\, w\in W^J\}$. Now apply the anti-involution $\delta$ and recall that $\mathbb{H}_{J}^Y=\delta(\mathbb{H}_{J}^X)$. 
\end{proof}
\begin{lemma}\label{mbasisvdmonomial}
For $\mu\in Q^\vee\cap\overline{E}_-$ and $v\in W_0$ such that $\tau(\mu)v\in W^J$ we have 
\[
m_{\tau(\mu)v;t}^J=x^\mu T_vm_t^J.
\]
\end{lemma}
\begin{proof}
By \eqref{lengthwmu} and \eqref{wmuneg} we have 
\[
m_{\tau(\mu)v;t}^J=\delta(T_{v^{-1}}T_{\tau(-\mu)})m_t^J.
\]
But $-\mu\in Q^\vee\cap\overline{E}_+$, hence $T_{\tau(-\mu)}=Y^{-\mu}$ by \eqref{YT}. Consequently
\[
m_{\tau(\mu)v;t}^J=\delta(Y^{-\mu})\delta(T_{v^{-1}})m_t^J=x^\mu T_vm_t^J.
\]
\end{proof}
\begin{remark}
For $w\in W$ write $\mu_w\in Q^\vee$ and $v_w\in W_0$ such that $w=\tau(\mu_w)v_w$.
We will prove in Theorem \ref{gbr}(3) that $\{x^{\mu_w}T_{v_w}m^J_t\,\,\, | \,\,\, w\in W^J\}$ is also a 
basis of $\mathbb{M}^J_t$. The quasi-polynomial realisation of $\mathbb{M}^J_t$, see Theorem \ref{gbr}, 
provides an explicit description of the $\mathbb{H}$-action on the
basis $\{x^{\mu_w}T_{v_w}m^J_t\,\,\, | \,\,\, w\in W^J\}$ of
$\mathbb{M}^J_t$, with the action described in terms of truncated Demazure-Lusztig operators.
\end{remark}

\section{Quasi-polynomial and quasi-rational representations}\label{polfct}
Let $\cc\in E$ be an element lying in the face $C^J$ of $\overline{C}_+$. In this section we will give explicit realisations of the $\mathbb{H}$-modules $\mathbb{M}^{J}_t$ ($t\in L_{J}$) on the space of quasi-polynomials with exponents in the affine Weyl group orbit $\mathcal{O}_\cc$. By localisation we also obtain nontrivial families of $W$-actions on spaces of quasi-rational functions.

\subsection{Spaces of quasi-polynomials}\label{qpsection}

The group algebra $\mathbf{F}[E]$ contains $\mathcal{P}_\Lambda$ as sub-algebra for all $\Lambda\in\mathcal{L}$. 
The formula $v(x^y):=x^{vy}$ ($v\in W_0$, $y\in E$) defines a $W_0$-action on $\mathbf{F}[E]$ by algebra automorphisms, which is compatible with the $W_0$-action on $\mathcal{P}_\Lambda$. We write
\[
x^{(y,\ell)}:=q^\ell x^y\in\mathbf{F}[E]\qquad (y,\ell)\in E\times\mathbb{Z}.
\]
For $\cc\in\overline{C}_+$ write
\[
\Pc:=\bigoplus_{y\in\mathcal{O}_\cc}\mathbf{F}x^y.
\]
We call $\Pc$ the space of quasi-polynomials with exponents in $\mathcal{O}_\cc$.

\begin{definition}\label{defty}
Let $\Lambda\in\mathcal{L}$, $\cc\in C^J$ and $\mathfrak{t}\in T_{\Lambda,J}$.
For $y\in\mathcal{O}_\cc$, define
\begin{equation}\label{deftyform}
\mathfrak{t}_y:=w_y\mathfrak{t}\in T_\Lambda.
\end{equation}
\end{definition}
Note that 
\[
\mathfrak{t}_{wy}=w\mathfrak{t}_y \qquad\quad (w\in W,\, y\in\mathcal{O}_{\cc})
\]
since $W^J=\{w_y\}_{y\in\mathcal{O}_\cc}$ and $\mathfrak{t}\in T_J\subseteq T^{W_J}$.
In particular,
\begin{equation}\label{integralshiftt}
\mathfrak{t}_{y+\mu}=q^\mu\mathfrak{t}_y\qquad (y\in\mathcal{O}_\cc,\, \mu\in Q^\vee).
\end{equation}

In the following lemma we identify $W\ltimes\mathcal{P}$ with $\mathbb{H}(\mathbf{1},q)$ by the isomorphism \eqref{isodaWgDAHA}.
\begin{lemma}\label{actiondeform}
For $\cc\in C^J$ and $\mathfrak{t}\in T_J$ the formulas $K_{\mathfrak{t}}x^y:=qx^y$ and
\begin{equation}\label{xyWt}
v_{\mathfrak{t}}x^y:=x^{vy},\quad \tau(\mu)_\mathfrak{t}x^y:=\mathfrak{t}_y^{-\mu}x^y,\quad
\nu_{\mathfrak{t}}(x^y):=x^{y+\nu}
\qquad (v\in W_0,\,\mu,\nu\in Q^\vee)
\end{equation}
for $y\in\mathcal{O}_\cc$ define a $\mathbf{F}$-linear $\mathbb{W}$-action on $\mathcal{P}^{(\cc)}$. The resulting $\mathbf{F}[\mathbb{W}]$-action on $\mathcal{P}^{(\cc)}$
descends to an action of $W\ltimes\mathcal{P}$ via the isomorphism \eqref{isodW}. Furthermore,
\[
\mathbb{M}_{\mathfrak{t}}^J\overset{\sim}{\longrightarrow}(\mathcal{P}^{(\cc)},\cdot_{\mathfrak{t}})\quad \hbox{ as modules over }\, \mathbb{H}(\mathbf{1},q)\simeq W\ltimes\mathcal{P}
\]
with the isomorphism defined by $(x^\mu v)m_{\mathfrak{t}}^J\mapsto x^{\mu+v\cc}$ \textup{(}$\mu\in Q^\vee,\, v\in W_0$\textup{)}.
\end{lemma}
\begin{proof}
It is a direct check that \eqref{xyWt} defines a $\mathbb{W}$-action on $\mathcal{P}^{(\cc)}$. For the last statement, note that $\tau(\mu)_{\mathfrak{t}}x^\cc=\mathfrak{t}^{-\mu}x^\cc$ for all $\mu\in Q^\vee$, and for $i\in J_0$ (see \eqref{J0}),
\[
\delta(s_i)_{\mathfrak{t}}x^\cc=s_{i,\mathfrak{t}}x^\cc=x^{s_i\cc}=x^\cc.
\]
If $0\in J$, then $\cc=s_0\cc=s_\varphi\cc+\varphi^\vee$ and $\delta(s_0)=\varphi^\vee s_\varphi\in\mathbb{W}$, hence
\[
\delta(s_0)_{\mathfrak{t}}x^\cc=(\varphi^\vee)_{\mathfrak{t}}x^{s_\varphi\cc}=x^{s_\varphi\cc+\varphi^\vee}=x^\cc.
\]
Hence there exists a unique morphism $\mathbb{M}_{\mathfrak{t}}^J\rightarrow (\mathcal{P}^{(\cc)},\cdot_{\mathfrak{t}})$ of $W\ltimes\mathcal{P}$-modules satisfying
$m_{\mathfrak{t}}^J\mapsto x^\cc$. Then $(x^\mu v)m_{\mathfrak{t}}^J$ maps to $(\mu v)_{\mathfrak{t}}x^\cc=x^{\mu+v\cc}$ for $\mu\in Q^\vee$ and $v\in W_0$.
\end{proof}
In the special case $\cc=0$ we necessarily have $\mathfrak{t}=1_T$. In this case the $W\ltimes\mathcal{P}$-action on $\mathcal{P}^{(0)}=\mathcal{P}$ reduces to the action introduced in Subsection \ref{ap},
\[
w_{1_T}x^\mu=w(x^\mu)\qquad\quad (w\in W,\, \mu\in Q^\vee).
\]

Note that formula \eqref{saxy} generalises to
\begin{equation}\label{saxyt}
s_{a,\mathfrak{t}}x^y=\mathfrak{t}_y^{-\ell\alpha^\vee}x^{s_\alpha y}\qquad (a=(\alpha,\ell)\in\Phi,\,\, y\in\mathcal{O}_\cc).
\end{equation}
In particular, $s_{0,\mathfrak{t}}x^y=\mathfrak{t}_y^{\varphi^\vee}x^{s_\varphi y}$ for $y\in\mathcal{O}_\cc$. 

The natural extension of formula \eqref{actx} to the present context is as follows.
\begin{lemma}\label{datumokcor}
Let $\cc\in C^J$, $\mathfrak{t}\in T_J$, $y\in\mathcal{O}_\cc$ and $a\in\Phi$ such that $a(y)\in\mathbb{Z}$. Then 
\begin{enumerate}
\item
$\mathfrak{t}_y^{Da^\vee}=q_{Da}^{Da(y)}$,
\item 
$ s_{a,\mathfrak{t}}x^y=x^{y-Da(y)a^\vee}$.
\end{enumerate}
\end{lemma}
\begin{proof}
Write $\alpha:=Da\in\Phi_0$ and $n:=\alpha(y)\in\mathbb{Z}$. The affine root $b:=w_y^{-1}(\alpha,-n)\in\Phi$
satisfies $b(\cc_y)=0$. Since $\cc_y\in C^J$ we conclude that $b\in\Phi_J$. 
Note that $\{\alpha_0^\vee,\ldots,\alpha_r^\vee\}$ is a basis for the affine coroot system $\Phi^\vee$, hence the coroot $b^\vee$ lies in $\bigoplus_{j\in J}\mathbb{Z}\alpha_j^\vee$, and hence $\mathfrak{t}^{b^\vee}=1$. Then
\[
1=\mathfrak{t}^{b^\vee}=(w_y\mathfrak{t})^{(\alpha,-n)^\vee}=q_\alpha^{-n}\mathfrak{t}_y^{\alpha^\vee},
\]
where the second equality follows from \eqref{actx}. Hence $\mathfrak{t}_y^{\alpha^\vee}=q_\alpha^{\alpha(y)}$, which proves (1).

Set $\ell:=a(0)\in\mathbb{Z}$. Then $a=(\alpha,\ell)$ and 
\[
s_{a,\mathfrak{t}}x^y=\mathfrak{t}_y^{-\ell\alpha^\vee}x^{s_\alpha y}=q_\alpha^{-\ell\alpha(y)}x^{s_\alpha y}=x^{y-Da(y)a^\vee}
\]
where we have used \eqref{saxyt} for the first equality, (1) for the second equality,
and \eqref{saxy} and \eqref{actx} for the last equality. This proves (2).
\end{proof}

Let $T_{\mathbb{R}}^\vee$ be the space of $Q^\vee$-orbits in $E$, with $Q^\vee$ acting on $E$ by translations. It is a real torus admitting a natural $W_0$-action.
For $\gamma\in T^\vee_{\mathbb{R}}$
set
\[
\mathcal{P}_\gamma:=\bigoplus_{y\in\gamma}\mathbf{F}x^y,
\]
which is a cyclic $\mathcal{P}$-submodule of $\mathbf{F}[E]$. It equals $\mathcal{P}$ when $\gamma$
is the neutral element of $T_{\mathbb{R}}^\vee$. Writing $[\cc]:=\cc+Q^\vee\in T_{\mathbb{R}}^\vee$, we have
\[
\Pc=\bigoplus_{\gamma\in W_0[\cc]}\mathcal{P}_\gamma
\]
since $W_0[\cc]\simeq W\cc/Q^\vee$, which allows one to think of $\Pc$ as vector-valued polynomial functions on $T$.

The spaces $\mathcal{P}_\gamma$ ($\gamma\in T_{\mathbb{R}}^\vee$) naturally arise in the Bethe ansatz for the trigonometric Gaudin model \cite{MV,MTV}. They also serve as the ambient spaces of the multivariable Baker-Akhiezer functions, see, e.g., \cite{Cha, ChE,S}. Multivariable Baker-Akhiezer functions are terminating power series solutions of the spectral problem for the Macdonald operators, which exist for special values of $k_j$ (for instance, when $k_j$ a nonpositive integral power of $q_{\alpha_j}$). 
In \cite{Cha,MV,MTV}, elements in $\mathcal{P}_\gamma$ are called quasi-polynomials or quasi-exponentials. 

For rational weights $\cc\in\mathbb{Q}\otimes_{\mathbb{Z}}Q^\vee$, we may view $\Pc$ as subspace of the space
of polynomials on the torus $\textup{Hom}(\Lambda,\mathbf{F}^\times)$ for an appropriate lattice $Q^\vee\subset\Lambda\subset E^\prime$. This is a natural viewpoint in the context of
metaplectic representation theory, see Subsection \ref{msec}.

In Subsection \ref{gbrsectionStatement} we deform the $\mathfrak{t}$-dependent $W\ltimes\mathcal{P}$-action on $\Pc$ to an action of the double affine Hecke algebra $\mathbb{H}$, with the $T_j$'s acting by truncated Demazure-Lusztig operators.

\subsection{Truncated divided difference operators}\label{truncsection}

For $a\in\Phi$ let $\nabla_a$ be the linear operator on $\mathbf{F}[E]$ defined by
\begin{equation}\label{nablaj}
\nabla_a(x^y):=\left(\frac{1-x^{-\lfloor Da(y)\rfloor a^\vee}}{1-x^{a^\vee}}\right)x^y\qquad (y\in E),
\end{equation}
where $\lfloor z\rfloor\in\mathbb{Z}$ is the floor of $z\in\mathbb{R}$. Note that $\nabla_a$ is well defined, since the numerator is divisible by $1-x^{a^\vee}$ in $\mathbf{F}[E]$. The subspaces $\Pc$ ($\cc\in\overline{C}_+$) are $\nabla_a$-stable.

We will write $\nabla_j:=\nabla_{\alpha_j}$, so that
\[
\nabla_0(x^y)=\left(\frac{1-(q_\varphi^{-1}x^{\varphi^\vee})^{\lfloor -\varphi(y)\rfloor}}
{1-q_\varphi x^{-\varphi^\vee}}\right)x^y,\qquad
\nabla_i(x^y)=
\left(\frac{1-x^{-\lfloor \alpha_i(y)\rfloor\alpha_i^\vee}}{1-x^{\alpha_i^\vee}}\right)x^y
\]
for $1\leq i\leq r$ and $y\in E$. 

\begin{lemma}
For $a\in\Phi$ and $y,y^\prime\in E$ such that $a(y^\prime)\in\mathbb{Z}$ we have
\begin{equation}\label{stepder}
\begin{split}
\nabla_a(x^{y^\prime})&=\left(\frac{1-x^{-Da(y^\prime)a^\vee}}{1-x^{a^\vee}}\right)x^{y^\prime},\\
\nabla_a(x^{y+y^\prime})&-x^{y^\prime-Da(y^\prime)a^\vee}\nabla_a(x^{y})=\nabla_a(x^{y^\prime})x^{y}.
\end{split}
\end{equation}
\end{lemma}
\begin{proof}
This is a straightforward verification which we leave to the reader.
\end{proof}
It follows from \eqref{actx} and the first formula in \eqref{stepder} that $\nabla_a$ restricts to the usual divided difference operator on $\mathcal{P}_\Lambda$ ($\Lambda\in\mathcal{L}$), 
\begin{equation}\label{usualDD}
\nabla_a(x^\mu)=\frac{x^\mu-s_a(x^\mu)}{1-x^{a^\vee}}=\frac{x^\mu-x^{s_a\cdot\mu}}{1-x^{a^\vee}}
\qquad (\mu\in\Lambda).
\end{equation}
Cherednik's \cite[Thm. 3.2.1]{Ch} polynomial representation $\pi: \mathbb{H}\rightarrow\textup{End}(\mathcal{P})$ is then explicitly defined by
\begin{equation}\label{Chpolrep}
\begin{split}
\pi(T_j)x^\mu&:=k_js_j(x^\mu)+(k_j-k_j^{-1})\nabla_j(x^\mu),\\
\pi(x^\lambda)x^\mu&:=x^{\lambda+\mu}
\end{split}
\end{equation}
for $0\leq j\leq r$ and $\lambda,\mu\in Q^\vee$.  

In the next subsection we introduce families of quasi-polynomial representations $\pi_{\cc,\mathfrak{t}}: \mathbb{H}\rightarrow\textup{End}(\mathcal{P}^{(\cc)})$ 
($\cc\in C^J,\,\mathfrak{t}\in T_J$). The action of $\mathcal{P}$ will still be by multiplication operators. The action of $T_j$ will be reminiscent to formula \eqref{Chpolrep} for $\pi(T_j)$; its first term will involve the $\mathfrak{t}$-twisted $W$-action on $\mathcal{P}^{(\cc)}$, and its second term the truncated divided difference operator $\nabla_j\vert_{\mathcal{P}^{(\cc)}}$.

\subsection{The quasi-polynomial realisation of $\mathbb{M}_t^J$}\label{gbrsectionStatement}

Denote by $\chi_B: \mathbb{R}\rightarrow\{0,1\}$ the characteristic function of the subset $B\subseteq\mathbb{R}$.
Define $\eta: \mathbb{R}\rightarrow\{-1,0,1\}$ by 
\begin{equation}\label{eta}
\eta=\chi_{\mathbb{Z}_{>0}}-\chi_{\mathbb{Z}_{\leq 0}},
\end{equation}
and set
\begin{equation}\label{kvy}
\kappa_v(y):=\prod_{\alpha\in\Pi(v)}k_\alpha^{-\eta(\alpha(y))}\qquad (v\in W_0,\, y\in E).
\end{equation}
Note that $\kappa_v\in\mathcal{F}_\Sigma(E,\mathbf{F}^\times)$ by Lemma \ref{faceconstantlem}.

\begin{theorem}\label{gbr}
Let $J\subsetneq [0,r]$, $\cc\in C^J$ and $\mathfrak{t}\in T_J$.
\begin{enumerate}
\item[{\textup{(1)}}] The formulas
\begin{equation}\label{actionformulas}
\begin{split}
\pi_{\cc,\mathfrak{t}}(T_j)x^y&:=k_j^{\chi_{\mathbb{Z}}(\alpha_j(y))}s_{j,\mathfrak{t}}x^y
+(k_j-k_j^{-1})\nabla_j(x^y),\\
\pi_{\cc,\mathfrak{t}}(x^\mu)x^y&:=x^{y+\mu}
\end{split}
\end{equation}
for $0\leq j\leq r$, $\mu\in Q^\vee$
and $y\in\mathcal{O}_\cc$ turn $\mathcal{P}^{(\cc)}$ in a $\mathbb{H}$-module, which we denote by $\mathcal{P}_{\mathfrak{t}}^{(\cc)}$.
\item[{\textup{(2)}}] 
There exists a unique isomorphism 
\[
\phi_{\cc,\mathfrak{t}}: \mathbb{M}^{J}_{\mathfrak{s}_{J}\mathfrak{t}}\overset{\sim}{\longrightarrow} \Pc_{\mathfrak{t}}
\]
of $\mathbb{H}$-modules satisfying $\phi_{\cc,\mathfrak{t}}\bigl(m^{J}_{\mathfrak{s}_{J}\mathfrak{t}}\bigr)=x^\cc$. 
\item[{\textup{(3)}}] 
We have
\begin{equation}\label{phiformula}
\phi_{\cc,\mathfrak{t}}\bigl(x^\mu T_vm^{J}_{\mathfrak{s}_{J}\mathfrak{t}}\bigr)=
\kappa_v(\cc)x^{\mu+v\cc}\qquad (\mu\in Q^\vee,\, v\in W_0).
\end{equation}
In particular, $\{x^{\mu_w}T_{v_w}m^{J}_{\mathfrak{s}_{J}\mathfrak{t}}\,\, | \,\, w\in W^J\}$
is a basis of $\mathbb{M}^{J}_{\mathfrak{s}_{J}\mathfrak{t}}$. 
\end{enumerate}
\end{theorem}
Note that the $\mathbb{H}$-action on $\Pc_{\mathfrak{t}}$ reduces to the $\mathfrak{t}$-dependent $W\ltimes\mathcal{P}$-action on $\Pc$
from Lemma \ref{actiondeform} when $k_a=1$ for all $a\in\Phi$. 

Furthermore, for $\cc=0$ and $\mathfrak{t}=1_T$, $\pi_{0,1_T}: \mathbb{H}\rightarrow \textup{End}(\mathcal{P})$ is Cherednik's polynomial representation $\pi$ (see \eqref{Chpolrep}). 
In this case the theorem is due to Cherednik, see \cite[Thm. 2.3]{C} and \cite[Thm. 3.1]{ChInt}.
We call $\pi_{\cc,\mathfrak{t}}$ the {\it quasi-polynomial representation} of $\mathbb{H}$.

The proof of Theorem \ref{gbr} is relegated to Section \ref{mainproof}. The proof of part (1), which will be given in Subsection \ref{proof1section}, will follow closely the proof of \cite[Thm. 3.7]{SSV} introducing the metaplectic affine Hecke algebra action on Laurent polynomials (in Section \ref{MetaplecticSection} we will show how \cite[Thm. 3.7]{SSV} follows from Theorem \ref{gbr}). Similarly as in the context of Cherednik's basic representation and its associated Macdonald polynomials (the special case $\cc=0$ and $\mathfrak{t}=1_T$), the proof of parts (2) and (3) in Subsections \ref{POsection}--\ref{POsection2}, as well as the existence of the quasi-polynomial eigenfunctions of $\pi_{\cc,\mathfrak{t}}(Y^\lambda)$ ($\lambda\in Q^\vee$) in Section \ref{QuasiSection}, will require a detailed analysis of the triangularity properties of the truncated Demazure-Lusztig operators $\pi_{\cc,\mathfrak{t}}(T_j)$ ($j\in [0,r]$). For a detailed discussion of the triangularity properties of Demazure-Lusztig operators in the context of Cherednik's polynomial representation  see, e.g., \cite[\S 4.6]{Ma}.

 In the following two subsections we give some direct consequences of Theorem \ref{gbr}. 

\subsection{The face dependence of the quasi-polynomial representation}\label{FaceSection}

The quasi-polynomial representations $\pi_{\cc,\mathfrak{t}}$ are equivalent for different values of $\cc\in C^J$ by Theorem \ref{gbr}(2). For $\mathfrak{t}\in T_J$ and
$\cc,\cc^\prime\in C^J$, the isomorphism is explicitly given by
\begin{equation}\label{isomface}
\mathcal{P}_{\mathfrak{t}}^{(\cc)}\overset{\sim}{\longrightarrow}\mathcal{P}_{\mathfrak{t}}^{(\cc^\prime)},\qquad x^{w\cc}\mapsto x^{w\cc^\prime}\quad (w\in W^J).
\end{equation}
As a direct consequence of \eqref{isomface}, we obtain the following result.
\begin{corollary}\label{expansionHcor}
Let $J\subsetneq [0,r]$, $\mathfrak{t}\in T_{J}$, $w\in W^J$ and $h\in\mathbb{H}$. Then
\[
\pi_{\cc,\mathfrak{t}}(h)x^{w\cc}=
\sum_{w^\prime\in W^J}e_{w,w^\prime;\mathfrak{t}}^{J,h}
x^{w^\prime\cc}\qquad \forall\,\cc\in C^J
\]
with coefficients $e_{w,w^\prime;\mathfrak{t}}^{J,h}\in\mathbf{F}$ \textup{(}$w^\prime\in W^J$\textup{)} independent of the choice of $\cc\in C^J$.
\end{corollary}
The closure relations between the faces $C^J$ of $\overline{C}_+$ are translated to quotient relations for the quasi-polynomial representations $\pi_{\cc,\mathfrak{t}}$
(see Remark \ref{specNJt}(1)). 
\begin{corollary}\label{projectioncor}
Let $J\subseteq J^\prime\subsetneq [0,r]$, $\cc\in C^J$, $\cc^\prime\in C^{J^\prime}$ and $\mathfrak{t}^\prime\in T_{J^\prime}$. 
\begin{enumerate}
\item There exists a unique epimorphism 
\[
\textup{pr}_{\cc,\cc^\prime}^{\mathfrak{t}^\prime}:\Pc_{\mathfrak{s}_J^{-1}\mathfrak{s}_{J^\prime}\mathfrak{t}^\prime}
\twoheadrightarrow\mathcal{P}^{(\cc^\prime)}_{\mathfrak{t}^\prime}
\]
of $\mathbb{H}$-modules mapping $x^\cc$ to $x^{\cc^\prime}$. 
\item We have
\[
\textup{pr}_{\cc,\cc^\prime}^{\mathfrak{t}^\prime}(x^{\mu+v\cc})=\frac{\kappa_v(\cc^\prime)}{\kappa_v(\cc)}x^{\mu+v\cc^\prime}\qquad  (\mu\in Q^\vee,\, v\in W_0).
\]
\end{enumerate}
\end{corollary}
\begin{proof}
Write $t^\prime:=\mathfrak{s}_{J^\prime}\mathfrak{t}^\prime\in L_{J^\prime}\subseteq L_J=\mathfrak{s}_JT_J$.
Recall from Remark \ref{specNJt}(1) that we have an epimorphism $\psi: \mathbb{M}^J_{t^\prime}\twoheadrightarrow\mathbb{M}^{J^\prime}_{t^\prime}$ of $\mathbb{H}$-modules mapping 
$m^J_{t^\prime}$ to $m^{J^\prime}_{t^\prime}$. Then 
\[
\textup{pr}^{\mathfrak{t}^\prime}_{\cc,\cc^\prime}:=\phi_{\cc^\prime,\mathfrak{t}^\prime}\circ\psi\circ\phi_{\cc,\mathfrak{s}_J^{-1}t^\prime}^{-1}: \Pc_{\mathfrak{s}_J^{-1}\mathfrak{s}_{J^\prime}\mathfrak{t}^\prime} 
\twoheadrightarrow\mathcal{P}^{(\cc^\prime)}_{\mathfrak{t}^\prime}
\]
has the desired properties (use Theorem \ref{gbr}\,(2)\&(3)).
 \end{proof}

\begin{remark}\label{projectionrem}
Let $J\subseteq [1,r]$ and $\cc\in C^J$. Then Corollary \ref{projectioncor} gives an epimorphism $\textup{pr}_{\cc,0}^{1_{T}}: \Pc_{\mathfrak{s}_J^{-1}\mathfrak{s}_0}
\twoheadrightarrow\mathcal{P}^{(0)}_{1_T}$ 
of $\mathbb{H}$-modules. The co-domain is Cherednik's basic representation, and
\[
\textup{pr}_{\cc,0}^{1_T}\bigl(x^{\mu+v\cc}\bigr)=\Bigl(\prod_{\alpha\in\Pi(v)}k_\alpha^{1+\eta(\alpha(\cc))}\Bigr)x^\mu\qquad (\mu\in Q^\vee,\, v\in W_0).
\]
\end{remark}

\subsection{Affine Weyl group actions on quasi-rational functions}\label{aWaSection}
We derive an affine Weyl group action on quasi-rational functions by localizing the quasi-poly\-no\-mial representation $\pi_{\cc,\mathfrak{t}}$ (Theorem \ref{gbr}(1)).  In Section \ref{MfinalSection}, we will show that this action gives rise to an affine analog of Chinta and Gunnells' $W_0$-action on rational functions \cite{CG07, CG}.
The approach in this subsection follows the paper \cite{SSV}, where the Chinta-Gunnells' $W_0$-action was obtained by localization of an appropriate metaplectic affine Hecke algebra representation. 
 
Consider the $\mathcal{Q}$-algebra 
\[\mathbf{F}_{\mathcal{Q}}[E]:=\mathcal{Q}\otimes_{\mathcal{P}}\mathbf{F}[E],
\] 
which is isomorphic to the localization of $\mathbf{F}[E]$ by $\mathcal{P}^\times$. The algebra $W_0\ltimes\mathcal{Q}$ naturally act on $\mathbf{F}_{\mathcal{Q}}[E]$.

For $\gamma\in T_{\mathbb{R}}^\vee$ denote by
$\mathcal{Q}_\gamma$ the $\mathcal{Q}$-submodule of $\mathbf{F}_{\mathcal{Q}}[E]$ generated
by $\mathcal{P}_\gamma$. Then 
\[
\mathcal{Q}_\gamma=\mathcal{Q}x^y\qquad \forall\, y\in\gamma.
\]
Write $\mathcal{Q}^{(\cc)}$ for the $\mathcal{Q}$-submodule of $\mathbf{F}_{\mathcal{Q}}[E]$ generated by $\Pc$
($\cc\in\overline{C}_+$), so that
\[
\mathcal{Q}^{(\cc)}=\bigoplus_{\gamma\in W_0[\cc]}\mathcal{Q}_\gamma.
\]
Note that $\mathcal{Q}^{(\cc)}$ is a cyclic $W_0\ltimes\mathcal{Q}$-submodule of $\mathbf{F}_{\mathcal{Q}}[E]$. 

Fix $\cc\in C^J$ and $\mathfrak{t}\in T_J$. The $W\ltimes\mathcal{P}$-action on $\Pc$ from Lemma \ref{actiondeform}
extends to a $W\ltimes\mathcal{Q}$-action on $\mathcal{Q}^{(\cc)}$. The following theorem provides a $\mathbf{k}$-deformation of this action.
\begin{theorem}\label{aCG}
Let $J\subsetneq [0,r]$, $\cc\in C^J$ and $\mathfrak{t}\in T_J$. The formulas
\begin{equation}\label{sigmaaction}
\begin{split}
\sigma_{\cc,\mathfrak{t}}(s_j)(x^y)&:=
\frac{k_j^{\chi_{\mathbb{Z}}(\alpha_j(y))}(x^{\alpha_j^\vee}-1)}{(k_jx^{\alpha_j^\vee}-k_j^{-1})}
s_{j,\mathfrak{t}}x^y+\frac{(k_j-k_j^{-1})}{(k_jx^{\alpha_j^\vee}-k_j^{-1})}x^{y-\lfloor D\alpha_j(y)
\rfloor\alpha_j^\vee},\\
\sigma_{\cc,\mathfrak{t}}(f)(x^y)&:=fx^y
\end{split}
\end{equation}
for $0\leq j\leq r$, $y\in\mathcal{O}_\cc$ and $f\in\mathcal{Q}$ turn $\mathcal{Q}^{(\cc)}$ into a 
$W\ltimes\mathcal{Q}$-module, which we denote by $\mathcal{Q}_{\mathfrak{t}}^{(\cc)}$.
\end{theorem}
\begin{proof}
Let $\check{s}_j$ be the linear operator on $\Pc$ defined by
\begin{equation}\label{checksj}
\check{s}_j(x^y):=x^{y-\lfloor D\alpha_j(y)\rfloor\alpha_j^\vee}\qquad (y\in\mathcal{O}_\cc).
\end{equation}
Note that 
\begin{equation}\label{compatibleQ}
\check{s}_j(px^y)=s_j(p)\check{s}_j(x^y)
\end{equation}
for $0\leq j\leq r$, $y\in\mathcal{O}_\cc$ and $p\in\mathcal{P}$. We extend 
$\check{s}_j$ to a linear operator on $\mathcal{Q}^{(\cc)}$
by requiring \eqref{compatibleQ} to hold true for all $p\in\mathcal{Q}$. 

Transporting the $\mathbb{H}^{X-\textup{loc}}$-action of the induced $\mathbb{H}^{X-\textup{loc}}$-module
$\mathbb{H}^{X-\textup{loc}}\otimes_{\mathbb{H}}\Pc_{\mathfrak{t}}$
to $\mathcal{Q}^{(\cc)}$ through the linear isomorphism 
\[
\mathcal{Q}^{(\cc)}\overset{\sim}{\longrightarrow}\mathbb{H}^{X-\textup{loc}}\otimes_{\mathbb{H}}\Pc_{\mathfrak{t}},\qquad
fx^y\mapsto f\otimes_{\mathbb{H}}x^y\qquad (f\in\mathcal{Q},\, y\in\mathcal{O}_\cc)
\]
(see \eqref{factorloc}) gives a representation 
\[
\pi_{\cc,\mathfrak{t}}^{X-\textup{loc}}: \mathbb{H}^{X-\textup{loc}}\rightarrow
\textup{End}(\mathcal{Q}^{(\cc)}).
\]
By \eqref{crossloc} and \eqref{compatibleQ} the representation map $\pi_{\cc,\mathfrak{t}}^{X-\textup{loc}}$ is explicitly given by
\begin{equation}\label{hatform}
\begin{split}
\pi_{\cc,\mathfrak{t}}^{X-\textup{loc}}(T_j)(fx^y)&=k_j^{\chi_{\mathbb{Z}}(\alpha_j(y))}s_{j,\mathfrak{t}}(fx^y)+(k_j-k_j^{-1})\Bigl(\frac{fx^y-\check{s}_j(fx^y)}{1-x^{\alpha_j^\vee}}\Bigr),\\
\pi_{\cc,\mathfrak{t}}^{X-\textup{loc}}(f^\prime)(fx^y)&=f^\prime fx^y
\end{split}
\end{equation}
for $0\leq j\leq r$, $y\in\mathcal{O}_\cc$ and $f,f^\prime\in\mathcal{Q}$. By Theorem \ref{locTHM} we obtain a representation $\sigma_{\cc,\mathfrak{t}}: W\ltimes\mathcal{Q}\rightarrow\textup{End}(\mathcal{Q}^{(\cc)})$, defined by
\begin{equation}\label{localaction}
\sigma_{\cc,\mathfrak{t}}(w\otimes f):=\pi_{\cc,\mathfrak{t}}^{X-\textup{loc}}(\widetilde{S}_w^Xf)
\qquad (w\in W, f\in\mathcal{Q}).
\end{equation}
Then $\sigma_{\cc,\mathfrak{t}}(f^\prime)(fx^y)=ff^\prime x^y$ for $f,f^\prime\in\mathcal{Q}$ and $y\in\mathcal{O}_\cc$, and 
\begin{equation*}
\begin{split}
\sigma_{\cc,\mathfrak{t}}(s_j)&=\pi_{\cc,\mathfrak{t}}^{X-\textup{loc}}(\widetilde{S}_j^X)\\
&=\pi_{\cc,\mathfrak{t}}^{X-\textup{loc}}\Bigl(
\Bigl(\frac{x^{\alpha_j^\vee}-1}{k_jx^{\alpha_j^\vee}-k_j^{-1}}\Bigr)T_j+
\Bigl(\frac{k_j-k_j^{-1}}{k_jx^{\alpha_j^\vee}-k_j^{-1}}\Bigr)\Bigr)\\
&=\frac{k_j^{\chi_{\mathbb{Z}}(\alpha_j(y))}(x^{\alpha_j^\vee}-1)}{(k_jx^{\alpha_j^\vee}-k_j^{-1})}
s_{j,\mathfrak{t}}+\frac{(k_j-k_j^{-1})}{(k_jx^{\alpha_j^\vee}-k_j^{-1})}\check{s}_j
\end{split}
\end{equation*}
as linear operators on $\mathcal{Q}^{(\cc)}$.
\end{proof}

\begin{remark}\label{linkCGlocalrem}\hfill
\begin{enumerate}
\item
Note that for $0\leq j\leq r$, $f\in\mathcal{Q}$ and $y\in\mathcal{O}_\cc$ we have
\[
s_{j,\mathfrak{t}}(fx^y)=
\check{s}_j(fx^y)\,\,\hbox{ if }\, \alpha_j(y)\in\mathbb{Z}
\]
by Lemma \ref{datumokcor}(2), and hence 
$\sigma_{\cc,\mathfrak{t}}(s_j)(fx^y)=\check{s}_j(fx^y)$ if $\alpha_j(y)\in\mathbb{Z}$. In particular, 
$\sigma_{0,1_T}$ reduces 
to the standard $W\ltimes\mathcal{Q}$-action on
$\mathcal{Q}^{(0)}=\mathcal{Q}$ as described in Subsection \ref{ap}.
\item We have
\begin{equation}\label{linkCGlocal}
\sigma_{\cc,\mathfrak{t}}=\pi_{\cc,\mathfrak{t}}^{X-\textup{loc}}\circ\ss
\end{equation}
by Theorem \ref{locTHM} and \eqref{localaction}.
\end{enumerate}
\end{remark}
\begin{remark}\label{DLform}
Theorem \ref{aCG} allows to rewrite the truncated Demazure-Lusztig type operators $\pi_{\cc,\mathfrak{t}}^{X-\textup{loc}}(T_j)$ (see \eqref{actionformulas})
as standard Demazure-Lusztig operators involving the $W$-action $\sigma_{\cc,\mathfrak{t}}$ on $\mathcal{Q}^{(\cc)}$,
\begin{equation*}
\pi_{\cc,\mathfrak{t}}^{X-\textup{loc}}(T_j)f=k_jf+k_j^{-1}\Bigl(\frac{1-k_j^2x^{\alpha_j^\vee}}{1-x^{\alpha_j^\vee}}\Bigr)\bigl(\sigma_{\cc,\mathfrak{t}}(s_j)f-f\bigr)
\qquad (f\in\mathcal{Q}^{(\cc)}).
\end{equation*}
This follows from \eqref{linkCGlocal} and the fact that 
\[
\ss^{-1}(T_j)=k_j+k_j^{-1}\Bigl(\frac{1-k_j^2x^{\alpha_j^\vee}}{1-x^{\alpha_j^\vee}}\Bigr)(s_j-1)
\]
in $W\ltimes\mathcal{Q}$. This formula shows that $\pi_{\cc,\mathfrak{t}}^{X-\textup{loc}}(T_j)$ is a so-called metaplectic Demazure-Lusztig operator (see \cite{CGP,PP,PP2}) in the 
context of metaplectic representation theory of reductive groups over non-archimedean local fields, see Section \ref{MfinalSection} for further details.
\end{remark}
\begin{remark}\label{k0limit}
In the limit case where $\mathbf{k}=\mathbf{0}$ (meaning $k_{\alpha} = 0$ for all $\alpha \in \Phi_0$), the $W$-action from
Theorem \ref{aCG} does not depend on $\mathfrak{t}$ and satisfies
\[
\sigma_{\cc,\mathfrak{t}}^{\mathbf{k}=\mathbf{0}}(s_j)(x^y)=\check{s}_j(x^y)\qquad (0\leq j\leq r,\, y\in\mathcal{O}_\cc),
\]
with $\check{s}_j$ defined by \eqref{checksj}. For $y \in\mathcal{O}_\cc$ let $\xi_y\in y+P^\vee$ be the unique element satisfying 
$0 \leq \alpha_i(\xi_y) < 1$ for all $1 \leq i \leq r$. Then
$\check{s}_i(x^y)=x^{\xi_y+s_i(y-\xi_y)}$ for $i\in [1,r]$, hence
\begin{equation*}
\sigma_{c, \mathfrak{t}}^{\mathbf{k}=\mathbf{0}}(v)(x^y) = x^{\xi_y + v (y-\xi_y)}
\end{equation*}
for all $v \in W_0$.  In particular, this is a shift of the standard $W_0$-action.  
\end{remark}

\section{The proof of the main theorem}\label{mainproof}

This section is dedicated to the proof of Theorem \ref{gbr}. As a byproduct we obtain triangularity properties of the commuting
operators $\pi_{\cc,\mathfrak{t}}(Y^\mu)$ ($\mu\in Q^\vee$) which will play an important role in the construction of the quasi-polynomial
generalisations of the Macdonald polynomials in Section \ref{QuasiSection}.

\subsection{Properties of the base point $\mathfrak{s}_J$}\label{idsection}
Recall the step function $\eta$ defined by \eqref{eta}. It satisfies the identities
\begin{equation}\label{etaaffine}
\eta(x)+\eta(-x)=-2\chi_{\{0\}}(x),\qquad 
\eta(x)+\eta(1-x)=0
\end{equation}
for all $x\in\mathbb{R}$. 
\begin{definition}\label{fraksdef}
Let $\Lambda\in\mathcal{L}$. For $y\in E$ set 
\begin{equation}\label{ffrak}
\mathfrak{s}_y:=\prod_{\alpha\in\Phi^+_0}k_\alpha^{\eta(\alpha(y))\alpha}\in T_{\Lambda}.
\end{equation}
\end{definition}
Note that $\mathfrak{s}_y\vert_\Lambda=j_\Lambda(\mathfrak{s}_y\vert_{P^\vee})$ since $\mathfrak{s}_y$ maps $\textup{pr}_{E_{\textup{co}}}(\Lambda)$ to $1$.

For $J\subsetneq [0,r]$ the element $\mathfrak{s}_J\in T_{P^\vee}$ (see \eqref{fraksprime}) can be recovered from $\mathfrak{s}_y$ as follows.
\begin{lemma}\label{facecors}
Let $\Lambda\in\mathcal{L}$. The function $y\mapsto \mathfrak{s}_y$ lies in $\mathcal{F}_\Sigma(E,T_{\Lambda})$. Furthermore,
\begin{equation}\label{srelation}
\mathfrak{s}_\cc=\mathfrak{s}_{J}\,\, \hbox{ if }\,\, \cc\in C^J.
\end{equation}
\end{lemma}
\begin{proof}
The first statement follows from Lemma \ref{faceconstantlem} applied to $g=\eta$. The second statement
follows from \eqref{etaJ} and \eqref{fraksprime}.
\end{proof}
We now analyse how the affine Weyl group $W$ acts on $\mathfrak{s}_y\in T_{\Lambda}$.
\begin{lemma}\label{fraks}
Let $y\in E$ and $j\in [0,r]$. 
\begin{enumerate}
\item If $s_j\in W_y$ then $\mathfrak{s}_y^{D\alpha_j^\vee}=k_j^{-2}$  and $s_{D\alpha_j}\mathfrak{s}_y=k_j^{2D\alpha_j}\mathfrak{s}_y$,
\item If $s_j\in W\setminus W_y$ then $s_{D\alpha_j}\mathfrak{s}_y=\mathfrak{s}_{s_jy}$.
\end{enumerate}
\end{lemma}
\begin{proof}
Let $1\leq i\leq r$. Then $\Pi(s_i)=\{\alpha_i\}$, hence
\[
s_i\mathfrak{s}_y=\prod_{\alpha\in s_i\Phi_0^+}k_\alpha^{\eta(\alpha(s_iy))\alpha}=
\mathfrak{s}_{s_iy}k_i^{-(\eta(\alpha_i(y))+\eta(-\alpha_i(y)))\alpha_i}.
\]

Suppose that $s_iy=y$. Then $\alpha_i(y)=0$ hence $\eta(\alpha_i(y))+\eta(-\alpha_i(y))=-2$, so 
$s_i\mathfrak{s}_y=\mathfrak{s}_yk_i^{2\alpha_i}$ in $T_{P^\vee}$. On the other hand, $s_{i}\mathfrak{s}_y=\mathfrak{s}_y\bigl(\mathfrak{s}_y^{-\alpha_i^\vee}\bigr)^{\alpha_i}$
by \eqref{actiononT}, so $\bigl(\mathfrak{s}_y^{-\alpha_i^\vee}\bigr)^{\alpha_i}=k_i^{2\alpha_i}$ in $T_{\Lambda}$. This in particular holds true for $\Lambda=P^\vee$.
Hence $\mathfrak{s}_y^{\alpha_i^\vee}=k_i^{-2}$.

Suppose that $s_iy\not=y$. Then $\alpha_i(y)\not=0$, hence $\eta(\alpha_i(y))+\eta(-\alpha_i(y))=0$ by the first equality of \eqref{etaaffine}. It follows that $s_i\mathfrak{s}_y=\mathfrak{s}_{s_iy}$.

The proof for $j=0$ is a bit more involved. First note that
\begin{equation}\label{sphirel}
s_\varphi\mathfrak{s}_y=\mathfrak{s}_{s_0y}\prod_{\alpha\in\Phi_0^+\setminus\Pi(s_\varphi)}
k_\alpha^{(\eta(\alpha(s_\varphi y))-\eta(\alpha(s_0y)))\alpha}
\prod_{\beta\in\Pi(s_\varphi)}k_\beta^{-(\eta(-\beta(s_\varphi y))+\eta(\beta(s_0y)))\beta}
\end{equation}
since $k_{s_\varphi\gamma}=k_\gamma$ for all $\gamma\in\Phi_0$. Now $\alpha\in\Phi_0^+\setminus\Pi(s_\varphi)$ implies
$\alpha(\varphi^\vee)=0$ by \eqref{thetalength}, hence
\[\eta(\alpha(s_\varphi y))-\eta(\alpha(s_0y))=\eta(\alpha(y))-\eta(\alpha(y))=0.
\] 
Hence the product over $\alpha\in\Phi_0^+\setminus\Pi(s_\varphi)$ in \eqref{sphirel} is equal to $1$.
By the explicit description \eqref{thetalength} of $\Pi(s_\varphi)$ we then have
\[
s_\varphi\mathfrak{s}_y=\mathfrak{s}_{s_0y}k_\varphi^{-(\eta(2-\varphi(y))+\eta(\varphi(y)))\varphi}
\prod_{\beta\in\Phi_0^+:\, \beta(\varphi^\vee)=1}
k_\beta^{-(\eta(-\beta(s_\varphi y))+\eta(\beta(s_\varphi y)+1))\beta}.
\]
The product over $\beta$ is equal to $1$ by \eqref{etaaffine}, so we conclude that
\[
s_\varphi\mathfrak{s}_y=\mathfrak{s}_{s_0y}k_\varphi^{-(\eta(2-\varphi(y))+\eta(\varphi(y)))\varphi}.
\]
Suppose that $s_0y=y$. Then $\varphi(y)=1$, hence $\eta(2-\varphi(y))+\eta(\varphi(y))=2\eta(1)=2$ and
$s_\varphi\mathfrak{s}_y=\mathfrak{s}_yk_\varphi^{-2\varphi}$ in $T_{\Lambda}$. Similarly as before, this implies that 
 $\mathfrak{s}_y^{-\varphi^\vee}=k_\varphi^{-2}$. 
 
 Suppose that $s_0y\not=y$. Then $\varphi(y)\not=1$
and $\eta(2-\varphi(y))+\eta(\varphi(y))=0$ (indeed, it is $0+0$ if $\varphi(y)\in\mathbb{R}\setminus\mathbb{Z}$, it is $-1+1$ if $\varphi(y)\in\mathbb{Z}_{>1}$ and it is $1-1$ if $\varphi(y)\in\mathbb{Z}_{<1}$). Hence $s_\varphi\mathfrak{s}_y=\mathfrak{s}_{s_0y}$.
\end{proof}

\begin{proposition}\label{facecor}
Let $\Lambda\in\mathcal{L}$. We have
\[
\mathfrak{s}_y=(Dw_y)\mathfrak{s}_{\cc_y}
\]
in $T_{\Lambda}$ for all $y\in E$.
\end{proposition}
\begin{proof}
Write $\cc:=\cc_y\in\overline{C}_+$ and $J:=\mathbf{J}(\cc)$. Fix a reduced expression $w_y=s_{j_1}\cdots s_{j_\ell}$ for $w_y\in W^J$ and write $y_m:=s_{j_m}\cdots s_{j_\ell}\cc$ ($1\leq m\leq \ell$) and $y_{\ell+1}:=\cc$. Then
$s_{j_{m}}y_{m+1}\not=y_m$ for $m=1\ldots,\ell$. Repeated application of Lemma \ref{fraks}(2) then gives $\mathfrak{s}_y=(Dw_y)\mathfrak{s}_\cc$. 
\end{proof}
 
Recall from Proposition \ref{lemmabasepoint} that $L_{\Lambda,J}=\mathfrak{s}_JT_{\Lambda,J}$ for $\Lambda\in\mathcal{L}$.
Furthermore, for $\cc\in C^J$ and $\mathfrak{t}\in T_{\Lambda,J}$ we defined $\mathfrak{t}_y\in T_\Lambda$ ($y\in \mathcal{O}_\cc$) by $\mathfrak{t}_y:=w_y\mathfrak{t}$ (see Definition \ref{defty}).
\begin{corollary}\label{datumok} Let $c\in C^J$ and $\mathfrak{t}\in T_{\Lambda,J}$.
For $y\in\mathcal{O}_\cc$ we have
\[
\mathfrak{s}_y\mathfrak{t}_y=w_y(\mathfrak{s}_J\mathfrak{t})
\]
in $T_\Lambda$.
\end{corollary}
\begin{proof}
This follows from 
\[
w_y(\mathfrak{s}_J\mathfrak{t})=((Dw_y)\mathfrak{s}_J)(w_y\mathfrak{t})=
\mathfrak{s}_y\mathfrak{t}_y
\]
for $y\in\mathcal{O}_\cc$, where the second equality follows from Proposition \ref{facecor}.
\end{proof}
We will also be needing analogous results for scalar-valued functions on $E$, in particular for
$\kappa_v: E\rightarrow\mathbf{F}^\times$ (see \eqref{kvy}).
Recall the characteristic function $\chi_{\pm}: \Phi_0\rightarrow\{0,1\}$ of $\Phi_0^{\pm}$.
\begin{lemma}\label{kylemma1}
\hfill
\begin{enumerate}
\item For $u,v\in W_0$ and $y\in E$,
\begin{equation}\label{krel}
\kappa_{uv}(y)=\kappa_{v}(y)\prod_{\alpha\in v^{-1}\Pi(u)}k_\alpha^{-\chi(\alpha)
\eta(\chi(\alpha)\alpha(y))},
\end{equation}
where $\chi:=\chi_+-\chi_-$. 
\item If $y\in E^{\textup{reg}}$ then $\kappa_{uv}(y)=\kappa_u(vy)\kappa_v(y)$ for all $u,v\in W_0$.
\end{enumerate}
\end{lemma}
\begin{proof}
(1) Formula \eqref{krel} follows from a direct computation, using that $\Pi(uv)$ is the disjoint union of the subsets 
$\Pi(v)\setminus (\Phi_0^+\cap (-v^{-1}\Pi(u)))$ and $\Phi_0^+\cap v^{-1}\Pi(u)$.\\
(2) By \eqref{etaaffine} we have $\eta(-x)=-\eta(x)$ for $x\in\mathbb{R}^\times$. Hence, for $y\in E^{\textup{reg}}$ formula \eqref{krel} reduces to
\[
\kappa_{uv}(y)=\kappa_{v}(y)\prod_{\alpha\in v^{-1}\Pi(u)}k_\alpha^{-\eta(\alpha(y))}=
\kappa_v(y)\kappa_u(vy).
\]
\end{proof}

We next introduce an affine version of $\kappa_v(y)$. Set
\begin{equation}\label{ky}
k(y):=\prod_{\alpha\in\Phi_0^+}k_\alpha^{\frac{\eta(\alpha(y))}{2}}\qquad (y\in E),
\end{equation}
and define
\begin{equation}\label{kawy}
k_w(y):=\frac{k(wy)}{k(y)}\qquad (y\in E, w\in W).
\end{equation}
The $k_w(y)$ satisfy the cocycle condition $k_{ww^\prime}(y)=k_w(w^\prime y)k_{w^\prime}(y)$ for
$w,w^\prime\in W$ and $y\in E$. Note furthermore that
\[
k_{\tau(\mu)}(y)=\frac{k(y+\mu)}{k(y)}=\prod_{\alpha\in\Phi_0^+}k_\alpha^{(\eta(\alpha(y)+\alpha(\mu))-
\eta(\alpha(y)))/2}
\]
for $y\in E$ and $\mu\in Q^\vee$. 

By Lemma \ref{faceconstantlem} we have $k,k_w\in\mathcal{F}_\Sigma(E,\mathbf{F}^\times)$ for all $w\in W$.
\begin{lemma}\label{kylemma2}
For $0\leq j\leq r$ 
and $y\in E$,
\begin{equation}\label{kjcombi}
k_{s_j}(y)=\kappa_{Ds_j}(y)\,\, \hbox{ if }\, \alpha_j(y)\not=0.
\end{equation}
\end{lemma}
\begin{proof}
We need to show that 
\begin{equation}\label{kj}
\begin{split}
k_{s_i}(y)&=k_i^{-\eta(\alpha_i(y))}\qquad\qquad\quad \hbox{ if }\,\, \alpha_i(y)\not=0,\\
k_{s_0}(y)&=\prod_{\alpha\in\Pi(s_\varphi)}k_\alpha^{-\eta(\alpha(y))}\qquad\,\,
\hbox{ if }\,\, \alpha_0(y)\not=0
\end{split}
\end{equation}
for $1\leq i\leq r$.
The first line follows directly from Lemma \ref{lengthadd} and \eqref{etaaffine}. For the second line one computes,
using \eqref{thetalength},
\[
k(s_0y)=k_{\varphi}^{\eta(-\varphi(y)+2)/2}\prod_{\alpha\in\Pi(s_\varphi)\setminus\{\varphi\}}
k_\alpha^{\eta(\alpha(s_\varphi y)+1)/2}
\prod_{\beta\in\Phi_0^+\setminus\Pi(s_\varphi)}k_\beta^{\eta(\beta(y))/2}.
\]
By \eqref{etaaffine} one then gets
\[
k_{s_0}(y)=k_{\varphi}^{(\eta(-\varphi(y)+2)+\eta(\varphi(y)))/2}\prod_{\alpha\in\Pi(s_\varphi)}k_\alpha^{-\eta(\alpha(y))}.
\]
Since $\eta(2-x)+\eta(x)=0$ if $x\not=1$ (see the proof of Lemma \ref{fraks}), this reduces to the second formula in \eqref{kj}.
\end{proof}

\subsection{The $\mathbb{H}$-action on $\mathcal{P}^{(\cc)}$}\label{proof1section}
In this subsection we give the proof of Theorem \ref{gbr}(1).
Fix $\cc\in C^J$ and $\mathfrak{t}\in T_J$. An important role in the proof 
is played by the following map.

\begin{definition}
Define the surjective linear map $\psi_{\cc,\mathfrak{t}}: \mathbb{H}\twoheadrightarrow \Pc$ by 
\begin{equation}\label{psic}
\psi_{\cc,\mathfrak{t}}(x^\mu T_v Y^\nu):=\kappa_v(\cc)(\mathfrak{s}_\cc\mathfrak{t})^{-\nu} x^{\mu+v\cc}
\end{equation}
for $\mu,\nu\in Q^\vee$ and $v\in W_0$. 
\end{definition}
Note that $\psi_{\cc,\mathfrak{t}}(1)=x^\cc$. 

\begin{lemma}\label{generatoractiontransfer}
Let $h^\prime\in\mathbb{H}$. Then
\begin{equation*}
\psi_{\cc,\mathfrak{t}}(hh^\prime)=\pi_{\cc,\mathfrak{t}}(h)\psi_{\cc,\mathfrak{t}}(h^\prime)
\end{equation*}
for $h=T_j$ \textup{(}$0\leq j\leq r$\textup{)} and $h=x^\mu$ \textup{(}$\mu\in Q^\vee$\textup{)}.
\end{lemma}
\begin{proof}
For the duration of the proof we write 
\begin{equation}\label{taff}
t:=\mathfrak{s}_\cc \mathfrak{t}\in L_{J}.
\end{equation}
Recall that $\chi_{\pm}: \Phi_0\rightarrow\{0,1\}$ is the characteristic function of $\Phi_0^{\pm}$, and $\chi:=\chi_+-\chi_-$.
Note that for $v\in W_0$ we have $\chi(v^{-1}\alpha_i)=\pm 1$ iff $\ell(s_iv)=\ell(v)\pm 1$ by Lemma \ref{lengthadd}.

The lemma is immediate for $h=x^\nu$ ($\nu\in Q^\vee$).
We split the proof for $h=T_j$ in two cases.\\
 {\bf Case 1:} $h=T_i$ with $1\leq i\leq r$.\\
This case is similar to the proof of \cite[Thm. 3.7]{SSV}, see in particular \cite[Lemma 3.2]{SSV}. We shortly discuss it here for completeness.
Fix $v\in W_0$ and $\mu,\nu\in Q^\vee$. By the cross relation \eqref{crossX} in $\mathbb{H}$ and the fact that $T_iT_v=\chi_-(v^{-1}\alpha_i)(k_i-k_i^{-1})T_v+T_{s_iv}$, we have for $v\in W_0$ and $\mu,\nu\in Q^\vee$,
\begin{equation}\label{Til}
\begin{split}
\psi_{\cc,\mathfrak{t}}(T_ix^\mu T_vY^\nu)&=
\kappa_{s_iv}(\cc)t^{-\nu}x^{s_i\mu+s_iv\cc}\\
&+\kappa_v(\cc)t^{-\nu}(k_i-k_i^{-1})\bigl(\nabla_i(x^\mu)x^{v\cc}+\chi_-(v^{-1}\alpha_i)x^{s_i\mu+v\cc}\bigr).
\end{split}
\end{equation}
By Lemma \ref{lengthadd} we have $0\leq \chi(v^{-1}\alpha_i)\alpha_i(v\cc)\leq 1$ since $\cc\in\overline{C}_+$, and 
\[
\kappa_{s_iv}(\cc)=\kappa_v(\cc)k_i^{-\chi(v^{-1}\alpha_i)\eta(\chi(v^{-1}\alpha_i)\alpha_i(v\cc))}
\]
by \eqref{krel}. 
So \eqref{Til} reduces to
\begin{equation}\label{Til2}
\begin{split}
\psi_{\cc,\mathfrak{t}}(T_ix^\mu T_vY^\nu)=\kappa_v(\cc)t^{-\nu}\Bigl(&k_i^{-\chi(v^{-1}\alpha_i)\eta(\chi(v^{-1}\alpha_i)\alpha_i(v\cc))}
x^{s_i\mu+s_iv\cc}\\
&+
(k_i-k_i^{-1})\bigl(\nabla_i(x^\mu)x^{v\cc}+\chi_-(v^{-1}\alpha_i)x^{s_i\mu+v\cc}\bigr)\Bigr).
\end{split}
\end{equation}
On the other hand,
\begin{equation}\label{Tir}
\pi_{\cc,\mathfrak{t}}(T_i)\psi_{\cc,\mathfrak{t}}(x^\mu T_vY^\nu)=
\kappa_v(\cc)t^{-\nu}\Bigl(k_i^{\chi_{\mathbb{Z}}(\alpha_i(v\cc))}x^{s_i\mu+s_iv\cc}
+(k_i-k_i^{-1})\nabla_i(x^{\mu+v\cc})\Bigr),
\end{equation}
and
\begin{equation}\label{nablaiaction}
\nabla_i(x^{\mu+v\cc})=
\begin{cases}
\nabla_i(x^\mu)x^{v\cc}+x^{s_i\mu+v\cc}\quad
&\hbox{ if }\,\, -1\leq \alpha_i(v\cc)<0,\\
\nabla_i(x^\mu)x^{v\cc}\quad &\hbox{ if }\,\, 0\leq\alpha_i(v\cc)<1,\\
\nabla_i(x^\mu)x^{v\cc}-x^{s_i\mu+s_iv\cc}
\quad
&\hbox{ if }\,\, \alpha_i(v\cc)=1.
\end{cases}
\end{equation}
To show that \eqref{Til2} is equal to \eqref{Tir},
we distinguish subcases depending on whether $v\cc$ lies on one of the three affine root hyperplanes $H_{i,m}:=\{z\in E\,\, | \,\, \alpha_i(z)=m\}$ ($m\in\{-1,0,1\}$) or that it lies between $H_{i,m}$ and $H_{i,m+1}$ for $m=-1$ and $m=0$. The cases $v\cc\in H_{i,-1}$, and $v\cc$ lying between $H_{i,-1}$ and $H_{i,0}$, are taken together as case 1a.\\
{\it Case 1a:} $-1\leq\alpha_i(v\cc)<0$. Then $v^{-1}\alpha_i\in\Phi_0^-$,
$\eta(-\alpha_i(v\cc))=
\chi_{\mathbb{Z}}(\alpha_i(v\cc))$ and $\chi_-(v^{-1}\alpha_i)=1$. Taking \eqref{nablaiaction} into account, it follows now directly that \eqref{Til2} is equal to \eqref{Tir}.\\
{\it Case 1b:} $\alpha_i(v\cc)=0$. Then $s_iv\cc=v\cc$ and $\eta(\chi(v^{-1}\alpha_i)\alpha_i(v\cc))=-1$, but it is undetermined whether $v^{-1}\alpha_i$ is a positive or a negative root. 
The expression \eqref{Til2} simplifies to
\begin{equation*}
\begin{split}
\psi_{\cc,\mathfrak{t}}(&T_ix^\mu T_vY^\nu)=\\
&=\kappa_v(\cc)t^{-\nu}
\Bigl(k_i^{\chi(v^{-1}\alpha_i)}x^{s_i\mu+v\cc}+(k_i-k_i^{-1})\bigl(\nabla_i(x^\mu)x^{v\cc}+\chi_-(v^{-1}\alpha_i)x^{s_i\mu+v\cc}\bigr)\Bigr)\\
&=\kappa_v(\cc)t^{-\nu}\Bigl(k_ix^{s_i\mu+v\cc}+(k_i-k_i^{-1})\nabla_i(x^\mu)x^{v\cc}\Bigr).
\end{split}
\end{equation*}
Note here that for $v^{-1}\alpha_i\in\Phi_0^-$ the last equality follows by combining the first and third term, which has the effect that it flips the prefactor $k_i^{-1}$ of the first term to $k_i$. Formula \eqref{nablaiaction}  now reduces this expression to \eqref{Tir}.\\
{\it Case 1c:} $0<\alpha_i(v\cc)<1$. Then $v^{-1}\alpha_i\in\Phi_0^+$ so $\chi_-(v^{-1}\alpha_i)=0$,
and furthermore $\eta(\alpha_i(v\cc))=0=\chi_{\mathbb{Z}}(\alpha_i(v\cc))$. Taking also \eqref{nablaiaction} into account, it follows that \eqref{Til2} is equal to \eqref{Tir}.\\
{\it Case 1d:} $\alpha_i(v\cc)=1$. Then $\chi_-(v^{-1}\alpha_i)=0$ and $\eta(\alpha_i(v\cc))=1$, so 
\begin{equation}\label{Tirpart}
\psi_{\cc,\mathfrak{t}}(T_ix^\mu T_vY^\nu)=
\kappa_v(\cc)t^{-\nu}\Bigl(k_i^{-1}x^{s_i\mu+s_iv\cc}+(k_i-k_i^{-1})\nabla_i(x^\mu)x^{v\cc}\Bigr),
\end{equation}
while \eqref{Tir} by \eqref{nablaiaction} becomes
\[
\pi_{\cc,\mathfrak{t}}(T_i)\psi_{\cc,\mathfrak{t}}(x^\mu T_vY^\nu)=
\kappa_v(\cc)t^{-\nu}\Bigl(k_ix^{s_i\mu+s_iv\cc}+
(k_i-k_i^{-1})(\nabla_i(x^\mu)x^{v\cc}-x^{s_i\mu+s_iv\cc})\Bigr).
\]
Moving the correction term of the divided difference operator to the front is altering the prefactor
$k_i$ of $x^{s_i\mu+s_iv\cc}$ to $k_i^{-1}$. This reduces the last expression to \eqref{Tirpart}.\\ 
{\bf Case 2:} $h=T_0$. Fix $v\in W_0$ and $\mu,\nu\in Q^\vee$. By \eqref{YT0} we have
\[
T_0T_v=T_{s_\varphi v}Y^{-v^{-1}\varphi^\vee}+\chi_+(v^{-1}\varphi)(k_\varphi-k_\varphi^{-1})T_v
\]
for $v\in W_0$ (recall here that $k_0=k_\varphi$). Combined with the cross relation \eqref{crossexplicit} in $\mathbb{H}$ and the definition of $\psi_{\cc,\mathfrak{t}}$, we get
\begin{equation*}
\begin{split}
\psi_{\cc,\mathfrak{t}}(T_0x^\mu T_vY^\nu)&=
\kappa_{s_\varphi v}(\cc)t^{-\nu+v^{-1}\varphi^\vee}q_\varphi^{\varphi(\mu)}
x^{s_\varphi\mu+s_\varphi v\cc}\\
&+
\kappa_v(\cc)t^{-\nu}(k_\varphi-k_\varphi^{-1})
\bigl(\nabla_0(x^\mu)x^{v\cc}+\chi_+(v^{-1}\varphi)q_\varphi^{\varphi(\mu)}x^{s_\varphi \mu+v\cc}
\bigr).
\end{split}
\end{equation*}
Now apply \eqref{krel} to $\kappa_{s_\varphi v}(\cc)$ and use \eqref{thetalength} to arrive at
\begin{equation}\label{kbusiness}
\kappa_{s_\varphi v}(\cc)=\kappa_v(\cc)k_{\varphi}^{\chi(v^{-1}\varphi)
\eta(\chi(v^{-1}\varphi)\varphi(v\cc))}\prod_{\alpha\in\Phi_0^+}k_\alpha^{\epsilon_v(\alpha)}
\end{equation}
with $\epsilon_v(\alpha):=-\chi(v^{-1}\alpha)\eta(\chi(v^{-1}\alpha)\alpha(v\cc))\alpha(\varphi^\vee)$ for $\alpha\in\Phi_0$. Since  
$\epsilon_v(-\alpha)=\epsilon_v(\alpha)$ for all $\alpha\in\Phi_0$, the product over $\alpha\in\Phi_0^+$ in \eqref{kbusiness} is unchanged when taking a different choice of positive roots for $\Phi_0^+$.
Choosing $v\Phi_0^+$ as set of positive roots we get
\begin{equation*}
\begin{split}
\kappa_{s_\varphi v}(\cc)&=\kappa_v(\cc)k_{\varphi}^{\chi(v^{-1}\varphi)
\eta(\chi(v^{-1}\varphi)\varphi(v\cc))}\prod_{\alpha\in\Phi_0^+}k_\alpha^{-\eta(\alpha(\cc))\alpha(v^{-1}\varphi^\vee)}\\
&=\kappa_v(\cc)k_{\varphi}^{\chi(v^{-1}\varphi)
\eta(\chi(v^{-1}\varphi)\varphi(v\cc))}\mathfrak{s}_\cc^{-v^{-1}\varphi^\vee},
\end{split}
\end{equation*}
where the last equality follows from the explicit expression of $\mathfrak{s}_y$ (see Definition
\ref{fraksdef}). By \eqref{taff} we conclude that
\begin{equation}\label{psiT0final}
\begin{split}
\psi_{\cc,\mathfrak{t}}(T_0x^\mu T_vY^\nu)=\kappa_v(\cc)&t^{-\nu}
\Bigl(\mathfrak{t}^{v^{-1}\varphi^\vee}k_\varphi^{\chi(v^{-1}\varphi)\eta(\chi(v^{-1}\varphi)
\varphi(v\cc))}q_\varphi^{\varphi(\mu)}x^{s_\varphi\mu+s_\varphi v\cc}\\
&\,\,+
(k_\varphi-k_\varphi^{-1})
\bigl(\nabla_0(x^\mu)x^{v\cc}+\chi_+(v^{-1}\varphi)q_\varphi^{\varphi(\mu)}x^{s_\varphi \mu+v\cc}
\bigr)\Bigr).
\end{split}
\end{equation}
On the other hand,
\begin{equation}\label{piT0final}
\begin{split}
\pi_{\cc,\mathfrak{t}}(T_0)&\psi_{\cc,\mathfrak{t}}(x^\mu T_vY^\nu)=\\
&=\kappa_v(\cc)t^{-\nu}\Bigl(\mathfrak{t}^{v^{-1}\varphi^\vee}k_\varphi^{\chi_{\mathbb{Z}}(\varphi(v\cc))}q_\varphi^{\varphi(\mu)}x^{s_\varphi \mu+s_\varphi v\cc}
+(k_\varphi-k_\varphi^{-1})\nabla_0(x^{\mu+v\cc})\Bigr)
\end{split}
\end{equation}
since $\mathfrak{t}_{\mu+v\cc}^{\varphi^\vee}=q_\varphi^{\varphi(\mu)}\mathfrak{t}^{v^{-1}\varphi^\vee}$ by Definition \ref{defty}. Furthermore,
\begin{equation}\label{nabla0rel}
\nabla_0(x^{\mu+v\cc})=
\begin{cases}
\nabla_0(x^\mu)x^{v\cc}-q_{\varphi}^{\varphi(\mu+v\cc)}x^{s_\varphi \mu+s_\varphi v\cc}
\quad &\hbox{ if }\,\,\,\, \varphi(v\cc)=-1,\\
\nabla_0(x^\mu)x^{v\cc}\quad &\hbox{ if }\,\,\,\, -1<\varphi(v\cc)\leq 0,\\
\nabla_0(x^\mu)x^{v\cc}+q_{\varphi}^{\varphi(\mu)}x^{s_\varphi\mu+v\cc}\quad
&\hbox{ if }\,\,\,\, 0<\varphi(v\cc)\leq 1.
\end{cases}
\end{equation}
Now as in case 1 we prove that the expression \eqref{psiT0final} is equal to \eqref{piT0final} by distinguishing various cases related to the position of
$v\cc$ relative to the affine root hyperplanes $\{z\in E \,\, | \,\, \varphi(z)=m\}$ ($m=-1,0,1$). Since there are various additional subtleties, we discuss the cases in detail.\\
{\it Case 2a:} $0<\varphi(v\cc)\leq 1$. Then $v^{-1}\varphi\in\Phi_0^+$, hence $\chi(v^{-1}\varphi)=1$, $\eta(\varphi(v\cc))=\chi_{\mathbb{Z}}(\varphi(v\cc))$ and $\chi_+(v^{-1}\varphi)=1$. The desired result now follows immediately from \eqref{nabla0rel}.\\
{\it Case 2b:} $\varphi(v\cc)=0$. Then $\eta(\varphi(v\cc))=-1$, $s_\varphi v\cc=v\cc$, and $\mathfrak{t}^{v^{-1}\varphi^\vee}=1$ by Lemma \ref{datumokcor}. Substitution into \eqref{psiT0final} gives
\begin{equation*}
\begin{split}
\psi_{\cc,\mathfrak{t}}(T_0x^\mu T_vY^\nu)=
\kappa_v(\cc)t^{-\nu}&\Bigl(k_\varphi^{-\chi(v^{-1}\varphi)}q_{\varphi}^{\varphi(\mu)}
x^{s_\varphi\mu+v\cc}\\
&\,\,+(k_\varphi-k_\varphi^{-1})\bigl(\nabla_0(x^\mu)x^{v\cc}+\chi_+(v^{-1}\varphi)
q_\varphi^{\varphi(\mu)}x^{s_\varphi\mu+v\cc}\bigr)\Bigr)\\
=\kappa_v(\cc)t^{-\nu}&\Bigl(k_\varphi q_\varphi^{\varphi(\mu)}x^{s_\varphi\mu+v\cc}
+(k_\varphi-k_\varphi^{-1})\nabla_0(x^\mu)x^{v\cc}\Bigr).
\end{split}
\end{equation*}
Note here that for $v^{-1}\varphi\in\Phi_0^+$, the second equality follows by combining the first and third term. By \eqref{nabla0rel} this expression reduces to \eqref{piT0final} (using once more $\mathfrak{t}^{v^{-1}\varphi^\vee}=1$).\\
{\it Case 2c:} $-1<\varphi(v\cc)<0$. Then $v^{-1}\varphi\in\Phi_0^-$, $\eta(-\varphi(v\cc))=0$ and $\chi_+(v^{-1}\varphi)=0$. Using \eqref{nabla0rel} it follows that \eqref{psiT0final} equals \eqref{piT0final}.\\ 
{\it Case 2d:} $\varphi(v\cc)=-1$. Then $v^{-1}\varphi\in\Phi_0^-$, $\eta(-\varphi(v\cc))=1$,
$\chi_+(v^{-1}\varphi)=0$, and $\mathfrak{t}^{v^{-1}\varphi^\vee}=q_\varphi^{-1}$ by Lemma \ref{datumokcor}. Hence
\begin{equation}\label{T01}
\psi_{\cc,\mathfrak{t}}(T_0x^\mu T_vY^\nu)=
\kappa_v(\cc)t^{-\nu}\Bigl(k_\varphi^{-1}q_\varphi^{\varphi(\mu+v\cc)}
x^{s_\varphi\mu+s_\varphi v\cc}+(k_\varphi-k_\varphi^{-1})\nabla_0(x^\mu)x^{v\cc}\Bigr).
\end{equation}
By \eqref{nabla0rel} the expression \eqref{piT0final} becomes
\begin{equation}\label{T02}
\begin{split}
\pi_{\cc,\mathfrak{t}}(T_0)\psi_{\cc,\mathfrak{t}}(x^\mu T_vY^\nu)=
\kappa_v(\cc)&t^{-\nu}\Bigl(k_\varphi q^{\varphi(\mu+v\cc)}_\varphi x^{s_\varphi\mu+s_\varphi v\cc}\\
&+(k_\varphi-k_\varphi^{-1})\bigl(\nabla_0(x^\mu)x^{v\cc}-q_\varphi^{\varphi(\mu+v\cc)}x^{s_\varphi\mu+s_\varphi v\cc}\bigr)\Bigr).
\end{split}
\end{equation}
Then \eqref{T02} reduces to \eqref{T01} by combining the first and third term in \eqref{T02} into a single term.
\end{proof}
Lemma \ref{generatoractiontransfer} implies that $\textup{ker}(\psi_{\cc,\mathfrak{t}})\subseteq\mathbb{H}$ is a left ideal, hence $\mathbb{H}/\textup{ker}(\psi_{\cc,\mathfrak{t}})$
is canonically a left $\mathbb{H}$-module. Now consider $\Pc$ as left $\mathbb{H}$-module by transporting the left $\mathbb{H}$-action on $\mathbb{H}/\textup{ker}(\psi_{\cc,\mathfrak{t}})$ through the linear isomorphism $\mathbb{H}/\textup{ker}(\psi_{\cc,\mathfrak{t}})\overset{\sim}{\longrightarrow}\Pc$ induced by $\psi_{\cc,\mathfrak{t}}$.
By Lemma \ref{generatoractiontransfer} again, its representation map is characterised by \eqref{actionformulas}. This completes the proof of Theorem \ref{gbr}(1).
\begin{corollary}\label{monomialcor}
View $\mathbb{H}$ as left $\mathbb{H}$-module with respect to the left-regular $\mathbb{H}$-action.
\begin{enumerate}
\item The surjective linear map $\psi_{\cc,\mathfrak{t}}:\mathbb{H}\twoheadrightarrow\Pc_{\mathfrak{t}}$, defined by \eqref{psic}, is $\mathbb{H}$-linear.
\item We have
\[
\psi_{\cc,\mathfrak{t}}(h)=\pi_{\cc,\mathfrak{t}}(h)x^\cc\qquad \forall\, h\in\mathbb{H}.
\]
In particular, $\pi_{\cc,\mathfrak{t}}(T_v)x^\cc=\kappa_v(\cc)x^{v\cc}$ for all $v\in W_0$.
\end{enumerate}
\end{corollary}
\begin{proof}
(1) This is immediate from the paragraph preceding the corollary.\\
(2) By (1) we have
\[
\psi_{\cc,\mathfrak{t}}(h)=\pi_{\cc,\mathfrak{t}}(h)\psi_{\cc,\mathfrak{t}}(1)=
\pi_{\cc,\mathfrak{t}}(h)x^\cc\qquad \forall\, h\in\mathbb{H}.
\]
The second statement then follows from \eqref{psic}.
\end{proof}
\subsection{The $\mathbb{H}$-module $\mathcal{P}^{(\cc)}_{\mathfrak{t}}$ as quotient of $\mathbb{M}^J_{\mathfrak{s}_\cc\mathfrak{t}}$}\label{mapsection}
Take $\cc\in C^J$ and $\mathfrak{t}\in T_{J}$. Then $\mathfrak{s}_\cc\mathfrak{t}\in L_J$, and we have 
the one-dimensional representation $\chi_{J,\mathfrak{s}_\cc\mathfrak{t}}: \mathbb{H}_J^Y\rightarrow\mathbf{F}$ from \eqref{chit}.
\begin{lemma}\label{actionHt}
We have $\pi_{\cc,\mathfrak{t}}(h)x^\cc=\chi_{J,\mathfrak{s}_{\cc}\mathfrak{t}}(h)x^\cc$ for all $h\in\mathbb{H}_{J}^Y$.
\end{lemma}
\begin{proof}
Write $t:=\mathfrak{s}_\cc\mathfrak{t}\in L_{J}$ for the duration of this proof. 

For $\mu\in Q^\vee$ we have by Corollary \ref{monomialcor},
\[
\pi_{\cc,\mathfrak{t}}(Y^\mu)x^\cc=\psi_{\cc,\mathfrak{t}}(Y^\mu)=t^{-\mu}x^\cc=\chi_{J,t}(Y^\mu)x^\cc.
\]
For $i\in J_0$ we have $\kappa_{s_i}(\cc)=k_i$ and $s_i\cc=\cc$ since $\alpha_i(\cc)=0$, hence 
\[
\pi_{\cc,\mathfrak{t}}(\delta(T_i))x^\cc=\pi_{\cc,\mathfrak{t}}(T_i)x^\cc=
k_ix^\cc=\chi_{J,t}(\delta(T_i))x^\cc,
\]
where we used Corollary \ref{monomialcor}(2) for the second equality.

Suppose that $0\in J$, so that $s_0\cc=\cc$ and
$\varphi(\cc)=1$. We need to show that $\pi_{\cc,\mathfrak{t}}(\delta(T_0))x^\cc=
k_{\varphi}x^\cc$.  
Using $\delta(T_0^{-1})=x^{\varphi^\vee}T_0^{-1}Y^{\varphi^\vee}$ we compute
\begin{equation*}
\pi_{\cc,\mathfrak{t}}(\delta(T_0^{-1}))x^\cc=t^{-\varphi^\vee}\pi_{\cc,\mathfrak{t}}(x^{\varphi^\vee}T_0^{-1})x^\cc=q_\varphi k_\varphi t^{-\varphi^\vee}x^\cc,
\end{equation*}
where the second equality follows from a direct computation using that
\[
\pi_{\cc,\mathfrak{t}}(T_j)x^y=k_jx^y+\Bigl(\frac{k_jx^{\alpha_j^\vee}-k_j^{-1}}{x^{\alpha_j^\vee}-1}\Bigr)
(x^{y-D\alpha_j(y)\alpha_j^\vee}-x^y)\,\,\, \hbox{ if }\, y\in\mathcal{O}_\cc\, \hbox{ and }\, \alpha_j(y)\in\mathbb{Z}.
\]
Now $\varphi(\cc)=1$ and $\mathfrak{t}\in T_{J}$ implies that $\mathfrak{t}^{-\varphi^\vee}=q_\varphi^{-1}$ by Lemma \ref{datumokcor}, and $\mathfrak{s}_\cc^{-\varphi^\vee}=k_\varphi^{-2}$ by Lemma \ref{fraks}(1). Hence $\pi_{\cc,\mathfrak{t}}(\delta(T_0^{-1}))x^\cc=k_\varphi^{-1}x^\cc$, as desired.
\end{proof}
\begin{corollary}\label{factorcorollary}
There exists a unique epimorphism 
$\phi_{\cc,\mathfrak{t}}: \mathbb{M}^{J}_{\mathfrak{s}_\cc\mathfrak{t}}\twoheadrightarrow \Pc_{\mathfrak{t}}$ of $\mathbb{H}$-modules such that
\[
\psi_{\cc,\mathfrak{t}}(h)=\phi_{\cc,\mathfrak{t}}\bigl(hm^{J}_{\mathfrak{s}_\cc\mathfrak{t}}\bigr)\qquad
\forall\, h\in\mathbb{H}.
\]
Furthermore,
\begin{equation}\label{mmcomp}
\phi_{\cc,\mathfrak{t}}\bigl(x^\mu T_vm^{J}_{\mathfrak{s}_\cc\mathfrak{t}}\bigr)=
\kappa_v(\cc)x^{\mu+v\cc}
\qquad (\mu\in Q^\vee,\, v\in W_0).
\end{equation}
\end{corollary}
\begin{proof}
The first part is an immediate consequence of Corollary \ref{monomialcor}(1) and Lemma \ref{actionHt}. For
\eqref{mmcomp} note that
\[
\phi_{\cc,\mathfrak{t}}\bigl(x^\mu T_vm^{J}_{\mathfrak{s}_\cc\mathfrak{t}}\bigr)=
\psi_{\cc,\mathfrak{t}}(x^\mu T_v)=\kappa_v(\cc)x^{\mu+v\cc}
\]
for $\mu\in Q^\vee$ and $v\in W_0$, where the second equality follows from \eqref{psic}.
\end{proof}

\subsection{The parabolic Bruhat order on $E$}\label{POsection}
The restrictions of the Bruhat order on $W$ to the subsets $W^J$ ($J\subsetneq [0,r]$) of minimal coset representatives can be glued together to a partial order on $E$, which we call the parabolic Bruhat order on $E$. This subsection is devoted to a detailed study of its properties.
In the next subsection we will continue with the proof of Theorem \ref{gbr} by
showing  
that the action of $\delta(T_j)$ on $\Pc_{\mathfrak{t}}$ 
is triangular with respect to the basis $\{x^y\}_{y\in\mathcal{O}_\cc}$ of monomials, partially ordered by the restriction of the parabolic Bruhat order to the set 
$\mathcal{O}_\cc$ of quasi-exponents. 

Let $\leq_B$ be the Bruhat order on $W$ with respect to the simple reflections $s_0,\ldots,s_r$ (see \cite[Chpt. 5]{Hu} for details). 
We write $w<_Bw^\prime$ if $w\leq_Bw^\prime$ and $w\not=w^\prime$. The following two properties of $\leq_B$ are well known.
\begin{lemma}\label{Bruhattrans}\hfill
\begin{enumerate}
\item Let $w,w^\prime\in W$ and $0\leq j\leq r$ satisfying
$w<_Bw^\prime<_Bs_jw^\prime$. Then we have $s_jw<_Bs_jw^\prime$.
\item For $a\in\Phi$ and $w\in W$ we have
\[
ws_a<_Bw\,\,\, \Leftrightarrow\,\,\, a\in\Pi(w)\cup -\Pi(w)\,\,\,\Leftrightarrow\,\,\,
\ell(ws_a)<\ell(w).
\] 
\end{enumerate}
\end{lemma}
\begin{proof}
(1) See, e.g., \cite[Prop. 5.9]{Hu}.\\
(2) See \cite[(2.3.3)]{Ma}.
\end{proof}

\begin{definition}\label{orderB}
For $y,z\in E$ we write $y\leq z$ if $y\in Wz$
and $w_y\leq_Bw_z$.
\end{definition}
In other words, elements in $E$ from different $W$-orbits are incomparable, and the restriction of $\leq$ to a $W$-orbit $\mathcal{O}_\cc$ ($\cc\in C^J$) corresponds to the 
Bruhat order on $W^J$ through the natural bijection $\mathcal{O}_\cc\overset{\sim}{\longrightarrow} W^J$, $y\mapsto w_y$ (which, in turn, can be regarded as the 
natural partial order on the coset space $W/W_\cc\simeq W^J$ induced by $\leq_B$). 
In particular, $\leq$ defines a partial order on $E$. 
 
 We will now establish an alternative description of the partial order $\leq$.
 
\begin{definition}\label{precalpha}
Let $\alpha\in\Phi^+_0$ and $y,z\in E$. Write $y\prec_\alpha z$ if 
the following two conditions hold true:
\begin{enumerate}
\item $y=s_{a}z$ for some $a\in\Phi$ with $Da=\alpha$,
\item either $|\alpha(y)|<|\alpha(z)|$ or $\alpha(y)=-\alpha(z)>0$.
\end{enumerate}
Let $\prec$ be the transitive closure of the relations $\prec_\alpha$ \textup{(}$\alpha\in\Phi^+_0$\textup{)} on $E$.
\end{definition}
\begin{remark}\label{remprec}\hfill
\begin{enumerate}
\item For $\alpha\in\Phi^+_0$ we have $y\prec_\alpha s_\alpha y$ iff $\alpha(y)>0$. 
\item Elements from different $W$-orbits in $E$ are incomparable with respect to $\prec$.
\item The relation $\prec$ on $Q^\vee$ coincides with the relation $\prec_{\Phi_0^+}$ introduced by Haiman in \cite[(5.14)]{Ha}.
\end{enumerate}
\end{remark}
Let $\textup{sgn}: \mathbb{Z}\rightarrow\{-1,1\}$ be the sign function, so $\textup{sgn}(m)=\frac{m}{|m|}$ for $m\not=0$ and $\textup{sgn}(0)=1$.
\begin{proposition}\label{propequivrel}
Let $\cc\in\overline{C}_+$, $w\in W$ and $a=(\alpha,\ell)\in\Phi$ with $\alpha\in\Phi_0^+$. The following three statements are equivalent.
\begin{enumerate}
\item $s_aw<_Bw$ and $a(w\cc)\not=0$.
\item $\textup{sgn}(\ell)a(w\cc)<0$.
\item $s_aw\cc\prec_\alpha w\cc$.
\end{enumerate}
\end{proposition} 
\begin{proof}
(1)$\Leftrightarrow$(2): By Lemma \ref{Bruhattrans}(2), the identity $w^{-1}s_aw=s_{w^{-1}a}$, and
the description \eqref{Phiplusexplicit} of $\Phi^+$, we have 
\[
s_aw<_Bw\,\,\,\Leftrightarrow\,\,\,\textup{sgn}(\ell)w^{-1}a\in\Phi^-.
\]
Furthermore, for $b\in\Phi$ with $b(\cc)\not=0$ we have $b\in\Phi^-$ iff $b(\cc)<0$, since $\cc\in\overline{C}_+$. The result follows from these observations and the fact that
$(w^{-1}a)(\cc)=a(w\cc)$.\\
(2)$\Rightarrow$(3): If $\ell=0$ then $\alpha(w\cc)<0$, hence $s_aw\cc=s_\alpha w\cc\prec_\alpha w\cc$ by Remark \ref{remprec}(1). If $\ell\not=0$ then $0<|\ell|<-\textup{sgn}(\ell)\alpha(w\cc)$, which implies that
\[
|\alpha(w\cc)|=-\textup{sgn}(\ell)\alpha(w\cc)>\pm\textup{sgn}(\ell)(\alpha(w\cc)+2\ell).
\]
Here the inequality for the minus-sign follows from $|\ell |>0$, and the inequality for the plus-sign follows from $|\ell |<-\textup{sgn}(\ell)\alpha(w\cc)$. Hence $|\alpha(w\cc)+2\ell |<|\alpha(w\cc)|$, and
we conclude that
\[
|\alpha(s_aw\cc)|=|\alpha(w\cc)+2\ell|<
|\alpha(w\cc)|.
\]
Hence $s_aw\cc\prec_\alpha w\cc$.\\
(3)$\Rightarrow$(2): We have either $|\alpha(w\cc)+2\ell |<|\alpha(w\cc)|$ or
$\alpha(w\cc)+2\ell=\alpha(w\cc)<0$. In the second case we have $\ell=0$, and hence $\textup{sgn}(\ell)a(w\cc)=\alpha(w\cc)<0$, as desired. Now suppose that
$|\alpha(w\cc)+2\ell |<|\alpha(w\cc)|$. Then 
$\alpha(w\cc)\not=0$. If $\alpha(w\cc)>0$ then 
$\ell\in\mathbb{Z}_{<0}$ and $-\alpha(w\cc)-2\ell<\alpha(w\cc)$, i.e., $-a(w\cc)<0$. If $\alpha(w\cc)<0$
then $\ell\in\mathbb{Z}_{>0}$ and 
$\alpha(w\cc)+2\ell<-\alpha(w\cc)$, i.e., $a(w\cc)<0$. So again we have $\textup{sgn}(\ell)a(w\cc)<0$, as desired.
\end{proof}
\begin{corollary}\label{equivalenceorder1} 
Let $w,w^\prime\in W$ with $w\leq_B w^\prime$. Then $w\cc\preceq w^\prime\cc$
for all $\cc\in\overline{C}_+$.
\end{corollary}
\begin{proof}
Let $\cc\in\overline{C}_+$ and suppose that $w<_Bw^\prime$. Then there exists a chain
\[
w=:w_0^\prime<_Bw_1^\prime<_B\cdots <_Bw_{m-1}^\prime<_Bw_m^\prime:=w^\prime
\]
with $w_j^\prime\in W$ satisfying $w_{j-1}^\prime=s_{b_j}w_j^\prime$ ($1\leq j\leq m$) for affine roots $b_j\in\Phi$ with $\beta_j:=Db_j\in\Phi_0^+$. If $b_j(w_{j-1}^\prime\cc)=0$ then $w_{j-1}^\prime\cc=w_j^\prime\cc$. If $b_j(w_{j-1}^\prime\cc)\not=0$
then $w_{j-1}^\prime\cc\prec_{\beta_j}w_j^\prime\cc$ by Proposition \ref{propequivrel}. Hence $w\cc\preceq w^\prime\cc$.
\end{proof}

\begin{proposition}\label{equivalenceorder2}\hfill
The relation $\preceq$ on $E$ coincides with the partial order $\leq$ on $E$.
\end{proposition}
\begin{proof}
Let $y,z\in E$. If $y\leq z$ then $w_y\leq_Bw_z$ and $\cc_y=\cc_z$. Write $\cc:=\cc_y$, then we get $y=w_y\cc\preceq w_z\cc=z$ by Corollary \ref{equivalenceorder1}.

Suppose that $y\prec z$. Then $y$ and $z$ lie in the same $W$-orbit by Remark \ref{remprec}(2), hence $\cc_y=\cc_z$. Write $\cc:=\cc_y$.
There exists
a chain 
\[
y=:y_0\prec_{\beta_1}y_1\prec_{\beta_2}\cdots\prec_{\beta_{m-1}}y_{m-1}\prec_{\beta_m}y_m:=z
\]
with $\beta_1,\ldots,\beta_m\in\Phi_0^+$ and $y_j\in\mathcal{O}_\cc$ satisfying $y_{j-1}=s_{b_j}y_j$ for some $b_j\in\Phi$ with $Db_j=\beta_j$. Then $s_{b_j}w_{y_j}<_Bw_{y_j}$ ($1\leq j\leq m$) by Proposition \ref{propequivrel}. Since $w_{y_{j-1}}\cc=y_{j-1}=s_{b_j}w_{y_j}\cc$ and $w_{y_{j-1}}$
is a minimal coset representative of $w_{y_{j-1}} W_\cc$, we have
$w_{y_{j-1}}\leq_B s_{b_j}w_{y_j}$, hence
$w_{y_{j-1}}<_{B}w_{y_j}$ ($1\leq j\leq m$). This implies $w_y<_{B}w_z$, hence $y<z$.
\end{proof}
\begin{corollary}\label{Bruhatproperty}
Let $y,z\in E$ and $0\leq j\leq r$, and suppose that $y<z<s_jz$. Then $s_jy<s_jz$.
\end{corollary}
\begin{proof}
Let $\cc\in\overline{C}_+$ such that $y,z\in \mathcal{O}_\cc$.
Since 
$w_{s_jz}$ is a minimal coset representative of the coset
$s_jw_z W_\cc$, we have $w_{s_jz}\leq_B s_jw_z$. 
By the assumption $z<s_jz$ we have 
$w_z<_Bw_{s_jz}$. Combining the two relations we get $w_z<_Bs_jw_z$. Since furthermore $y<z$,
we conclude that
\[
w_y<_Bw_z<_Bs_jw_z,
\]
hence $s_jw_y<_Bs_jw_z$ by Lemma \ref{Bruhattrans}(1). Then $s_jy<s_jz$ by Corollary
\ref{equivalenceorder1} and Proposition \ref{equivalenceorder2}.
\end{proof}
\begin{proposition}\label{reflectionrelation}
Let $y\in E$ and $a\in\Phi^+$. 
Then $y<s_ay$ iff $a(y)>0$.
\end{proposition}
\begin{proof}
Suppose that $y<s_ay$. Then $a(y)\not=0$ and 
$w_y<_B w_{s_ay}$.
Writing $s_a w_y=w_{s_ay}w^\prime$ with $w^\prime\in W_{\cc_y}$, we have 

\[
w_y<_Bw_{s_ay}\leq_B w_{s_ay}w^\prime=s_a w_y
\]
since $\ell(w_{s_ay}w^\prime)=
\ell(w_{s_ay})+\ell(w^\prime)$.
So $w_y<_Bs_a w_y$, which implies that $w_y^{-1}a\in\Phi^+$ by \cite[Prop. 5.7]{Hu}. 
It follows that $a(y)=(w_y^{-1}a)(\cc_y)\geq 0$.
Since $a(y)\not=0$ this forces $a(y)>0$.

For the converse implication, let $a=(\alpha,\ell)\in\Phi^+$ such that $a(y)=\ell+\alpha(y)>0$. The condition that $a$ is a positive root implies that $\ell\in\mathbb{Z}_{\geq 0}$ if $\alpha\in\Phi_0^+$, and
$\ell\in\mathbb{Z}_{>0}$ if $\alpha\in\Phi_0^-$. It follows that
\[
\alpha(s_ay)=-\alpha(y)-2\ell<0,
\]
hence $|\alpha(s_ay)|=\alpha(y)+2\ell$. Since $\ell\in\mathbb{Z}_{\geq 0}$ we conclude that 
$\alpha(y)\leq |\alpha(s_ay)|$.
If $\alpha(y)=|\alpha(s_ay)|$ then $\ell=0$, so $a(y)=\alpha(y)>0$. 
Then $\alpha(y)=-\alpha(s_\alpha y)=-\alpha(s_ay)>0$. We conclude that $y\prec s_ay$, and hence $y<s_ay$ by Proposition \ref{equivalenceorder2}. 
If $\alpha(y)<|\alpha(s_ay)|$ then $\ell\in\mathbb{Z}_{>0}$ and
\[
-\alpha(y)=-(\alpha(y)+\ell)+\ell<
(\alpha(y)+\ell)+\ell=|\alpha(s_ay)|,
\]
hence $|\alpha(y)|<|\alpha(s_ay)|$. Consequently $y\prec s_ay$, so $y<s_ay$ by Proposition \ref{equivalenceorder2}.
\end{proof}
\begin{corollary}\label{newroots}
Let $y\in E$ and $a\in\Phi^+$. Then
$s_ay<y$ iff $a\in\Pi(w_y^{-1})$ iff $a(y)<0$.
\end{corollary}
\begin{proof}
Combine Proposition \ref{Pidescription} and Proposition \ref{reflectionrelation}.
\end{proof}
\begin{corollary}\label{wyrules}
Let $\cc\in C^J$, $y\in\mathcal{O}_\cc$ and $j\in [0,r]$.
\begin{enumerate}
\item If $\alpha_j(y)\not=0$ then $w_{s_jy}=s_jw_y$ and 
\[
\ell(s_jw_y)=\ell(w_y)\pm 1\quad \Leftrightarrow\quad \alpha_j(y)\gtrless 0.
\]
\item If $\alpha_j(y)=0$ then $w_{s_jy}=w_y$ and $\ell(s_jw_y)=\ell(w_y)+1$.
\end{enumerate}
\end{corollary}
\begin{proof}
(1) If $\alpha_j(y)\not=0$ then $s_jy\not=y$ and so $s_jw_y\not\in w_yW_{\cc}$. Hence
$s_jw_y\in W^{J}$ by Lemma \ref{cosetcomb}(2). Furthermore, $s_jw_y\cc=s_jy=w_{s_jy}\cc$,
hence $s_jw_y=w_{s_jy}$. If $\alpha_j(y)>0$ then $w_y<_Bs_jw_y$ by Proposition \ref{reflectionrelation}, hence $\ell(s_jw_y)=\ell(w_y)+1$ by Lemma \ref{Bruhattrans}(2). 
Similarly, if $\alpha_j(y)<0$ then $w_{s_jy}<_Bw_y$ by Corollary \ref{newroots},
hence $\ell(s_jw_y)=\ell(w_y)-1$ by Lemma \ref{Bruhattrans}(2) again.\\
(2) If $\alpha_j(y)=0$ then $s_jy=y$, so the first equality is trivial. Then $s_jw_y\in w_yW_{\cc}$,
hence the second part of the statement follows from Lemma \ref{cosetcomb}(2).
\end{proof}
We conclude this subsection with the following lemma, which explicitly describes root strings with respect to $\prec_\alpha$. 
\begin{lemma}\label{techlem}
Let $\alpha\in\Phi_0^+$ and $y\in\mathcal{O}_\cc$. Then we have
\begin{enumerate}
\item If $\alpha(y)\in\mathbb{R}_{<0}$ then $s_\alpha y-\ell\alpha^\vee\prec_\alpha y$ for $0\leq \ell\leq -1-\lfloor\alpha(y)\rfloor$,
\item If $\alpha(y)\in\mathbb{Z}_{>1}$ then $y-\ell\alpha^\vee\prec_\alpha y$ for $1\leq \ell\leq \alpha(y)-1$,
\item If $\alpha(y)\in\mathbb{R}_{>1}\setminus\mathbb{Z}_{>1}$ then $s_\alpha y+\ell\alpha^\vee\prec_\alpha y$ for $1\leq \ell\leq \lfloor \alpha(y)\rfloor$.
\end{enumerate}
\end{lemma}
\begin{proof}
(1) Write $y_\ell:=s_\alpha y-\ell\alpha^\vee$ with $0\leq \ell\leq -1-\lfloor\alpha(y)\rfloor$. Then
$y_\ell=s_ay$ with $a:=(\alpha,\ell)\in\Phi$. If $\ell=0$ then $\alpha(y_0)=-\alpha(y)>0$ and so 
$y_0\prec_\alpha y$. If $1\leq \ell\leq -1-\lfloor\alpha(y)\rfloor$ then 
\[
\alpha(y)<2+2\lfloor\alpha(y)\rfloor-\alpha(y)\leq -2\ell-\alpha(y)=\alpha(y_\ell)<-\alpha(y).
\]
Hence $|\alpha(y_\ell)|<|\alpha(y)|$, and so $y_\ell\prec_\alpha y$.\\
(2) Write $y_\ell:=y-\ell\alpha^\vee$ with $1\leq \ell\leq \alpha(y)-1$. Then we have $y_\ell=s_ay$ with $a=(\alpha,\ell-\alpha(y))\in\Phi$ since $\alpha(y)\in\mathbb{Z}$, and $-\alpha(y)+2\leq\alpha(y)-2\ell=\alpha(y_\ell)\leq \alpha(y)-2$. Hence $|\alpha(y_\ell)|<|\alpha(y)|$, and so $y_\ell\prec_\alpha y$.\\
(3) Write $y_\ell:=s_\alpha y+\ell\alpha^\vee$ with $1\leq \ell\leq\lfloor \alpha(y)\rfloor$. Then $y_\ell=s_ay$ with
$a=(\alpha,-\ell)\in\Phi$, and 
\[
-\alpha(y)<2-\alpha(y)\leq 2\ell-\alpha(y)=\alpha(y_\ell)\leq 2\lfloor \alpha(y)\rfloor-\alpha(y)<\alpha(y),
\]
where the last inequality follows from the fact that $\alpha(y)\not\in\mathbb{Z}$. Hence
$|\alpha(y_\ell)|<|\alpha(y)|$, and so $y_\ell\prec_\alpha y$.
\end{proof}

\subsection{Triangularity properties}\label{POsection2}
Throughout this subsection we fix $\cc\in C^J$ and $\mathfrak{t}\in T_{J}$.
We will analyse the triangularity properties of the action of $Y^\mu$ ($\mu\in Q^{\vee}$)
and $\delta(T_j)$ ($0\leq j\leq r$) on $\Pc_{\mathfrak{t}}$ with respect to the basis $\{x^y\}_{y\in\mathcal{O}_\cc}$ of quasi-monomials, partially ordered
by the parabolic Bruhat order on $\mathcal{O}_\cc$.

\begin{definition}
For $f\in\mathbf{F}[E]$ and $y\in E$ we write 
\begin{equation}\label{ftriang}
f=dx^y+\textup{l.o.t.}
\end{equation}
if $f$ is of the form
\[
f=dx^y+\sum_{z<y}d_{z}x^z\qquad (d,d_z\in\mathbf{F}).
\] 
If $d\not=0$ then $y$ is called the leading exponent of $f$,
and $d$ its leading coefficient.
\end{definition}
\begin{remark}\label{orderfinite}
If $f\in\mathbf{F}[E]$ is of the form \eqref{ftriang} then $f$ is a quasi-polynomial. In fact, 
$z<y$ forces $z\in Wy$, so that $f\in\mathcal{P}^{(\cc_y)}$.
Furthermore, for $y\in \mathcal{O}_\cc$ ($\cc\in C^J$) we have
\[
\#\{z\in E \,\, | \,\, z<y\}<\infty
\]
since $\#\{w\in W^J\,\, | \,\, w<_Bw_y\}<\infty$.
\end{remark}

We now first show that the linear operators $\pi_{\cc,\mathfrak{t}}(Y^\mu)$
($\mu\in Q^\vee$) are triangular. 
For an affine root $a\in\Phi$ define
$G_{\mathfrak{t}}(a)\in\textup{End}(\Pc)$ by
\[
G_{\mathfrak{t}}(a)x^y:=k_a^{\chi_{\mathbb{Z}}(a(y))}x^y+
(k_a-k_a^{-1})\Bigl(\frac{1-x^{\lfloor Da(y)\rfloor a^\vee}}{1-x^{-a^\vee}}\Bigr)s_{a,\mathfrak{t}}x^y\qquad (y\in\mathcal{O}_{\cc}).
\]
Using Lemma \ref{actiondeform} and the fact that $w_{\mathfrak{t}}x^y$ for $w\in W$ is a scalar 
multiple of $x^{(Dw)y}$, one shows that 
\begin{equation}\label{eq1}
G_{\mathfrak{t}}(\alpha_j)=s_{j,\mathfrak{t}}\pi_{\cc,\mathfrak{t}}(T_j)\qquad (j=0,\ldots,r)
\end{equation}
and 
\begin{equation}\label{eq2}
G_{\mathfrak{t}}(wa)=w_{\mathfrak{t}}G_{\mathfrak{t}}(a)w_{\mathfrak{t}}^{-1}\qquad (w\in W,\,\, a\in\Phi)
\end{equation}
as linear operators on $\Pc$.

\begin{lemma}\label{triangularG}
Fix $a=(\alpha,\ell)\in\Phi$ with $\alpha\in\Phi_0^+$ and $\ell\in\mathbb{Z}$.
Then
\[
G_{\mathfrak{t}}(a)x^y=k_\alpha^{-\eta(\alpha(y))}x^y+\textup{l.o.t.}
\]
for all $y\in \mathcal{O}_{\cc}$.
\end{lemma}
\begin{proof}
Note that 
\begin{equation*}
k_\alpha^{-\eta(\alpha(y))}=
\begin{cases} k_a^{-1}\quad &\hbox{ if }\alpha(y)\in\mathbb{Z}_{>0},\\
k_a^{\chi_{\mathbb{Z}}(a(y))}\quad &\textup{ if }
\alpha(y)\in\mathbb{R}\setminus\mathbb{Z}_{>0}.
\end{cases}
\end{equation*}
Keeping in mind Lemma \ref{datumokcor}, one then obtains by direct computations,
\begin{equation*}
\begin{split}
G_{\mathfrak{t}}(a)x^y&-k_\alpha^{-\eta(\alpha(y))}x^y=\\
&=
\begin{cases}
(k_a-k_a^{-1})(1+x^{-a^\vee}+\cdots+x^{(1+\lfloor\alpha(y)\rfloor)a^\vee})s_{a,\mathfrak{t}}x^{y}
&\hbox{ if } \alpha(y)\in\mathbb{R}_{<0},\\
0 &\hbox{ if }  0\leq \alpha(y)\leq 1,\\
(k_a^{-1}-k_a)(x^{-a^\vee}+x^{-2a^\vee}+\cdots+x^{(1-\alpha(y))a^\vee})x^y
&\hbox{ if } \alpha(y)\in\mathbb{Z}_{>1},\\
(k_a^{-1}-k_a)(x^{a^\vee}+x^{2a^\vee}+\cdots+x^{\lfloor \alpha(y)\rfloor a^\vee})s_{a,\mathfrak{t}}x^{y}
\,\,&\hbox{ if } \alpha(y)\in \mathbb{R}_{>1}\setminus\mathbb{Z}_{>1}.
\end{cases}
\end{split}
\end{equation*}
Now note that $x^{a^\vee}=q_\alpha^\ell x^{\alpha^\vee}$, and $s_{a,\mathfrak{t}}x^y=\mathfrak{t}_y^{-\ell\alpha^\vee}x^{y-\alpha(y)\alpha^\vee}$ by \eqref{saxyt}, hence the result follows from
Lemma \ref{techlem}.
\end{proof}
\begin{proposition}\label{trianY}
For $\mu\in Q^\vee$ and $y\in\mathcal{O}_{\cc}$ we have
\[
\pi_{\cc,\mathfrak{t}}(Y^\mu)x^y=(\mathfrak{s}_y\mathfrak{t}_y)^{-\mu}x^y+\textup{l.o.t.}
\]
in $\Pc$.
\end{proposition}
\begin{proof}
It suffices to prove the proposition for $\mu\in Q^{\vee}\cap\overline{E}_+$.

If $\tau(\mu)=s_{j_1}\cdots s_{j_m}$ is a reduced expression and $b_1,\ldots,b_m$ are the associated positive affine roots in $\Pi(\tau(\mu))$ (see \eqref{rootsdescription}), then 
\[
\pi_{\cc,\mathfrak{t}}(Y^\mu)=\tau(\mu)_\mathfrak{t}G_{\mathfrak{t}}(b_1)\cdots G_{\mathfrak{t}}(b_m)
\]
as linear operators on $\Pc$ by the equivariance properties \eqref{eq1} and \eqref{eq2} of $G_{\mathfrak{t}}(a)$.
Furthermore, since $\mu\in Q^{\vee}\cap\overline{E}_+$ we have 
\begin{equation*}
\Pi(\tau(\mu))=
\bigsqcup_{\alpha\in\Phi^+_0} \{(\alpha,\ell)\,\, | \,\, 0\leq \ell<\alpha(\mu)\}.
\end{equation*}
Combined with Lemma \ref{triangularG} we get for $y\in\mathcal{O}_{\cc}$,
\begin{equation*}
\begin{split}
\pi_{\cc,\mathfrak{t}}(Y^\mu)x^y&=
\Bigl(
\prod_{a\in\Pi(\tau(\mu))}k_{Da}^{-\eta(Da(y))}\Bigr)\tau(\mu)_{\mathfrak{t}}x^y+\textup{ l.o.t.}\\
&=
\Bigl(\prod_{\alpha\in\Phi_0^+}k_\alpha^{-\eta(\alpha(y))\alpha(\mu)}\Bigr)\mathfrak{t}_y^{-\mu}x^y+\textup{l.o.t.}\\
&=(\mathfrak{s}_y\mathfrak{t}_y)^{-\mu}x^y+\textup{ l.o.t.}
\end{split}
\end{equation*}
in $\Pc$, as desired.
\end{proof}
Recall $k_w\in\mathcal{F}_\Sigma(E,\mathbf{F}^\times)$ ($w\in W$), defined by \eqref{kawy}. 
\begin{proposition}\label{triangularitymbasis}
For $\cc\in C^J$ and $w\in W^J$ we have
\[
\pi_{\cc,\mathfrak{t}}(\delta(T_{w^{-1}}))x^\cc=k_w(\cc)
x^{w\cc}+\textup{l.o.t.}
\]
\textup{(}here $w\cc$ is the action of $w$ on $\cc\in E$ by translations and reflections\textup{)}.
\end{proposition}
\begin{proof}
Let $w=s_{j_1}\cdots s_{j_m}$ be a reduced expression for $w\in W^{J}$, so that
\[
\pi_{\cc,\mathfrak{t}}(\delta(T_{w^{-1}}))x^\cc=
\pi_{\cc,\mathfrak{t}}(\delta(T_{j_1}))\cdots\pi_{\cc,\mathfrak{t}}(\delta(T_{j_m}))x^\cc.
\]
Write $\Pi(w)=\{b_1,\ldots,b_m\}$ with $b_i:=s_{j_m}\cdots s_{j_{i+1}}\alpha_{j_i}$ ($1\leq i<m$) and $b_m:=\alpha_{j_m}$.
Then
\[
\alpha_{j_i}(s_{{j_{i+1}}}\cdots s_{j_{m}}\cc)=b_i(\cc)>0\qquad (i=1,\ldots,m)
\]
since $w\in W^{J}$. It thus suffices to show for $j=0,\ldots,r$ that
\begin{equation}\label{deltaTjcase}
\pi_{\cc,\mathfrak{t}}(\delta(T_j))x^y=k_{s_j}(y)x^{s_jy}+\textup{l.o.t.}\qquad
\hbox{ if }\,\,\, y\in\mathcal{O}_\cc\, \hbox{ and }\, \alpha_j(y)>0.
\end{equation}
 
For $y\in\mathcal{O}_\cc$
and $1\leq i\leq r$ we have
\begin{equation}\label{claimm}
\pi_{\cc,\mathfrak{t}}(T_i^{-1})x^y=
\begin{cases}
k_i^{-1}x^{s_iy}+\textup{ l.o.t.} \qquad &\hbox{ if }\,\, \alpha_i(y)\in
\mathbb{Z}_{\geq 0},\\
k_i^{\chi_{\mathbb{Z}}(\alpha_i(y))}x^{s_iy}+\textup{ l.o.t.}\qquad
&\hbox{ otherwise}.
\end{cases}
\end{equation}
This follows in a similar way as Lemma \ref{triangularG}, using as extra ingredients Proposition \ref{reflectionrelation} and the fact that
\begin{equation}\label{floorrel}
\lfloor z\rfloor + \lfloor -z\rfloor=
\begin{cases}
-1\qquad &\hbox{ if } z\in\mathbb{R}\setminus\mathbb{Z},\\
0\qquad &\hbox{ if }  z\in\mathbb{Z}
\end{cases}
\end{equation}
(below we give the similar, but more involved, proof of the variant \eqref{claimU0} of formula \eqref{claimm} for the simple reflection $s_0$).
As an easy consequence of Proposition \ref{reflectionrelation} and the Hecke relations \eqref{Hr} for $T_i$, we conclude that for $1\leq i\leq r$ and $y\in\mathcal{O}_\cc$ with $\alpha_i(y)>0$,
\begin{equation}\label{tipm}
\pi_{\cc,\mathfrak{t}}(T_i^{\pm 1})x^y=k_i^{-\eta(\alpha_i(y))}x^{s_iy}
+\textup{ l.o.t.}
\end{equation}
Then \eqref{deltaTjcase} for $j>0$ follows from \eqref{kjcombi}, \eqref{tipm} and the fact
that $\delta(T_i)=T_i$ for $1\leq i\leq r$.

For \eqref{deltaTjcase} with $j=0$, we first note that for $U_0=q_\varphi^{-1}x^{\varphi^\vee}T_0^{-1}\in\mathbb{H}$ (see \eqref{Uzero}),
\begin{equation}\label{claimU0pre}
\pi_{\cc,\mathfrak{t}}(U_0^{-1})x^y=k_\varphi^{\chi_{\mathbb{Z}}(\alpha_0(y))}\mathfrak{t}_y^{\varphi^\vee}
x^{s_{\varphi}y-\alpha_0^\vee}+
(k_\varphi-k_\varphi^{-1})\left(\frac{1-x^{-\lfloor \alpha_0(y)+1\rfloor \alpha_0^\vee}}{1-x^{\alpha_0^\vee}}
\right)x^{y+\alpha_0^\vee}
\end{equation}
for $y\in\mathcal{O}_{\cc}$.
Here we used that $\mathfrak{t}_{y-\varphi^\vee}^{\varphi^\vee}=q_\varphi^{-2}\mathfrak{t}_y^{\varphi^\vee}$ (see Definition \ref{defty}), and the fact that
$q_\varphi x^{-\varphi^\vee}=x^{\alpha_0^\vee}$.
We will now first show that for $y\in\mathcal{O}_{\cc}$,
\begin{equation}\label{claimU0}
\pi_{\cc,\mathfrak{t}}(U_0^{-1})x^y=
\begin{cases}
k_\varphi^{-1}q_\varphi^{-\alpha_0(y)}x^{s_0y}+\textup{ l.o.t.} \qquad &\hbox{ if }  \varphi(y)\in\mathbb{Z}_{\leq 1},\\
k_\varphi^{\chi_{\mathbb{Z}}(\varphi(y))}\mathfrak{t}_y^{\varphi^\vee}q_\varphi^{-1}x^{s_{0}y}+
\textup{ l.o.t.}\qquad
&\hbox{ otherwise}
\end{cases}
\end{equation}
(note here that $x^{s_0y}=x^{\tau(\varphi^\vee)s_\varphi y}=x^{s_\varphi y+\varphi^\vee}=q_\varphi x^{s_\varphi y-\alpha_0^\vee}$).
If $\varphi(y)\in\mathbb{Z}_{\leq 1}$ then $\alpha_0(y)\geq 0$ and it follows from \eqref{claimU0pre}
and Lemma \ref{datumokcor} that
\begin{equation*}
\begin{split}
\pi_{\cc,\mathfrak{t}}(U_0^{-1})x^y&=
k_{\varphi}x^{y-\alpha_0(y)\alpha_0^\vee}+(k_\varphi^{-1}-k_\varphi)\bigl(
1+x^{\alpha_0^\vee}+\cdots+x^{\alpha_0(y)\alpha_0^\vee}\bigr)x^{y-\alpha_0(y)\alpha_0^\vee}\\
&=k_\varphi^{-1}x^{y-\alpha_0(y)\alpha_0^\vee}+(k_\varphi^{-1}-k_\varphi)\bigl(1+x^{-\alpha_0^\vee}+\cdots+
x^{(1-\alpha_0(y))\alpha_0^\vee}\bigr)x^y,
\end{split}
\end{equation*}
with the second equation interpreted as $k_\varphi^{-1}x^{y}$ when $\alpha_0(y)=0$. 

When $\varphi(y)=1$, i.e., $\alpha_0(y)=0$, it is now clear that \eqref{claimU0}
is correct. If $\varphi(y)\in\mathbb{Z}_{<1}$ then $\alpha_0(y)\in\mathbb{Z}_{>0}$, and hence $y<s_0y$ by Proposition \ref{reflectionrelation}. In particular, for
$\varphi(y)=0$ we get $\pi_{\cc,\mathfrak{t}}(U_0^{-1})x^y=k_\varphi^{-1}x^{y-\alpha_0(y)\alpha_0^\vee}+
(k_\varphi^{-1}-k_\varphi)x^y=k_\varphi^{-1}q_\varphi^{-\alpha_0(y)}x^{s_0y}+\textup{l.o.t.}$, as desired. If $\varphi(y)\in\mathbb{Z}_{<0}$ then Lemma \ref{techlem}(1) and Proposition \ref{equivalenceorder2} give $y+\ell\varphi^\vee<y$ for $1\leq\ell\leq -\varphi(y)$. In addition $y<s_0y$ by Proposition \ref{reflectionrelation}, and so
 \eqref{claimU0} is correct. Suppose that $\varphi(y)\in\mathbb{R}_{<1}\setminus\mathbb{Z}_{<1}$. Then \eqref{claimU0pre} gives
 \[
 \pi_{\cc,\mathfrak{t}}(U_0^{-1})x^y=\mathfrak{t}_y^{\varphi^\vee}
 x^{s_{\varphi}y-\alpha_0^\vee}+(k_\varphi^{-1}-k_\varphi)\bigl(1+x^{-\alpha_0^\vee}+\cdots
 +x^{-\lfloor\alpha_0(y)\rfloor\alpha_0^\vee}\bigr)x^y.
 \]
We have $\varphi(s_0y)=\alpha_0(y)+1=2-\varphi(y)\in\mathbb{R}_{>1}\setminus\mathbb{Z}_{>1}$, and hence
$y+\ell\varphi^\vee<s_0y$ for $0\leq \ell\leq -\lfloor \varphi(y)\rfloor$ by Lemma \ref{techlem}(3) and Proposition \ref{equivalenceorder2}. Furthermore, $\lfloor \alpha_0(y)\rfloor=-\lfloor\varphi(y)\rfloor$
by \eqref{floorrel}, so \eqref{claimU0} is correct. If $1<\varphi(y)\leq 2$  then $\lfloor \alpha_0(y)+1\rfloor=0$, and
\eqref{claimU0} immediately follows from \eqref{claimU0pre}.
Finally we consider the case $\varphi(y)\in\mathbb{R}_{>2}$. Then $1+\alpha_0(y)\in\mathbb{R}_{<0}$, and so \eqref{claimU0pre} gives
\[
\pi_{\cc,\mathfrak{t}}(U_0^{-1})x^y=
k_\varphi^{\chi_{\mathbb{Z}}(\alpha_0(y))}\mathfrak{t}_y^{\varphi^\vee}x^{s_{\varphi}y-\alpha_0^\vee}
+(k_\varphi-k_\varphi^{-1})\bigl(x^{\alpha_0^\vee}+x^{2\alpha_0^\vee}+
\cdots+x^{(-2-\lfloor-\varphi(y)\rfloor)\alpha_0^\vee}\bigr)x^y.
\]
Now $\varphi(s_0y)\in\mathbb{R}_{<0}$ so Lemma \ref{techlem}(1) and Proposition \ref{equivalenceorder2} give
$y-\ell\varphi^\vee<s_0y$ for $1\leq\ell\leq -2-\lfloor-\varphi(y)\rfloor$. 
This completes the proof of \eqref{claimU0}.

Fix $y\in\mathcal{O}_\cc$ with $\alpha_0(y)>0$. Formula \eqref{claimU0}, the Hecke relation for $U_0$, 
Lemma \ref{datumokcor} and Proposition \ref{reflectionrelation} imply that
\begin{equation}\label{t0pm}
\pi_{\cc,\mathfrak{t}}(U_0^{\pm 1})x^y=k_\varphi^{-\eta(\alpha_0(y))}\mathfrak{t}_y^{\varphi^\vee}q_\varphi^{-1}x^{s_0y}
+\textup{ l.o.t.}
\end{equation}
By the formula $\delta(T_0)=q_\varphi^{-1}Y^{-\varphi^\vee}U_0^{-1}$, \eqref{t0pm} and
Proposition \ref{trianY}, we conclude that 
\[
\pi_{\cc,\mathfrak{t}}(\delta(T_0))x^y=k_\varphi^{-\eta(\alpha_0(y))}
(\mathfrak{s}_{s_0y}\mathfrak{t}_{s_0y})^{\varphi^\vee}\mathfrak{t}_y^{\varphi^\vee}q_\varphi^{-2}x^{s_0y}+\textup{ l.o.t.}
\]
But $\alpha_0(y)>0$ implies $s_0y\not=y$, hence Lemma \ref{fraks}(2) and Corollary \ref{datumok} give
\[
(\mathfrak{s}_{s_0y}\mathfrak{t}_{s_0y})^{\varphi^\vee}=(s_\varphi\mathfrak{s}_y)^{\varphi^\vee}(s_0\mathfrak{t}_{y})^{\varphi^\vee}
=q_\varphi^2\mathfrak{s}_y^{-\varphi^\vee}\mathfrak{t}_y^{-\varphi^\vee}.
\]
So we obtain 
\[
\pi_{\cc,\mathfrak{t}}(\delta(T_0))x^y=\mathfrak{s}_y^{-\varphi^\vee}k_\varphi^{\eta(\varphi(y))}
x^{s_0y}+\textup{ l.o.t.},
\]
where we have also used \eqref{etaaffine} to rewrite the $k_\varphi$-factor. Now observe that
\[
\mathfrak{s}_y^{-\varphi^\vee}=k_\varphi^{-\eta(\varphi(y))}
\prod_{\alpha\in\Pi(s_\varphi)}k_\alpha^{-\eta(\alpha(y))}
\]
by Definition \ref{fraksdef} and \eqref{thetalength}. We conclude that
\[
\pi_{\cc,\mathfrak{t}}(\delta(T_0))x^y=\Bigl(\prod_{\alpha\in\Pi(s_\varphi)}k_\alpha^{-\eta(\alpha(y))}
\Bigr)x^{s_0y}+\textup{l.o.t.}
\]
The proof of  \eqref{deltaTjcase} for $j=0$ now follows from \eqref{kj}.
\end{proof}
\subsection{Completion of the proof}\label{POsection3}
Fix $\cc\in C^J$ and $\mathfrak{t}\in T_{J}$. To complete the proof of Theorem \ref{gbr} it remains to show that the epimorphism
$\phi_{\cc,\mathfrak{t}}: \mathbb{M}^{J}_{\mathfrak{s}_\cc\mathfrak{t}}\twoheadrightarrow \Pc_{\mathfrak{t}}$ of $\mathbb{H}$-modules, defined in Corollary \ref{factorcorollary}, is an isomorphism.

Note that $\phi_{\cc,\mathfrak{t}}$ maps the standard basis element $m_{w,\mathfrak{s}_\cc\mathfrak{t}}^{J}=\delta(T_{w^{-1}})
m^{J}_{\mathfrak{s}_\cc\mathfrak{t}}$ 
($w\in W^{J}$) of $\mathbb{M}^{J}_{\mathfrak{s}_\cc\mathfrak{t}}$ to
\[
\pi_{\cc,\mathfrak{t}}(\delta(T_{w^{-1}}))x^\cc=k_w(\cc)x^{w\cc}+\textup{l.o.t.}
\]
It follows that $\{\phi_{\cc,\mathfrak{t}}(m_{w,\mathfrak{s}_\cc\mathfrak{t}}^{J}\bigr)\,\, | \,\, w\in W^{J}\}$ is linear independent in 
$\Pc$, hence indeed $\phi_{\cc,\mathfrak{t}}$ is an isomorphism of $\mathbb{H}$-modules.

\section{Quasi-polynomial generalisations of Macdonald polynomials}
\label{QuasiSection}
In this section $J$ will be a proper subset of $[0,r]$. We show that for $\cc\in C^J$, the commuting operators $\pi_{\cc,\mathfrak{t}}(Y^\mu)$ ($\mu\in Q^\vee$) are simultaneously diagonalisable under suitable restrictions on the character $\mathfrak{t}\in T_J$. We study their simultaneous eigenfunctions, which are quasi-polynomial generalisations of the Macdonald polynomials.

\subsection{Simple spectrum conditions}\label{genSection}
Let $c\in C^J$. We will construct in this section simultaneous eigenfunctions for the commuting operators $\pi_{\cc,\mathfrak{t}}(Y^\mu)$ ($\mu\in Q^\vee$) when the character $\mathfrak{t}$ lies in the following subset of $T_J$.
\begin{definition}\label{simplespec}
Denote by $T_J^\prime\subseteq T_J$ the set of elements $\mathfrak{t}\in T_J$ for which the map
\begin{equation}\label{spectralmap}
W^J\rightarrow T,\qquad w\mapsto w(\mathfrak{s}_J\mathfrak{t})
\end{equation}
is injective. 
\end{definition}
If $\cc\in C^J$ then $W^J=\{w_y\}_{y\in\mathcal{O}_\cc}$, and the condition can be rephrased as the requirement that
\[
\mathcal{O}_\cc\rightarrow T,\qquad y\mapsto \mathfrak{s}_y\mathfrak{t}_y=\mathfrak{t}_y\prod_{\alpha\in\Phi_0^+}k_\alpha^{\eta(\alpha(y))\alpha}
\]
is injective, where $\mathfrak{t}_y=w_y\mathfrak{t}$ (see Corollary \ref{datumok}). We will now show that for $J=I\subseteq [1,r]$ we have $T_I^\prime\not=\emptyset$ under generic conditions on $k_a$, and that for $0\in J\subsetneq [0,r]$ we have $T_J^\prime\not=\emptyset$ if we in addition assume that $\mathbf{F}$ contains a $(2h)^{th}$ root of $q$.

In the classical context of Cherednik's basic representation we have $J=[1,r]$, $T_{[1,r]}=\{1_T\}$, $C^{[1,r]}=\{0\}$ and $\mathcal{O}_0=Q^\vee$.
The requirement that $1_T\in T_{[1,r]}^\prime$ then amounts to generic conditions on $k_a$, since $q$ is not a root of unity and 
\[
\mathfrak{s}_\mu(1_T)_\mu=q^\mu\prod_{\alpha\in\Phi_0^+}k_\alpha^{\eta(\alpha(\mu))\alpha}\qquad (\mu\in Q^\vee)
\]
\textup{(}here we use that $(1_T)_\mu=q^\mu$, in view of \eqref{integralshiftt}\textup{)}. 

If $J=I\subseteq [1,r]$ then $W_I$ is already a parabolic subgroup of $W_0$, and
we denote by $W_0^I$ the resulting minimal coset representatives of $W_0/W_I$. For $\cc\in C^I$ each $y\in\mathcal{O}_\cc$ can be uniquely written as
$y=\mu+v\cc$ with $\mu\in Q^\vee$ and $v\in W_0^I$, and 
\[
\mathfrak{s}_y\mathfrak{t}_{y}=q^\mu(v\mathfrak{t})\prod_{\alpha\in\Phi_0^+}k_\alpha^{\eta(\alpha(\mu+v\cc))\alpha}.
\]
For generic $k_a$ the maps $Q^\vee\rightarrow T$, $\mu\mapsto q^\mu\prod_{\alpha\in\Phi_0^+}k_\alpha^{\eta(\alpha(\mu+v\cc))\alpha}$  ($v\in W_0^I$) are injective, and
the requirement $\mathfrak{t}\in T_I^\prime$ results in generic conditions on $\mathfrak{t}\in T_I$.

Finally, consider the general case $J\subsetneq [0,r]$ when
$\mathbf{F}$ contains a $(2h)^{th}$ root $q^{\frac{1}{2h}}$ of $q$. Recall that this assumption on $q$ ensures that $T_J\not=\emptyset$ by Lemma \ref{nonzeroLem}. In fact, one can construct elements
in $T_J$ as follows. For $y\in \frac{1}{2h}P^\vee$ let $q^{y}\in T$ be the character which takes the value
\[
q^{\langle y,\alpha^\vee\rangle}=q_\alpha^{\alpha(y)}
\]
at the coroot $\alpha^\vee\in\Phi_0^\vee$. Using the table \cite[Table 2, Chpt. 3]{Hu0} containing the explicit values of the expansion coefficients $n_i(\varphi)$ of $\varphi$ as positive integral linear combination of simple roots, as well as the corresponding explicit values $h=1+\sum_{i=1}^rn_i(\varphi)$ for the Coxeter number, it follows by a straightforward case-by-case check that
$C^J\cap\frac{1}{2h}P^\vee\not=\emptyset$
(cf. the proof Lemma \ref{nonzeroLem}).
Then for 
\begin{equation}\label{lambdaJ}
\lambda_J\in C^J\cap\tfrac{1}{2h}P^\vee
\end{equation}
we have $q^{\lambda_J}\in T_J$
and $w_yq^{\lambda_J}=q^y$ ($y\in\mathcal{O}_{\lambda_J}$). Since the values of $\mathfrak{s}_y=v_y^{-1}\mathfrak{s}_{J}$ ($y\in\mathcal{O}_{\lambda_J}$) are monomials
in ${}^{\textup{sh}}k$ and ${}^{\textup{lg}}k$ (with possibly negative exponents), we have $q^{\lambda_J}\in T_J^\prime$ for generic $k_a$.
\begin{remark}\label{parametersREM}
Let $P\subseteq E^*$ be the weight lattice of $\Phi_0$. Initially it is actually a full lattice in $(E^\prime)^*$, which we view as a sublattice of $E^*$ by declaring the weights to be zero on $E_{\textup{co}}$.
Let $\{\varpi_i\}_{i=1}^r\subset E^*$ be the fundamental weights with respect to $\Delta_0$. Write $I^{\textup{co}}:=[1,r]\setminus I$ for a subset $I\subseteq [1,r]$. 
For a subset $J\subsetneq [0,r]$ and 
$\mathbf{z}:=(z_i)_{i\in J_0^{\textup{co}}}\in (\mathbf{F}^\times)^{\#J_0^{\textup{co}}}$ set
\[
\mathfrak{t}(\mathbf{z}):=\prod_{j\in J_0^{\textup{co}}}z_j^{\varpi_j}.
\]
Then
\begin{equation*}
T_J^{\textup{red}}=
\begin{cases}
\{\,\mathfrak{t}(\mathbf{z})\,\, | \,\, \mathbf{z}\in\mathbf{F}^{\#J_0^{\textup{co}}}\,\}\quad &\hbox{ if }\,\, J_0=J,\\
\{\,\mathfrak{t}(\mathbf{z})\,\, | \,\, \mathbf{z}\in\mathbf{F}^{\#J_0^{\textup{co}}}\,\, \&\,\, \prod_{i\in J_0^{\textup{co}}}z_i^{m_i}=1\,\}\quad 
&\hbox{ if }\,\, J_0\not=J,
\end{cases}
\end{equation*}
where 
$\varphi^\vee=\sum_{i=1}^rm_i\alpha_i^\vee$ ($m_i\in\mathbb{Z}_{>0}$) is the expansion of $\varphi^\vee\in\Phi_0^\vee$ in simple coroots.
\end{remark}

\subsection{The monic $Y$-eigenbasis of $\Pc_{\mathfrak{t}}$}\label{monicSection}

For $\cc\in C^J$ and $\mathfrak{t}\in T_J$ we have encountered so far two natural bases of $\Pc_{\mathfrak{t}}$, 
the basis $\{x^y\}_{y\in\mathcal{O}_\cc}$ of quasi-monomials and the basis $\{m_y^{J}(x;\mathfrak{t})\}_{y\in\mathcal{O}_\cc}$ with
\begin{equation}\label{mwP}
m_y^{J}(x;\mathfrak{t}):=\pi_{\cc,\mathfrak{t}}\bigl(\delta(T_{w_y^{-1}})\bigr)x^\cc\qquad (y\in\mathcal{O}_\cc).
\end{equation}
The latter basis corresponds to the basis $\{m_{w;\mathfrak{s}_J\mathfrak{t}}^J\}_{w\in W^J}$ of $\mathbb{M}_{\mathfrak{s}_J\mathfrak{t}}^J$ through the isomorphism $\phi_{\cc,\mathfrak{t}}$ from Theorem \ref{gbr}(2). The change of basis matrix between these two bases is triangular with respect to the partially ordered set $(\mathcal{O}_\cc,\leq)$, 
\begin{equation}\label{trmx}
m_y^{J}(x;\mathfrak{t})=k_{w_y}(\cc)x^{y}+\sum_{y^\prime<y}d_{y,y^\prime;\mathfrak{t}}^{J}x^{y^\prime}\qquad (y\in\mathcal{O}_\cc)
\end{equation}
($d_{y,y^\prime;\mathfrak{t}}^{J}\in\mathbf{F}$) by Proposition \ref{triangularitymbasis}. Note that for $u,w\in W^J$ satisfying $u<_Bw$,
the coefficient 
\[
d_{w,u;\mathfrak{t}}^J:=d_{w\cc,u\cc;\mathfrak{t}}^{J}
\]
of $x^{u\cc}$ in the expansion \eqref{trmx} of $m_{w\cc}^{J}(x;\mathfrak{t})$ does not depend on $\cc\in C^J$ by Corollary \ref{expansionHcor}.
We now construct for $\mathfrak{t}\in T_J^\prime$ a third basis of $\Pc_{\mathfrak{t}}$ consisting of simultaneous eigenfunctions of $\pi_{\cc,\mathfrak{t}}(Y^\mu)$ ($\mu\in Q^\vee$), which is always triangular with respect to the basis of quasi-monomials.

We start with introducing some general terminology. For a left $\mathbb{H}$-module $M$ and $s\in T$ write
\[
M[s]:=\{m\in M \,\,\, | \,\,\, p(Y)m=p(s^{-1})m\quad \forall\, p\in\mathcal{P} \}.
\]
We call $s\in T$ a {\it $Y$-weight} of $M$ if $M[s]\not=0$. A nonzero vector in $M[s]$ is called a 
{\it $Y$-weight vector} of weight $s$. The set of all $Y$-weights of $M$ is called the {\it $Y$-spectrum} of $M$, and will be denoted by $\mathcal{S}(M)$. Following Cherednik \cite[\S 3.6]{Ch} we say that $M$ is {\it $Y$-semisimple} if 
\[
M=\bigoplus_{s\in\mathcal{S}(M)}M[s].
\]
The $Y$-spectrum of a $Y$-semisimple $\mathbb{H}$-module $M$ is called {\it simple} if $\dim_{\mathbf{F}}(M[s])=1$ for all $s\in\mathcal{S}(M)$. A basis of $M$ consisting of $Y$-weight vectors is called a {\it $Y$-eigenbasis} of $M$.
\begin{theorem}\label{Edef}
Let $\cc\in C^J$ and $\mathfrak{t}\in T_J^\prime$.
\begin{enumerate}
\item For $y\in\mathcal{O}_\cc$ there exists a unique simultaneous eigenfunction of the commuting operators $\pi_{\cc,\mathfrak{t}}(Y^\mu)\in\textup{End}(\mathcal{P}^{(\cc)})$ \textup{(}$\mu\in Q^\vee$\textup{)}
of the form
\[
E_y^J(x;\mathfrak{t})=x^y+\sum_{y^\prime<y}e_{y,y^\prime;\mathfrak{t}}^{J}x^{y^\prime}\qquad (e_{y,y^\prime;\mathfrak{t}}^{J}\in\mathbf{F}).
\]
\item $E_y^J(x;\mathfrak{t})\in\mathcal{P}^{(\cc)}_{\mathfrak{t}}[\mathfrak{s}_y\mathfrak{t}_y]$ for $y\in\mathcal{O}_\cc$.
\item $\{E_y^J(x;\mathfrak{t})\}_{y\in\mathcal{O}_\cc}$ is a $Y$-eigenbasis of $\Pc_{\mathfrak{t}}$. In particular, $\Pc_{\mathfrak{t}}\simeq\mathbb{M}_{\mathfrak{s}_\cc\mathfrak{t}}^J$ is $Y$-semisimple, with simple $Y$-spectrum
$\{\mathfrak{s}_y\mathfrak{t}_y\}_{y\in\mathcal{O}_\cc}$.
\item For $u,w\in W^J$ such that $u<_Bw$, the coefficient
\[
e_{w,u;\mathfrak{t}}^J:=e_{w\cc,u\cc;\mathfrak{t}}^{J}
\]
of $x^{u\cc}$ in the monomial expansion of $E_{w\cc}^J(x;\mathfrak{t})$ does not depend on $\cc\in C^J$.
\end{enumerate}
\end{theorem}
\begin{proof}
Parts (1)\&(2) are due to Proposition \ref{trianY}, part (3) is a direct consequence of (1),(2) and Theorem \ref{gbr}(2), and part (4) follows from \eqref{isomface}.
\end{proof} 
\begin{remark}\label{trRemark}\hfill
\begin{enumerate}
\item
If $1_T\in T_{[1,r]}^\prime$ and $\mu\in\mathcal{O}_0=Q^\vee$ then
\[
E_\mu^{[1,r]}(x;1_T)\in\mathcal{P}^{(0)}_{1_T}
\]
is the monic non-symmetric Macdonald polynomial of degree $\mu$ (see, e.g., \cite{Ch,Ma}).  
\item For $\cc\in C^J$, $\mathfrak{t}\in T_J$ and $y\in\mathcal{O}_\cc$ we have by \eqref{trmx} and Theorem \ref{Edef}(1),
\[
E_y^J(x;\mathfrak{t})=k_{w_y}(\cc)^{-1}m_y^{J}(x;\mathfrak{t})+ \cdots
\]
with $\cdots$ meaning a linear combination of the $m_{y^\prime}^{J}(x;\mathfrak{t})$ with $y^\prime<y$.
\end{enumerate}
\end{remark}
The following corollary is immediate.
\begin{corollary}
Let $\cc\in C^J$ and $\mathfrak{t}\in T_J^\prime$. 
\begin{enumerate}
\item
The transition matrices between the three bases $\{x^y\}_{y\in\mathcal{O}_\cc}$, $\{m_y^{J}(x;\mathfrak{t})\}_{y\in\mathcal{O}_\cc}$ and 
$\{E_y^{J}(x;\mathfrak{t})\}_{y\in\mathcal{O}_\cc}$ of $\Pc_{\mathfrak{t}}$ are triangular with respect to $\leq$.
\item
Identifying $(\mathcal{O}_\cc,\leq)$ with $(W^J,\leq_B)$ as partially ordered sets by the map
\[
W^J\overset{\sim}{\longrightarrow}\mathcal{O}_\cc,\qquad w\mapsto w\cc,
\]
the transition matrices $A=(a_{u,w})_{u,w\in W^J}$ between the bases in (1) do not depend on $\cc\in C^J$.
\end{enumerate}
\end{corollary}
In the following lemma we give two elementary properties of the monic quasi-polynomials $E_y^J(x;\mathfrak{t})$.
\begin{lemma}\label{remE1}
For $\cc\in C^J$ and $\mathfrak{t}\in T_J^\prime$ we have
\begin{enumerate}
\item
$E_{\cc}^J(x,\mathfrak{t})=x^{\cc}$.
\item $E_{y+z}^J(x;\mathfrak{t})=
x^zE_y^J(x;\mathfrak{t})$ when $y\in\mathcal{O}_\cc$ and $z\in E_{\textup{co}}$.
\end{enumerate}
\end{lemma}
\begin{proof}
(1) By Theorem \ref{gbr}(2) we have $x^\cc\in\Pc_{\mathfrak{t}}[\mathfrak{s}_J\mathfrak{t}]$. Now apply Theorem \ref{Edef}.\\
(2) This follows from the fact that $\cc+z\in C^J$, $\mathcal{O}_{\cc+z}=\mathcal{O}_\cc+z$ and 
\[
\pi_{\cc,\mathfrak{t}}(h)x^{y+z}=x^z\pi_{\cc,\mathfrak{t}}(h)x^{y}
\]
for all $h\in\mathbb{H}$.
\end{proof}

\subsection{Creation operators and irreducibility conditions}\label{CreationSection}
The $Y$-semisimple $\mathbb{H}$-module $\Pc_{\mathfrak{t}}$ ($\cc\in C^J$, $\mathfrak{t}\in T_J^\prime$)
has simple $Y$-spectrum $\{w(\mathfrak{s}_J\mathfrak{t})\}_{w\in W^J}$ by Theorem \ref{Edef}. In this subsection we will use the action of $Y$-intertwiners $S_w^Y$ ($w\in W^J$) to create the quasi-polynomial $E_{w\cc}^J(x;\mathfrak{t})\in\Pc_{\mathfrak{t}}[w(\mathfrak{s}_J\mathfrak{t})]$ from the quasi-monomial 
$x^\cc\in\Pc_{\mathfrak{t}}[\mathfrak{s}_J\mathfrak{t}]$.

\begin{lemma}\label{YintAction}
Let $M$ be a $\mathbb{H}$-module, $s\in\mathcal{S}(M)$, and $w\in W$. Then $S_w^Y$ maps $M[s]$ to $M[ws]$.
\end{lemma}
\begin{proof}
For a polynomial $p=\sum_{\mu\in Q^\vee}a_\mu x^\mu\in\mathcal{P}$ and $w\in W$
write 
\[
p_w^\iota:=\sum_{\mu\in Q^\vee}a_\mu w(x^{-\mu})=\sum_{\mu\in Q^\vee}a_\mu x^{-w\cdot\mu}\in\mathcal{P}.
\] 
Then we have $p(Y)S_w^Y
=S_w^Yp_w^\iota(Y^{-1})$
by \eqref{IpropY}. For $m\in M[s]$ we then have
\begin{equation*}
p(Y)S_w^Ym=S_{w}^Yp^\iota_w(Y^{-1})m
=p^\iota_w(s)S_w^Ym=p((ws)^{-1})S_w^Ym,
\end{equation*}
hence $S_w^Ym\in M[ws]$.
\end{proof}
Applying the lemma to $\Pc_{\mathfrak{t}}$
($\cc\in C^J$, $\mathfrak{t}\in T_J$) we conclude that
\begin{equation}\label{bwJt}
\pi_{\cc,\mathfrak{t}}(S_{w}^Y)x^\cc\in\mathcal{P}_{\mathfrak{t}}^{(\cc)}[w(\mathfrak{s}_J\mathfrak{t})]\qquad (w\in W),
\end{equation}
since $x^\cc\in\Pc_{\mathfrak{t}}[\mathfrak{s}_J\mathfrak{t}]$.

For $w\in W$ set
\begin{equation}\label{dwt}
d_{w}(x):=\prod_{a\in\Pi(w)}(x^{a^\vee}-1)\in\mathcal{P}.
\end{equation}
\begin{proposition}\label{triangular1}
Let $\cc\in C^J$ and $\mathfrak{t}\in T_J$.
\begin{enumerate}
\item $\pi_{\cc,\mathfrak{t}}(S_w^Y)x^\cc=0$ if $w\in W\setminus W^J$.
\item There exist $\epsilon_{w,u}\in\mathbf{F}$ \textup{(}$u,w\in W^J$\textup{:} $u<_Bw$\textup{)} such that 
\begin{equation}\label{todo}
\pi_{\cc,\mathfrak{t}}(S_w^Y)x^\cc=d_{w}(\mathfrak{s}_J\mathfrak{t})m_{w\cc}^{J}(x;\mathfrak{t})+\sum_{u\in W^J: u<_Bw}\epsilon_{w,u}m_{u\cc}^{J}(x;\mathfrak{t})
\end{equation}
for all $w\in W^J$.
\end{enumerate}
\end{proposition}
\begin{proof}
For the duration of the proof we write $t:=\mathfrak{s}_J\mathfrak{t}\in L_J$.\\
(1) Let $j\in J$. It suffices to show that $\pi_{\cc,\mathfrak{t}}(S_j^Y)x^\cc=0$.
Since $x^\cc\in\Pc_{\mathfrak{t}}[t]$ and $\pi_{\cc,\mathfrak{t}}(\delta(T_j))x^\cc=k_jx^\cc$ by Theorem \ref{gbr}, substitution of the explicit expression \eqref{buildingYintertwiners} for $S_j^Y$ gives
\begin{equation*}
\pi_{\cc,\mathfrak{t}}(S_j^Y)x^\cc=\bigl(k_j(t^{\alpha_j^\vee}-1)+k_j-k_j^{-1}\bigr)x^\cc.
\end{equation*}
This vanishes since $t\in L_J$.\\
(2) The proof is by induction to $\ell(w)$. 
Fix $w\in W^{J}$ with
$\ell(w)>0$. Write $w=s_{j}u$ with $j\in [0,r]$ and $\ell(s_ju)=\ell(u)+1$. Then $u\in W^{J}$ by Lemma \ref{cosetcomb}, and $u<_Bw$. Furthermore, 
$S_w^Y=S_j^YS_u^Y$ and 
\begin{equation}\label{eq:d_relation}
d_{w}(x)=\bigl(x^{(u^{-1}\alpha_j)^\vee}-1\bigr)d_{u}(x)
\end{equation}
by Lemma \ref{lengthadd}.
By the induction hypothesis, the explicit expression \eqref{buildingYintertwiners} for $S_j^Y$ and \eqref{bwJt}
we then have for $\cc\in C^J$ and $\mathfrak{t}\in T_J$,
\begin{equation*}
\begin{split}
&\pi_{\cc,\mathfrak{t}}(S_w^Y)x^\cc=
\bigl((t^{(u^{-1}\alpha_j)^\vee}-1)\pi_{\cc,\mathfrak{t}}(\delta(T_j))+k_j-k_j^{-1}\bigr)\pi_{\cc,\mathfrak{t}}(S_u^Y)x^\cc\\
&\,\,\,\,\,\,\,\in d_{w}(t)m_{w\cc}^{J}(x;\mathfrak{t})+\sum_{w^\prime\in W^{J}:\, w^\prime<_Bu}
\bigl(t^{(u^{-1}\alpha_j)^\vee}-1\bigr)\epsilon_{u,w^\prime}\pi_{\cc,\mathfrak{t}}(\delta(T_j))m_{w^\prime\cc}^{J}(x;\mathfrak{t})+\mathcal{P}^{(\cc)}_{<w}
\end{split}
\end{equation*}
for some $\epsilon_{u,w^\prime}\in\mathbf{F}$, with 
\[
\mathcal{P}^{(\cc)}_{<w}:=\bigoplus_{w^{\prime\prime}\in W^J:\, w^{\prime\prime}<_Bw}\mathbf{F}m_{w^{\prime\prime}\cc}^J(x;\mathfrak{t}).
\]
It thus remains to show that $\pi_{\cc,\mathfrak{t}}(\delta(T_j))m_{w^\prime\cc}^{J}(x;\mathfrak{t})\in\mathcal{P}^{(\cc)}_{<w}$ when $w^\prime\in W^{J}$ and $w^\prime<_Bu$. There are three cases to consider.\\

\noindent
{\it Case 1:} If $\ell(s_jw^\prime)=\ell(w^\prime)+1$ and $s_jw^\prime\not\in W^{J}$ then $s_jw^\prime=w^\prime s_{j^\prime}$ with $j^\prime\in J$ by Lemma \ref{cosetcomb}, and hence
$\pi_{\cc,\mathfrak{t}}(\delta(T_j))m_{w^\prime\cc}^{J}(x;\mathfrak{t})=k_jm_{w^\prime\cc}^{J}(x;\mathfrak{t})\in\mathcal{P}^{(\cc)}_{<w}$ since $w^\prime<_Bu<_Bw$.\\

\noindent
{\it Case 2:} If $\ell(s_jw^\prime)=\ell(w^\prime)+1$ and $s_jw^\prime\in W^{J}$ then 
$\pi_{\cc,\mathfrak{t}}(\delta(T_j))m_{w^\prime\cc}^{J}(x;\mathfrak{t})=m_{s_jw^\prime\cc}^{J}(x;\mathfrak{t})$
lies in $\mathcal{P}^{(\cc)}_{<w}$ since $s_jw^\prime<_Bs_ju=w$ by Lemma \ref{Bruhattrans}.\\

\noindent
{\it Case 3:} If $\ell(s_jw^\prime)=\ell(w^\prime)-1$ then $s_jw^\prime<_Bw^\prime<w$ and $s_jw^\prime\in W^{J}$ by Lemma \ref{cosetcomb}, hence 
\[
\pi_{\cc,\mathfrak{t}}(\delta(T_j))m_{w^\prime\cc}^{J}(x;\mathfrak{t})=
m_{s_jw^\prime\cc}^{J}(x;\mathfrak{t})+(k_j-k_j^{-1})m_{w^\prime\cc}^{J}(x;\mathfrak{t})
\]
lies in $\mathcal{P}_{<w}^{(\cc)}$.  
\end{proof}
The basis elements $m_{w\cc}^J(x;\mathfrak{t})$ ($w\in W^J$) of $\mathcal{P}_{\mathfrak{t}}^{(\cc)}$ can now be expanded in terms of $\pi_{\cc,\mathfrak{t}}(S_u^Y)x^\cc$ ($u\in W^J$) when the leading coefficients $d_w(\mathfrak{s}_J\mathfrak{t})$ ($w\in W^J$) in \eqref{todo} are nonzero. This will be ensured by the following generic condition on $\mathfrak{t}\in T_J$.

\begin{definition}
We say that $\mathfrak{t}\in T_J$ is $J$-regular if $(\mathfrak{s}_J\mathfrak{t})^{a^\vee}\not=1$ for all $a\in\Phi\setminus\Phi_{J}$. 
\end{definition}
We now have the following consequence of Proposition \ref{triangular1}(2).
\begin{corollary}\label{14} 
Let $\cc\in C^J$ and $\mathfrak{t}\in T_J$. Assume that $\mathfrak{t}$ is $J$-regular. 
\begin{enumerate}
\item For $w\in W^J$ we have
\[
m_{w\cc}^{J}(x;\mathfrak{t})=\frac{1}{d_{w}(\mathfrak{s}_J\mathfrak{t})}\pi_{\cc,\mathfrak{t}}(S_w^Y)x^\cc+\sum_{u\in W^{J}:\, u<_B w}\epsilon^\prime_{w,u}\pi_{\cc,\mathfrak{t}}(S_u^Y)x^\cc
\]
for some $\epsilon_{w,u}^\prime\in\mathbf{F}$.
\item
$\{\pi_{\cc,\mathfrak{t}}(S_w^Y)x^\cc\,\, | \,\, w\in W^J\}$
is a $Y$-eigenbasis of $\Pc_{\mathfrak{t}}$.
\end{enumerate}
\end{corollary}
\begin{proof}
For  $w\in W^{J}$ we have $\Pi(w)\subseteq\Phi^+\setminus\Phi^+_{J}$ by \eqref{Phipos}, hence $d_{w}\in\mathcal{P}$ does not vanish at $\mathfrak{s}_J\mathfrak{t}$ for $J$-regular $\mathfrak{t}\in T_J$. The corollary now follows from Lemma \ref{mbasis} and Proposition \ref{triangular1}.
\end{proof}
In Theorem \ref{Edef}(3) we described the $Y$-spectrum of $\mathcal{P}_{\mathfrak{t}}^{(\cc)}$ when $\mathfrak{t}\in T_J^\prime$, in which case the $Y$-spectrum is simple. When the condition $\mathfrak{t}\in T_J^\prime$ is replaced by the assumption that $\mathfrak{t}\in T_J$ is $J$-regular, we have the following result.
\begin{corollary}\label{14cor}
Let $\cc\in C^J$ and suppose that $\mathfrak{t}\in T_J$ is $J$-regular. Then
\begin{enumerate}
\item $\Pc_{\mathfrak{t}}$ is $Y$-semisimple.
\item For $s\in T$ we have
\[
\textup{Dim}_{\mathbf{F}}\bigl(\Pc_{\mathfrak{t}}[s]\bigr)=\#\{w\in W^J\,\, 
| \,\, s=w(\mathfrak{s}_J\mathfrak{t})\}.
\]
\end{enumerate}
\end{corollary}
\begin{remark}\label{14rem}
Suppose that $\mathfrak{t}\in T_J$ is $J$-regular. Then $\Pc_{\mathfrak{t}}$ ($\cc\in C^J$) is $Y$-semisimple with finite dimensional $Y$-weight spaces, since $q$ is not a root of unity.
In the terminology of Cherednik \cite[\S 3.6.1]{Ch}, $\Pc_{\mathfrak{t}}\simeq\mathbb{M}^{J}_{\mathfrak{s}_J\mathfrak{t}}$ is a $Y$-cyclic, $Y$-semisimple $\mathbb{H}$-module in category $\mathcal{O}_Y$.
\end{remark}
In the following theorem we relate $\pi_{\cc,\mathfrak{t}}(S_w^Y)x^\cc$ for $\cc\in C^J$, $\mathfrak{t}\in T_J$ and $w\in W^J$ to the monic quasi-polynomial simultaneous eigenfunction $E_{w\cc}^J(x;\mathfrak{t})$ of the commuting operators
$\pi_{\cc,\mathfrak{t}}(Y^\mu)$ ($\mu\in Q^\vee$)
from Theorem \ref{Edef}. This requires $\mathfrak{t}\in T_J^\prime$, but $\mathfrak{t}$ does not have to be $J$-regular.
\begin{theorem}\label{prop:I_and_E}
Let $\cc\in C^J$, $\mathfrak{t}\in T_J^\prime$ and $w\in W^J$. Then
\begin{equation}\label{relS}
\pi_{\cc,\mathfrak{t}}(S_w^Y)x^\cc=
d_w(\mathfrak{s}_J\mathfrak{t})k_w(\cc)E_{w\cc}^J(x;
\mathfrak{t})
\end{equation}
with $k_w(y)\in\mathbf{F}$ and $d_w\in\mathcal{P}$ defined by \eqref{kawy} and \eqref{dwt}, respectively. 
\end{theorem}
\begin{proof}
Let $w\in W^J$.
We have $\pi_{\cc,\mathfrak{t}}(S_w^Y)x^\cc=\epsilon_wE_{w\cc}^J(x;\mathfrak{t})$ for some $\epsilon_w\in\mathbf{F}$ because the $Y$-weight space
$\Pc_{\mathfrak{t}}[w(\mathfrak{s}_J\mathfrak{t})]$ is one-dimensional. 
By \eqref{todo} and \eqref{trmx} we have
\[
\pi_{\cc,\mathfrak{t}}(S_w^Y)x^\cc=
d_w(\mathfrak{s}_J\mathfrak{t})k_w(\cc)x^{w\cc}+ \textup{l.o.t.},
\]
hence Theorem \ref{Edef}(1) implies that $\epsilon_w=d_w(\mathfrak{s}_J\mathfrak{t})k_w(\cc)$.
This concludes the proof.
\end{proof}
\begin{corollary}
With the assumptions of Theorem \ref{prop:I_and_E} we have
\begin{enumerate}
\item
$\pi_{\cc,\mathfrak{t}}(S_w^Y)x^\cc=0$ $\Leftrightarrow$ $d_w(\mathfrak{s}_J\mathfrak{t})=0$ for $w\in W^J$.
\item If $\mathfrak{t}$ is $J$-regular then $\{\pi_{\cc,\mathfrak{t}}(S_w^Y)x^\cc\}_{w\in W^J}$ is a $Y$-eigenbasis of $\mathcal{P}_{\mathfrak{t}}^{(\cc)}$.
\end{enumerate}
\end{corollary}
\begin{proof}
(1) is immediate from \eqref{relS}. For (2) it suffices to recall that the $J$-regularity of $\mathfrak{t}$ ensures that $d_w(\mathfrak{s}_J\mathfrak{t})\not=0$ for all $w\in W^J$.
\end{proof}
\begin{theorem}\label{irredTHM}
Let $\cc\in C^J$. 
Suppose that $\mathfrak{t}\in T_J^\prime$ is $J$-regular and that $(\mathfrak{s}_J\mathfrak{t})^{a^\vee}\not=k_a^{-2}$
for all $a\in\Phi\setminus\Phi_J$. Then $\mathfrak{s}_J\mathfrak{t}\in L^J$, and $\Pc_{\mathfrak{t}}$ is irreducible.
\end{theorem}
\begin{proof}
The first statement follows immediately from the definition of $L^J\subseteq L_J$ (see Remark \ref{specNJt}(1)).

Suppose that $M\subseteq\Pc_{\mathfrak{t}}$ is a nonzero subrepresentation.
By the assumptions on $\mathfrak{t}$, $\Pc_{\mathfrak{t}}$ is $Y$-semisimple with simple $Y$-spectrum, and $\{\pi_{\cc,\mathfrak{t}}(S_w^Y)x^\cc\}_{w\in W^{J}}$ is a $Y$-eigenbasis of $\Pc_{\mathfrak{t}}$. Hence $\pi_{\cc,\mathfrak{t}}(S_w^Y)x^\cc\in M$ for some
$w\in W^{J}$. 
By \eqref{IpropY}, we then have
\begin{equation}\label{dSaction}
\pi_{\cc,\mathfrak{t}}(S_{w^{-1}}^Y)\pi_{\cc,\mathfrak{t}}(S_w^Y)x^\cc=\textup{n}_w(\mathfrak{s}_J\mathfrak{t})x^\cc\in M,
\end{equation}
with $\textup{n}_w\in\mathcal{P}$ defined by
\begin{equation}
\label{nw}
\textup{n}_w(x):=\prod_{a\in\Pi(w)}(k_a^{-1}-k_ax^{a^\vee})(k_a^{-1}-k_ax^{-a^\vee}).
\end{equation}
Note that $\Pi(w)\subseteq\Phi^+\setminus\Phi_{J}^+$ by \eqref{Phipos}, hence $n_w(\mathfrak{s}_J\mathfrak{t})\not=0$ by
the assumption that $(\mathfrak{s}_J\mathfrak{t})^{a^\vee}\not=k_a^{-2}$ for all $a\in\Phi\setminus\Phi_J$. We conclude that $x^\cc\in M$, hence
$M=\Pc_{\mathfrak{t}}$.
\end{proof}

\subsection{Closure relations}\label{ClosureSection}
For $\cc\in C^J$, $\mathfrak{t}\in T_J$ and $w\in W^J$ the coefficients $e_{w,u;\mathfrak{t}}^J$ of $x^{u\cc}$ ($u\in W^J$) in the expansion of $E_{w\cc}^J(x;\mathfrak{t})$ in quasi-monomials do not dependent on $\cc\in C^J$ (see Theorem \ref{Edef}). In this subsection we express the coefficients $e_{w^\prime,u^\prime;\mathfrak{t}^\prime}^{J^\prime}$ in terms of
$e_{w,u;\mathfrak{t}^\prime}^J$ when $J\subseteq J^\prime$ (i.e., when $C^{J^\prime}\subseteq\overline{C^J}$).
  
Recall the map $\textup{pr}_{\cc,\cc^\prime}^{\mathfrak{t}^\prime}$, defined in Corollary \ref{projectioncor}. 
\begin{proposition}\label{projectionprop}
Let $J\subseteq J^\prime$. Fix 
a $J^\prime$-regular $\mathfrak{t}^\prime\in T_{J^\prime}^\prime$ such that $\mathfrak{s}_J^{-1}\mathfrak{s}_{J^\prime}\mathfrak{t}^\prime\in T_J$ lies in $T_J^\prime$ and is $J$-regular.
Let $\cc\in C^J$ and $\cc^\prime\in C^{J^\prime}$. Then 
\begin{equation}\label{ttddnew}
\textup{pr}_{\cc,\cc^\prime}^{\mathfrak{t}^\prime}\Bigl(k_w(\cc)E_{w\cc}^J(x;\mathfrak{s}_J^{-1}\mathfrak{s}_{J^\prime}\mathfrak{t}^\prime)\Bigr)=
\begin{cases}
0\quad &\hbox{ if }\, w\in W^J\setminus W^{J^\prime},\\
k_w(\cc^\prime)E_{w\cc^\prime}^{J^\prime}(x;\mathfrak{t}^\prime)\quad &\hbox{ if }\, w\in W^{J^\prime}.
\end{cases}
\end{equation}
\end{proposition}
\begin{proof}
Let $w\in W^J$. By \eqref{relS}, the left hand side of \eqref{ttddnew} equals 
\[
\frac{1}{d_w(\mathfrak{s}_{J^\prime}\mathfrak{t}^\prime)}\textup{pr}_{\cc,\cc^\prime}^{\mathfrak{t}^\prime}\bigl(\pi_{\cc,\mathfrak{s}_J^{-1}\mathfrak{s}_{J^\prime}\mathfrak{t}^\prime}(S_w^Y)x^\cc
\bigr).
\]
By Proposition \ref{triangular1}(1) and \eqref{relS} it then suffices to show that
\[
\textup{pr}_{\cc,\cc^\prime}^{\mathfrak{t}^\prime}\bigl(\pi_{\cc,\mathfrak{s}_J^{-1}\mathfrak{s}_{J^\prime}\mathfrak{t}^\prime}(S_w^Y)x^\cc\bigr)=
\pi_{\cc^\prime,\mathfrak{t}^\prime}(S_w^Y)x^{\cc^\prime}.
\]
This follows from Corollary \ref{projectioncor}.
\end{proof}
Keep the assumptions of Proposition \ref{projectionprop}. 
For $w\in W^J$ and $u^\prime\in W^{J^\prime}$, consider the expression
\begin{equation}\label{combicoeff}
e_{w,u^\prime;\mathfrak{t}^\prime}^{J,J^\prime}:=\sum_{u\in W^J\cap u^\prime W_{J^\prime}:\,u\leq_Bw}
\frac{\kappa_{Du}(\cc^\prime)}{\kappa_{Du}(\cc)}e_{w,u;\mathfrak{s}_J^{-1}\mathfrak{s}_{J^\prime}\mathfrak{t}^\prime}^J
\end{equation}
involving the coefficients of $E_{w\cc}^J(x;\mathfrak{s}_J^{-1}\mathfrak{s}_{J^\prime}\mathfrak{t}^\prime)$ in its expansion in quasi-monomials.
The following consequence of Proposition \ref{projectionprop} shows that \eqref{combicoeff} provides the coefficients of $E_{w\cc}^{J^\prime}(x;\mathfrak{t}^\prime)$ in its expansion in quasi-monomials when $w\in W^{J^\prime}$.

\begin{corollary}\label{closurecor} 
Keep the assumptions of Proposition \ref{projectionprop}. Then
\begin{enumerate}
\item If $w\in W^J\setminus W^{J^\prime}$ then
$e_{w,u^\prime;\mathfrak{t}^\prime}^{J,J^\prime}=0$ for all $u^\prime\in W^{J^\prime}$.
\item If $w\in W^{J^\prime}$ and $u^\prime\not\leq_Bw$ then $e_{w,u^\prime;\mathfrak{t}^\prime}^{J,J^\prime}=0$.
\item If $w\in W^{J^\prime}$ then $e_{w,w;\mathfrak{t}^\prime}^{J,J^\prime}=k_w(\cc^\prime)/k_w(\cc)$.
\item If $u^\prime,w\in W^{J^\prime}$ and $u^\prime\leq_Bw$ then 
\[
e_{w,u^\prime;\mathfrak{t}^\prime}^{J^\prime}=\frac{k_w(\cc)}{k_w(\cc^\prime)}e_{w,u^\prime;\mathfrak{t}^\prime}^{J,J^\prime}.
\]
\end{enumerate}
\end{corollary}
\begin{proof}
By Corollary \ref{projectioncor} we have
\begin{equation*}
\textup{pr}_{\cc,\cc^\prime}^t\bigl(E_{w\cc}^J(x;\mathfrak{s}_J^{-1}\mathfrak{s}_{J^\prime}\mathfrak{t}^\prime)\bigr)=\sum_{u^\prime\in W^{J^\prime}}e_{w,u^\prime;\mathfrak{t}^\prime}^{J,J^\prime}x^{u^\prime\cc^\prime}\qquad (w\in W^J).
\end{equation*}
The result now follows from Proposition \ref{projectionprop}.
\end{proof}
Consider the special case $J^\prime=[1,r]$, $\cc^\prime=0$ and $\mathfrak{t}^\prime=1_T$, cf. Remark \ref{projectionrem}. 
In this case $W^{[1,r]}=\{w_\mu\,\, | \,\, \mu\in Q^\vee\}$. Let $J\subseteq J^\prime$ and $\cc\in C^J$. Then we conclude from Proposition \ref{projectionprop} that 
\[
\textup{pr}_{\cc,0}^{1_T}(E_{w_\lambda\cc}^J(x;\mathfrak{s}_J^{-1}t_{\textup{sph}}))=\frac{k_{w_\lambda}(0)}{k_{w_\lambda}(\cc)}E_\lambda^{[1,r]}
(x;1_T)\qquad (\lambda\in Q^\vee).
\]
Corollary \ref{closurecor}(3) then yields the expression
\[
e_{w_\lambda,w_\mu;1_T}^{[1,r]}=\frac{k_{w_\lambda}(\cc)}{k_{w_\lambda}(0)}\sum_{u\in W^J\cap w_\mu W_0:\, u\leq_Bw_\lambda}\frac{\kappa_{Du}(0)}{\kappa_{Du}(\cc)}
e_{w_\lambda,u;\mathfrak{s}_J^{-1}t_{\textup{sph}}}^J
\]
for the coefficient of $x^\mu$ in the expansion of
the nonsymmetric Macdonald polynomial $E_\lambda^{[1,r]}(x;1_T)$ in monomials.

\subsection{The normalised $Y$-eigenbasis and pseudo-duality}\label{S64}
 The normalised version of the quasi-polynomial $E_y^{J}(x;\mathfrak{t})$ that we will introduce in this subsection, requires the following additional constraint on $\mathfrak{t}\in T_J^\prime$.
\begin{definition}
We say that
$\mathfrak{t}\in T_J$ is $J$-generic if $\mathfrak{t}$ is $J$-regular and $(\mathfrak{s}_J\mathfrak{t})^{a^\vee}\not=k_a^{-2}$ for all $a\in\Phi^+\setminus\Phi_J^+$.
\end{definition}
Note that $\mathfrak{s}_J\mathfrak{t}\in L^J$ if $\mathfrak{t}\in T_J$ is $J$-generic.
Furthermore, for $\cc\in C^J$ and $J$-generic $\mathfrak{t}\in T_J^\prime$ the irreducibility criterion for the quasi-polynomial representation $\Pc_{\mathfrak{t}}$ (see Theorem \ref{irredTHM}) requires the additional constraints $(\mathfrak{s}_J\mathfrak{t})^{a^\vee}\not=k_a^{-2}$ for all $a\in\Phi^-\setminus\Phi_J^-$. 

For $w\in W$ 
define $r_{w}\in\mathcal{P}$ by 
\begin{equation}\label{rw}
r_{w}(x):=\prod_{a\in\Pi(w)}(k_a x^{a^\vee}-k_a^{-1}).
\end{equation}
Note that $\widetilde{S}^Y_w=S_w^Yr_w(Y^{-1})^{-1}$ in $\mathbb{H}^{Y-\textup{loc}}$ by \eqref{RelToUnnormY}.
For $w\in W^J$ and $J$-generic $\mathfrak{t}\in T_J$ we have $r_{w}(\mathfrak{s}_J\mathfrak{t})\not=0$ by \eqref{Phipos}, and hence we may define for $c\in C^J$,
\begin{equation}\label{boldbw}
P_{w\cc}^{J}(x;\mathfrak{t}):=
\frac{\pi_{\cc,\mathfrak{t}}(S_w^Y)x^\cc}{r_w(\mathfrak{s}_J\mathfrak{t})}\in\Pc_{\mathfrak{t}}[w(\mathfrak{s}_J\mathfrak{t})]\qquad\quad (w\in W^J).
\end{equation}
Note that $\{P_y^J(x;\mathfrak{t})\}_{y\in\mathcal{O}_\cc}$ is a $Y$-eigenbasis of $\mathcal{P}_{\mathfrak{t}}^{(\cc)}$ by Corollary \ref{14}. 
 \begin{remark}\label{projectionrem2}
Adding to the assumptions of Proposition \ref{projectionprop} the assumption that $\mathfrak{t}^\prime$ is $J^\prime$-generic, we have
\[
\textup{pr}_{\cc,\cc^\prime}^{\mathfrak{t}^\prime}\bigl(P_{w\cc}^J(x;\mathfrak{s}_J^{-1}\mathfrak{s}_{J^\prime}\mathfrak{t}^\prime)\bigr)=
P_{w\cc^\prime}^{J^\prime}(x;\mathfrak{t}^\prime)\qquad (w\in W^{J^\prime}).
\]
This follows from \eqref{boldbw} and the fact that $\textup{pr}_{\cc,\cc^\prime}^{\mathfrak{t}^\prime}(x^\cc)=x^{\cc^\prime}$.
\end{remark}
The precise relation of the normalised quasi-polynomials \eqref{boldbw} with the monic
$Y$-eigenbasis $\{E_y^J(x;\mathfrak{t})\}_{y\in\mathcal{O}_\cc}$ of $\mathcal{P}_{\mathfrak{t}}^{(\cc)}$ is as follows.
\begin{proposition}\label{leadcoeff}
Let $\cc\in C^J$. For $J$-generic $\mathfrak{t}\in T_J^\prime$ and $y\in\mathcal{O}_\cc$ set 
\begin{equation}\label{ty}
t_y:=w_y(\mathfrak{s}_J\mathfrak{t})=\mathfrak{s}_y\mathfrak{t}_y\in T.
\end{equation}
Then 
\[
P_y^J(x;\mathfrak{t})=\frac{k(y)}{k(\cc)}\prod_{a\in\Phi^+: a(y)<0}
\Bigl(\frac{t_y^{-a^\vee}-1}{k_at_y^{-a^\vee}-k_a^{-1}}\Bigr)
E_y^J(x;\mathfrak{t})\qquad\quad (y\in\mathcal{O}_\cc).
\]
\end{proposition}
\begin{proof}
By Theorem \ref{prop:I_and_E} we have 
\begin{equation}\label{PYrel}
P_{w\cc}^J(x;\mathfrak{t})=\frac{d_w(\mathfrak{s}_J\mathfrak{t})k_w(\cc)}{r_w(\mathfrak{s}_J\mathfrak{t})}E_{w\cc}^J(x;\mathfrak{t})\qquad\quad (w\in W^J).
\end{equation}
Now substitute $w=w_y\in W^J$ ($y\in\mathcal{O}_\cc$). Then $k_{w_y}(\cc)=k(y)/k(\cc)$ by \eqref{kawy}, and
\[
d_{w_y}(\mathfrak{s}_J\mathfrak{t})=
\prod_{a\in\Pi(w_y^{-1})}\bigl((\mathfrak{s}_J\mathfrak{t})^{-(w_y^{-1}a)^\vee}-1\bigr)=
\prod_{a\in\Phi^+: a(y)<0}\bigl(t_y^{-a^\vee}-1\bigr)
\]
by \eqref{dwt}, $\Pi(w_y)=-w_y^{-1}\Pi(w_y^{-1})$, \eqref{actx} and Corollary
\ref{newroots}. In a similar manner one verifies that
\[
r_{w_y}(\mathfrak{s}_J\mathfrak{t})=\prod_{a\in\Phi^+: a(y)<0}\bigl(k_at_y^{-a^\vee}-k_a^{-1}\bigr).
\]
Substituting these expressions in \eqref{PYrel} gives the desired result.
\end{proof}
\begin{remark}
Suppose that 
$1_T\in T_{[1,r]}^\prime$ and that $1_T$ is $[1,r]$-generic (this holds true for suitably generic values of $k_a\in\mathbf{F}$, cf. Subsection \ref{genSection}).
Then
\[
P_\mu^{[1,r]}(x;1_T)\in\mathcal{P}\qquad
 (\mu\in \mathcal{O}_0=Q^\vee)
 \]
are the normalised nonsymmetric Macdonald polynomials, i.e., $P_\mu^{[1,r]}(x;1_T)$ is the normalisation of
the monic nonsymmetric Macdonald polynomials $E_\mu^{[1,r]}(x;1_T)$
such that 
\[
P_\mu^{[1,r]}(t_{\textup{sph}};1_T)=1.
\]
This follows from Proposition \ref{leadcoeff} and the evaluation formula for the nonsymmetric monic Macdonald polynomials,
\[
E_\mu^{[1,r]}(t_{\textup{sph}};1_T)=\frac{k(0)}{k(\mu)}\prod_{a\in\Phi^+: a(\mu)<0}
\Bigl(\frac{k_a\mathfrak{s}_\mu^{-a^\vee}-k_a^{-1}}{\mathfrak{s}_\mu^{-a^\vee}-1}\Bigr)
\qquad (\mu\in Q^\vee),
\]
see, e.g., \cite[(5.2.14)]{Ma} and \cite[\S 5]{Cm}.
\end{remark}

We will now describe the action of $\delta(T_j)$ on $\mathcal{P}_{\mathfrak{t}}^{(\cc)}$ in terms of a discrete Demazure-Lusztig operator acting on $\mathcal{O}_\cc$ (see Theorem \ref{17}). We first need to analyse the action of the affine Weyl group $W$ on $L_J$ in more detail.

For $w\in W$ set
\[
\rho_w:=\prod_{a\in \Pi(w^{-1})}k_a^{2Da}\in T.
\]
\begin{lemma}\label{7}
If $w\in W_J$ and $t\in L_J$ then $wt=\rho_wt$.
\end{lemma}
\begin{proof}
This is correct for $w=s_j$ ($j\in J$) since $\Pi(s_j)=\{\alpha_j\}$ and $t\in L_J$ (see \eqref{actiononT}). We now apply induction to $\ell(w)$.
If $w=us_j$ with $u \in W_{J}$, $j\in J$ and
$\ell(u s_j)=\ell(u)+1$, then 
\begin{equation}\label{hoh}
\rho_w=k_j^{2D(u\alpha_j)}\rho_{u}
\end{equation}
by Lemma \ref{lengthadd}, and
\[
wt=u(k_j^{2D\alpha_j}t)
=(Du)(k_j^{2D\alpha_j})u t
=k_j^{2D(u\alpha_j)}ut.
\]
Applying the induction hypothesis to $ut$ and using \eqref{hoh}, we obtain $wt=\rho_wt$.
\end{proof}

Suppose that $w\in W^J$ and $j\in [0,r]$ satisfy $s_jw\not\in wW_J$. Then $s_jw\in W^J$ by Lemma \ref{cosetcomb}, hence both
$wt$ and $s_jwt$ are $Y$-weights of $\mathbb{M}^J_t$
when $t\in\mathfrak{s}_JT_J^\prime\subseteq L_J$.

We now analyse what happens with the $Y$-weights of $\mathbb{M}_t^J$ when $w\in W^J$ and $j\in [0,r]$ satisfy $s_jw\in wW_{J}$.

\begin{lemma}\label{8} 
Let $t\in L_J$, $w\in W^J$ and $j\in [0,r]$ such that $s_jw\in wW_{J}$. Then
\[
(wt)^{\alpha_j^\vee}=k_j^{-2},
\]
i.e., $j\in J(wt)$ \textup{(}see \eqref{Jt}\textup{)}. Furthermore,  $s_jwt=k_j^{2D\alpha_j}wt$.
\end{lemma}
\begin{proof}
By Lemma \ref{cosetcomb}, the assumption $s_jw\in wW_J$ implies that 
$s_jw=ws_{j^\prime}$ for some $j^\prime\in J$, and $\ell(s_jw)=\ell(w)+1=\ell(ws_{j^\prime})$. Then $w\alpha_{j^\prime}=\alpha_j$,
and hence
\[
(wt)^{\alpha_j^\vee}=t^{(w^{-1}\alpha_j)^\vee}=t^{\alpha_{j^\prime}^\vee}=k_{j^\prime}^{-2}=k_j^{-2}.
\]
The second statement follows from 
\eqref{actiononT}.
\end{proof}

\begin{corollary}\label{10} 
Let $\cc\in C^J$ and $\mathfrak{t}\in T_J$. Suppose that $w\in W^J$ and $j\in [0,r]$ such that $s_jw\in wW_J$. Then
\begin{equation}\label{tddd}
(\pi_{\cc,\mathfrak{t}}(\delta(T_j))-k_j)\pi_{\cc,\mathfrak{t}}(S_w^Y)x^\cc=0
\end{equation}
if $k_j^2\not=1$.
\end{corollary}
\begin{proof}
For the duration of the proof we write $t:=\mathfrak{s}_J\mathfrak{t}\in L_J$.\\
Let $w\in W^J$ and $j\in [0,r]$ such that $s_jw\in wW_J$. Then $s_jw=ws_{j^\prime}$ for some $j^\prime\in J$, and
$\ell(s_jw)=\ell(w)+1=\ell(ws_{j^\prime})$ (see the proof of Lemma \ref{8}). Hence 
\[
\pi_{\cc,\mathfrak{t}}(S_j^Y)\pi_{\cc,\mathfrak{t}}(S_w^Y)x^\cc=\pi_{\cc,\mathfrak{t}}(S_{s_jw}^Y)x^\cc=\pi_{\cc,\mathfrak{t}}(S_w^YS_{j^\prime}^Y)x^\cc=0
\]
by (the proof of) Proposition \ref{triangular1}(1). We have $\pi_{\cc,\mathfrak{t}}(S_w^Y)x^\cc\in\mathcal{P}_{\mathfrak{t}}^{(\cc)}[wt]$, and $j\in J(wt)$ by Lemma \ref{8}, 
hence 
\[
\pi_{\cc,\mathfrak{t}}\bigl((Y^{-1})^{\alpha_j^\vee}\bigr)\pi_{\cc,\mathfrak{t}}(S_w^Y)x^\cc=k_j^{-2}\pi_{\cc,\mathfrak{t}}(S_w^Y)x^\cc.
\] 
Using the explicit expression \eqref{buildingYintertwiners} for $S_j^Y$ we conclude that 
\[
0=\pi_{\cc,\mathfrak{t}}(S_j^Y)\pi_{\cc,\mathfrak{t}}(S_w^Y)x^\cc=((k_j^{-2}-1)\pi_{\cc,\mathfrak{t}}(\delta(T_j))+k_j-k_j^{-1})\pi_{\cc,\mathfrak{t}}(S_w^Y)x^\cc.
\]
The result then follows from the assumption that $k_j^2\not=1$.
\end{proof}
\begin{remark}\label{ksquaredone}
If $\mathfrak{t}\in T_J^\prime$ and $\mathfrak{t}$ is $J$-regular, then the condition $k_j^2\not=1$ may be omitted in Corollary \ref{10}. Indeed, assume that $\mathfrak{t}\in T_J^\prime$ is $J$-regular and that $k_j^2=1$,
and fix $w\in W^J$ and $j\in [0,r]$ with $s_jw\in wW_J$. Then we have $s_jwt=wt$ by Lemma \ref{8}, where $t=\mathfrak{s}_J\mathfrak{t}$. The cross relations \eqref{crossX} then show that $\pi_{\cc,\mathfrak{t}}(\delta(T_j))\pi_{\cc,\mathfrak{t}}(S_w^Y)x^\cc$ lies in $\mathcal{P}_{\mathfrak{t}}^{(\cc)}[wt]$. Since $\mathfrak{t}\in T_J^\prime$ it follows that $\pi_{\cc,\mathfrak{t}}(\delta(T_j))\pi_{\cc,\mathfrak{t}}(S_w^Y)x^\cc$ is a constant multiple
of $\pi_{\cc,\mathfrak{t}}(S_w^Y)x^\cc$. The constant multiple can be computed by expanding $\pi_{\cc,\mathfrak{t}}(S_w^Y)x^\cc$ in terms of the basis $\{m_{u\cc}^J(x;\mathfrak{t})\}_{u\in W^J}$. Since $\mathfrak{t}$ is $J$-regular, $d_w(t)m_{w\cc}^J(x;\mathfrak{t})$ is the nonzero leading term in this expansion (see Proposition \ref{triangular1}(2)). By the proof of Proposition \ref{triangular1}, the leading term of $\pi_{\cc,\mathfrak{t}}(\delta(T_j))\pi_{\cc,\mathfrak{t}}(S_w^Y)x^\cc$ in its expansion along
$\{m_{u\cc}^J(x;\mathfrak{t})\}_{u\in W^J}$ will then be $d_w(t)\pi_{\cc,\mathfrak{t}}(\delta(T_j))m_w^J(x;\mathfrak{t})=k_jd_w(t)m_w^J(x;\mathfrak{t})$, where the equality follows from Lemma \ref{chitJ}.
  We conclude that \eqref{tddd} is valid.
\end{remark}

In the following theorem we show that the action of $\delta(T_j)$ ($0\leq j\leq r$) on the normalised $Y$-eigenbasis $\{P_y^J(x;\mathfrak{t})\}_{y\in\mathcal{O}_\cc}$ of $\mathcal{P}_{\mathfrak{t}}^{(\cc)}$ is described by discrete Demazure-Lusztig operators acting on $y\in\mathcal{O}_\cc$. We call this the {\it pseudo-duality property} of 
the normalised quasi-polynomials $P_y^J(x;\mathfrak{t})$ ($y\in\mathcal{O}_\cc$).
\begin{theorem}\label{17}
Let $\cc\in C^J$ and $\mathfrak{t}\in T_J^\prime$. Suppose that $\mathfrak{t}$ is $J$-generic. Then 
\begin{equation}\label{DLopgbr}
\pi_{\cc,\mathfrak{t}}(\delta(T_j))P_{y}^J(x;\mathfrak{t})=k_jP_{y}^J(x;\mathfrak{t})+
\Bigl(\frac{k_j^{-1}-k_jt_y^{\alpha_j^\vee}}{1-t_y^{\alpha_j^\vee}}\Bigr)
\bigl(P_{s_jy}^J(x;\mathfrak{t})-P_{y}^J(x;\mathfrak{t})\bigr)
\end{equation}
for $0\leq j\leq r$ and $y\in\mathcal{O}_\cc$. Here $t_y\in T$ is defined by \eqref{ty} and the right hand side of \eqref{DLopgbr} should always be read as $k_jP_y^J(x;\mathfrak{t})$ when $s_jy=y$
\end{theorem}
\begin{proof}
Let $y\in\mathcal{O}_\cc$ and $0\leq j\leq r$. 

The case that $s_jy=y$ follows from Remark \ref{ksquaredone}. Note that in this case we have
$s_jw_y\in w_yW_J$, so Lemma \ref{8} gives 
\[
t_{y}^{\alpha_j^\vee}=(\mathfrak{s}_J\mathfrak{t})^{(w_y^{-1}\alpha_j)^\vee}=k_j^{-2},
\] 
hence the right hand side of \eqref{DLopgbr} is well defined when $k_j^2\not=1$ and it simplifies to $k_jP_y^J(x;\mathfrak{t})$. 

In the remainder of the proof we assume that $s_jy\not=y$. This is equivalent to 
$s_jw_y=w_{s_jy}$ by Lemma \ref{cosetcomb} and Corollary \ref{wyrules}.
Then 
\[
w_y^{-1}\alpha_j\in\Pi(w_y)\cup\Pi(s_jw_y)\subseteq\Phi\setminus\Phi_J
\] 
by Lemma \ref{lengthadd} and \eqref{Phipos}, and
\[
t_{y}^{\alpha_j^\vee}=(\mathfrak{s}_J\mathfrak{t})^{(w_y^{-1}\alpha_j)^\vee}\not=1
\]
since $\mathfrak{t}$ is $J$-generic, hence in particular $J$-regular. So the right hand side of \eqref{DLopgbr} is well defined, and formula \eqref{DLopgbr} is equivalent to
\begin{equation}\label{intermediate0}
\pi_{\cc,\mathfrak{t}}(\delta(T_j))P_{y}^J(x;\mathfrak{t})=\Bigl(\frac{k_j-k_j^{-1}}{1-t_y^{\alpha_j^\vee}}\Bigr)P_y^J(x;\mathfrak{t})+
\Bigl(\frac{k_j^{-1}-k_jt_y^{\alpha_j^\vee}}{1-t_y^{\alpha_j^\vee}}\Bigr)P_{s_jy}^J(x;\mathfrak{t}).
\end{equation}
To prove \eqref{intermediate0}, we consider two cases.\\

\noindent
{\bf Case 1:} $\ell(w_{s_jy})=\ell(w_y)+1$.\\

\noindent
Consider the expansion
\[
\pi_{\cc,\mathfrak{t}}(\delta(T_j))\pi_{\cc,\mathfrak{t}}(S_{w_y}^Y)x^\cc
=\sum_{y^\prime\in\mathcal{O}_\cc}C_{y^\prime}
\pi_{\cc,\mathfrak{t}}(S_{w_{y^\prime}}^Y)x^\cc
\qquad (C_{y^\prime}\in\mathbf{F})
\]
in the $Y$-eigenbasis $\{\pi_{\cc,\mathfrak{t}}(S_{y^\prime}^Y)x^\cc\}_{y^{\prime}\in\mathcal{O}_\cc}$ of $\Pc_{\mathfrak{t}}$. 
Since $\ell(w_{s_jy})=\ell(w_y)+1$ we then have
\begin{equation}\label{spectralexpansion}
\begin{split}
\pi_{\cc,\mathfrak{t}}&(S_{w_{s_jy}}^Y)x^\cc=\pi_{\cc,\mathfrak{t}}(S_j^YS_{w_y}^Y)x^\cc\\
=&\bigl((t_y^{\alpha_j^\vee}-1)\pi_{\cc,\mathfrak{t}}(\delta(T_j))+k_j-k_j^{-1}\bigr)\pi_{\cc,\mathfrak{t}}(S_{w_y}^Y)x^\cc\\
=&\bigl((t_y^{\alpha_j^\vee}-1)C_y+k_j-k_j^{-1}\bigr)\pi_{\cc,\mathfrak{t}}(S_{w_y}^Y)x^\cc
+\sum_{y^\prime\in\mathcal{O}_\cc\setminus\{y\}}(t_y^{\alpha_j^\vee}-1)
C_{y^\prime}\pi_{\cc,\mathfrak{t}}(S_{w_{y^\prime}}^Y)x^\cc.
\end{split}
\end{equation}
The assumptions on $\mathfrak{t}$ ensure that the $t_y$ ($y\in\mathcal{O}_\cc$) are pairwise different $Y$-weights of $\mathcal{P}_{\mathfrak{t}}^{(\cc)}$ and that
$\Pc_{\mathfrak{t}}[t_y]$ is spanned by $\pi_{\cc,\mathfrak{t}}(S_{w_y}^Y)x^\cc$ for $y\in\mathcal{O}_\cc$.
Comparing the different $Y$-weight contributions on both sides of \eqref{spectralexpansion} thus gives
\begin{equation*}
\begin{split}
C_{y^\prime}&=0\quad \textup{ unless }\, y^\prime\in \{y,s_jy\},\\
C_y&=\frac{k_j-k_j^{-1}}{1-t_y^{\alpha_j^\vee}},\qquad
C_{s_jy}=\frac{1}{t_y^{\alpha_j^\vee}-1}.
\end{split}
\end{equation*}
We conclude that
\begin{equation}\label{intcase2}
\pi_{\cc,\mathfrak{t}}(\delta(T_j))\pi_{\cc,\mathfrak{t}}(S_{w_y}^Y)x^\cc
=\Bigl(\frac{k_j-k_j^{-1}}{1-t_y^{\alpha_j^\vee}}\Bigr)\pi_{\cc,\mathfrak{t}}(S_{w_y}^Y)x^\cc+
\Bigl(\frac{1}{t_y^{\alpha_j^\vee}-1}\Bigr)\pi_{\cc,\mathfrak{t}}(S_{w_{s_jy}}^Y)x^\cc.
\end{equation}
Then \eqref{boldbw} and the identity
$r_{w_{s_jy}}(\mathfrak{s}_J\mathfrak{t})=(k_jt_y^{\alpha_j^\vee}-k_j^{-1})r_{w_y}(\mathfrak{s}_J\mathfrak{t})$ give \eqref{intermediate0}.\\

\noindent
{\it Case 2:} $\ell(w_{s_jy})=\ell(w_y)-1$.\\

\noindent
We can now apply the first case to $s_jy$, yielding
\begin{equation}\label{intermediate}
\pi_{\cc,\mathfrak{t}}(\delta(T_j))P_{s_jy}^J(x;\mathfrak{t})=
\Bigl(\frac{k_j-k_j^{-1}}{1-t_{s_jy}^{\alpha_j^\vee}}\Bigr)P_{s_jy}^{J}(x;\mathfrak{t})
+\Bigl(\frac{k_j^{-1}-k_jt_{s_jy}^{\alpha_j^\vee}}{1-t_{s_jy}^{\alpha_j^\vee}}\Bigr)P_{y}^J(x;\mathfrak{t}).
\end{equation}
Now act on both sides with $\delta(T_j)$, apply the Hecke relation 
$\delta(T_j)^2=(k_j-k_j^{-1})\delta(T_j)+1$ to the left hand side, and substitute \eqref{intermediate} for the resulting 
$\pi_{\cc,\mathfrak{t}}(\delta(T_j))P_{s_jy}^J(x;\mathfrak{t})$'s in the equality. A straightforward computation, avoiding rescaling of the equation,  then shows that
\begin{equation*}
\begin{split}
&\Bigl(\frac{k_j^{-1}-k_jt_{s_jy}^{\alpha_j^\vee}}{1-t_{s_jy}^{\alpha_j^\vee}}\Bigr)\pi_{\cc,\mathfrak{t}}(\delta(T_j))P_y^J(x;\mathfrak{t})=\\
&\qquad=\frac{(k_j^{-1}-k_jt_{s_jy}^{\alpha_j^\vee})(k_j-k_j^{-1})}{(1-t_{s_jy}^{\alpha_j^\vee})
(1-t_{s_jy}^{-\alpha_j^\vee})}P_y^J(x;\mathfrak{t})+
\frac{(k_j^{-1}-k_jt_{s_jy}^{\alpha_j^\vee})(k_j^{-1}-k_jt_{s_jy}^{-\alpha_j^\vee})}
{(1-t_{s_jy}^{\alpha_j^\vee})(1-t_{s_jy}^{-\alpha_j^\vee})}P_{s_jy}^J(x;\mathfrak{t}).
\end{split}
\end{equation*}
Now dividing out the prefactor of $\pi_{\cc,\mathfrak{t}}(\delta(T_j))P_y^J(x;\mathfrak{t})$
(which is nonzero since $\mathfrak{t}$ is $J$-generic) and using that $t_{s_jy}^{\alpha_j^\vee}=t_y^{-\alpha_j^\vee}$, we obtain the desired formula \eqref{intermediate0}.
\end{proof}
\begin{remark}
\hfill
\begin{enumerate}
\item 
In case of Cherednik's basic representation $\mathcal{P}^{(0)}_{1_T}$, Theorem \ref{17} is an important intermediate step in proving the well-known duality formula
\[
P_\mu^{[1,r]}(\mathfrak{s}_\nu;1_T)=P_\nu^{[1,r]}(\mathfrak{s}_\mu;1_T)\qquad (\mu,\nu\in Q^\vee)
\]
for the normalised nonsymmetric Macdonald polynomials, see \cite[Thm 5.1]{Cm}. An analogue of the duality formula for the quasi-polynomial generalisations $P_y^J(x;\mathfrak{t})$ of the normalised nonsymmetric Macdonald polynomials is not known.
\item
The quasi-duality property \eqref{DLopgbr}, together with the eigenvalue equations
\[
\pi_{\cc,\mathfrak{t}}(Y^\mu)P_y^J(x;\mathfrak{t})=t_y^{-\mu}P_y^J(x;\mathfrak{t})\qquad (y\in\mathcal{O}_\cc,\, \mu\in Q^\vee)
\]
for $\cc\in C^J$ and $J$-generic $\mathfrak{t}\in T_J^\prime$, allows one to describe the $\mathbb{H}$-action on $\Pc_{\mathfrak{t}}$ entirely in terms of 
an $\mathbb{H}$-action on the space of $\mathbf{F}$-valued functions on $\mathcal{O}_\cc$, with $\delta(T_j)$ ($0\leq j\leq r$) acting by discrete Demazure-Lusztig type operators (cf. Remark \ref{realIND}). 
The subspace of finitely supported $\mathbf{F}$-valued functions on $\mathcal{O}_\cc$ forms a $\mathbb{H}$-submodule, which alternatively can be obtained as a quotient of 
Cherednik's \cite[\S 3.4.2]{Ch} $\mathbb{H}$-representation on the space of finitely supported $\mathbf{F}$-valued functions on $W$ (cf. Remark \ref{realIND}(3)).
\end{enumerate}
\end{remark}
\begin{corollary}\label{17cor}
Let $\cc\in C^J$ and $\mathfrak{t}\in T_J^\prime$. Suppose that $\mathfrak{t}$ is $J$-generic. Then 
\begin{equation}\label{DLopgbrcor}
\pi_{\cc,\mathfrak{t}}(S_j^Y)P_y^J(x;\mathfrak{t})=\bigl(k_jt_y^{\alpha_j^\vee}-k_j^{-1})P_{s_jy}^J(x;\mathfrak{t})
\end{equation}
for $j\in [0,r]$ and $y\in\mathcal{O}_\cc$, where $t_y\in T$ is defined by \eqref{ty}. In particular, 
\begin{equation}\label{DLopgbrcor1}
\pi_{\cc,\mathfrak{t}}(S_j^Y)P_y^J(x;\mathfrak{t})=0\,\,\hbox{ if }\, s_jy=y.
\end{equation}
\end{corollary}
\begin{proof}
If $s_jy=y$ then $t_y^{\alpha_j^\vee}=k_j^{-2}$, hence the right hand side of \eqref{DLopgbrcor} is zero and the resulting equation \eqref{DLopgbrcor1} follows from \eqref{buildingYintertwiners} and Remark \ref{ksquaredone}. If $s_jy\not=y$ then \eqref{DLopgbrcor} follows from \eqref{buildingYintertwiners} and \eqref{DLopgbr} by a straightforward computation.
\end{proof}
\begin{remark}
Under additional assumptions on $\mathfrak{t}$, a proof of Corollary \ref{17cor} that avoids Theorem \ref{17} is as follows.

Formula \eqref{DLopgbrcor} follows directly from 
the definition of $P_y^J(x;\mathfrak{t})$ (see \eqref{boldbw}) when $s_jw_y\not\in W^J$, and when $s_jw_y\in W^J$ and $\ell(s_jw_y)=\ell(w_y)+1$. In the first case use Proposition \ref{triangular1}(1) and Lemma \ref{8}, in the second case use $s_jw_y=w_{s_jy}$ and 
$r_{s_jw_y}(x)=(k_jx^{(w_y^{-1}\alpha_j)^\vee}-k_j^{-1})r_{w_y}(x)$. 
When $s_jw_y\in W^J$ and $\ell(s_jw_y)=\ell(w_y)-1$, acting by $S_j^Y$ on \eqref{DLopgbrcor} and applying \eqref{IpropY} gives
\[
(k_jt_y^{\alpha_j^\vee}-k_j^{-1})\pi_{\cc,\mathfrak{t}}(S_j^Y)P_{s_jy}^J(x;\mathfrak{t})=(k_jt_y^{\alpha_j^\vee}-k_j^{-1})(k_jt_y^{-\alpha_j^\vee}-k_j^{-1})P_y^J(x;\mathfrak{t}).
\]
This yields \eqref{DLopgbrcor} if the scalar factor $k_jt_y^{\alpha_j^\vee}-k_j^{-1}$ can be divided out.
This is ensured if $(\mathfrak{s}_J\mathfrak{t})^{a^\vee}\not=k_a^{-2}$ for
all $a\in\Phi^-\setminus\Phi_J^-$. Note that these are exactly the additional requirements on $\mathfrak{t}$ ensuring the irreducibility of $\pi_{\cc,\mathfrak{t}}$, see Theorem \ref{irredTHM}. 
\end{remark}

The action of $\delta(T_j)$ on the monic quasi-polynomials $E_y^J(x;\mathfrak{t})$ can be described as follows.
Recall the constants $\kappa_v(y)\in\mathbf{F}^\times$ ($y\in E$, $v\in W_0$), defined by \eqref{kvy}.
\begin{corollary}\label{Whittakerprop}
Let $\cc\in C^J$ and assume that $\mathfrak{t}\in T_J^\prime$ is $J$-generic.

The monic quasi-polynomials $E_y^J(x;\mathfrak{t})\in\Pc_{\mathfrak{t}}$ \textup{(}$y\in\mathcal{O}_\cc$\textup{)} satisfy for $0\leq j\leq r$,
\begin{equation*}
\pi_{\cc,\mathfrak{t}}(\delta(T_j))E_y^J(x;\mathfrak{t})=
\begin{cases}
k_jE_y^J(x;\mathfrak{t})\qquad &\hbox{ if }\,\, \alpha_j(y)=0,\\
\Bigl(\frac{k_j-k_j^{-1}}{1-t_y^{\alpha_j^\vee}}\Bigr)E_y^J(x;\mathfrak{t})+\kappa_{Ds_j}(y)E_{s_jy}^J(x;\mathfrak{t})
\qquad &\hbox{ if }\,\, \alpha_j(y)>0,
\end{cases}
\end{equation*}
where $t_y\in T$ is defined by \eqref{ty}. Furthermore, for $y\in\mathcal{O}_\cc$ such that $\alpha_j(y)<0$,
\begin{equation*}
\begin{split}
\pi_{\cc,\mathfrak{t}}(\delta(T_j))E_y^J(x;\mathfrak{t})&=
\Bigl(\frac{k_j-k_j^{-1}}{1-t_y^{\alpha_j^\vee}}\Bigr)E_y^J(x;\mathfrak{t})\\
&+
\kappa_{Ds_j}(y)\frac{(k_j-k_j^{-1}t_y^{\alpha_j^\vee})(k_j-k_j^{-1}t_y^{-\alpha_j^\vee})}
{\bigl(1-t_y^{\alpha_j^\vee}\bigr)\bigl(1-t_y^{-\alpha_j^\vee}\bigr)}E_{s_jy}^J(x;\mathfrak{t}).
\end{split}
\end{equation*}
\end{corollary}
\begin{proof}
The statement when $\alpha_j(y)=0$ is clear.
If $\alpha_j(y)>0$ then $s_jw_y=w_{s_jy}\in W^J$ and $\ell(s_jw_y)=\ell(w_y)+1$ by Corollary
\ref{wyrules}. Furthermore $t_y^{\alpha_j^\vee}\not=1$, see the proof of Theorem \ref{17}. Then \eqref{intcase2} and \eqref{relS} give
\[
\pi_{\cc,\mathfrak{t}}(\delta(T_j))E_y^J(x;\mathfrak{t})=\Bigl(\frac{k_j-k_j^{-1}}{1-t_y^{\alpha_j^\vee}}\Bigr)
E_y^J(x;\mathfrak{t})+\frac{d_{s_jw_y}(\mathfrak{s}_J\mathfrak{t})k_{s_jw_y}(\cc)}{d_{w_y}(\mathfrak{s}_J\mathfrak{t})k_{w_y}(\cc)}
\Bigl(\frac{1}{t_y^{\alpha_j^\vee}-1}\Bigr)E_{s_jy}^J(x;\mathfrak{t}).
\]
Simplifying the second factor using \eqref{eq:d_relation}, \eqref{kawy} and Lemma \ref{kylemma2} yields the desired result.

If $\alpha_j(y)<0$ then the formula follows from a direct computation using the previous case, the Hecke relation \eqref{Hr} and Lemma \ref{kylemma2}.
\end{proof}

\subsection{(Anti)symmetrisation}\label{S65}
By \eqref{kvy} we have $\kappa_v(0)=\prod_{\alpha\in\Pi(v)}k_\alpha$. The element
\[
\mathbf{1}_+:=\sum_{v\in W_0}\kappa_v(0)T_v=\kappa_{w_0}(0)^2\sum_{v\in W_0}\kappa_v(0)^{-1}T_{v^{-1}}^{-1}
\]
is, up to a normalisation factor, the trivial idempotent of $H_0$. Concretely, it satisfies 
\[
T_i\mathbf{1}_+=k_i\mathbf{1}_+=\mathbf{1}_+T_i\qquad (1\leq i\leq r)
\]
and $\mathbf{1}_+^2=\bigl(\sum_{v\in W_0}\kappa_v(0)^2\bigr)\mathbf{1}_+$. 

The following result follows by a direct computation.
\begin{lemma}\label{SymmInt}
Let $i\in [1,r]$. Then
\begin{equation*}
S_i^X\mathbf{1}_+=(k_ix^{\alpha_i^\vee}-k_i^{-1})\mathbf{1}_+,\qquad
\mathbf{1}_+S_i^Y=\mathbf{1}_+(k_i(Y^{-1})^{\alpha_i^\vee}-k_i^{-1})
\end{equation*}
in $\mathbb{H}$. In particular, $\widetilde{S}_i^X\mathbf{1}_+=\mathbf{1}_+$ in 
$\mathbb{H}^{X-\textup{loc}}$ and $\mathbf{1}_+\widetilde{S}_i^Y=\mathbf{1}_+$
in $\mathbb{H}^{Y-\textup{loc}}$.
\end{lemma}
For $\cc\in C^J$ and $\mathfrak{t}\in T_J$ consider the $\mathcal{P}^{W_0}$-submodule
\[
\mathcal{P}^{(\cc),+}_{\mathfrak{t}}:=\pi_{\cc,\mathfrak{t}}(\mathbf{1}_+)\Pc_{\mathfrak{t}}
\]
of $\Pc_{\mathfrak{t}}$. The interpretation of $\mathcal{P}^{(\cc),+}_{\mathfrak{t}}$ as $W_0$-invariant quasi-polynomials requires the 
$\mathbb{H}^{X-\textup{loc}}$-module $\mathcal{Q}^{(\cc)}_{\mathfrak{t}}$ from the proof of Theorem \ref{aCG}, which contains $\mathcal{P}_{\mathfrak{t}}^{(\cc)}$ as
$\mathbb{H}$-submodule.
\begin{lemma}\label{symm_space}
Let $\cc\in C^J$ and $\mathfrak{t}\in T_J$. The $\mathcal{Q}^{W_0}$-submodule $\mathcal{Q}^{(\cc),+}_{\mathfrak{t}}:=\pi_{\cc,\mathfrak{t}}^{X-\textup{loc}}(\mathbf{1}_+)\mathcal{Q}^{(\cc)}_{\mathfrak{t}}$
of $\mathcal{Q}_{\mathfrak{t}}^{(\cc)}$ is contained in 
$\{f\in\mathcal{Q}^{(\cc)}_{\mathfrak{t}} \,\, | \,\, \sigma_{\cc,\mathfrak{t}}(v)f=f\quad \forall\, v\in W_0\}$.
\end{lemma}
\begin{proof}
For $f\in\mathcal{Q}_{\mathfrak{t}}^{(\cc),+}$ we have $\pi^{X-\textup{loc}}_{\cc,\mathfrak{t}}(\widetilde{S}_v^X)f=f$ 
for all $v\in W_0$ by Lemma \ref{SymmInt}. 
The lemma now follows from the definition \eqref{localaction} of the $W\ltimes\mathcal{Q}$-action $\sigma_{\cc,\mathfrak{t}}$.
\end{proof}

\begin{definition}
Let $\cc\in C^J$ and $\mathfrak{t}\in T_J^\prime$.  
Define $E_y^{J,+}(x;\mathfrak{t})\in\mathcal{P}^{(\cc),+}_{\mathfrak{t}}$ \textup{(}$y\in\mathcal{O}_\cc$\textup{)} by
\[
E_y^{J,+}(x;\mathfrak{t}):=\pi_{\cc,\mathfrak{t}}(\mathbf{1}_+)E_y^J(x;\mathfrak{t}).
\]
If furthermore $\mathfrak{t}$ is $J$-generic, then we define normalised versions $P_y^{J,+}(x;\mathfrak{t})\in\mathcal{P}_{\mathfrak{t}}^{(\cc),+}$
of $E_y^{J,+}(x;\mathfrak{t})$ by 
\begin{equation*}
P_y^{J,+}(x;\mathfrak{t}):=\pi_{\cc,\mathfrak{t}}(\mathbf{1}_+)P_y^J(x;\mathfrak{t})\qquad (y\in\mathcal{O}_\cc).
\end{equation*}
\end{definition}
Note that $E_y^{J,+}(x;\mathfrak{t})\in\mathcal{P}^{(\cc),+}_{\mathfrak{t}}$ ($y\in\mathcal{O}_\cc$) for $\cc\in C^J$ and $\mathfrak{t}\in T_J^\prime$ satisfies the eigenvalue equations
\begin{equation}\label{Wde}
\pi_{\cc,\mathfrak{t}}(p(Y))E_y^{J,+}(x;\mathfrak{t})=p(\mathfrak{s}_y^{-1}\mathfrak{t}_y^{-1})E_y^{J,+}(x;\mathfrak{t})\qquad \forall\, p\in\mathcal{P}^{W_0},
\end{equation}
since $\mathcal{P}_Y^{W_0}$ is the center of $H$. The same holds true for $P_y^{J,+}(x;\mathfrak{t})$.
\begin{remark}\label{remsymMac}
If $1_T\in T_{[1,r]}^\prime$ then $E_\mu^{[1,r],+}(x;1_T)\in\mathcal{P}^{W_0}$ ($\mu\in Q^\vee$) is, up to a normalisation factor, the symmetric Macdonald polynomial of degree $\mu\in Q^\vee$ (see, e.g., \cite[\S 5.3]{Ma}).
\end{remark}

\begin{proposition}\label{ydepSymm}
Let $\cc\in C^J$ and assume that $\mathfrak{t}\in T_J^\prime$ is $J$-generic. 
\begin{enumerate}
\item For $i\in [1,r]$ such that $\alpha_i(y)>0$ we have
\[
E_{s_iy}^{J,+}(x;\mathfrak{t})=k_i^{\eta(\alpha_i(y))}\Bigl(\frac{k_i^{-1}-k_it_y^{\alpha_i^\vee}}{1-t_y^{\alpha_i^\vee}}\Bigr)E_y^{J,+}(x;\mathfrak{t}),
\]
where $t_y\in T$ is given by \eqref{ty}.
\item $P_{s_iy}^{J,+}(x;\mathfrak{t})=P_y^{J,+}(x;\mathfrak{t})$
for $i\in [1,r]$ and $y\in\mathcal{O}_\cc$.
\end{enumerate}
\end{proposition}
\begin{proof}
(1) Let $i\in [1,r]$ such that $\alpha_i(y)>0$. Then $t_y^{\alpha_i^\vee}\not=1$, see the proof of Theorem \ref{17}. We furthermore have $\kappa_{s_i}(y)=k_i^{-\eta(\alpha_i(y))}$, hence
Corollary \ref{Whittakerprop} gives
\begin{equation*}
\begin{split}
E_{s_iy}^{J,+}(x;\mathfrak{t})&=k_i^{\eta(\alpha_i(y))}\pi_{\cc,\mathfrak{t}}(\mathbf{1}_+)
\Bigl(\pi_{\cc,\mathfrak{t}}(T_i)-\Bigl(\frac{k_i-k_i^{-1}}{1-t_y^{\alpha_i^\vee}}\Bigr)\Bigr)
E_y^J(x;\mathfrak{t})\\
&=k_i^{\eta(\alpha_i(y))}\Bigl(\frac{k_i^{-1}-k_it_y^{\alpha_i^\vee}}{1-t_y^{\alpha_i^\vee}}\Bigr)
E_y^{J,+}(x;\mathfrak{t}),
\end{split}
\end{equation*}
where the second equality follows by a direct computation using 
$\mathbf{1}_+T_i=k_i\mathbf{1}_+$.\\
(2) It suffices to prove the statement under the additional assumption that $\alpha_i(y)>0$. Then
$w_{s_iy}=s_iw_y$, 
$\ell(w_{s_iy})=\ell(w_y)+1$ and $\Pi(w_{s_iy})=\{w_y^{-1}\alpha_i\}\cup\Pi(w_y)$, hence 
\[
P_{s_iy}^{J,+}(x;\mathfrak{t})=\frac{\pi_{\cc,\mathfrak{t}}(\mathbf{1}_+S_{w_{s_iy}}^Y)x^\cc}{r_{w_{s_iy}}(\mathfrak{s}_J\mathfrak{t})}=
\frac{\pi_{\cc,\mathfrak{t}}(\mathbf{1}_+S_i^YS_{w_y}^Y)x^\cc}{r_{w_y}(\mathfrak{s}_J\mathfrak{t})
(k_it_y^{\alpha_i^\vee}-k_i^{-1})}.
\]
Applying Lemma \ref{SymmInt} we conclude that
\[
P_{s_iy}^{J,+}(x;\mathfrak{t})=\frac{\pi_{\cc,\mathfrak{t}}(\mathbf{1}_+(k_i^{-1}-k_i(Y^{-1})^{\alpha_i^\vee}))P_y^J(x;\mathfrak{t})}
{(k_i^{-1}-k_it_y^{\alpha_i^\vee})}=
\pi_{\cc,\mathfrak{t}}(\mathbf{1}_+)P_y^J(x;\mathfrak{t})=P_y^{J,+}(x;\mathfrak{t}),
\]
as desired.
\end{proof}
For $\cc\in\overline{C}_+$ write $\mathcal{O}_\cc^+:=\mathcal{O}_\cc\cap\overline{E}_+$, which is a fundamental domain for the $W_0$-action on
$\mathcal{O}_\cc$.
\begin{corollary}\label{basisW}
Let $\cc\in C^J$ and assume that $\mathfrak{t}\in T_J^\prime$ is $J$-generic. 

Then 
$\{E_y^{J,+}(x;\mathfrak{t})\}_{y\in\mathcal{O}_\cc^+}$ and $\{P_y^{J,+}(x;\mathfrak{t})\}_{y\in\mathcal{O}_\cc^+}$ are bases of $\mathcal{P}^{(\cc),+}_{\mathfrak{t}}$ consisting of simultaneous eigenfunctions of the commuting operators
$\pi_{\cc,\mathfrak{t}}(p(Y))$ \textup{(}$p\in\mathcal{P}^{W_0}$\textup{)}. 
\end{corollary}
\begin{proof}
By Proposition \ref{ydepSymm} and \eqref{Wde} it suffices to show that $\{E_y^{J,+}(x;\mathfrak{t})\}_{y\in\mathcal{O}_\cc^+}$ is a linear independent set. But Theorem \ref{17} shows
that $E_y^{J,+}(x;\mathfrak{t})$ lies in subspace spanned by $E_{vy}^J(x;\mathfrak{t})$ ($v\in W_0$), hence the linear independence is a consequence of 
Theorem \ref{Edef}(3).
\end{proof}
The coefficients in the expansion of $E_y^{J,+}(x;\mathfrak{t})$ as linear combination of the quasi-polynomials $E_{vy}^J(x;\mathfrak{t})$ ($v\in W_0$) can be explicitly computed as follows.

Recall that $y_{\pm}$ denotes the unique element in $\overline{E}_{\pm}\cap W_0y$. 
Write $W_{0,y_{\pm}}$ for the subgroup of $W_0$ consisting of the elements
$v\in W_0$ fixing $y_{\pm}$. It is a parabolic subgroup. Denote by $W_0^{y_{\pm}}$ the minimal coset representatives of $W_0/W_{0,y_{\pm}}$. Denote by $g_y\in W_0$ the unique element of minimal length such that $y_-=g_yy$. Note that $y\mapsto g_y^{-1}$ defines a bijection $W_0y\overset{\sim}{\longrightarrow} W_0^{y_-}$. 
Observe furthermore that $g_{y_+}$ is the minimal coset representative of $w_0W_{0,y_+}$. In particular, $y_-=g_{y_+}y_+=w_0y_+$.
\begin{theorem}\label{Weylformula}
Let $\cc\in C^J$ and assume that $\mathfrak{t}\in T_J^\prime$ is $J$-generic. Set $t_y:=w_y(\mathfrak{s}_J\mathfrak{t})$ for $y\in\mathcal{O}_\cc$ \textup{(}see \eqref{ty}\textup{)}. 
For $y\in\mathcal{O}_\cc$ we have
\begin{equation}\label{expgen}
E_y^{J,+}(x;\mathfrak{t})=\sum_{y^\prime\in W_0y}\textup{C}_y^+(y^\prime)E_{y^\prime}^{J}(x;\mathfrak{t})
\end{equation}
with the coefficients
$C_y^+(y^\prime)\in\mathbf{F}$ determined by the following two properties:
\begin{equation}\label{ch1}
C_{s_iy}^+(y^\prime)=k_i^{\eta(\alpha_i(y))}\Bigl(\frac{k_i-k_i^{-1}t_y^{-\alpha_i^\vee}}
{1-t_y^{-\alpha_i^\vee}}\Bigr)C_y^+(y^\prime)\,\, \hbox{ for }\,\, 1\leq i\leq r\,
\hbox{ such that }\, \alpha_i(y)>0
\end{equation}
and
\begin{equation}\label{ch2}
C_{y_+}^+(y^\prime)=\Bigl(\sum_{v\in W_{0,y_+}}\kappa_v(0)^2\Bigr)\kappa_{g_{y_+}}(0)\kappa_{g_{y_+}}(y_+)
\prod_{\beta\in\Pi(g_{y^\prime}^{-1})}k_\beta^{-\eta(\beta(y_-))}
\Bigl(\frac{k_\beta-k_\beta^{-1}t_{y_-}^{-\beta^\vee}}{1-t_{y_-}^{-\beta^\vee}}\Bigr)
\end{equation}
for $y^\prime\in W_0y$.
\end{theorem}
\begin{proof}
Note that there exists an expansion of the form \eqref{expgen} for unique $C_y^+(y^\prime)\in\mathbf{F}$, cf. Corollary \ref{basisW}. It is clear that the two conditions
\eqref{ch1} and \eqref{ch2} determine the coefficients $C_y(y^\prime)\in\mathbf{F}$ 
uniquely.

The recursion relation
\eqref{ch1} follows from Proposition \ref{ydepSymm}(1). It thus remains to prove \eqref{ch2}.

Fix $y\in\mathcal{O}_\cc$. Note that
\[
\mathbf{1}_+=\mathbf{1}_+^{y_+}\sum_{v\in W_{0,y_+}}\kappa_v(0)T_v
\]
with $\mathbf{1}_+^{y_+}:=\sum_{u\in W_0^{y_+}}\kappa_u(0)T_u$. Then Corollary \ref{Whittakerprop} gives
\[
E_{y_+}^{J,+}(x;\mathfrak{t})=\Bigl(\sum_{v\in W_{0,y_+}}\kappa_v(0)^2\Bigr)
\pi_{\cc,\mathfrak{t}}(\mathbf{1}_+^{y_+})E_{y_+}^{J}(x;\mathfrak{t}).
\]
Note that for $u\in W_0^{y_+}$ and
$1\leq i\leq r$ one has $\alpha_i(uy_+)\gtrless 0$ iff $s_iu\in W_0^{y_+}$ and $\ell(s_iu)=\ell(u)\pm 1$ (this follows from the Weyl group variant of Corollary \ref{wyrules} with
$W$ and its fundamental alcove $\overline{C}_+$ replaced by $W_0$ and its fundamental Weyl chamber $\overline{E}_{+}$). Again invoking
Corollary \ref{Whittakerprop} we then get from Lemma \ref{kylemma1},
\begin{equation}\label{lt}
C_{y_+}^+(y_-)=\Bigl(\sum_{v\in W_{0,y_+}}\kappa_v(0)^2\Bigr)\kappa_{g_{y_+}}(0)\kappa_{g_{y_+}}(y_+).
\end{equation}
We now derive the coefficients $C_{y_+}^+(y^\prime)$ for arbitrary $y^\prime\in W_0y$ recursively
using the recursion relation
\begin{equation}\label{hulpagain}
C_{y^{\prime\prime}}^+(s_iy^\prime)=C_{y^{\prime\prime}}^+(y^\prime)k_i^{-\eta(\alpha_i(y^\prime))}
\Bigl(
\frac{k_i-k_i^{-1}t_{y^\prime}^{-\alpha_i^\vee}}{1-t_{y^\prime}^{-\alpha_i^\vee}}\Bigr)
\, \hbox{ for }\, y^\prime\in W_0y\, \hbox{ such that }\, \alpha_i(y^\prime)<0
\end{equation}
for $y^{\prime\prime}\in W_0y$.
The recursion relation \eqref{hulpagain} follows from the identity 
\[
\pi_{\cc,\mathfrak{t}}(T_i)E_{y^{\prime\prime}}^{J,+}(x;\mathfrak{t})=k_iE_{y^{\prime\prime}}^{J,+}(x;\mathfrak{t}),
\]
expanding $E_{y^{\prime\prime}}^{J,+}(x;\mathfrak{t})$ as linear combination of quasi-polynomials $E_{y^{\prime\prime\prime}}^J(x;\mathfrak{t})$ ($y^{\prime\prime\prime}\in W_0y$), and subsequently comparing the coefficient of $E_{y^\prime}^J(x;\mathfrak{t})$ on both sides using Corollary \ref{Whittakerprop}.

Now take $y^\prime\in W_0y$ and let $g_{y^\prime}^{-1}=s_{i_1}\cdots s_{i_\ell}$ be a reduced expression. Then $\Pi(g_{y^\prime}^{-1})=\{\beta_1,\ldots,\beta_\ell\}$
with $\beta_m:=s_{i_\ell}\cdots s_{i_{m+1}}(\alpha_{i_m})$ for $1\leq m\leq\ell$. Furthermore,
\[
\beta_m(y_-)=\alpha_{i_m}(s_{i_{m+1}}\cdots s_{i_\ell}y_-)<0\qquad\quad (1\leq m\leq\ell)
\] 
since $g_{y^\prime}^{-1}\in W_0^{y_-}$ (this follows from the Weyl group variant of Corollary \ref{wyrules} with
$W$ and $\overline{C}_+$ replaced by $W_0$ and $\overline{E}_{-}$). By repeated application of \eqref{hulpagain} 
we conclude that
\[
C_{y_+}^+(y^\prime)=C_{y_+}^+(y_-)\prod_{\beta\in\Pi(g_{y^\prime}^{-1})}k_\beta^{-\eta(\beta(y_-))}\Bigl(\frac{k_\beta-k_\beta^{-1}t_{y_-}^{-\beta^\vee}}{1-t_y^{-\beta^\vee}}\Bigr).
\]
This completes the proof of the theorem. 
\end{proof}
When considering the Whittaker limit of $E_y^{J,+}(x;\mathfrak{t})$ in Subsection \ref{rationalsection} it is convenient to use the explicit expansion formula for $E_y^{J,+}(x;\mathfrak{t})$ when $y\in\overline{E}_-\cap\mathcal{O}_\cc$.
\begin{corollary}\label{corplus}
Let $\cc\in C^J$, and assume that $\mathfrak{t}\in T_J^\prime$ is $J$-generic. 
For $y\in\overline{E}_-\cap\mathcal{O}_\cc$ we have
\begin{equation*}
\begin{split}
E_{y}^{J,+}(x;\mathfrak{t})&=\Bigl(\sum_{v\in W_{0,y}}\kappa_{g_{y_+}^{-1}v}(0)^2\Bigr)\prod_{\alpha\in\Pi(g_{y_+}^{-1})}\Bigl(\frac{1-k_\alpha^{-2}t_{y}^{\alpha^\vee}}
{1-t_{y}^{\alpha^\vee}}\Bigr)\\
&\qquad\quad\times
\sum_{v\in W_0^{y}}\frac{\kappa_v(y)}{\kappa_v(0)}\Bigl(\prod_{\beta\in\Pi(v)}\Bigl(\frac{1-k_\beta^{2}t_{y}^{\beta^\vee}}{1-t_{y}^{\beta^\vee}}\Bigr)\Bigr)
E_{vy}^J(x;\mathfrak{t}),
\end{split}
\end{equation*}
where $t_y\in T$ is defined by \eqref{ty}.
\end{corollary}
\begin{proof}
Using Lemma \ref{kylemma1} and \eqref{ch1} we have for $y\in\overline{E}_-\cap\mathcal{O}_\cc$,
\[
C_{y}^+(y^\prime)=
\frac{\kappa_{g_{y_+}^{-1}}(0)^2}{\kappa_{g_{y_+}}(0)\kappa_{g_{y_+}}(y_+)}\Bigl(\prod_{\alpha\in\Pi(g_{y_+}^{-1})}\Bigl(
\frac{1-k_\alpha^{-2}t_{y}^{\alpha^\vee}}{1-t_{y}^{\alpha^\vee}}\Bigr)\Bigr)C_{y_+}^+(y^\prime).
\]
Substituting \eqref{ch2}, the desired result follows after a direct computation using the fact that $\kappa_v(0)=\prod_{\alpha\in\Pi(v)}k_\alpha$ and 
$\kappa_v(y)=\prod_{\alpha\in\Pi(v)}k_\alpha^{-\eta(\alpha(y))}$ for $v\in W_0^y$ (cf. Lemma \ref{kylemma1}).
\end{proof}
Within the ambient $\mathbb{H}^{X-\textup{loc}}$-module $\mathcal{Q}_{\mathfrak{t}}^{(\cc)}$ of $\mathcal{P}_{\mathfrak{t}}^{(\cc)}$, we can express $E_y^{J,+}(x;\mathfrak{t})$ as  linear combination of the $\sigma_{\cc,\mathfrak{t}}(W_0)$-translates of $E_y^J(x;\mathfrak{t})$ as follows.
\begin{proposition}\label{Namaraprop}
Let $\cc\in C^J$ and assume that $\mathfrak{t}\in T_J^\prime$ is $J$-generic. For $y\in\mathcal{O}_\cc$ we have
\[
E_y^{J,+}(x;\mathfrak{t})=\sum_{v\in W_0}\kappa_v(0)^2\Bigl(\prod_{\alpha\in\Phi_0^+}\frac{1-k_\alpha^{2\chi(v^{-1}\alpha)}x^{-\alpha^\vee}}{1-x^{-\alpha^\vee}}\Bigr)
\sigma_{\cc,\mathfrak{t}}(v)E_y^J(x;\mathfrak{t})
\]
where \textup{(}recall\textup{)} $\chi(\beta)=\pm 1$ if $\beta\in\Phi_0^{\pm}$.
\end{proposition}
\begin{proof}
By 
\cite[(5.5.14)]{Ma} we have
\[
\ss^{-1}(\mathbf{1}_+)=\Bigl(\sum_{v\in W_0}v\Bigr)\prod_{\alpha\in\Phi_0^+}\Bigl(\frac{1-k_\alpha^2x^{-\alpha^\vee}}{1-x^{-\alpha^\vee}}\Bigr)
\]
and hence, in view of \eqref{linkCGlocal} and the definition of $E_y^{J,+}(x;\mathfrak{t})$,
\[
E_y^{J,+}(x;\mathfrak{t})=\pi_{\cc,\mathfrak{t}}^{X-\textup{loc}}(\mathbf{1}_+)E_y^J(x;\mathfrak{t})=
\sum_{v\in W_0}\sigma_{\cc,\mathfrak{t}}(v)\Bigl(E_y^J(x;\mathfrak{t})\prod_{\alpha\in\Phi_0^+}\Bigl(\frac{1-k_\alpha^2x^{-\alpha^\vee}}{1-x^{-\alpha^\vee}}\Bigr)
\Bigr)
\]
in $\mathcal{Q}^{(\cc)}$. The result now follows by a straightforward computation.
\end{proof}
Similar results can be obtained with respect to Hecke anti-symmetrisation. Write
\begin{equation}\label{basicminus0}
\mathbf{1}_-:=\sum_{v\in W_0}(-1)^{\ell(v)}\kappa_v(0)^{-1}T_v=\kappa_{w_0}(0)^{-2}\sum_{v\in W_0}(-1)^{\ell(v)}\kappa_v(0)T_{v^{-1}}^{-1},
\end{equation}
which satisfies $T_i\mathbf{1}_-=-k_i^{-1}\mathbf{1}_-=\mathbf{1}_-T_i$ ($1\leq i\leq r$). 
\begin{definition}
Let $\cc\in C^J$ and $\mathfrak{t}\in T_J^\prime$. Define $E_y^-(x;\mathfrak{t})\in\mathcal{P}_{\mathfrak{t}}^{(\cc)}$ \textup{(}$y\in\mathcal{O}_\cc$\textup{)} by
\begin{equation}\label{basicminus}
E_y^{J,-}(x;\mathfrak{t}):=\pi_{\cc,\mathfrak{t}}(\mathbf{1}_-)E_y^J(x;\mathfrak{t}).
\end{equation}
\end{definition}
We state below the basic formulas for $E_y^{J,-}(x;\mathfrak{t})$ but we do not provide detailed proofs, since they are similar to the proofs for $E_y^{J,+}(x;\mathfrak{t})$.

Let $\cc\in C^J$ and $\mathfrak{t}\in T_J^\prime$. Suppose that $\mathfrak{t}$ is $J$-generic. Let $t_y\in T$ be given by \eqref{ty}.
Then
\[
E_{s_iy}^{J,-}(x;\mathfrak{t})=-k_i^{\eta(\alpha_i(y))}\Bigl(\frac{k_i-k_i^{-1}t_y^{\alpha_i^\vee}}{1-t_y^{\alpha_i^\vee}}\Bigr)E_y^{J,-}(x;\mathfrak{t})\,\hbox{ for }\,
1\leq i\leq r\, \hbox{ such that }\, \alpha_i(y)>0
\]
and
\begin{equation}\label{minussigma}
E_y^{J,-}(x;\mathfrak{t})=\Bigl(\prod_{\alpha\in\Phi_0^+}\frac{1-k_\alpha^{-2}x^{-\alpha^\vee}}{1-x^{-\alpha^\vee}}\Bigr)
\sum_{v\in W_0}(-1)^{\ell(v)}\sigma_{\cc,\mathfrak{t}}(v)E_y^J(x;\mathfrak{t}).
\end{equation}
Furthermore,
\begin{equation}\label{Eminzero}
E_y^{J,-}(x;\mathfrak{t})=0\,\, \hbox{ if }\,\, y\not\in E^{\textup{reg}}\cap\mathcal{O}_\cc 
\end{equation}
and
\begin{equation}\label{Eminexp}
E_y^{J,-}(x;\mathfrak{t})=\sum_{y^\prime\in W_0y}C_y^-(y^\prime)E_{y^\prime}^J(x;\mathfrak{t})\qquad (y\in\mathcal{O}_\cc)
\end{equation}
with $C_y^-(y^\prime)\in\mathbf{F}$ satisfying, for all $y^\prime\in W_0y$,
\[
C_{s_iy}^-(y^\prime)=k_i^{\eta(\alpha_i(y))}\Bigl(\frac{k_i^{-1}-k_it_y^{-\alpha_i^\vee}}{t_y^{-\alpha_i^\vee}-1}\Bigr)C_y^-(y^\prime)\,
\hbox{ for }\, 1\leq i\leq r\, \hbox{ such that }\, \alpha_i(y)>0,
\]
$C_{y_+}^-(y^\prime)=0$ if $y_+\in\overline{E}_+\setminus E_+$
and, for $y\in E^{\textup{reg}}\cap\mathcal{O}_\cc$,
\[
C_{y_+}^-(y^\prime)=(-1)^{\ell(w_0)}\frac{\kappa_{w_0}(y_+)}{\kappa_{w_0}(0)}\prod_{\beta\in\Pi(g_{y^\prime}^{-1})}
k_\beta^{-\eta(\beta(y_-))}\Bigl(\frac{k_\beta^{-1}-k_\beta t_{y_-}^{-\beta^\vee}}{t_{y_-}^{-\beta^\vee}-1}\Bigr)\qquad (y^\prime\in W_0y).
\]

For the Whittaker limit of $E_y^{J,-}(x;\mathfrak{t})$ (see Subsection \ref{rationalsection}), it is again convenient to use the explicit expansion of $E_{y_-}^{J,-}(x;\mathfrak{t})$ in the quasi-polynomials $E_{vy_-}^J(x;\mathfrak{t})$ ($v\in W_0$) when $y_-\in E_-$.
\begin{proposition}\label{Wformulaminus}
Let $\cc\in C^J$ and suppose that $\mathfrak{t}\in T_J^\prime$ is $J$-generic. For $y\in E_-\cap\mathcal{O}_\cc$ we have
\begin{equation*}
\begin{split}
E_y^{J,-}(x;\mathfrak{t})=\kappa_{w_0}(0)^{-2}&\prod_{\alpha\in\Phi_0^+}\Bigl(\frac{1-k_\alpha^{2}t_y^{\alpha^\vee}}{1-t_y^{\alpha^\vee}}\Bigr)\\
&\times
\sum_{v\in W_0}(-1)^{\ell(v)}\kappa_v(0)\kappa_v(y)\Bigl(\prod_{\beta\in\Pi(v)}\Bigl(\frac{1-k_\beta^{-2}t_y^{\beta^\vee}}{1-t_y^{\beta^\vee}}\Bigr)\Bigr)E_{vy}^J(x;\mathfrak{t})
\end{split}
\end{equation*}
where $t_y\in T$ is defined by \eqref{ty}.
\end{proposition}

\subsection{Pseudo-unitarity and orthogonality relations}\label{unitaritySection}
Let $d\mapsto d^\ast$ be a nontrivial automorphism of $\mathbf{F}$ of order two, and write
\begin{equation*}
\begin{split}
\mathbf{F}_u^\times&:=\{d\in\mathbf{F}^\times \,\, | \,\, d^\ast=d^{-1} \},\\
\mathbf{F}_r&:=\{d\in\mathbf{F} \,\, | \,\, d^\ast=d\}.
\end{split}
\end{equation*}
for the subgroup (resp. subfield) of unitary (resp. real) elements in $\mathbf{F}$. Write $T_u:=\textup{Hom}(Q^\vee,\mathbf{F}_u^\times)$ and
$T_r:=\textup{Hom}(Q^\vee,\mathbf{F}_r^\times)$ for the associated compact and real tori, respectively.
We assume in this subsection that $q,k_a\in\mathbf{F}_u^\times$ for all $a\in\Phi$. In case a root of $q$ is required (see Subsection \ref{IndPar}), then we assume it to lie in $\mathbf{F}_u^\times$ too. Note that under these assumptions, $\mathfrak{s}_J\in T_u$ for all $J\subsetneq [0,r]$.

The formulas
\[
T_w^\ast=T_w^{-1}\quad (w\in W),\qquad (x^\mu)^*=x^{-\mu} \quad (\mu\in Q^\vee)
\]
extend the automorphism $d\mapsto d^\ast$ of $\mathbf{F}$ to a $\mathbf{F}_r$-algebra anti-involution $h\mapsto h^\ast$ of $\mathbb{H}$. It satisfies $(Y^\mu)^*=Y^{-\mu}$ for all $\mu\in Q^\vee$. The following definition is from \cite[\S 3.6]{Ch}.
\begin{definition} Let $M$ be a $\mathbb{H}$-module. 
\begin{enumerate}
\item A $\mathbf{F}_r$-bilinear map $(\cdot,\cdot): M\times M\rightarrow \mathbf{F}$ is said to be $\ast$-sesquilinear if 
\begin{enumerate}
\item $(\cdot,\cdot)$ is $\mathbf{F}$-linear in the first component,
\item $(m,m^\prime)=(m^\prime,m)^*$ for all $m,m^\prime\in M$.
\end{enumerate}
\item We say that $M$ is pseudo-unitarizable if there exists a nonzero $\ast$-sesquilinear form $(\cdot,\cdot)$ on $M$ such that 
\[
(hm,m^\prime)=(m,h^*m^\prime)\qquad \forall\,  
h\in\mathbb{H},\,\,\,\forall\, m,m^\prime\in M.
\]
\end{enumerate}
\end{definition}

Define $N^w\in\mathcal{Q}$ ($w\in W$) by 
\begin{equation}\label{norm}
N^w(x):=\prod_{a\in\Pi(w)}\Bigl(
\frac{k_a^{-1}x^{a^\vee}-k_a}{k_ax^{a^\vee}-k_a^{-1}}\Bigr).
\end{equation}
The following theorem relates to \cite[Thm. 3.6.1]{Ch} via Remark \ref{14rem}.
\begin{theorem}\label{unitaritytheorem}
Let $\cc\in C^J$ and fix a $J$-generic $\mathfrak{t}\in T_J^\prime\cap\textup{Hom}(Q^\vee,\mathbf{F}_u^\times)$. 
\begin{enumerate}
\item There exists a unique $\ast$-sesquilinear form $(\cdot,\cdot)_{J,\mathfrak{t}}$ on $\mathcal{P}_{\mathfrak{t}}^{(\cc)}$ such that
\begin{equation}\label{norms}
\bigl(P_y^J(x;\mathfrak{t}),P_{y^\prime}^J(x;\mathfrak{t})\bigr)_{J,\mathfrak{t}}
:=\delta_{y,y^\prime}N^{w_y}(\mathfrak{s}_J\mathfrak{t})
\qquad \forall\, y,y^\prime\in\mathcal{O}_\cc.
\end{equation}
\item $\mathcal{P}_{\mathfrak{t}}^{(\cc)}$ is pseudo-unitarizable. 
\item
If $(\cdot,\cdot)$ is a nonzero $\ast$-sesquilinear form on $\Pc_{\mathfrak{t}}$ realising the pseudo-unitarity of $\Pc_{\mathfrak{t}}$, then 
$(\cdot,\cdot)=d\,(\cdot,\cdot)_{J,\mathfrak{t}}$ for some $d\in\mathbf{F}^\times_r$.
\end{enumerate}
\end{theorem}
\begin{proof}
Write $t_y:=\mathfrak{s}_y\mathfrak{t}_y$ for $y\in\mathcal{O}_\cc$ (see \eqref{ty}). We have $t_y\in\textup{Hom}(Q^\vee,\mathbf{F}_u^\times)$ by the assumptions on $q,k_a$ and $\mathfrak{t}$.\\
(1) 
For $w\in W^J$ the quadratic norm $N^w(\mathfrak{s}_J\mathfrak{t})$ is well-defined since $\mathfrak{t}$ is $J$-generic (it may be zero).
Furthermore, 
$\{P_y^J(x;\mathfrak{t})\}_{y\in\mathcal{O}_\cc}$ is a $\mathbf{F}$-basis of $\Pc_{\mathfrak{t}}$. The result then follows directly from the observation that
$N^w(\mathfrak{s}_J\mathfrak{t})\in\mathbf{F}_r$ for all $w\in W^J$.\\
(2) We show that $\Pc_{\mathfrak{t}}$ is pseudo-unitary with respect to $(\cdot,\cdot)_{J,\mathfrak{t}}$. 
It suffices to show that
\[
\bigl(\pi_{\cc,\mathfrak{t}}(\delta(h))P_y^J(x;\mathfrak{t}),P_{y^\prime}^J(x;\mathfrak{t})\bigr)_{J,\mathfrak{t}}=\bigl(P_y^J(x;\mathfrak{t}),
\pi_{\cc,\mathfrak{t}}(\delta(h)^*)P_{y^\prime}^J(x;\mathfrak{t})\bigr)_{J,\mathfrak{t}}
\]
for $h=x^\mu$ ($\mu\in Q^\vee$) and $h=T_j$ ($0\leq j\leq r$).

For $h=x^\mu$ we have $\delta(x^\mu)=Y^{-\mu}$, hence
$\delta(x^{\mu})^*=Y^\mu$. Since $t_y^\mu\in\mathbf{F}_u$ and the
$P_y^J(x;\mathfrak{t})$ are orthogonal with respect to $(\cdot,\cdot)_{J,\mathfrak{t}}$, we have
\begin{equation*}
\begin{split}
\bigl(\pi_{\cc,\mathfrak{t}}(Y^{-\mu})P_y^J(x;\mathfrak{t}),P_{y^\prime}^J(x;\mathfrak{t})\bigr)_{J,\mathfrak{t}}&=
t_y^\mu\bigl(P_y^J(x;\mathfrak{t}),P_{y^\prime}^J(x;\mathfrak{t})\bigr)_{J,\mathfrak{t}}\\
&=t_{y^\prime}^\mu N^{w_y}(\mathfrak{s}_J\mathfrak{t})\delta_{y,y^\prime}=
\bigl(P_y^J(x;\mathfrak{t}),\pi_{\cc,\mathfrak{t}}(Y^\mu)P_{y^\prime}^J(x;\mathfrak{t})\bigr)_{J,\mathfrak{t}}
\end{split}
\end{equation*}
for $y,y^\prime\in\mathcal{O}_\cc$.

Finally, consider the case that $h=T_j$ ($0\leq j\leq r$). Then $\delta(T_j)^*=\delta(T_j^{-1})$. 
By Theorem \ref{17} and Corollary \ref{wyrules}
the following two statements are equivalent:
\begin{enumerate}
\item[(a)] For all $y,y^\prime\in\mathcal{O}_\cc$,
\[
\bigl(\pi_{\cc,\mathfrak{t}}(\delta(T_j))P_y^J(x;\mathfrak{t}),P_{y^\prime}^J(x;\mathfrak{t}))_{J,\mathfrak{t}}=
(P_y^J(x;\mathfrak{t}),\pi_{\cc,\mathfrak{t}}(\delta(T_j^{-1}))P_{y^\prime}^J(x;\mathfrak{t})\bigr)_{J,\mathfrak{t}}.
\]
\item[(b)] For $w\in W^J$ satisfying $s_jw\in W^J$,
\[
\Bigl(\frac{k_jt_{w\cc}^{\alpha_j^\vee}-k_j^{-1}}{t_{w\cc}^{\alpha_j^\vee}-1}\Bigr)
N^{s_jw}(\mathfrak{s}_J\mathfrak{t})=\Bigl(\frac{k_j^{-1}t_{w\cc}^{\alpha_j^\vee}-k_j}{t_{w\cc}^{\alpha_j^\vee}-1}\Bigr)
N^w(\mathfrak{s}_J\mathfrak{t}).
\]
\end{enumerate}
Statement (b) can be checked directly using
Lemma \ref{lengthadd} and \eqref{norm}, which completes the proof.\\
(3) Suppose that $(\cdot,\cdot)$ is a nonzero $\ast$-sesquilinear form on $\Pc_{\mathfrak{t}}$ which also realises the pseudo-unitarity of $\Pc_{\mathfrak{t}}$. Set
\[d:=(x^\cc,x^\cc)\in\mathbf{F}_r.
\]
We will show that $(\cdot,\cdot)=d\,(\cdot,\cdot)_{J,\mathfrak{t}}$ (and consequently $d\in\mathbf{F}_r^\times$).

By the simplicity of the $Y$-spectrum $\mathcal{S}(\mathcal{P}_{\mathfrak{t}}^{(\cc)})$ we have 
\[
(P_y^{J}(x;\mathfrak{t}),P_{y^\prime}^J(x;\mathfrak{t}))=\delta_{y,y^\prime}\bigl(P_y^J(x;\mathfrak{t}),P_y^J(x;\mathfrak{t})\bigr)\qquad (y,y^\prime\in\mathcal{O}_\cc).
\] 
We have $(S_j^Y)^*=S_j^Y$ for $j\in [0,r]$ (this can be verified by looking at its $\delta$-image and using \eqref{crossX}), hence
\[
(S_w^Y)^*=S_{w^{-1}}^Y\qquad (w\in W).
\]
For $w\in W^J$ we then obtain
\begin{equation*}
\begin{split}
\bigl(P_{w\cc}^J(x;\mathfrak{t}),P_{w\cc}^J(x;\mathfrak{t})\bigr)&=\frac{\bigl(\pi_{\cc,\mathfrak{t}}(S_{w^{-1}}^YS_w^Y)x^\cc,x^\cc\bigr)}{r_w(\mathfrak{s}_J\mathfrak{t})^*r_w(\mathfrak{s}_J\mathfrak{t})}\\
&=\frac{(x^\cc,x^\cc)\textup{n}_w(\mathfrak{s}_J\mathfrak{t})}{r_w(\mathfrak{s}_J\mathfrak{t})^*r_w(\mathfrak{s}_J\mathfrak{t})}\\
&=\bigl(x^\cc,x^\cc\bigr)N^w(\mathfrak{s}_J\mathfrak{t})=
d\,\bigl(P_{w\cc}^J(x;\mathfrak{t}),P_{w\cc}^J(x;\mathfrak{t})\bigr)_{J,\mathfrak{t}},
\end{split}
\end{equation*}
where the first equality is by definition of $P_y^J(x;\mathfrak{t})$ (see \eqref{boldbw}), the
second equality follows from \eqref{dSaction}, 
and the third equality
 from the explicit expressions \eqref{nw}, \eqref{rw} and \eqref{norm} of $\textup{n}_w(x)$, $r_w(x)$ and $N^w(x)$, respectively. 
It follows that $(\cdot,\cdot)=d\,(\cdot,\cdot)_{J,\mathfrak{t}}$, as desired.
\end{proof}
It is easy to check that
\[
\Phi^+=\bigcup_{w\in W}\Pi(w).
\]
For $J\subsetneq [0,r]$ this implies that
\[
\Phi^+\setminus\Phi_J^+=\bigcup_{(w,u)\in W^J\times W_J}u\bigl(\Pi(w)\bigr).
\]
In the present context, Theorem \ref{unitaritytheorem} gives the following sharpening of Theorem \ref{irredTHM}.
\begin{corollary}\label{correducible}
Keep the assumptions of Theorem \ref{unitaritytheorem}. Then 
\begin{enumerate}
\item $(\cdot,\cdot)_{J,\mathfrak{t}}$ is nondegenerate iff 
\[
(\mathfrak{s}_J\mathfrak{t})^{a^\vee}\not=k_a^2\qquad \forall\, a\in\bigcup_{w\in W^J}\Pi(w).
\]
\item $\mathcal{P}_{\mathfrak{t}}^{(\cc)}$ is reducible if $(\mathfrak{s}_J\mathfrak{t})^{a^\vee}=k_a^{2}$ for some $a\in\bigcup_{w\in W^J}\Pi(w)$.
\end{enumerate}
\end{corollary}
\begin{proof}
(1) By \eqref{norms} the sesquilinear form $(\cdot,\cdot)_{J,\mathfrak{t}}$ is degenerate iff $N^w(\mathfrak{s}_J\mathfrak{t})=0$ for some $w\in W^J$. In view of \eqref{norm}, this is equivalent to $(\mathfrak{s}_J\mathfrak{t})^{a^\vee}=k_a^2$ for some $a\in\Pi(w)$.\\
(2)
The radical 
\[
\textup{Rad}_{J,\mathfrak{t}}:=\{p\in\mathcal{P}\,\, | \,\, (p,\cdot)_{J,\mathfrak{t}}\equiv 0\}
\]
of $(\cdot,\cdot)_{J,\mathfrak{t}}$ is a proper $\mathbb{H}$-submodule of
$\mathcal{P}_{\mathfrak{t}}^{(\cc)}$ by Theorem \ref{unitaritytheorem}. It is nonzero iff $(\cdot,\cdot)_{J,\mathfrak{t}}$ is degenerate. The result now follows from part (1).
\end{proof}

Note that by $\Pi(w_y)=-w_y^{-1}\Pi(w_y^{-1})$, \eqref{actx}, Corollary \ref{newroots},
 and \eqref{norm}, we have the explicit expressions
 \[
 N^{w_y}(\mathfrak{s}_J\mathfrak{t})=\prod_{a\in\Phi^+: a(y)<0}
 \Bigl(\frac{k_a^{-1}t_y^{-a^\vee}-k_a}{k_at_y^{-a^\vee}-k_a^{-1}}\Bigr)\qquad (y\in\mathcal{O}_\cc)
 \]
 for the quadratic norms, where $t_y=\mathfrak{s}_y\mathfrak{t}_y\in T$.
 By Proposition \ref{leadcoeff}, the quadratic norms for the monic quasi-polynomials $E_y^J(x;\mathfrak{t})$ are
 \[
 \bigl(E_y^J(x;\mathfrak{t}),E_y^J(x;\mathfrak{t})\bigr)_{J,\mathfrak{t}}=\prod_{a\in\Phi^+: a(y)<0}\frac{(k_at_y^{-a^\vee}-k_a^{-1})(k_a^{-1}t_y^{-a^\vee}-k_a)}
 {(t_y^{-a^\vee}-1)^2}\qquad (y\in\mathcal{O}_\cc).
 \]

 \begin{remark}
 In case of Cherednik's polynomial representation $\pi=\pi_{0,1_T}$,
 there exists a concrete realization of the $\ast$-sesquilinear form $(\cdot,\cdot)_{0,1_T}: \mathcal{P}\times\mathcal{P}\rightarrow\mathbf{F}$. It is of the form 
 \[
 (p_1,p_2)_{[1,r],1_T}=\textup{ct}(p_1p_2^*\mathcal{W})\qquad (p_1,p_2\in\mathcal{P})
 \]
 with $p^\ast:=\sum_{\mu\in Q^\vee}d_\mu^\ast x^{-\mu}$ for $p=\sum_{\mu\in Q^\vee}d_\mu x^\mu\in\mathcal{P}$, with
 $\textup{ct}$ the constant term map (picking the coefficient of $x^0$ in the expansion in monomials), and with $\mathcal{W}$ an explicit weight function (see, e.g., \cite[\S 5.1]{Ma}).  Such a realisation for $(\cdot,\cdot)_{J,\mathfrak{t}}$ is not known when $(J,\mathfrak{t})\not=([1,r],1_T)$.
 \end{remark}

\subsection{The Whittaker limit}\label{rationalsection}
In this section we take
$\mathbf{F}=\mathbf{K}(q^{\frac{1}{2h}})$ for some field $\mathbf{K}$ of characteristic zero
and we assume
that ${}^{\textup{sh}}k^{\frac{1}{2}}$ and ${}^{\textup{lg}}k^{\frac{1}{2}}$ are in $\mathbf{K}$. 
Set
\[
\overline{\mathcal{P}}^{(\cc)}:=\bigoplus_{y\in\mathcal{O}_\cc}\mathbf{K}x^y\subset\mathbf{K}[E]\qquad (\cc\in\overline{C}_+).
\]
Let $\overline{\mathcal{Q}}$ be the quotient field of $\overline{\mathcal{P}}:=\overline{\mathcal{P}}^{(0)}$, and set
\[
\mathbf{K}_{\overline{\mathcal{Q}}}[E]:=\overline{\mathcal{Q}}\otimes_{\overline{\mathcal{P}}}\mathbf{K}[E].
\]
For $\cc\in\overline{C}_+$, let $\overline{\mathcal{Q}}^{(\cc)}\subset\mathbf{K}_{\overline{\mathcal{Q}}}[E]$ be the $\overline{\mathcal{Q}}$-submodule generated by
$\overline{\mathcal{P}}^{(\cc)}$. We view $\overline{\mathcal{Q}}^{(\cc)}$ as $\mathbf{K}$-submodule of $\mathcal{Q}^{(\cc)}$ in the natural way.

Denote by $\overline{H}_0\subset\overline{H}$ the finite and affine Hecke algebra over $\mathbf{K}$, respectively. They are regarded as $\mathbf{K}$-subalgebras of $H$ in the natural way.
\begin{lemma}\label{Wlemma}
\hfill
\begin{enumerate}
\item The formulas
\begin{equation}
\label{actionWhittaker}
(fx^y)\blacktriangleleft T_i:=k_i^{\chi_{\mathbb{Z}}(\alpha_i(y))}s_i(fx^{y})+(k_i-k_i^{-1})\Bigl(\frac{fx^y-s_i(f)x^{y-\lfloor \alpha_i(y)\rfloor\alpha_i^\vee}}{1-x^{\alpha_i^\vee}}\Bigr)
\end{equation}
for $1\leq i\leq r$, $f\in\overline{\mathcal{Q}}$ and $y\in E$ define a $\mathbf{K}$-linear right $\overline{H}_0$-action on $\mathbf{K}_{\overline{\mathcal{Q}}}[E]$. 
\item The $\overline{H}_0$-action \eqref{actionWhittaker} on $\mathbf{K}_{\overline{\mathcal{Q}}}[E]$ extends to a $\mathbf{K}$-linear right $\overline{H}$-action by 
\begin{equation}\label{actionWhittaker0}
(fx^y)\blacktriangleleft T_0:=(fx^{y-\varphi^\vee})\blacktriangleleft T_{s_\varphi}^{-1}
\end{equation}
for $f\in\overline{\mathcal{Q}}$ and $y\in E$. 
\item $\mathbf{K}[E]$, $\overline{\mathcal{Q}}^{(\cc)}$ and $\overline{\mathcal{P}}^{(\cc)}$ are 
$\overline{H}$-submodules of $\mathbf{K}_{\overline{\mathcal{Q}}}[E]$ for all $\cc\in\overline{C}_+$.
\item Let $\cc\in C^J$. The restriction $\pi_{\cc,\mathfrak{t}}^{X-\textup{loc}}\vert_{\delta(H)}$ of the $X$-localised quasi-polynomial representation $\pi_{\cc,\mathfrak{t}}^{X-\textup{loc}}: \mathbb{H}^{X-\textup{loc}}\rightarrow \textup{End}(\mathcal{Q}^{(\cc)})$ to $\delta(H)\subset\mathbb{H}^{X-\textup{loc}}$ does not depend on $\mathfrak{t}\in T_J$. Furthermore,
\begin{equation}\label{startpi}
f\blacktriangleleft h=\pi_{\cc,\mathfrak{t}}^{X-\textup{loc}}(\delta(h))(f)\qquad \forall\, f\in\overline{\mathcal{Q}}^{(\cc)},\, \forall\, h\in\overline{H}.
\end{equation}
\end{enumerate}
\end{lemma}
\begin{proof}
Let $\cc\in C^J$ and $\mathfrak{t}\in T_J$. For $i\in [1,r]$, $f\in\overline{\mathcal{Q}}$ and $y\in\mathcal{O}_\cc$ we have in $\overline{\mathcal{Q}}^{(\cc)}$,
\begin{equation}\label{hu1}
\begin{split}
\pi_{\cc,\mathfrak{t}}^{X-\textup{loc}}(\delta(T_i))(fx^y)&=\pi_{\cc,\mathfrak{t}}^{X-\textup{loc}}(T_i)(fx^y)\\
&=k_i^{\chi_{\mathbb{Z}}(\alpha_i(y))}s_i(fx^{y})+(k_i-k_i^{-1})\Bigl(\frac{fx^y-s_i(f)x^{y-\lfloor \alpha_i(y)\rfloor\alpha_i^\vee}}{1-x^{\alpha_i^\vee}}\Bigr)
\end{split}
\end{equation}
by \eqref{hatform}. Furthermore, we have by \eqref{delta0},
\begin{equation}\label{hu2}
\pi_{\cc,\mathfrak{t}}^{X-\textup{loc}}(\delta(T_0))(fx^y)=\pi_{\cc,\mathfrak{t}}^{X-\textup{loc}}(T_{s_\varphi}^{-1}x^{-\varphi^\vee})(fx^y)=\pi_{\cc,\mathfrak{t}}^{X-\textup{loc}}(\delta(T_{s_\varphi}^{-1}))(fx^{y-\varphi^\vee}).
\end{equation}
This shows that $\pi_{\cc,\mathfrak{t}}^{X-\textup{loc}}\vert_{\delta(H)}$ does not depend on $\mathfrak{t}\in T_J$. Hence \eqref{startpi} defines a $\mathfrak{t}$-independent $\mathbf{K}$-linear right $\overline{H}$-action $\blacktriangleleft$ on
$\overline{\mathcal{Q}}^{(\cc)}$. By the explicit formulas \eqref{hu1} and \eqref{hu2}, the $\overline{H}$-action $\blacktriangleleft$
on $\overline{\mathcal{Q}}^{(\cc)}$ is characterised by the formulas \eqref{actionWhittaker} and \eqref{actionWhittaker0}. The remaining statements now follow immediately.
\end{proof}
The adjoint torus over $\mathbf{K}$ is denoted by
\[
\overline{T}:=\textup{Hom}(Q^\vee,\mathbf{K})\subset T=\textup{Hom}(Q^\vee,\mathbf{F}).
\]
Let $J\subsetneq [0,r]$ and set
\[
\overline{T}_J^{\textup{red}}:=\{\mathfrak{t}\in\overline{T} \,\, | \,\, \mathfrak{t}^{D\alpha_j^\vee}=1\quad \forall\, j\in J \}.
\]
Choose $\lambda_J\in C^J$ such that $\lambda_J\in\frac{1}{2h}P^\vee$ (see Subsection \ref{genSection}). Recall that
$q^{\lambda_J}\in T_J$ is the character which takes the value
\[
q^{\langle \lambda_J,\alpha^\vee\rangle}=q_\alpha^{\alpha(\lambda_J)}\in\mathbb{Q}\bigl(q^{\frac{1}{2h}}\bigr)\subset\mathbf{F}
\]
at the coroot $\alpha^\vee\in\Phi_0^\vee$  (see Subsection \ref{genSection}). Note that
\begin{equation}\label{qseperated}
q^{\lambda_J}\overline{T}_J^{\textup{red}}\subseteq\{\mathfrak{t}\in T_J^\prime \,\, | \,\, \mathfrak{t}\, \hbox{ is }\, J\textup{-generic}\}.
\end{equation}

Let $\mathbf{F}_{\textup{reg}}\subset\mathbf{F}$ be the subring consisting of the elements $d\in\mathbf{F}=\mathbf{K}(q^{\frac{1}{2h}})$
which are regular at $q^{\,-\frac{1}{2h}}=0$. Specialisation at $q^{-\frac{1}{2h}}=0$ defines a ring homomorphism $\mathbf{F}_{\textup{reg}}\rightarrow\mathbf{K}$, $d\mapsto \overline{d}$ which we extend to a ring homomorphism
\[
\mathbf{F}_{\textup{reg}}[E]\rightarrow\mathbf{K}[E],\qquad f\mapsto\overline{f}:=\sum_{y\in E}\overline{d}_yx^y\quad (f=\sum_{y\in E}d_yx^y\in\mathbf{F}_{\textup{reg}}[E]).
\]
We call $\overline{f}\in\mathbf{K}[E]$ the {\it Whittaker limit} of $f\in\mathbf{F}_{\textup{reg}}[E]$. Write 
\[
\mathcal{P}_{\textup{reg}}^{(\cc)}:=\bigoplus_{y\in\mathcal{O}_\cc}\mathbf{F}_{\textup{reg}}x^y.
\]
\begin{proposition}\label{prop:qlimit}
Let $\cc\in C^J$, $\mathfrak{t}\in\overline{T}_J^{\textup{red}}$ and $y\in\mathcal{O}_\cc$. Then
\begin{enumerate}
\item $E_y^J(x;q^{\lambda_J}\mathfrak{t})\in\mathcal{P}_{\textup{reg}}^{(\cc)}$.
\item[\textup{(2)}] The Whittaker limit $\overline{E}_y^J(x)\in\overline{\mathcal{P}}^{(\cc)}$ of $E_y^J(x;q^{\lambda_J}\mathfrak{t})$
is independent of $\lambda_J$ and $\mathfrak{t}$.  
\item[\textup{(3)}] For $j\in [0,r]$ we have
\begin{equation}
\label{qlimit_inv_intertwine}
\overline{E}_{s_jy}^J(x)\blacktriangleleft T_j=
\begin{cases}
\kappa_{Ds_j}(y)^{-1}\overline{E}_y^J(x)\quad &\hbox{ if }\, \alpha_j(y)<0,\\
k_j\overline{E}_y^J(x)\quad &\hbox{ if }\, \alpha_j(y)=0.
\end{cases}
\end{equation}
\end{enumerate}
\end{proposition}
\begin{remark}
Note that $\kappa_{Ds_j}(y)\not=k_j^{-1}$ when $\alpha_j(y)=0$, so it is necessary to distinguish the two cases in part (3) of the proposition (an explicit formula for $\overline{E}_{s_jy}^J(x)\blacktriangleleft T_j$ when $\alpha_j(y)>0$ follows from \eqref{qlimit_inv_intertwine} and the Hecke relation \eqref{Hr} for $T_j$).
\end{remark}
\begin{proof}
For $y\in\mathcal{O}_\cc$ write $t_y:=w_y(q^{\lambda_J}\mathfrak{s}_J\mathfrak{t})\in T$ (see \eqref{qseperated} and \eqref{ty}). Then 
\begin{equation}\label{steepp}
t_y^{\alpha_j^\vee}=q_{\alpha_j}^{\alpha_j(w_y\lambda_J)}(\mathfrak{s}_y\mathfrak{t}_y)^{D\alpha_j^\vee}
\end{equation}
with $q_{\alpha_j}^{\alpha_j(w_y\lambda_j)}\in\mathbb{Q}\bigl(q^{\frac{1}{2h}}\bigr)$ and $(\mathfrak{s}_y\mathfrak{t}_y)^{D\alpha_j^\vee}\in\mathbf{K}$, and
Corollary \ref{Whittakerprop} gives
\begin{equation}\label{step}
\begin{split}
\pi_{\cc,q^{\lambda_J}\mathfrak{t}}(\delta(T_j))E_{s_jy}^J(x;q^{\lambda_J}\mathfrak{t})+
&\left(\frac{k_j-k_j^{-1}}{t_{y}^{-\alpha_j^\vee}-1}\right)E_{s_jy}^J(x;q^{\lambda_J}\mathfrak{t})=\\
&\quad=
\kappa_{Ds_j}(y)^{-1}E_y^J(x;q^{\lambda_J}\mathfrak{t})\quad \hbox{ if }\, \alpha_j(y)<0
\end{split}
\end{equation}
(we used here that $\kappa_{Ds_i}(s_iy)=\kappa_{Ds_i}(y)^{-1}$ by \eqref{kjcombi}, and $t_{s_jy}^{\alpha_j^\vee}=t_y^{-\alpha_j^\vee}$). We now use \eqref{step} to prove
(1) and (2) by induction to $\ell(w_y)$.

If $\ell(w_y)=0$ then $y=\cc$ and $E_\cc^J(x;q^{\lambda_J}\mathfrak{t})=x^\cc\in\mathcal{P}_{\textup{reg}}^{(\cc)}$. Furthermore, $\overline{E}_\cc^J(x)=x^\cc$.

Let $y\in\mathcal{O}_\cc$ with $\ell(w_y)>0$ and suppose that for all $y^\prime\in\mathcal{O}_\cc$ with $\ell(w_{y^\prime})<\ell(w_y)$ we have $E_{y^\prime}^J(x;q^{\lambda_J}\mathfrak{t})\in\mathcal{P}_{\textup{reg}}^{(\cc)}$ with 
$\overline{E}_{y^\prime}^J(x)$ independent of $\lambda_J$ and $\mathfrak{t}$. Take $0\leq j\leq r$ such that
$\ell(s_jw_y)=\ell(w_y)-1$. By Corollary \ref{wyrules} we have $\alpha_j(y)<0$ and $s_jw_y=w_{s_jy}$, in particular $\ell(w_{s_jy})<\ell(w_y)$. Since both $\cc$ and $\lambda_J$ lie in $C^J$ and $\alpha_j(y)=\alpha_j(w_y\cc)<0$, we also have $\alpha_j(w_y\lambda_J)<0$ by Corollary \ref{wyrules}(1). Taking also \eqref{steepp} into account we conclude that the left hand side of \eqref{step} lies in $\mathcal{P}_{\textup{reg}}^{(\cc)}$.  Hence $E_y^J(x;q^{\lambda_J}\mathfrak{t})\in\mathcal{P}_{\textup{reg}}^{(\cc)}$, and 
the Whittaker limit of \eqref{step} gives
\[
\overline{E}_{y}^J(x)=\kappa_{Ds_j}(y)\overline{E}_{s_jy}^J(x)\blacktriangleleft T_j.
\]
The right hand side does not depend on $\lambda_J$ and $\mathfrak{t}$ by the induction hypothesis. This completes the proof of the induction step.

The Whittaker limit of \eqref{step} 
yields the first case of \eqref{qlimit_inv_intertwine}. The second case  is immediate from Corollary \ref{Whittakerprop}.
\end{proof}

\begin{corollary}\label{cor:qlim_wy}
Let $\cc\in C^J$ and $y\in\mathcal{O}_\cc$. Then
\begin{equation*}
\overline{E}_y^J(x)=\frac{k(\cc)}{k(y)}\,x^\cc\blacktriangleleft T_{w_y^{-1}}.
\end{equation*}
\end{corollary}
\begin{proof}
This follows from Proposition \ref{prop:qlimit}, 
Corollary \ref{wyrules} and \eqref{kjcombi} by induction to $\ell(w_y)$.
\end{proof}
In Theorem \ref{qlim_formula} we will give an explicit formula for $\overline{E}_{y}^J(x)$ only involving the $\overline{H}_0$-action on $\mathbf{K}[E]$.
We need for this the following preparatory lemma.
\begin{lemma} \label{lem:dominant_factorization}
If $\cc\in\overline{C}_+$ and $y \in \mathcal{O}_{\cc}\cap \overline{E}_-$, then $y = v\cc + \mu$ for some $v \in W_0$ and $\mu \in Q^{\vee}\cap\overline{E}_-$.
\end{lemma}
\begin{proof}
We may assume without loss of generality that $E=E^\prime$ (i.e., $\Phi_0$ spans $E$).

First suppose that $y \in\mathcal{O}_\cc\cap E_-$.  Write $y=v\cc+\mu$ with $v\in W_0$ and $\mu \in Q^{\vee}$. 
For $\alpha \in \Phi_0^{+}$ we have $\alpha(\mu) =  \alpha(y)- \alpha(v\cc) < 1$, since $\alpha(y)<0$ and $|\alpha(v\cc)| \leq 1$.  Since $\mu \in Q^\vee$, it follows that $\alpha(\mu) \leq 0$, so $\mu \in Q^{\vee}\cap\overline{E}_-$.

The general case follows by continuity.  Indeed, let $y \in \mathcal{O}_{\cc}\cap\overline{E}_-$.  Let $\{y_\ell\}_{\ell=1}^{\infty}$ be a sequence in $E_-$ converging to $y$.
  Denote
$\cc_{y_\ell}\in\overline{C}_+$ by $\cc_\ell$.
By the first paragraph, we have $y_\ell = v_\ell \cc_\ell + \mu_\ell$ for some $v_\ell \in W_0$ and $\mu_\ell \in Q^{\vee}\cap\overline{E}_-$.  Since $W_0$ is finite, we may assume by passing to a subsequence that $v_\ell = v\in W_0$ is constant.  Since $\overline{C}_{+}$ is compact (by the assumption that $E=E^\prime$), we may assume by passing to a subsequence that $\cc_\ell$ converges to some $\cc^\prime \in \overline{C}_{+}$.  Since $Q^{\vee}\cap\overline{E}_-$ is discrete, it then follows that $\mu_\ell=\mu$ is constant for $\ell$ sufficiently large.  We then have $y = v\cc^\prime + \mu$, so $\cc^\prime = \cc$. The result now follows, since $\mu\in Q^\vee\cap\overline{E}_-$.	
\end{proof}
Recall that $g_y\in W_0$ denotes the unique element of minimal length such that $y_-=g_yy$.
\begin{theorem}\label{qlim_formula}
For $\cc\in C^J$ and $y\in\mathcal{O}_\cc$ we have
\begin{equation}\label{fff}
\overline{E}_y^J(x)=\frac{k(y_-)}{k(y)}\,x^{y_-}\blacktriangleleft T_{g_y^{-1}}^{-1}.
\end{equation}
\end{theorem}
\begin{proof}
First we prove the theorem for $y\in\mathcal{O}_\cc\cap\overline{E}_-$. 

In this case the right hand side of \eqref{fff} is simply $x^y$.
Since $\overline{E}_y^J(x)=x^y+\textup{l.o.t.}$ it suffices to show that $\overline{E}_y^J(x)\sim x^y$, where we write
$p\sim p^\prime$ if $p$ and $p^\prime$ in $\overline{\mathcal{P}}^{(\cc)}$ differ by a nonzero constant multiple. 

Write $y\in\mathcal{O}_\cc\cap\overline{E}_-$ as $y=v\cc+\mu$ with $v\in W_0$ and $\mu\in Q^{\vee}\cap\overline{E}_-$, see Lemma \ref{lem:dominant_factorization}. Set $w:=\tau(\mu)v\in W$, then $w=w_yh_y$ for a unique $h_y\in W_J$ and 
\[
\ell(v)+\ell(\tau(\mu))=\ell(w)=\ell(w_y)+\ell(h_y),
\]
where the first equality is due to \eqref{fundlength}. It follows that 
\begin{equation}\label{dT}
\delta(T_{w^{-1}})=\delta(T_{v^{-1}}Y^{-\mu})=x^\mu T_v
\end{equation} 
and $T_{w_y^{-1}}=T_{h_y^{-1}}^{-1}T_{w^{-1}}$. Since $x^\cc\blacktriangleleft T_{h_y^{-1}}^{-1}\sim x^\cc$ by Proposition \ref{prop:qlimit}(3), we get from Corollary \ref{cor:qlim_wy} that
\begin{equation*}
\overline{E}_y^J(x)\sim x^\cc\blacktriangleleft T_{w_y^{-1}}\sim x^\cc\blacktriangleleft T_{w^{-1}}.
\end{equation*}
By Lemma \ref{Wlemma} and \eqref{dT} we then have for $\mathfrak{t}\in T_J$,
\begin{equation*}
\overline{E}_y^J(x)\sim\pi_{\cc,\mathfrak{t}}(\delta(T_{w^{-1}}))x^\cc=x^\mu\pi_{\cc,\mathfrak{t}}(T_v)x^\cc\sim x^{\mu+v\cc}=x^y,
\end{equation*}
where the final step is due to Corollary \ref{monomialcor}. Hence $\overline{E}_y^J(x)=x^y$, which completes the proof of \eqref{fff} when $y\in\mathcal{O}_\cc\cap\overline{E}_-$.

Suppose now that $y\in\mathcal{O}_\cc$ is arbitrary. Take a reduced expression 
$g_y^{-1}=s_{i_1}\cdots s_{i_\ell}$. Since $g_y^{-1}\in W_0^{y_-}$, formula \eqref{rootsdescription} implies that
\[
\alpha_{i_m}(s_{i_{m+1}}\cdots s_{i_\ell}y_-)<0
\]
for $1\leq m\leq\ell$.
Hence by \eqref{kjcombi} and by repeated application of Proposition \ref{prop:qlimit}(3),
\[
\overline{E}^J_y(x)=\overline{E}^J_{g_y^{-1}y_-}(x)=
\frac{k(y_-)}{k(y)}\overline{E}^J_{y_-}(x)\blacktriangleleft T_{g_y^{-1}}^{-1}.
\]
The result now follows since $\overline{E}_{y^-}^J(x)=x^{y_-}$ by the previous paragraph.
\end{proof}
Let $\cc\in C^J$, $\mathfrak{t}\in\overline{T}^{\textup{red}}_J$ and $y\in\mathcal{O}_\cc$. By Lemma \ref{Wlemma} and Proposition \ref{prop:qlimit}, both $E_y^{J,+}(x;q^{\lambda_J}\mathfrak{t})$ and $E_y^{J,-}(x;q^{\lambda_J}\mathfrak{t})$ lie in $\mathbf{F}_{\textup{reg}}[E]$, and their Whittaker limits $\overline{E}_y^{J,+}(x)$ and $\overline{E}_y^{J,-}(x)$ do not depend on $\lambda_J$ and $\mathfrak{t}$. We now take a close look at $\overline{E}_y^{J,-}(x)$, since they are closely related to spherical metaplectic Whittaker functions (see Section \ref{MfinalSection}). We list the main formulas for $\overline{E}_y^{J,+}(x)$ at the end of the subsection. 

First note that 
\[
\overline{E}_y^{J,-}(x)=0\quad \hbox{ if }\, y\not\in E^{\textup{reg}}\cap\mathcal{O}_\cc
\]
by \eqref{Eminzero}. Furthermore, $\overline{E}_{vy}^{J,-}(x)$ ($v\in W_0$) is a nonzero constant multiple of $\overline{E}_y^{J,-}(x)$ when $y\in E^{\textup{reg}}\cap\mathcal{O}_\cc$.
\begin{corollary}\label{linkmWcor}
Suppose $\cc\in C^J$ and $y\in E_-\cap\mathcal{O}_\cc$. Then 
\begin{equation}\label{linkmW1}
\overline{E}_y^{J,-}(x)=\kappa_{w_0}(0)^{-2}\sum_{v\in W_0}(-1)^{\ell(v)}\kappa_v(0)\kappa_v(y)\overline{E}_{vy}^J(x).
\end{equation}
Furthermore, 
\begin{equation}\label{linkmW2}
\kappa_v(y)\overline{E}_{vy}^J(x)=x^y\blacktriangleleft T_v^{-1}\qquad\quad(v\in W_0).
\end{equation}
\end{corollary}
\begin{proof}
We first prove \eqref{linkmW2}. For $v\in W_0$ we have $g_{vy}=v^{-1}$ since $y\in E_-$, and hence by \eqref{fff},
\[
\overline{E}_{vy}^J(x)=\frac{k(y)}{k(vy)}x^y\blacktriangleleft T_v^{-1}.
\]
Then use that $k(vy)/k(y)=k_v(y)=\kappa_v(y)$, since $y\in E^{\textup{reg}}$.

There are two ways one can arrive at \eqref{linkmW1}. First, one can take the Whittaker limit of \eqref{basicminus} after fixing $\mathfrak{t}:=q^{\lambda_J}\mathfrak{t}^\prime\in T_J^\prime$ with $\mathfrak{t}^\prime\in\overline{T}_J^{\textup{red}}$. By Lemma \ref{Wlemma}, Proposition \ref{prop:qlimit} and \eqref{fff} it then 
follows that 
\begin{equation}\label{linkmW3}
\overline{E}_y^{J,-}(x)=\kappa_{w_0}(0)^{-2}\sum_{v\in W_0}(-1)^{\ell(v)}\kappa_v(0)\,x^y\blacktriangleleft T_v^{-1},
\end{equation}
and \eqref{linkmW1} follows from \eqref{linkmW2}. Second, one can take the Whittaker limit of the formula in Proposition \ref{Wformulaminus}
after fixing $\mathfrak{t}:=q^{\lambda_J}\mathfrak{t}^\prime\in T_J^\prime$ with $\mathfrak{t}^\prime\in\overline{T}_J^{\textup{red}}$. To perform the specialisation in this case, one notes that for $\beta\in\Phi_0^+$, 
\[
t_y^{\beta^\vee}=q_\beta^{\beta(w_y\lambda_J)}\bigl(\mathfrak{s}_y\mathfrak{t}_y\bigr)^{\beta^\vee}
\]
and the $q$-power is strictly negative, since $y=w_y\cc\in E_-$ implies $w_y\lambda_J\in E_-$. The Whittaker limit of the formula in Proposition \ref{Wformulaminus} then
gives \eqref{linkmW1} (compare with the proof of Corollary \ref{corplus}).
\end{proof}
\begin{lemma}\label{dotWzeroaction}
The formulas
\begin{equation*}
\begin{split}
s_i\blacktriangleright x^y&:=\frac{k_i^{\chi_{\mathbb{Z}}(\alpha_i(y))}(x^{\alpha_i^\vee}-1)}{(k_ix^{\alpha_i^\vee}-k_i^{-1})}x^{s_iy}+\frac{(k_i-k_i^{-1})}{(k_ix^{\alpha_i^\vee}-k_i^{-1})}
x^{y-\lfloor\alpha_i(y)\rfloor\alpha_i^\vee},\\
f\blacktriangleright x^y&:=fx^y
\end{split}
\end{equation*}
for $i\in [1,r]$, $f\in\overline{\mathcal{Q}}$ and $y\in\mathcal{O}_\cc$ define a left $W_0\ltimes\overline{\mathcal{Q}}$-action on $\mathbf{K}_{\overline{\mathcal{Q}}}[E]$. Furthermore, if $\cc\in C^J$ then $\sigma_{\cc,\mathfrak{t}}\vert_{W_0\ltimes\overline{\mathcal{Q}}}$ does not depend on $\mathfrak{t}\in T_J$, and
\begin{equation}\label{startsigma}
v\blacktriangleright (fx^y)=\sigma_{\cc,\mathfrak{t}}(v)(fx^y)\qquad\quad (v\in W_0,\, f\in\overline{\mathcal{Q}},\, y\in\mathcal{O}_\cc).
\end{equation}
\end{lemma}
\begin{proof}
It suffices to note that \eqref{startsigma} is correct for $v=s_i$ ($1\leq i\leq r$) and $y\in\mathcal{O}_\cc$. This follows immediately from the definition of $\sigma_{\cc,\mathfrak{t}}$, see Theorem \ref{aCG}.
\end{proof}

In the following lemma we describe the right $\overline{H}_0$-action on $\mathbf{K}_{\overline{\mathcal{Q}}}[E]$ in terms of Demazure-Lusztig type operators.
\begin{lemma}\label{lemDKtransfer}
For $i\in [1,r]$ and $f\in\mathbf{K}_{\overline{\mathcal{Q}}}[E]$ we have
\begin{equation}\label{DLformspec}
f\blacktriangleleft T_i=k_if+k_i^{-1}\Bigl(\frac{1-k_i^2x^{\alpha_i^\vee}}{1-x^{\alpha_i^\vee}}\Bigr)\bigl(s_i\blacktriangleright f-f\bigr).
\end{equation}
\end{lemma}
\begin{proof}
This follows from \eqref{startpi}, \eqref{startsigma} and Remark \ref{DLform}.
\end{proof}
\begin{proposition}
Let $\cc\in C^J$ and $y\in\mathcal{O}_\cc$. Then 
\[
\overline{E}_y^{J,-}(x)=\Bigl(\prod_{\alpha\in\Phi_0^+}\frac{1-k_\alpha^{-2}x^{-\alpha^\vee}}{1-x^{-\alpha^\vee}}\Bigr)\sum_{v\in W_0}(-1)^{\ell(v)}v\blacktriangleright\overline{E}_y^J(x).
\]
In particular, 
\begin{equation}\label{CStype}
\overline{E}_y^{J,-}(x)=\Bigl(\prod_{\alpha\in\Phi_0^+}\frac{1-k_\alpha^{-2}x^{-\alpha^\vee}}{1-x^{-\alpha^\vee}}\Bigr)\sum_{v\in W_0}(-1)^{\ell(v)}v\blacktriangleright x^y\qquad (y\in E_-\cap\mathcal{O}_\cc).
\end{equation}
\end{proposition}
\begin{proof}
The first formula follows by taking the Whittaker limit of \eqref{minussigma}
after fixing $\mathfrak{t}=q^{\lambda_J}\mathfrak{t}^\prime\in T_J^\prime$ with $\mathfrak{t}^\prime\in\overline{T}_J^{\textup{red}}$. Then \eqref{CStype} follows from \eqref{fff}.
\end{proof}
We now give the main formulas for $\overline{E}_y^{J,+}(x)$. Let $\cc\in C^J$ and $y\in\mathcal{O}_\cc$. Then $E_{vy}^{J,+}(x)$ ($v\in W_0$) is a nonzero scalar multiple of $E_y^{J,+}(x)$, and 
\[
\overline{E}_y^{J,+}(x)=\kappa_{w_0}(0)^2\sum_{v\in W_0}\kappa_v(0)^{-1}\overline{E}_y^J(x)\blacktriangleleft T_v^{-1}.
\]
In particular, for $y\in\overline{E}_-\cap\mathcal{O}_\cc$ it follows from \eqref{qlimit_inv_intertwine} and \eqref{linkmW2} that
\begin{equation}\label{ChintaGunnellsplus}
\begin{split}
\overline{E}_y^{J,+}(x)&=\kappa_{w_0}(0)^2\Bigl(\sum_{v\in W_{0,y}}\kappa_v(0)^{-2}\Bigr)\sum_{v\in W_0^y}\kappa_v(0)^{-1}x^y\blacktriangleleft T_v^{-1}\\
&=\kappa_{w_0}(0)^2\Bigl(\sum_{v\in W_{0,y}}\kappa_v(0)^{-2}\Bigr)\sum_{v\in W_0^y}\frac{\kappa_v(y)}{\kappa_v(0)}\overline{E}_{vy}^J(x).
\end{split}
\end{equation}
The latter formula can also be directly obtained by taking the Whittaker limit of the formula in Corollary \ref{corplus} after taking $\mathfrak{t}=q^{\lambda_J}\mathfrak{t}^\prime$
with $\mathfrak{t}^\prime\in\overline{T}^{\textup{red}}_J$. Then one uses the fact that $\beta(w_y\lambda_J)<0$ for $\beta\in\Pi(v)$ ($v\in W_0^y$) since $\beta(y)<0$, and that
$\sum_{v\in W_{0,y}}\kappa_{g_{y_+}^{-1}v}(0)^2=\kappa_{w_0}(0)^2\sum_{v\in W_{0,y}}\kappa_v(0)^{-2}$.

Taking the Whittaker limit of the formula in Proposition \ref{Namaraprop} we obtain the expression
\[
\overline{E}_y^{J,+}(x)=\Bigl(\prod_{\alpha\in\Phi_0^+}\frac{1}{1-x^{-\alpha^\vee}}\Bigr)\sum_{v\in W_0}\kappa_v(0)^2\Bigl(\prod_{\alpha\in\Phi_0^+}\bigl(1-k_\alpha^{2\chi(v^{-1}\alpha)}x^{-\alpha^\vee}\bigr)\Bigr)\,
v\blacktriangleright\overline{E}_y^J(x).
\]
In particular, for $y\in\overline{E}_-\cap\mathcal{O}_\cc$,
\begin{equation}\label{CasselmannShalikaplus}
\begin{split}
\overline{E}_y^{J,+}(x)&=\Bigl(\prod_{\alpha\in\Phi_0^+}\frac{1}{1-x^{-\alpha^\vee}}\Bigr)
\sum_{v\in W_0}\kappa_v(0)^2\Bigl(\prod_{\alpha\in\Phi_0^+}\bigl(1-k_\alpha^{2\chi(v^{-1}\alpha)}x^{-\alpha^\vee}\bigr)\Bigr)\,v\blacktriangleright x^y\\
&=\sum_{v\in W_0}\Bigl(\prod_{\alpha\in\Phi_0^+}\Bigl(\frac{1-k_\alpha^2x^{-v\alpha^\vee}}{1-x^{-v\alpha^\vee}}\Bigr)\Bigr)\,v\blacktriangleright x^y,
\end{split}
\end{equation}
cf. the proof of Proposition \ref{Namaraprop} for the second equality.
\begin{remark}\label{HLcase}
For $\mu\in P^\vee$ and $v\in W_0$ we have $v\blacktriangleright x^\mu=\sigma_{\cc,\mathfrak{t}}(v)x^\mu=x^{v\mu}$ by \eqref{startsigma} and Remark 
\ref{linkCGlocalrem}(1). Hence the right hand side of 
\eqref{CasselmannShalikaplus} for $y=\mu\in P^{\vee,-}\cap\mathcal{O}_\cc$ gives the familiar explicit expression of Macdonald's spherical function on a $\mathfrak{p}$-adic Lie group \cite{Macpadic} (which is the Hall-Littlewood polynomial in type $A$). Since $E_\mu^{[1,r],+}(x)$ ($\mu\in\overline{E}_-\cap Q^\vee$) is a symmetric Macdonald polynomial by Remark \ref{remsymMac}, and the monic symmetric Macdonald polynomials are invariant under the simultaneous inversion of the parameters $q$ and $k_\alpha$ ($\alpha\in\Phi_0$), this is in agreement with the interpretation of Macdonald's spherical function as the $q\rightarrow 0$ limit of the symmetric Macdonald polynomial, see \cite[\S 10]{Mac}.
\end{remark}
For $\mu\in P^{\vee,-}$ write 
\[
s_\mu(x)=\Bigl(\prod_{\alpha\in\Phi_0^+}\frac{1}{1-x^{-\alpha^\vee}}\Bigr)\sum_{v\in W_0}(-1)^{\ell(v)}x^{-\rho^\vee+v(\mu+\rho^\vee)}=
\sum_{v\in W_0}\prod_{\alpha\in\Phi_0^+}\Bigl(\frac{x^{v\mu}}{1-x^{-v\alpha^\vee}}\Bigr),
\]
which is Weyl's formula for the character of the associated finite dimensional irreducible representation 
 of the complex simply connected semisimple Lie group with underlying root system $\Phi_0^\vee$ (the highest weight $\mu$ is relative to $\Phi_0^{\vee,-}$). It is the Schur function for type $A$.

Recall that $\xi_y$ is the unique element in $y+P^\vee$ such that $0\leq \alpha_i(\xi_y)<1$ for $i\in [1,r]$ (see Remark \ref{k0limit}). 
\begin{proposition}
Let $\cc\in C^J$ and $y\in\overline{E}_-\cap\mathcal{O}_\cc$. For the quasi-polynomial $\overline{E}_y^{J,+}(x)=\overline{E}_y^{J,+}(x;\mathbf{k})$ in the limiting case $\mathbf{k}=\mathbf{0}$ we have
\[
\overline{E}_y^{J,+}(x;\mathbf{0})=x^{\xi_y}s_{y-\xi_y}(x).
\]
\end{proposition}
\begin{proof}
Note that $y-\xi_y\in P^{\vee,-}$ since $y\in\overline{E}_-$. By \eqref{startsigma} and Remark \ref{k0limit} we have for $v\in W_0$,
\[
v\blacktriangleright_{\mathbf{k}=\mathbf{0}}x^y=\sigma_{\cc,\mathfrak{t}}^{\mathbf{k}=\mathbf{0}}(v)x^y=x^{\xi_y+v(y-\xi_y)},
\]
and hence the result follows from \eqref{CasselmannShalikaplus}.
\end{proof}

We will see in Section \ref{MfinalSection} that the $\overline{E}_y^J(x)$ and $\overline{E}_y^{J,-}(x)$ are, after some elementary modifications, the metaplectic Iwahori-Whittaker and metaplectic spherical Whittaker functions from \cite{PP} and \cite{CGP}, respectively. Hence $E_{y}^{J}(x;\mathfrak{t})$ and $E_{y}^{J,-}(x;\mathfrak{t})$ are $q$-analogues of metaplectic Whittaker functions, depending on the additional parameters $\mathfrak{t}\in T_J$.  The results in this subsection then relates to the metaplectic theory
as follows:
\begin{enumerate}
\item Lemma \ref{dotWzeroaction} corresponds to Chinta's and Gunnells' \cite{CG07,CG} Weyl group action on rational functions, used to construct the local parts of Weyl group multiple Dirichlet series (see also \cite[\S 3.3]{SSV}).
\item Formula \eqref{linkmW2} corresponds via \eqref{DLformspec} with Patnaik's and Puskas' \cite[Cor. 5.4]{PP} expression of the metaplectic Iwahori-Whittaker function in terms of metaplectic Demazure-Lusztig operators.
\item Formula \eqref{CStype} corresponds to McNamara's \cite[Thm. 15.2]{McN} Casselman-Shalika type formula for the metaplectic spherical Whittaker functions.
 \item Formula \eqref{linkmW3} corresponds Chinta's, Gunnells', and Puskas' \cite[Thm. 16]{CGP} expression of the metaplectic spherical Whittaker function in terms of metaplectic Demazure-Lusztig operators.
 \end{enumerate}
 In Section \ref{MfinalSection} we give further details about these connections with metaplectic representation theory and metaplectic Whittaker functions.

\section{Quasi-polynomial representations of extended double affine Hecke algebras}\label{Extsection}
In this section we generalise the results of the previous sections to extended root data. We assume throughout this section that $J$ is a proper subset of $[0,r]$.

\subsection{Extended root datum and extended affine Weyl groups}\label{S70}
 For $e\in\mathbb{Z}_{>0}$ let $(\mathcal{L}^{\times 2})_e$ be the set of pairs $(\Lambda,\Lambda^\prime)$ with 
$\Lambda,\Lambda^\prime\in\mathcal{L}$  (see \eqref{latticecond}) satisfying the additional requirement that
\begin{equation}\label{latticecompatibility}
\langle \Lambda,\Lambda^\prime\rangle\subseteq e^{-1}\mathbb{Z}.
\end{equation}
Our choice of normalisation of the root system $\Phi_0$ implies that $Q^\vee\subseteq Q$, hence $(Q^\vee,\Lambda^\prime)\in (\mathcal{L}^{\times 2})_e$
for all $\Lambda^\prime\in\mathcal{L}$.

Furthermore, if either $\Lambda$ or $\Lambda^\prime$ is contained in $E^\prime$ then $(\Lambda,\Lambda^\prime)\in (\mathcal{L}^{\times 2})_{e}$ when $h\vert e$. Indeed, 
if either $\Lambda$ or $\Lambda^\prime$ is contained in $E^\prime$ then $\langle\Lambda,\Lambda^\prime\rangle\subseteq \langle P^\vee,P^\vee\rangle$, and
$\langle P^\vee,P^\vee\rangle\subseteq h^{-1}\mathbb{Z}$ since $Q^\vee\subseteq Q$ and $P^\vee\subseteq \frac{1}{h}Q^\vee$ (see Proposition \ref{latticecompatibilityprop}).

Note that if $(\Lambda_2,\Lambda_2^\prime)\in (\mathcal{L}^{\times 2})_e$ and if $\Lambda_1,\Lambda_1^\prime\in\mathcal{L}$ are such that $\Lambda_1\subseteq\Lambda_2$ and $\Lambda_1^\prime\subseteq\Lambda_2^\prime$, then $(\Lambda_1,\Lambda_1^\prime)\in (\mathcal{L}^{\times 2})_e$. 
\begin{definition}\label{defpartorder}
For $(\Lambda_1,\Lambda_1^\prime),(\Lambda_2,\Lambda_2^\prime)\in (\mathcal{L}^{\times 2})_e$ we write 
\[
(\Lambda_1,\Lambda_1^\prime)\leq (\Lambda_2,\Lambda_2^\prime)
\]
if $\Lambda_1\subseteq\Lambda_2$ and $\Lambda_1^\prime\subseteq\Lambda_2^\prime$.
\end{definition}
Definition \ref{defpartorder} turns $(\mathcal{L}^{\times 2})_e$ into a partially ordered set with least element $(Q^\vee,Q^\vee)$.

The lattice $\Lambda^\prime\subset E$ is $W_0$-invariant. The resulting group
\[
W_{\Lambda^\prime}:=W_0\ltimes \Lambda^\prime
\]
is called the extended affine Weyl group. It contains $W$ as normal subgroup, and $W_{\Lambda^\prime}/W\simeq
\Lambda^\prime/Q^\vee$. We denote its elements by $v\tau(\lambda)$ ($v\in W_0$, $\lambda\in\Lambda^\prime$), and write $(w,y)\mapsto wy$ ($w\in W_{\Lambda^\prime}$, $y\in E$) for the $W_{\Lambda^\prime}$-action on $E$ by reflections and translations.
Extend the group epimorphism $D: W\twoheadrightarrow W_0$ to a group epimorphism $D: W_{\Lambda^\prime}\twoheadrightarrow W_0$ by $D(v\tau(\lambda)):=v$ for $v\in W_0$ and $\lambda\in\Lambda^\prime$.

The contragredient $W_{\Lambda^\prime}$-action provides a linear action on the space $E^*\oplus\mathbb{R}$
of affine linear functionals on $E$. The action is explicitly given by \eqref{linearaction}. 
The condition that $\alpha(\Lambda^\prime)\subseteq\mathbb{Z}$ for all $\alpha\in\Phi_0$ implies that
the affine root system $\Phi\subset E^*\oplus\mathbb{R}$ is $W_{\Lambda^\prime}$-stable. We extend the length function $\ell: W\rightarrow \mathbb{Z}_{\geq 0}$ to $W_{\Lambda^\prime}$ using formula \eqref{length}. It satisfies \eqref{fundlength} for $\mu\in\Lambda^\prime$ and $v\in W_0$.

In Definition \ref{wydef} we introduced, for $y\in E$, the element $w_y\in W$ as the unique element of minimal length such that $w_y\cc_y=y$, as well as the finite Weyl group element $v_y:=w_y^{-1}\tau(\mu_y)\in W_0$, where $\mu_y=w_y(0)$. For $y=\lambda\in Q^\vee$ one has $\mu_\lambda=\lambda$, hence
\begin{equation}\label{decompstandard}
\tau(\lambda)=w_\lambda v_\lambda \qquad\quad (\lambda\in Q^\vee).
\end{equation}
We now generalise the factorisation \eqref{decompstandard} to $\lambda\in\Lambda^\prime$ as follows.

Let $\lambda\in\Lambda^\prime$. By, e.g., \cite[(2.4.5)]{Ma}, there exists a unique 
$w_{\lambda,\Lambda^\prime}\in\tau(\lambda)W_0\subset W_{\Lambda^\prime}$ of minimal length. The resulting finite Weyl group element
\[
v_{\lambda,\Lambda^\prime}:=w_{\lambda,\Lambda^\prime}^{-1}\tau(\lambda)\in W_0 \qquad (\lambda\in\Lambda^\prime)
\]
is the shortest element in $W_0$ such that $v_{\lambda,\Lambda^\prime}\lambda=\lambda_-$. Furthermore, 
$\ell(w_{\lambda,\Lambda^\prime}v)=\ell(w_{\lambda,\Lambda^\prime})+\ell(v)$ for $v\in W_0$ (for details see, e.g., \cite[\S 2.4]{Ma}). We thus have
the decomposition
\[
\tau(\lambda)=w_{\lambda,\Lambda^\prime}v_{\lambda,\Lambda^\prime}\qquad\quad (\lambda\in\Lambda^\prime)
\]
with $\ell(\tau(\lambda))=
\ell(w_{\lambda,\Lambda^\prime})+\ell(v_{\lambda,\Lambda^\prime})$.  Note that $w_{\lambda,\Lambda^\prime}\not=w_\lambda$ for $\lambda\in\Lambda^\prime\setminus Q^\vee$, but we do have
\begin{equation}\label{IsEqualQ}
w_{\lambda,\Lambda^\prime}=w_\lambda,\qquad\quad v_{\lambda,\Lambda^\prime}=v_\lambda\qquad (\lambda\in Q^\vee).
\end{equation}
To compare $w_{\lambda,\Lambda^\prime}$ and $w_\lambda$ for $\lambda\in\Lambda^\prime\setminus Q^\vee$ we first need to describe the elements of the subgroup 
\[
\Omega_{\Lambda^\prime}:=\{\omega\in W_{\Lambda^\prime} \,\, | \,\, \ell(\omega)=0\}
\]
of $W_{\Lambda^\prime}$ in detail. It stabilises $C_+$ and $\Delta$. For $j\in [0,r]$ and $\omega\in\Omega_{\Lambda^\prime}$ we write
$\omega(j)\in [0,r]$ for the index such that $\omega(\alpha_j)=\alpha_{\omega(j)}$. Then we have $\omega s_j\omega^{-1}=s_{\omega(j)}$.

We now recall some well known facts about the abelian quotient group $\Lambda^\prime/Q^\vee$ and $\Omega_{\Lambda^\prime}$
(see, e.g., \cite[\S 2.5]{Ma} for details).
First note that $\mu+Q^\vee=W\mu$ for all $\mu\in\Lambda^\prime$, since $\Lambda^\prime\in\mathcal{L}$.
Hence the set 
\[
\Lambda_{\textup{min}}^\prime:=\overline{C}_+\cap \Lambda^\prime=\{\lambda\in\Lambda^\prime \,\,\, | \,\,\, 0\leq \alpha(\lambda)\leq 1\quad
\forall\, \alpha\in\Phi_0^+\}
\]
of minuscule weights in $\Lambda^\prime$ forms a complete set of representatives of $\Lambda^\prime/Q^\vee$. The inverse of the set-theoretic bijection  $\Lambda_{\textup{min}}^\prime\overset{\sim}{\longrightarrow}\Lambda^\prime/Q^\vee$, $\zeta\mapsto \zeta+Q^\vee$ is concretely given by
$\mu+Q^\vee\mapsto\cc_\mu$ for $\mu\in\Lambda^\prime$.
The following result is well known (see, e.g., \cite[\S 2.5]{Ma}).
\begin{proposition}\label{Omegaprop}\hfill
\begin{enumerate}
\item The assignment $\omega\mapsto \omega(0)+Q^\vee$ defines a group isomorphism $\Omega_{\Lambda^\prime}\overset{\sim}{\longrightarrow}\Lambda^\prime/Q^\vee$.
Its inverse is explicitly given by $\mu+Q^\vee\mapsto w_{\cc_\mu,\Lambda^\prime}$ for $\mu\in\Lambda^\prime$.
\item $W_{\Lambda^\prime}=\Omega_{\Lambda^\prime}\ltimes W$.
\end{enumerate}
\end{proposition}
A reduced expression of $w\in W_{\Lambda^\prime}$ is an expression $w=\omega s_{j_1}\cdots s_{j_\ell}$ with $\omega\in\Omega_{\Lambda^\prime}$, $0\leq j_i\leq r$ and $\ell=\ell(w)$. The roots in $\Pi(w)$ are still described by \eqref{rootsdescription}.

The elements $v_{\zeta,\Lambda^\prime}\in W_0$ ($\zeta\in \Lambda_{\textup{min}}^\prime$) satisfy
\begin{equation}\label{zetalength}
\begin{split}
\Pi(v_{\zeta,\Lambda^\prime})&=\{\alpha\in\Phi^+_0 \,\, | \,\, \alpha(\zeta)=1\},\\
\Phi^+_0\setminus\Pi(v_{\zeta,\Lambda^\prime})&=\{\alpha\in\Phi^+_0 \,\, | \,\, \alpha(\zeta)=0\},
\end{split}
\end{equation}
see \cite[\S 2.5]{Ma}. Furthermore, for $\zeta\in \Lambda_{\textup{min}}^\prime$ we have 
\[
\zeta^\prime:=-w_0\zeta=-v_{\zeta,\Lambda}\zeta\in \Lambda_{\textup{min}}^\prime,
\]
$w_{\zeta^\prime,\Lambda^\prime}=w_{\zeta,\Lambda^\prime}^{-1}$, 
$v_{\zeta^\prime,\Lambda^\prime}=v_{\zeta,\Lambda^\prime}^{-1}$ (see \cite[(2.5.9)]{Ma}), and
\[
w_{\zeta,\Lambda^\prime}y
=\zeta+v_{\zeta,\Lambda^\prime}^{-1}y\qquad (y\in E).
\]

The minuscule weights admit the following explicit description. 
Recall the functions $n_i: \Phi_0\rightarrow\mathbb{Z}$
for $1\leq i\leq r$, which encode the expansion coefficients of the roots in simple roots: $\alpha=\sum_{i=1}^rn_i(\alpha)\alpha_i$ for all $\alpha\in\Phi_0$. 
Define
\begin{equation}\label{Ilambda}
I_{\Lambda^\prime}:=\{i\in [1,r] \,\,\, | \,\,\, n_i(\varphi)=1 \,\, \& \,\, (\varpi_i^\vee+E_{\textup{co}})\cap \Lambda^\prime\not=\emptyset\}.
\end{equation}
For $i\in I_{\Lambda^\prime}$, $(\varpi_i^\vee+E_{\textup{co}})\cap \Lambda^\prime$ is a coset in $\Lambda^\prime/(\Lambda^{\prime}\cap E_{\textup{co}})$, and we fix a representative
$\varpi_{i,\Lambda^\prime}^\vee$ once and for all. The set $\Lambda_{\textup{min}}^\prime$ of minuscule weights is invariant for the translation action of 
$\Lambda^{\prime}\cap E_{\textup{co}}$.
\begin{lemma}
$\{0\}\cup\{\varpi_{i,\Lambda^\prime}^\vee\}_{i\in I_{\Lambda^\prime}}$ is a complete set of representatives of the $\Lambda^{\prime}\cap E_{\textup{co}}$-orbits in $\Lambda_{\textup{min}}^\prime$.
\end{lemma}
\begin{proof}
See, e.g., \cite[\S 2.5]{Ma} and \cite[\S 9.3.4]{StAB}.
\end{proof}
For $\lambda\in\Lambda^\prime$ the element $w_{\lambda,\Lambda^\prime}\in W_{\Lambda^\prime}$
is related to the shortest element $w_\lambda\in W$ such that $\lambda=w_\lambda c_\lambda$
in the following way.
\begin{lemma}\label{relww}
For $\lambda\in\Lambda^\prime$ we have $\cc_\lambda\in\Lambda^\prime$ and $w_{\cc_\lambda,\Lambda^\prime}\in\Omega_{\Lambda^\prime}$. Furthermore,
\begin{equation}\label{IsEqualQ2}
w_\lambda=w_{\lambda,\Lambda^\prime}w_{\cc_\lambda,\Lambda^\prime}^{-1}
\end{equation}
in $W_{\Lambda^\prime}$.  
\end{lemma}
\begin{proof}
Fix $\lambda\in\Lambda^\prime$ and write $\zeta:=\cc_\lambda$. Then clearly $\zeta\in\overline{C}_+\cap\Lambda^\prime=\Lambda_{\textup{min}}^\prime$.
Hence $w_{\zeta,\Lambda^\prime}\in\Omega_{\Lambda^\prime}$ by Proposition \ref{Omegaprop}, proving the first statement.

For the second statement, first note that
\[
w_{\lambda,\Lambda^\prime} w_{\zeta,\Lambda^\prime}^{-1}(0)=\lambda-v_{\lambda,\Lambda^\prime}^{-1}v_{\zeta,\Lambda^\prime}\zeta.
\]
This lies in $Q^\vee$ since $\lambda\in\zeta+Q^\vee$, hence $w_{\lambda,\Lambda^\prime}w_{\zeta,\Lambda^\prime}^{-1}\in W$. 

Now note that 
\[
w_{\lambda,\Lambda^\prime}w_{\zeta,\Lambda^\prime}^{-1}(\zeta)=
w_{\lambda,\Lambda^\prime}(0)=\tau(\lambda)(0)=\lambda.
\]
So $w_{\lambda,\Lambda^\prime}w_{\zeta,\Lambda^\prime}^{-1}$ and $w_\lambda$ are two elements in $W$ mapping $\zeta$ to $\lambda$. Since $w_\lambda\in W^\zeta$ this implies that $\ell(w_{\lambda,\Lambda^\prime})=\ell(w_{\lambda,\Lambda^\prime}w_{\zeta,\Lambda^\prime}^{-1})\geq \ell(w_\lambda)$ (the equality follows from the fact that $w_{\zeta,\Lambda^\prime}^{-1}\in\Omega_{\Lambda^\prime}$).
On the other hand, $w_\lambda w_{\zeta,\Lambda^\prime}\in\tau(\lambda)W_0$, so by definition of $w_{\lambda,\Lambda^\prime}\in W_{\Lambda^\prime}$ we have $\ell(w_\lambda)=\ell(w_\lambda w_{\zeta,\Lambda^\prime})\geq \ell(w_{\lambda,\Lambda^\prime})$ with equality iff 
$w_\lambda w_{\zeta,\Lambda^\prime}=w_{\lambda,\Lambda^\prime}$. 
Combining the two inequalities we conclude that 
$\ell(w_{\lambda,\Lambda^\prime})=\ell(w_\lambda)$, and hence 
$w_\lambda=w_{\lambda,\Lambda^\prime}w_{\zeta,\Lambda^\prime}^{-1}$.
\end{proof}
\begin{remark}
Note that $c_\lambda=0$ iff $\lambda\in Q^\vee$, in which case $w_{0,\Lambda^\prime}=1$ and \eqref{IsEqualQ2} reduces to
\eqref{IsEqualQ}.
\end{remark}

\subsection{The extended double affine Hecke algebra}\label{S700}
Fix lattices $(\Lambda,\Lambda^\prime)\in (\mathcal{L}^{\times 2})_e$ and assume that $q\in\mathbf{F}^\times$ has an $e^{\textup{th}}$ root $q^{\frac{1}{e}}\in\mathbf{F}^\times$, which we fix once and for all. We write
\[
q^{\ell/e}:=(q^{\frac{1}{e}})^{\ell},\qquad q_\alpha^{\ell/e}:=(q^{\frac{1}{e}})^{\frac{2\ell}{\|\alpha\|^2}}
\]
for $\ell\in\mathbb{Z}$
(recall that $2/\|\alpha\|^2\in\mathbb{Z}$ by the choice of normalisation of $\Phi_0$).

The extended affine Weyl group $W_{\Lambda^\prime}$ then acts on $T_\Lambda$ by $q$-dilations and reflections. In particular, $\tau(\lambda)\in W_{\Lambda^\prime}$ ($\lambda\in\Lambda^\prime$) acts by $\tau(\lambda)t=q^\lambda t$ (here $q^\lambda\in T_{\Lambda}$ is the torus element mapping $\mu\in\Lambda$ to $q^{\langle\lambda,\mu\rangle}$). Later on we will also use the $W_\Lambda$-action on $T_{\Lambda^\prime}$ by $q$-dilations and reflections.

\begin{remark}
The natural map $j_\Lambda: T_{P^\vee}\rightarrow T_\Lambda$ from Corollary \ref{ilacor} is $W$-equivariant. If $(P^\vee,\Lambda^\prime)\in(\mathcal{L}^{\times 2})_e$, then $j_\Lambda$ is
 $W_{\Lambda^\prime}$-equivariant if either $\Lambda\subset E^\prime$ or $\Lambda^\prime\subset E^\prime$.
\end{remark}

The contragredient representation gives a $W_{\Lambda^\prime}$-action on $\mathcal{P}_\Lambda$ by algebra automorphisms. The explicit formulas for the action are given by \eqref{xyW}.
The action of the subgroup $\Omega_{\Lambda^\prime}$ on $\mathcal{P}_\Lambda$ is given by
\begin{equation}\label{ux}
w_{\zeta,\Lambda^\prime}(x^\lambda)=q^{-\langle v_{\zeta,\Lambda^\prime}\zeta,\lambda\rangle}x^{v_{\zeta,\Lambda^\prime}^{-1}\lambda}
\end{equation}
for $\zeta\in\Lambda_{\textup{min}}^\prime$
and $\lambda\in\Lambda$. The formulas $T_j\mapsto T_{\omega(j)}$
($\omega\in\Omega_{\Lambda^\prime}$, $0\leq j\leq r$) extend the action of $\Omega_{\Lambda^\prime}$ on $\mathcal{P}=\mathcal{P}_{Q^\vee}$
to an action of $\Omega_{\Lambda^\prime}$ on $\mathbb{H}$
by algebra automorphisms.

\begin{definition}
The smashed product algebra $\mathbb{H}_{Q^\vee,\Lambda^\prime}:=\Omega_{\Lambda^\prime}\ltimes\mathbb{H}$ is the
$Y$-extended double affine Hecke algebra.
\end{definition}
Concretely we have the commutation relation $\omega T_j=T_{\omega(j)}\omega$ for $\omega\in\Omega_{\Lambda^\prime}$ and $0\leq j\leq r$
in $\mathbb{H}_{Q^\vee,\Lambda^\prime}$,
as well as the commutation relations
\begin{equation}\label{crossOmegaexplicit}
w_{\zeta,\Lambda^\prime}x^\mu=q^{-\langle v_{\zeta,\Lambda^\prime}\zeta,\mu\rangle}x^{v_{\zeta,\Lambda^\prime}^{-1}\mu}\,w_{\zeta,\Lambda^\prime}
\end{equation}
in $\mathbb{H}_{Q^\vee,\Lambda^\prime}$ for $\zeta\in\Lambda_{\textup{min}}^\prime$ and $\mu\in Q^\vee$.

The extended affine Hecke algebra $H_{\Lambda^\prime}$ is the subalgebra $\Omega_{\Lambda^\prime}\ltimes H$ of
$\mathbb{H}_{Q^\vee,\Lambda^\prime}$. Multiplication defines a $\mathbf{F}$-linear isomorphism
$\mathcal{P}\otimes H_{\Lambda^\prime}\overset{\sim}{\longrightarrow}\mathbb{H}_{Q^\vee,\Lambda^\prime}$. If
$w=\omega s_{j_1}\cdots s_{j_\ell}\in W_{\Lambda^\prime}$ ($\omega\in\Omega_{\Lambda^\prime}$, $0\leq j_i\leq r$) is a reduced
expression, then 
\[
T_w:=\omega T_{j_1}\cdots T_{j_\ell}\in H_{\Lambda^\prime}
\]
is well defined. 

There exists a unique group homomorphism $\Lambda^\prime\rightarrow H_{\Lambda^\prime}$, $\lambda\mapsto Y^\lambda$, such that 
\begin{equation}\label{Ywhere}
Y^{\lambda}=T_{\tau(\lambda)}\qquad\quad (\lambda\in\Lambda^{\prime +}:=\Lambda^\prime\cap \overline{E}_+).
\end{equation}
Note that $Y^\mu$ for $\mu\in Q^\vee$ reduces to the corresponding element in $H$, as defined in Subsection \ref{DAHA}. Then 
\[
\{T_vY^\lambda \,\,\, | \,\,\, v\in W_0,\, \lambda\in\Lambda^\prime\}
\]
is a basis of
$H_{\Lambda^\prime}$, and the 
cross relations \eqref{crossrelation} in $H_{\Lambda^\prime}$ holds true for
$\mu\in\Lambda^\prime$ and $1\leq i\leq r$.

The elements in $\Omega_{\Lambda^\prime}$, viewed as elements in $\mathbb{H}_{Q^\vee,\Lambda^\prime}$, factorise as follows, 
\begin{equation}\label{OmegaHH}
w_{\zeta,\Lambda^\prime}=Y^{\zeta}T_{v_{\zeta,\Lambda^\prime}}^{-1}\qquad (\zeta\in \Lambda_{\textup{min}}^\prime).
\end{equation}

\begin{definition}
The extended double affine Hecke algebra $\mathbb{H}_{\Lambda,\Lambda^\prime}$ is the algebra generated by
$\mathcal{P}_\Lambda$ and $H_{\Lambda^\prime}$ subject to the cross relations \eqref{crossX} 
for $0\leq j\leq r$  and
$p\in\mathcal{P}_\Lambda$, as well as the relations \eqref{crossOmegaexplicit} for $\zeta\in\Lambda_{\textup{min}}^\prime$ and $\mu\in\Lambda$. 
\end{definition}
Note that 
\begin{equation}\label{smashedprodform}
\mathbb{H}_{\Lambda,\Lambda^\prime}\simeq\Omega_{\Lambda^\prime}\ltimes\mathbb{H}_{\Lambda,Q^\vee}
\end{equation}
with the action of $\Omega_{\Lambda^\prime}$ on $\mathbb{H}_{\Lambda,Q^\vee}$ by algebra automorphisms defined by \eqref{ux} and $\omega(T_j):=T_{\omega(j)}$
for $\omega\in\Omega_{\Lambda^\prime}$ and $0\leq j\leq r$. The algebra $\mathbb{H}_{\Lambda,Q^\vee}$ is the {\it $X$-extended double affine Hecke algebra}.

The duality anti-involution $\delta$ of $\mathbb{H}$ defined in Theorem \ref{deltatheorem} extends to an anti-algebra isomorphism $\delta:\mathbb{H}_{\Lambda,\Lambda^\prime}\rightarrow\mathbb{H}_{\Lambda^\prime,\Lambda}$ by 
\[
\delta(Y^\lambda)=x^{-\lambda}\qquad (\lambda\in\Lambda^\prime).
\] 
Note that $\delta$ restricts to an anti-algebra isomorphism $\mathbb{H}_{Q^\vee,\Lambda^\prime}\overset{\sim}{\longrightarrow}\mathbb{H}_{\Lambda^\prime,Q^\vee}$, and
\[
\delta\bigl(w_{\zeta,\Lambda^\prime}\bigr)=T_{v_{\zeta,\Lambda^\prime}^{-1}}^{-1}x^{-\zeta}\qquad
(\zeta\in\Lambda_{\textup{min}}^\prime).
\]

The extended $X$-intertwiners are defined by
\[
S_{\omega w}^X:= \omega S_w^X\in\mathbb{H}_{Q^\vee,\Lambda^\prime}\qquad (\omega\in\Omega_{\Lambda^\prime}, w\in W).
\]
They satisfy \eqref{Iprop} for $w\in W_{\Lambda^\prime}$ and $p\in\mathcal{P}_\Lambda$. The extended $Y$-intertwiners are 
\[
S_w^Y:=\delta(S_{w^{-1}}^X)\in\mathbb{H}_{\Lambda,Q^\vee}
\qquad (w\in W_\Lambda).
\]
In particular,
\begin{equation}
\label{buildingYintertwinersOmega}
S_{w_{\zeta,\Lambda}}^Y=\delta\bigl(w_{\zeta,\Lambda}^{-1})=x^\zeta T_{v_{\zeta,\Lambda}^{-1}}\qquad (\zeta\in \Lambda_{\textup{min}}).
\end{equation}
The extended $Y$-intertwiners $S_w^Y$ ($w\in W_\Lambda$) satisfy \eqref{IpropY} for  $p\in\mathcal{P}_{\Lambda^\prime}$.

\subsection{The extended quasi-polynomial representation}\label{S71}

Recall the definition of the base-point $\mathfrak{s}_y\in T_{\Lambda^\prime}$ ($y\in E$) from \eqref{ffrak}.
We have the following extension of Lemma \ref{fraks}.
\begin{lemma}\label{Omegas}
For $\omega\in\Omega_{\Lambda}$ and $y\in E$ we have $(D\omega)\mathfrak{s}_y=\mathfrak{s}_{\omega y}$ in $T_{\Lambda^\prime}$.
\end{lemma}
\begin{proof}
Let $\zeta\in\Lambda_{\textup{min}}$ such that $\omega=w_{\zeta,\Lambda}=\tau(\zeta)v_{\zeta,\Lambda}^{-1}$
(see Proposition \ref{Omegaprop}). 
By the definition \eqref{ffrak} of $\mathfrak{s}_y$ we have 
\begin{equation}\label{us}
v_{\zeta,\Lambda}^{-1}\mathfrak{s}_y=\mathfrak{s}_{w_{\zeta,\Lambda}y}
\prod_{\alpha\in\Phi_0^+\setminus\Pi(v_{\zeta,\Lambda})}k_\alpha^{(\eta(\alpha(v_{\zeta,\Lambda}^{-1}y))-
\eta(\alpha(w_{\zeta,\Lambda}y)))\alpha}
\prod_{\beta\in\Pi(v_{\zeta,\Lambda})}k_\beta^{-(\eta(-\beta(v_{\zeta,\Lambda}^{-1}y))+\eta(\beta(w_{\zeta,\Lambda}y)))\beta}
\end{equation}
in $T_{\Lambda^\prime}$.
The second line of \eqref{zetalength} gives
\[
\eta(\alpha(v_{\zeta,\Lambda}^{-1}y))-
\eta(\alpha(w_{\zeta,\Lambda}y))=\eta(\alpha(v_{\zeta,\Lambda}^{-1}y))-
\eta(\alpha(v_{\zeta,\Lambda}^{-1}y))=0
\]
for $\alpha\in \Phi_0^+\setminus\Pi(v_{\zeta,\Lambda})$, 
hence the product over $\alpha$ in \eqref{us} is equal to $1$.
Formula \eqref{etaaffine} and the first line of \eqref{zetalength} imply that the product over $\beta$ in \eqref{us} is equal to $1$.
\end{proof}

For $\cc\in C^J$ and $\mathfrak{t}\in T_{\Lambda^\prime,J}$ we extend the quasi-polynomial representation $\pi_{\cc,\mathfrak{t}\vert_{Q^\vee}}$ of $\mathbb{H}$ 
to $\mathbb{H}_{\Lambda,\Lambda^\prime}$ under suitably further restrictions on $\mathfrak{t}\in T_{\Lambda^\prime,J}$. It is convenient to discuss
the extension to the $Y$-extended double affine Hecke algebra $\mathbb{H}_{Q^\vee,\Lambda^\prime}$ first, in which case we do not need to impose additional conditions 
on $\mathfrak{t}\in T_{\Lambda^\prime,J}$.

Let  $\cc\in C^J$ and $\mathfrak{t}\in T_{\Lambda^\prime,J}$. Recall the notation
$\mathfrak{t}_y:=w_y\mathfrak{t}\in T_{\Lambda^\prime}$ ($y\in\mathcal{O}_\cc$) from Definition \ref{defty}.
We then have a $\mathfrak{t}$-dependent
$W_{\Lambda^\prime}\ltimes\mathcal{P}$-action $(w,p)\mapsto w_{\mathfrak{t}}p$ on $\mathcal{P}^{(\cc)}$, with the action of $v\in W_0$ and $\tau(\mu)$ ($\mu\in\Lambda^\prime$) given by
\eqref{xyWt}. Note that
\begin{equation}\label{OmegaPol1}
\bigl(w_{\zeta,\Lambda^\prime}\bigr)_{\mathfrak{t}}(x^y)=
\mathfrak{t}_{v_{\zeta,\Lambda^\prime}^{-1}y}^{-\zeta}x^{v_{\zeta,\Lambda^\prime}^{-1}y}\qquad
(\zeta\in\Lambda_{\textup{min}}^\prime,\, y\in\mathcal{O}_\cc).
\end{equation} 
\begin{lemma}\label{gluelemma}
Let $\cc\in C^J$ and $\mathfrak{t}\in T_{\Lambda^\prime,J}$. 

The quasi-polynomial representation
$\pi_{\cc,\mathfrak{t}\vert_{Q^\vee}}: \mathbb{H}\rightarrow\textup{End}(\mathcal{P}^{(\cc)})$ \textup{(}see Theorem \ref{gbr}\textup{)}
extends to a representation $\pi_{\cc,\mathfrak{t}}^{Q^\vee,\Lambda^\prime}: \mathbb{H}_{Q^\vee,\Lambda^\prime}\rightarrow\textup{End}(\mathcal{P}^{(\cc)})$ by
\[
\pi_{\cc,\mathfrak{t}}^{Q^\vee,\Lambda^\prime}(\omega)(x^y):=\omega_{\mathfrak{t}}(x^y)
\qquad (\omega\in\Omega_{\Lambda^\prime},\, y\in\mathcal{O}_\cc).
\]
Following the notations from Theorem \ref{gbr}, we denote the $\mathbb{H}_{Q^\vee,\Lambda^\prime}$-module $(\mathcal{P}^{(\cc)},
\pi_{\cc,\mathfrak{t}}^{Q^\vee,\Lambda^\prime})$ again by $\mathcal{P}^{(\cc)}_{\mathfrak{t}}$.
\end{lemma}
\begin{proof}
Note that $\mathfrak{t}|_{Q^\vee}\in T_J$, hence the quasi-polynomial representation $\pi_{\cc,\mathfrak{t}\vert_{Q^\vee}}$ of $\mathbb{H}$ is well-defined. Since 
\begin{equation}\label{wptrivial}
w_{\mathfrak{t}}(px^y)=w(p)w_{\mathfrak{t}}(x^y)\qquad (w\in W_{\Lambda^\prime}, p\in\mathcal{P}, y\in\mathcal{O}_\cc),
\end{equation}
the commutation relations $\omega x^\mu=\omega(x^\mu)\omega$ in $\mathbb{H}_{Q^\vee,\Lambda^\prime}$ for $\omega\in\Omega_{\Lambda^\prime}$ and $\mu\in Q^\vee$ are respected by the proposed extension $\pi_{\cc,\mathfrak{t}}^{Q^\vee,\Lambda^\prime}$ of $\pi_{\cc,\mathfrak{t}\vert_{Q^\vee}}$ to $\mathbb{H}_{Q^\vee,\Lambda^\prime}$. 
It remains to check that the same is true for the relations $\omega T_j=T_{\omega(j)}\omega$
($\omega\in\Omega_{\Lambda^\prime}$, $0\leq j\leq r$). Note that $k_{\omega(j)}=k_j$, $\omega s_j=s_{\omega(j)}\omega$ and 
\[
\chi_{\mathbb{Z}}(\alpha_{\omega(j)}(D\omega(y)))=
\chi_{\mathbb{Z}}(\alpha_{\omega(j)}(\omega y))=\chi_{\mathbb{Z}}(\alpha_j(y)),
\]
hence it suffices to show that 
\[
\omega_{\mathfrak{t}}\bigl(\nabla_j(x^y)\bigr)=\nabla_{\omega(j)}(\omega_{\mathfrak{t}}(x^y))
\]
for $\omega\in\Omega_{\Lambda^\prime}$, $0\leq j\leq r$ and $y\in\mathcal{O}_\cc$. 

Fix $\omega\in\Omega_{\Lambda^\prime}$ and $0\leq j\leq r$. By \eqref{actx}, which holds true for $w\in W_{\Lambda^\prime}$ and $a\in\Phi$, and \eqref{wptrivial}
we have
\begin{equation}\label{ttdd}
\omega_{\mathfrak{t}}\bigl(\nabla_j(x^y)\bigr)=
\left(\frac{1-x^{-\lfloor D\alpha_j(y)\rfloor \alpha_{\omega(j)}^\vee}}{1-x^{\alpha_{\omega(j)}^\vee}}
\right)\omega_{\mathfrak{t}}(x^y).
\end{equation}
Now note that $D\alpha_{\omega(j)}((D\omega)y)=D(\omega^{-1}\alpha_{\omega(j)})(y)=
D\alpha_j(y)$, hence the right hand side of \eqref{ttdd} is equal to $\nabla_{\omega(j)}(\omega_{\mathfrak{t}}(x^y))$, as desired.
\end{proof}
In terms of the new notation from Lemma \ref{gluelemma}, the quasi-polynomial representation $\pi_{\cc,\mathfrak{t}}$ is equal to $\pi_{\cc,\mathfrak{t}}^{Q^\vee,Q^\vee}$.

In the same way Theorem \ref{aCG} extends as follows. 
\begin{lemma}\label{aCGrem}
Let $\cc\in C^J$ and $\mathfrak{t}\in T_{\Lambda^\prime,J}$. The formulas 
\begin{equation*}
\begin{split}
\sigma_{\cc,\mathfrak{t}}^{Q^\vee,\Lambda^\prime}(s_j)(x^y)&:=
\frac{k_j^{\chi_{\mathbb{Z}}(\alpha_j(y))}(x^{\alpha_j^\vee}-1)}{(k_jx^{\alpha_j^\vee}-k_j^{-1})}
s_{j,\mathfrak{t}}(x^y)+\frac{(k_j-k_j^{-1})}{(k_jx^{\alpha_j^\vee}-k_j^{-1})}x^{y-\lfloor D\alpha_j(y)
\rfloor\alpha_j^\vee},\\
\sigma_{\cc,\mathfrak{t}}^{Q^\vee,\Lambda^\prime}(\omega)(x^y)&:=\omega_{\mathfrak{t}}(x^y),\\
\sigma_{\cc,\mathfrak{t}}^{Q^\vee,\Lambda^\prime}(f)(x^y)&:=fx^y
\end{split}
\end{equation*}
for $0\leq j\leq r$, $y\in\mathcal{O}_\cc$, $\omega\in\Omega_{\Lambda^\prime}$ and $f\in\mathcal{Q}$ define a representation 
\[
\sigma_{\cc,\mathfrak{t}}^{Q^\vee,\Lambda^\prime}: W_{\Lambda^\prime}\ltimes\mathcal{Q}\rightarrow\textup{End}(\mathcal{Q}^{(\cc)}).
\]
\end{lemma}

As the next step we introduce a natural extension of the quasi-polynomial representation  
from $\mathbb{H}_{Q^\vee,\Lambda^\prime}$ to $\mathbb{H}_{\Lambda,\Lambda^\prime}$. Its representation space will be
\[\
\mathcal{P}_\Lambda^{(\cc)}:=\bigoplus_{y\in\mathcal{O}_{\Lambda,\cc}}\mathbf{F}x^y,
\]
where $\mathcal{O}_{\Lambda,\cc}$ is the $W_\Lambda$-orbit of $\cc\in C^J$ in $E$. We realise it as a quotient of the 
induced $\mathbb{H}_{\Lambda,\Lambda^\prime}$-module 
\[
V_{\cc,\mathfrak{t}}:=\mathbb{H}_{\Lambda,\Lambda^\prime}\otimes_{\mathbb{H}_{Q^\vee,\Lambda^\prime}}\mathcal{P}_{\mathfrak{t}}^{(\cc)}
\]
under suitable further restrictions on $\mathfrak{t}\in T_{\Lambda^\prime,J}$.

We first consider the $\mathbb{H}_{\Lambda,\Lambda^\prime}$-module $V_{\cc,\mathfrak{t}}$ in more detail.
\begin{lemma}\label{Vct}
Let $\cc\in C^J$ and $\mathfrak{t}\in T_{\Lambda^\prime,J}$.
As a $\mathbb{H}_{Q^\vee,\Lambda^\prime}$-module we have
\[
V_{\cc,\mathfrak{t}}
\simeq \bigoplus_{\omega\in\Omega_\Lambda}\mathcal{P}_{\omega\mathfrak{t}}^{(\omega\cc)}.
\]
\end{lemma}
\begin{proof}
Applying $\delta$ to the decomposition \eqref{smashedprodform} it follows that the subspaces
\[
V_{\cc,\mathfrak{t}}(\omega):=\textup{span}\{S_\omega^Y\otimes_{\mathbb{H}_{Q^\vee,\Lambda^\prime}}f
\,\,\, | \,\,\, f\in\mathcal{P}_{\mathfrak{t}}^{(\cc)}\} 
\qquad (\omega\in\Omega_\Lambda)
\]
are $\mathbb{H}_{Q^\vee,\Lambda^\prime}$-submodules of $V_{\cc,\mathfrak{t}}$,
and 
\[
V_{\cc,\mathfrak{t}}=\bigoplus_{\omega\in\Omega_\Lambda}V_{\cc,\mathfrak{t}}(\omega).
\]
It is easy to check that $V_{\cc,\mathfrak{t}}(\omega)\simeq\mathcal{P}_{\omega\mathfrak{t}}^{(\omega\cc)}$
as $\mathbb{H}_{Q^\vee,\Lambda^\prime}$-modules, with the isomorphism mapping $S_\omega^Y\otimes_{\mathbb{H}_{Q^\vee,\Lambda^\prime}}x^\cc$
to $x^{\omega\cc}$.
\end{proof}
Let $\cc\in C^J$ and write 
\[
\Omega_{\Lambda,J}:=\{\omega\in\Omega_{\Lambda} \,\, 
| \,\, \omega(J)=J\}.
\]
Denote by $W_{\Lambda,\cc}$ the fix-point subgroup in $W_\Lambda$ of $\cc$, and by 
$\Omega_{\Lambda,\cc}$ the fix-point subgroups in $\Omega_{\Lambda}$ of $\cc$.
Then $\Omega_{\Lambda,\cc}\subseteq\Omega_{\Lambda,J}$ and 
$W_{\Lambda,\cc}=\Omega_{\Lambda,\cc}\ltimes W_J$. Write $\Omega_{\Lambda}^\cc$ for a complete set of representatives of $\Omega_{\Lambda}/\Omega_{\Lambda,\cc}$.
The $W_\Lambda$-orbit $\mathcal{O}_{\Lambda,\cc}:=W_\Lambda\cc$ then decomposes as
\[
\mathcal{O}_{\Lambda,\cc}=\bigsqcup_{\omega\in\Omega^\cc_\Lambda}\mathcal{O}_{\omega\cc}.
\] 
In particular, we have
\[
\mathcal{P}_\Lambda^{(\cc)}=
\bigoplus_{\omega\in\Omega_{\Lambda}^\cc}\mathcal{P}^{(\omega\cc)}.
\]
as $\mathcal{P}$-modules.

We now look for a quotient of the induced $\mathbb{H}_{\Lambda,\Lambda^\prime}$-module
$V_{\cc,\mathfrak{t}}$ admitting a quasi-polynomial realisation on $\mathcal{P}_{\Lambda}^{(\cc)}$.
For this we need to restrict to the following parameter subset of $T_{\Lambda^\prime,J}$ to the following subset.
 
\begin{definition}\label{Tres}
For $\cc\in C^J$ write
\[
{}^\Lambda T_{\Lambda^\prime}^{\cc}:=\{\mathfrak{t}\in T_{\Lambda^\prime,J} \,\, | \,\, 
\omega\mathfrak{t}=\mathfrak{t}\,\,\, \forall\,\omega\in\Omega_{\Lambda,\cc}\}.
\]
\end{definition}
The second and third example below should be compared to Lemma \ref{resLambdaQ}. 
\begin{example}\label{exampleTspec}\hfill
\begin{enumerate}
\item For $\omega\in\Omega_\Lambda$ and $\cc\in\overline{C}_+$ we have $\omega({}^{\Lambda}T_{\Lambda^\prime}^{\cc})={}^{\Lambda}T_{\Lambda^\prime}^{\omega\cc}$.
\item For $\cc\in C^J$ we have ${}^{Q^\vee}T_{\Lambda^\prime}^\cc=T_{\Lambda^\prime,J}$. Furthermore, the restriction map $T_{\Lambda^\prime,J}\rightarrow T_J$, $s\mapsto s\vert_{Q^\vee}$ restricts to a map ${}^\Lambda T_{\Lambda^\prime}^\cc\rightarrow {}^\Lambda T_{Q^\vee}^\cc$ \textup{(}Lemma \ref{resLambdaQ}(2) does not have an apparent extension to the present context\textup{)}.
\item For all $\cc\in\overline{C}_+$ the subset ${}^{\Lambda}T_{\Lambda^\prime}^\cc$ of $T_{\Lambda^\prime}$ is a $T_{\Lambda^\prime,[1,r]}$-subset \textup{(}use here the fact that $T_{\Lambda,[1,r]}\subseteq T_{\Lambda}^{W_0}$\textup{)}.
\item If $\zeta\in\Lambda_{\textup{min}}$ then 
\[
{}^{\Lambda}T_{\Lambda^\prime}^\zeta=q^\zeta T_{\Lambda^\prime,[1,r]}.
\]
Indeed, note that ${}^{\Lambda}T_{\Lambda^\prime}^0=T_{\Lambda^\prime,[1,r]}$ since $\Omega_{\Lambda,0}=\{1\}$ and $\mathbf{J}(0)=[1,r]$, and so
\[
{}^{\Lambda}T_{\Lambda^\prime}^\zeta=w_{\zeta,\Lambda}({}^{\Lambda}T_{\Lambda^\prime}^0)=q^\zeta v_{\zeta,\Lambda}^{-1}(T_{\Lambda^\prime,[1,r]})=
q^\zeta T_{\Lambda^\prime,[1,r]}.
\]
Here we used part (2) and the fact that $w_{\zeta,\Lambda}=\tau(\zeta)v_{\zeta,\Lambda}^{-1}\in\Omega_\Lambda$ maps $0$ to $\zeta$. 
\end{enumerate}
\end{example}
The lifting properties of characters as described in Lemma \ref{resLambdaQ} do not seem to extend easily to ${}^\Lambda T_{\Lambda^\prime}^\cc$.

For $\mathfrak{t}\in {}^{\Lambda}T_{\Lambda^\prime}^{\cc}$ the induced module $V_{\cc,\mathfrak{t}}$ admits the following natural family of intertwiners.
\begin{lemma}
Let $\cc\in C^J$ and $\mathfrak{t}\in {}^{\Lambda}T_{\Lambda^\prime}^\cc$.
For each $\omega\in\Omega_{\Lambda,\cc}$ there exists a unique intertwiner 
\[\phi_\omega\in\textup{End}_{\mathbb{H}_{\Lambda,\Lambda^\prime}}
\bigl(V_{\cc,\mathfrak{t}}\bigr)
\]
such that 
\begin{equation}\label{t10d10}
\phi_\omega\bigl(S_{\omega^\prime}^Y\otimes_{\mathbb{H}_{Q^\vee,\Lambda^\prime}}x^\cc\bigr)
=S_{\omega^\prime\omega}^Y\otimes_{\mathbb{H}_{Q^\vee,\Lambda^\prime}}x^\cc
\qquad (\omega^\prime\in\Omega_\Lambda).
\end{equation}
\end{lemma}
\begin{proof}
Let $\omega^\prime\in\Omega_\Lambda$. Since $\omega\in\Omega_{\Lambda,\cc}$, the condition $\omega\mathfrak{t}=\mathfrak{t}$ and the proof of Lemma \ref{Vct} 
provide isomorphisms
\[
V_{\cc,\mathfrak{t}}(\omega^\prime)\overset{\sim}{\longrightarrow} \mathcal{P}_{\omega^\prime\mathfrak{t}}^{(\omega^\prime\cc)}\overset{\sim}{\longrightarrow}
V_{\cc,\mathfrak{t}}(\omega^\prime\omega)
\]
of $\mathbb{H}_{Q^\vee,\Lambda^\prime}$-modules. The resulting $\mathbb{H}_{Q^\vee,\Lambda^\prime}$-linear isomorphism 
$V_{\cc,\mathfrak{t}}(\omega^\prime)\overset{\sim}{\longrightarrow}
V_{\cc,\mathfrak{t}}(\omega^\prime\omega)$ is characterised by
\[
S_{\omega^\prime}^Y\otimes_{\mathbb{H}_{Q^\vee,\Lambda^\prime}}x^\cc\mapsto 
S_{\omega^\prime\omega}^Y\otimes_{\mathbb{H}_{Q^\vee,\Lambda^\prime}}x^\cc.
\] 
Hence \eqref{t10d10} gives rise to a well defined intertwiner 
$\phi_\omega\in\textup{Hom}_{\mathbb{H}_{Q^\vee,\Lambda^\prime}}(V_{\cc,\mathfrak{t}})$.
A straightforward check shows that
$\phi_\omega$ also intertwines the $\delta(\Omega_\Lambda)$-action on $V_{\cc,\mathfrak{t}}$.
\end{proof}

Recall that for $\cc\in C^J$ and $\mathfrak{t}\in T_{\Lambda^\prime,J}$, we introduced $W$-translates $\mathfrak{t}_y\in T_{\Lambda^\prime}$
for $y\in\mathcal{O}_\cc$ in Definition \ref{defty}.
In case $\mathfrak{t}\in {}^{\Lambda}T^\cc_{\Lambda^\prime}\subseteq T_{\Lambda^\prime,J}$, we extend this construction to $W_\Lambda$-translates 
as follows.
\begin{definition}\label{deftyc}
Let $\cc\in C^J$ and $\mathfrak{t}\in {}^{\Lambda}T^\cc_{\Lambda^\prime}$. Define $\mathfrak{t}_{y;\cc}\in T_{\Lambda^\prime}$
for $y\in\mathcal{O}_{\Lambda,\cc}$ by 
\begin{equation}\label{tyc}
\mathfrak{t}_{w\cc;\cc}:=w\mathfrak{t} \qquad (w\in W_\Lambda).
\end{equation}
\end{definition}
Note that \eqref{tyc} is well defined since $\mathfrak{t}\in {}^{\Lambda}T^\cc_{\Lambda^\prime}$ implies that $w\mathfrak{t}=\mathfrak{t}$ for 
all $w\in W_{\Lambda,\cc}$. Furthermore, $\mathfrak{t}_{y;\cc}=\mathfrak{t}_y$ for $y\in\mathcal{O}_\cc$, and more generally,
\begin{equation}\label{O4}
\mathfrak{t}_{y;\cc}=(\omega\mathfrak{t})_y\,\,\,\hbox{ when }\, y\in\mathcal{O}_{\omega\cc}\, \hbox{ and }\, \omega\in\Omega_\Lambda.
\end{equation}
Furthermore, note that for $\omega\in\Omega_\Lambda$ we have
$\omega\bigl({}^{\Lambda}T_{\Lambda^\prime}^\cc\bigr)={}^{\Lambda}T_{\Lambda^\prime}^{\omega\cc}\subseteq T_{\Lambda^\prime,\omega(J)}$ by Example \ref{exampleTspec}(1), and 
\begin{equation}\label{O3}
(\omega\mathfrak{t})_{y;\omega\cc}=\mathfrak{t}_{y;\cc}\,\,\, \hbox{ when }\, y\in\mathcal{O}_{\Lambda,\cc}\, \hbox{ and }\, \omega\in\Omega_\Lambda.
\end{equation}

For $w\in W_\Lambda$ and $y\in\mathcal{O}_{\Lambda,\cc}$ we have $\mathfrak{t}_{wy;\cc}=w\mathfrak{t}_{y;\cc}$. For translations this implies that
\begin{equation}\label{Lambdatranslatet}
\mathfrak{t}_{y+\lambda;\cc}=q^{\lambda}\mathfrak{t}_{y;\cc}\qquad (y\in\mathcal{O}_{\Lambda,\cc},\, \lambda\in\Lambda).
\end{equation}

Lemma \ref{actiondeform} now immediately generalises as follows.
\begin{lemma}\label{actiondeformext}
For $\cc\in C^J$ and $\mathfrak{t}\in{}^{\Lambda}T^\cc_{\Lambda^\prime}$ the extended affine Weyl group
$W_{\Lambda^\prime}$ acts linearly on $\mathcal{P}_\Lambda^{(\cc)}$
by
\begin{equation*}
\begin{split}
v_{\mathfrak{t};\cc}(x^y)&:=x^{vy}\qquad\qquad\qquad (v\in W_0),\\
\tau(\mu)_{\mathfrak{t};\cc}(x^y)&:=
\mathfrak{t}_{y;\cc}^{-\mu}x^y
\qquad\qquad\quad (\mu\in\Lambda^\prime)
\end{split}
\end{equation*}
for $y\in\mathcal{O}_{\Lambda,\cc}$. 
\end{lemma}
Note that
$w_{\mathfrak{t};\cc}(pf)=w(p)w_{\mathfrak{t};\cc}(f)$ for $w\in W_{\Lambda^\prime}$, $p\in\mathcal{P}_\Lambda$ and
$f\in\mathcal{P}_\Lambda^{(\cc)}$, and
\begin{equation}\label{OmegaPol}
\begin{split}
s_{0,\mathfrak{t};\cc}(x^y)&=
\mathfrak{t}_{y;\cc}^{\varphi^\vee}x^{s_\varphi y},\\
(w_{\zeta,\Lambda^\prime})_{\mathfrak{t};\cc}(x^y)&=
\mathfrak{t}_{v_{\zeta,\Lambda^\prime}^{-1}y;\cc}^{-\zeta}x^{v_{\zeta,\Lambda^\prime}^{-1}y}
\end{split}
\end{equation}
for $y\in\mathcal{O}_{\Lambda,\cc}$ and $\zeta\in \Lambda_{\textup{min}}^\prime$ (cf. \eqref{OmegaPol1}).
Note furthermore that $\bigl(w_{\zeta,\Lambda^\prime}\bigr)_{\mathfrak{t};\cc}(x^y)$ is a constant multiple of $x^{(Dw_\zeta)y}$, and that for
$\cc\in C^J$, $\omega\in\Omega_{\Lambda}$ and $\mathfrak{t}\in{}^{\Lambda}T^\cc_{\Lambda^\prime}$,
\begin{equation}
\label{compatibleOmegatranslate}
w_{\mathfrak{t};\cc}(x^y)=w_{\omega\mathfrak{t}}(x^y)\qquad (w\in W_{\Lambda^\prime},\,
y\in\mathcal{O}_{\omega\cc}).
\end{equation}
We now have the following extension of Lemma \ref{gluelemma} to $\mathbb{H}_{\Lambda,\Lambda^\prime}$.
\begin{theorem}\label{glueprop}
Let $\cc\in C^J$ and $\mathfrak{t}\in {}^{\Lambda}T_{\Lambda^\prime}^\cc$. 
\begin{enumerate}
\item The formulas 
\begin{equation}\label{actionformulasext}
\begin{split}
\pi_{\cc,\mathfrak{t}}^{\Lambda,\Lambda^\prime}(T_j)x^y&:=k_j^{\chi_{\mathbb{Z}}(\alpha_j(y))}s_{j,\mathfrak{t};\cc}(x^y)
+(k_j-k_j^{-1})\nabla_j(x^y),\\
\pi_{\cc,\mathfrak{t}}^{\Lambda,\Lambda^\prime}(x^\lambda)x^y&:=x^{y+\lambda},\\
\pi_{\cc,\mathfrak{t}}^{\Lambda,\Lambda^\prime}(\omega)x^y&:=\omega_{\mathfrak{t};\cc}(x^y)
\end{split}
\end{equation}
for $0\leq j\leq r$, $\lambda\in\Lambda$, $\omega\in\Omega_{\Lambda^\prime}$ and $y\in\mathcal{O}_{\Lambda,\cc}$ define a representation 
\[
\pi_{\cc,\mathfrak{t}}^{\Lambda,\Lambda^\prime}: \mathbb{H}_{\Lambda,\Lambda^\prime}\rightarrow\textup{End}
(\mathcal{P}_\Lambda^{(\cc)}).
\]
We denote the resulting $\mathbb{H}_{\Lambda,\Lambda^\prime}$-module by $\mathcal{P}_{\Lambda,\mathfrak{t}}^{(\cc)}$.
\item As $\mathbb{H}_{Q^\vee,\Lambda^\prime}$-modules,
\begin{equation}\label{Pdec}
\mathcal{P}_{\Lambda,\mathfrak{t}}^{(\cc)}=\bigoplus_{\omega\in\Omega_\Lambda^\cc}\mathcal{P}_{\omega\mathfrak{t}}^{(\omega\cc)}.
\end{equation}
\item $\mathcal{P}_{\Lambda,\mathfrak{t}}^{(\cc)}$
is a quotient of $V_{\cc,\mathfrak{t}}$.
\end{enumerate}
\end{theorem}
\begin{proof}
Fix $\omega\in\Omega^\cc_\Lambda$. 
We first show that
\begin{equation}\label{tttddd}
\pi_{\cc,\mathfrak{t}}^{\Lambda,\Lambda^\prime}(h)x^y=\pi_{\omega\cc,\omega\mathfrak{t}}^{Q^\vee,\Lambda^\prime}(h)x^y\qquad
(y\in\mathcal{O}_{\omega\cc})
\end{equation}
for $h=T_j$ ($0\leq j\leq r$), $h=x^\mu$ ($\mu\in Q^\vee$) and $h=\omega^\prime$ ($\omega^\prime\in\Omega_{\Lambda^\prime}$).
Observe that the right hand side is well defined since $\omega\cc\in C^{\omega(J)}$ and $\omega\mathfrak{t}\in T_{\Lambda^\prime,\omega(J)}$. Formula \eqref{tttddd} is trivially correct for $h=x^\mu$
($\mu\in Q^\vee$), and it follows from \eqref{compatibleOmegatranslate} when $h=\omega^\prime$ ($\omega^\prime\in\Omega_{\Lambda^\prime}$). For $h=T_j$ ($0\leq j\leq r$), formula \eqref{tttddd} follows from the observation that 
\begin{equation*}
\begin{split}
s_{a,\omega\mathfrak{t}}(x^y)=
(\omega\mathfrak{t})_y^{-m\alpha^\vee}
x^{s_\alpha y}=
\mathfrak{t}_{y;\cc}^{-m\alpha^\vee}x^{s_\alpha y}=
s_{a,\mathfrak{t};\cc}(x^y)
\end{split}
\end{equation*}
for $y\in\mathcal{O}_{\omega\cc}$ and $a=(\alpha,m)\in\Phi$,
where we have used \eqref{O4} in the second equality.

Proof of (1): In view of Lemma \ref{gluelemma} and the previous paragraph it suffices to show that the formulas
\eqref{actionformulasext} respect the defining relations of $\mathbb{H}_{\Lambda,\Lambda^\prime}$ involving the generators $x^\lambda$ ($\lambda\in\Lambda$).
The properties of the action $(w,f)\mapsto w_{\mathfrak{t};\cc}(f)$ of $w\in W_{\Lambda^\prime}$
on $f\in\mathcal{P}_\Lambda^{(\cc)}$ mentioned directly after Lemma \ref{actiondeformext} imply that
$\pi_{\cc,\mathfrak{t}}^{\Lambda,\Lambda^\prime}$ respect the relation $\omega x^\lambda=\omega(x^\lambda)\omega$ in $\mathbb{H}_{\Lambda,\Lambda^\prime}$ ($\omega\in\Omega_{\Lambda^\prime}$, $\lambda\in\Lambda$). 
Using \eqref{stepder}, 
the cross relations \eqref{crossX} for the operators
$\pi_{\cc,\mathfrak{t}}^{\Lambda,\Lambda^\prime}(T_j)$ and $\pi_{\cc,\mathfrak{t}}^{\Lambda,\Lambda^\prime}(x^\lambda)$
($0\leq j\leq r$, $\lambda\in\Lambda$) are verified by a direct computation.\\
Proof of (2): This is immediate from \eqref{tttddd} and part (1).\\
Proof of (3): This is clear from part (2) and Lemma \ref{Vct}.
\end{proof}
\begin{remark}\label{O3rem}\hfill
\begin{enumerate}
\item Cherednik's polynomial representation of the extended double affine Hecke algebra is 
\[
\pi_{0,1_{T_{\Lambda^\prime}}}^{\Lambda,\Lambda^\prime}: \mathbb{H}_{\Lambda,\Lambda^\prime}\rightarrow\textup{End}(\mathcal{P}_\Lambda).
\] 
\item
Let $\cc\in\overline{C}_+$ and $\cc^\prime\in\overline{C}_+\cap\mathcal{O}_{\Lambda,\cc}$. Then $\cc^\prime=\omega\cc$ for a unique $\omega\in\Omega_\Lambda$, and
\[
\pi_{\omega\mathfrak{t};\omega\cc}^{\Lambda,\Lambda^\prime}=\pi_{\mathfrak{t};\cc}^{\Lambda,\Lambda^\prime}\qquad\quad (\omega\in\Omega_\Lambda)
\]
by Example \ref{exampleTspec}(2) and \eqref{O3}. In particular, for $\zeta\in\Lambda_{\textup{min}}$ we have $\mathcal{P}_\Lambda^{(\zeta)}=\mathcal{P}_\Lambda$ and 
\begin{equation}\label{lateid}
\pi_{\zeta,\mathfrak{t}}^{\Lambda,\Lambda^\prime}=\pi_{0,w_{\zeta,\Lambda}^{-1}\mathfrak{t}}^{\Lambda,\Lambda^\prime}=\pi_{0,q^{-\zeta}\mathfrak{t}}^{\Lambda,\Lambda^\prime}
\end{equation}
for $\mathfrak{t}\in {}^\Lambda T_{\Lambda^\prime}^\zeta=q^\zeta T_{\Lambda^\prime,[1,r]}$ (see Example \ref{exampleTspec}(4)). For the second equality of \eqref{lateid} we used that
$w_{\zeta,\Lambda^\prime}^{-1}\mathfrak{t}=v_{\zeta,\Lambda^\prime}(q^{-\zeta}\mathfrak{t})=q^{-\zeta}\mathfrak{t}$ since $q^{-\zeta}\mathfrak{t}\in T_{\Lambda^\prime,[1,r]}\subseteq T_{\Lambda^\prime}^{W_0}$.
\end{enumerate}
\end{remark}

We now compare extended quasi-polynomial representations for different choices of lattices $(\Lambda,\Lambda^\prime)$.
\begin{proposition}\label{ordercompatiblerep}
Let $(\Lambda_1,\Lambda_1^\prime), (\Lambda_2,\Lambda_2^\prime)\in (\mathcal{L}^{\times 2})_e$ such that $(\Lambda_1,\Lambda_1^\prime)\leq (\Lambda_2,\Lambda_2^\prime)$. Let $\cc\in C^J$ and $\mathfrak{t}\in {}^{\Lambda_2}T_{\Lambda_2^\prime}^\cc$. 

Then the inclusion map 
$\mathcal{P}_{\Lambda_1}^{(\cc)}\hookrightarrow\mathcal{P}_{\Lambda_2}^{(\cc)}$
defines an embedding
\[
\mathcal{P}_{\Lambda_1,\mathfrak{t}\vert_{\Lambda_1^\prime}}^{(\cc)}\hookrightarrow\mathcal{P}_{\Lambda_2,\mathfrak{t}}^{(\cc)}
\]
of $\mathbb{H}_{\Lambda_1,\Lambda_1^\prime}$-modules, where we view $\mathbb{H}_{\Lambda_1,\Lambda_1^\prime}$ as subalgebra of $\mathbb{H}_{\Lambda_2,\Lambda_2^\prime}$ in the natural way.
\end{proposition}
\begin{proof}
We have $\Omega_{\Lambda_1}\subseteq\Omega_{\Lambda_2}$, and hence the restriction map $T_{\Lambda_2^\prime}\rightarrow T_{\Lambda_1^\prime}$,
$\mathfrak{t}\mapsto \mathfrak{t}\vert_{\Lambda_1^\prime}$ restricts to a group homomorphism
\[
{}^{\Lambda_2}T_{\Lambda_2^\prime}^\cc\rightarrow{}^{\Lambda_1}T_{\Lambda_1^\prime}^\cc.
\]
Furthermore, a direct check shows that
\[
\pi^{\Lambda_1,\Lambda_1^\prime}_{\cc,\mathfrak{t}\vert_{\Lambda_1^\prime}}(h)=\pi^{\Lambda_2,\Lambda_2^\prime}_{\cc,\mathfrak{t}}(h)\vert_{\mathcal{P}_{\Lambda_1}^{(\cc)}}
\qquad \forall\, h\in\mathbb{H}_{\Lambda_1,\Lambda_1^\prime}
\]
for $\mathfrak{t}\in {}^{\Lambda_2}T_{\Lambda_2^\prime}^\cc$. 
\end{proof}

The action of the commuting family of operators $\pi_{\cc,\mathfrak{t}}^{\Lambda,\Lambda^\prime}(Y^\mu)$ ($\mu\in\Lambda^\prime$) on the cyclic vector $x^\cc$ is described as follows.
\begin{lemma}\label{zetaY}
Let $\cc\in C^J$ and $\mathfrak{t}\in {}^{\Lambda}T_{\Lambda^\prime}^\cc$. Then
\begin{equation}\label{eveqextended}
\pi_{\cc,\mathfrak{t}}^{\Lambda,\Lambda^\prime}(Y^\mu)x^\cc=(\mathfrak{s}_J\mathfrak{t})^{-\mu}x^\cc
\qquad \forall\,\mu\in\Lambda^\prime.
\end{equation}
\end{lemma}
\begin{proof}
By Proposition \ref{ordercompatiblerep} and 
Theorem \ref{gbr}(2), formula \eqref{eveqextended} is correct for $\mu\in Q^\vee$.
It thus suffices to prove the lemma for $\mu\in\Lambda_{\textup{min}}^\prime$.
Fix $\zeta\in\Lambda_{\textup{min}}^\prime$.  Then
\begin{equation*}
\begin{split}
\pi_{\cc,\mathfrak{t}}^{\Lambda,\Lambda^\prime}(Y^\zeta)x^\cc=
\pi_{\cc,\mathfrak{t}}^{\Lambda,\Lambda^\prime}(w_{\zeta,\Lambda^\prime})\pi_{\cc,\mathfrak{t}}(T_{v_{\zeta,\Lambda^\prime}})x^\cc
&=\kappa_{v_{\zeta,\Lambda^\prime}}(\cc)\pi_{\cc,\mathfrak{t}}^{\Lambda,\Lambda^\prime}(w_{\zeta,\Lambda^\prime})x^{v_{\zeta,\Lambda^\prime}\cc}\\
&=
\kappa_{v_{\zeta,\Lambda^\prime}}(\cc)\bigl(w_{\zeta,\Lambda^\prime}\bigr)_{\mathfrak{t};\cc}(x^{v_{\zeta,\Lambda^\prime}\cc})=
\kappa_{v_{\zeta,\Lambda^\prime}}(\cc)
\mathfrak{t}^{-\zeta}x^\cc
\end{split}
\end{equation*}
with the first equality by \eqref{OmegaHH} and Proposition \ref{ordercompatiblerep},
the second equality by Corollary \ref{monomialcor} and the fourth equality by
\eqref{OmegaPol}. The result now follows from the fact that
\[
\mathfrak{s}_J^{-\zeta}=\mathfrak{s}_\cc^{-\zeta}=
\prod_{\alpha\in\Phi_0^+}
k_\alpha^{-\eta(\alpha(\cc))\alpha(\zeta)}=
\kappa_{v_{\zeta,\Lambda^\prime}}(\cc),
\]
where the last equality follows from \eqref{zetalength}.
\end{proof}

Corollary \ref{monomialcor} now extends as follows.
\begin{proposition}\label{monomialpropext}
Let $\cc\in C^J$ and $\mathfrak{t}\in{}^{\Lambda}T_{\Lambda^\prime}^\cc$. Define
the surjective linear map $\psi_{\cc,\mathfrak{t}}^{\Lambda,\Lambda^\prime}: \mathbb{H}_{\Lambda,\Lambda^\prime}\twoheadrightarrow\mathcal{P}_\Lambda^{(\cc)}$
by
\[
\psi_{\cc,\mathfrak{t}}^{\Lambda,\Lambda^\prime}(x^\lambda T_v Y^{\mu}):=
\kappa_v(\cc)(\mathfrak{s}_J\mathfrak{t})^{-\mu}x^{\lambda+v\cc}
\qquad (\lambda\in\Lambda,\, v\in W_0,\, \mu\in\Lambda^\prime).
\]
Then 
\begin{equation}\label{ttttdddd}
\psi_{\cc,\mathfrak{t}}^{\Lambda,\Lambda^\prime}(h)=\pi_{\cc,\mathfrak{t}}^{\Lambda,\Lambda^\prime}(h)x^\cc\qquad
\qquad\forall\,h\in\mathbb{H}_{\Lambda,\Lambda^\prime}.
\end{equation}
\end{proposition}
\begin{proof}
First note that it suffices to prove 
\eqref{ttttdddd} for $h\in H_{\Lambda^\prime}$ since
$\psi_{\cc,\mathfrak{t}}^{\Lambda,\Lambda^\prime}(x^\lambda h^\prime)=
x^\lambda\psi_{\cc,\mathfrak{t}}^{\Lambda,\Lambda^\prime}(h^\prime)$ 
and $\pi_{\cc,\mathfrak{t}}^{\Lambda,\Lambda^\prime}(x^\lambda h)x^\cc=
x^\lambda\pi_{\cc,\mathfrak{t}}^{\Lambda,\Lambda^\prime}(h)(x^\cc)$
for $\lambda\in\Lambda$ and $h^\prime\in H_{\Lambda^\prime}$. By a similar argument it follows from Lemma \ref{zetaY} and Theorem \ref{glueprop}(2) 
that it suffices to prove \eqref{ttttdddd} for $h\in H_0$.
But $\psi_{\cc,\mathfrak{t}}^{\Lambda,\Lambda^\prime}\vert_{\mathbb{H}}=\psi_{\cc,\mathfrak{t}\vert_{Q^\vee}}$, hence this follows from
Corollary \ref{monomialcor}.
\end{proof}

\begin{corollary}\label{corOmega}
Let $\cc\in\overline{C}_+$ and $\mathfrak{t}\in{}^{\Lambda}T_{\Lambda^\prime}^\cc$. Then
\[
\pi_{\cc,\mathfrak{t}}^{\Lambda,\Lambda^\prime}(S_\omega^Y)x^\cc=
\kappa_{D\omega}(\cc)x^{\omega\cc}\qquad (\omega\in\Omega_\Lambda).
\]
In particular, $\pi_{\cc,\mathfrak{t}}^{\Lambda,\Lambda^\prime}(S_\omega^Y)x^\cc=x^\cc$ for $\omega\in\Omega_{\Lambda,\cc}$, and 
\[
S_\omega^Y\otimes_{\mathbb{H}_{Q^\vee,\Lambda^\prime}}f\mapsto 
\pi_{\cc,\mathfrak{t}}^{\Lambda,\Lambda^\prime}(S_\omega^Y)f\qquad
(\omega\in\Omega_\Lambda,\, f\in\mathcal{P}^{(\cc)})
\]
defines an epimorphism $V_{\cc,\mathfrak{t}}\twoheadrightarrow \mathcal{P}_{\Lambda,\mathfrak{t}}^{(\cc)}$ of $\mathbb{H}_{\Lambda,\Lambda^\prime}$-modules.
\end{corollary}
\begin{proof}
Write $\omega=w_{\zeta,\Lambda}\in\Omega_\Lambda$ ($\zeta\in\Lambda_{\textup{min}}$), Then we have 
\[
\pi_{\cc,\mathfrak{t}}^{\Lambda,\Lambda^\prime}(S_{w_{\zeta,\Lambda}}^Y)x^\cc=\pi_{\cc,\mathfrak{t}}^{\Lambda,\Lambda^\prime}(x^\zeta
T_{v_{\zeta,\Lambda}^{-1}})x^\cc=\kappa_{v_{\zeta,\Lambda}^{-1}}(\cc)x^{\zeta+v_{\zeta,\Lambda}^{-1}\cc}=
\kappa_{Dw_{\zeta,\Lambda}}(\cc)x^{w_{\zeta,\Lambda}\cc},
\]
where we used Proposition \ref{monomialpropext} in the second equality. The last statement now follows from Theorem \ref{glueprop} and the proof of Lemma \ref{Vct}.
\end{proof}
\subsection{Twist parameters}\label{Tp}
The group isomorphism $\Lambda^\prime/Q^\vee\overset{\sim}{\longrightarrow}\Omega_{\Lambda^\prime}$ from Proposition \ref{Omegaprop}(1)
allows one to identify ${}^{\Lambda}T_{\Lambda^\prime}^0=T_{\Lambda^\prime,[1,r]}\simeq\textup{Hom}(\Lambda^\prime/Q^\vee,\mathbf{F}^\times)$ with the character group $\textup{Hom}(\Omega_{\Lambda^\prime},\mathbf{F}^\times)$ of $\Omega_{\Lambda^\prime}$. Concretely, the multiplicative character of $\Omega_{\Lambda^\prime}$ corresponding to $\mathfrak{t}\in T_{\Lambda^\prime,[1,r]}$ is given by $w_{\zeta,\Lambda^\prime}\mapsto \mathfrak{t}^\zeta$ ($\zeta\in\Lambda_{\textup{min}}^\prime$). Note that
$w_{\cc_\mu,\Lambda^\prime}\mapsto \mathfrak{t}^\mu$ for $\mu\in\Lambda^\prime$. 

\begin{definition}
For $\mathfrak{t}\in T_{\Lambda^\prime,[1,r]}$ let $\Xi_{\mathfrak{t}}$ be the algebra automorphism
of the extended double affine Hecke algebra $\mathbb{H}_{\Lambda,\Lambda^\prime}\simeq\Omega_{\Lambda^\prime}\ltimes\mathbb{H}_{\Lambda,Q^\vee}$ defined by 
$h\mapsto h$\, \textup{(}$h\in\mathbb{H}_{\Lambda,Q^\vee}$\textup{)} and $w_{\zeta,\Lambda^\prime}\mapsto \mathfrak{t}^{-\zeta} w_{\zeta,\Lambda^\prime}$\, \textup{(}$\zeta\in\Lambda_{\textup{min}}^\prime$\textup{)}.
\end{definition}
 In particular, $\Xi_{\mathfrak{t}}(x^\lambda)=x^\lambda$ and $\Xi_{\mathfrak{t}}(T_j)=T_j$ for $\lambda\in\Lambda$ and $j\in [0,r]$.

\begin{proposition}\label{twistactionprop}
Let $\cc\in C^J$, $\mathfrak{t}^\prime\in {}^{\Lambda}T_{\Lambda^\prime}^\cc$ and $\mathfrak{t}\in T_{\Lambda^\prime,[1,r]}$. Then 
\begin{equation}\label{twistaction}
\pi_{\cc,\mathfrak{t}^\prime\mathfrak{t}}^{\Lambda,\Lambda^\prime}=\pi_{\cc,\mathfrak{t}^\prime}^{\Lambda,\Lambda^\prime}\circ\Xi_{\mathfrak{t}}.
\end{equation}
Furthermore, $\Xi_{\mathfrak{t}}(Y^\mu)=\mathfrak{t}^{-\mu}Y^\mu$ for all $\mu\in\Lambda^\prime$.
\end{proposition}
\begin{proof}
Fix $\mathfrak{t}^\prime\in {}^{\Lambda}T_{\Lambda^\prime}^\cc$ and $\mathfrak{t}\in T_{\Lambda^\prime,[1,r]}$.
Then $\mathfrak{t}^\prime\mathfrak{t}\in{}^{\Lambda^\prime}T_{\Lambda}^\cc$ by Example \ref{exampleTspec}(3), and
\begin{equation}\label{twistpart}
(\mathfrak{t}^\prime\mathfrak{t})_{y;\cc}=\mathfrak{t}^\prime_{y;\cc}\mathfrak{t}\qquad \forall\, y\in\mathcal{O}_{\Lambda,\cc}.
\end{equation}
Indeed, if $g\in W_{\Lambda}$ is such that $y=g\cc$ then $(\mathfrak{t}^\prime\mathfrak{t})_{y;\cc}=g(\mathfrak{t}^\prime\mathfrak{t})=(g\mathfrak{t}^\prime)((Dg)\mathfrak{t})=\mathfrak{t}^\prime_{y;\cc}\mathfrak{t}$ since $Dg\in W_0$ and $\mathfrak{t}\in T_{\Lambda^\prime}^{W_0}$.

Let $\Theta_{\mathfrak{t}}: W_{\Lambda^\prime}\rightarrow\mathbf{F}^\times$ be the group homomorphism defined by $w\mapsto 1$ ($w\in W$) and $w_{\zeta,\Lambda^\prime}\mapsto \mathfrak{t}^{-\zeta}$ ($\zeta\in\Lambda^\prime_{\textup{min}}$). Since $\tau(\mu)\in w_{\cc_\mu,\Lambda^\prime}W$ for all $\mu\in\Lambda^\prime$, we have
\[
\Theta_{\mathfrak{t}}(\tau(\mu))=\mathfrak{t}^{-\mu}\qquad\quad \forall\, \mu\in\Lambda^\prime.
\]
We claim that
\begin{equation}\label{twistactionclassical}
w_{\mathfrak{t}^\prime\mathfrak{t};\cc}(x^y)=\Theta_{\mathfrak{t}}(w)w_{\mathfrak{t}^\prime;\cc}(x^y)
\end{equation}
for $w\in W_{\Lambda^\prime}$ and $y\in\mathcal{O}_{\Lambda,\cc}$. 
Formula \eqref{twistactionclassical} is trivial for $w\in W_0$. If $w=\tau(\mu)\in W$ ($\mu\in\Lambda^\prime$) then it follows from \eqref{twistpart} that
\[
\tau(\mu)_{\mathfrak{t}^\prime\mathfrak{t};\cc}(x^y)=(\mathfrak{t}^\prime_{y;\cc}\mathfrak{t})^{-\mu}x^y=\Theta_{\mathfrak{t}}(\tau(\mu))\tau(\mu)_{\mathfrak{t}^\prime;\cc}(x^y)
\]
for $y\in\mathcal{O}_{\Lambda,\cc}$, which completes the proof of \eqref{twistactionclassical}.

Formula \eqref{twistaction} follows immediately from 
\eqref{twistactionclassical}.
Finally, we have $\Xi_{\mathfrak{t}}(Y^\mu)=\mathfrak{t}^{-\mu}Y^\mu$ for $\mu\in\Lambda^\prime$ since $Y^\mu\in w_{\cc_\mu,\Lambda^\prime}H$ by \eqref{Ywhere}.
\end{proof}
Proposition \ref{twistactionprop} implies, under suitable generic conditions on $\mathfrak{t}\in {}^\Lambda T_{\Lambda^\prime}^\cc$, that the simultaneous eigenfunctions of the commuting operators
$\pi_{\cc,\mathfrak{t}}^{\Lambda,\Lambda^\prime}(Y^\mu)$ ($\mu\in \Lambda^\prime$) will only depend on the coset $\mathfrak{t}T_{\Lambda^\prime,[1,r]}$ inside
${}^\Lambda T_{\Lambda^\prime}^\cc$. The precise statement will be given in Theorem \ref{propEVext0}(3).

It follows from Remark \ref{O3rem} and Proposition \ref{twistactionprop} that
for $\zeta\in\Lambda_{\textup{min}}$ and $\mathfrak{t}\in {}^{\Lambda}T_{\Lambda^\prime}^\zeta=q^{-\zeta}T_{\Lambda^\prime,[1,r]}$, the extended quasi-polynomial
representation
\[
\pi_{\zeta,\mathfrak{t}}^{\Lambda,\Lambda^\prime}:
\mathbb{H}_{\Lambda,\Lambda^\prime}\rightarrow\textup{End}(\mathcal{P}_\Lambda)
\]
is the twisted version $\pi_{0,1_{T_{\Lambda^\prime}}}^{\Lambda,\Lambda^\prime}\circ\Xi_{q^{-\zeta}\mathfrak{t}}$
of the Cherednik representation $\pi_{0,1_{T_{\Lambda^\prime}}}^{\Lambda,\Lambda^\prime}$. The corresponding commuting $Y$-operators are 
related by the formula
\[
\pi_{\zeta,\mathfrak{t}}^{\Lambda,\Lambda^\prime}(Y^\mu)=q^{\langle\zeta,\mu\rangle}\mathfrak{t}^{-\mu}\,\pi_{0,1_{T_{\Lambda^\prime}}}^{\Lambda,\Lambda^\prime}(Y^\mu)
\qquad\quad\forall\,\mu\in\Lambda^\prime.
\]

\subsection{The extended eigenvalue equations}\label{S72}
In this subsection we consider the monic quasi-polynomial simultaneous eigenfunctions of $\pi_{\cc,\mathfrak{t}}^{\Lambda,\Lambda^\prime}(Y^\mu)$ ($\mu\in\Lambda^\prime$).

\begin{theorem}\label{propEVext0}
Let $\cc\in C^J$, and let $\mathcal{O}$ be a $W$-orbit in $\mathcal{O}_{\Lambda,\cc}$. Let $\mathfrak{t}\in{}^{\Lambda}T_{\Lambda^\prime}^\cc$ such that 
\begin{equation}\label{eecond}
(\mathfrak{s}_y\mathfrak{t}_{y;\cc})\vert_{\Lambda^\prime}\not=(\mathfrak{s}_{y^\prime}\mathfrak{t}_{y^\prime;\cc})\vert_{\Lambda^\prime}\quad\hbox{ when }\,\, y,y^\prime\in 
\mathcal{O}\,\, \hbox{ and }\,\, y\not=y^\prime.
\end{equation}
For all $y\in\mathcal{O}$ we have
\begin{enumerate}
\item 
There exists a unique simultaneous eigenfunction
\[
E_{y;\cc}^{\Lambda,\Lambda^\prime}(x;\mathfrak{t})\in\mathcal{P}^{(\cc)}_\Lambda
\]
of the commuting operators
$\pi_{\cc,\mathfrak{t}}^{\Lambda,\Lambda^\prime}(Y^\mu)$ \textup{(}$\mu\in\Lambda^\prime$\textup{)} satisfying 
\begin{equation}\label{lotExt}
E_{y;\cc}^{\Lambda,\Lambda^\prime}(x;\mathfrak{t})=x^{y}+\textup{l.o.t.}
\end{equation}
\item We have
\begin{equation}\label{evYY}
\pi_{\cc,\mathfrak{t}}^{\Lambda,\Lambda^\prime}(Y^\mu)E_{y;\cc}^{\Lambda,\Lambda^\prime}(\cdot;\mathfrak{t})=
(\mathfrak{s}_y\mathfrak{t}_{y;\cc})^{-\mu}E_{y;\cc}^{\Lambda,\Lambda^\prime}(\cdot;\mathfrak{t})\qquad\quad\forall\, \mu\in\Lambda^\prime.
\end{equation}
\item 
If $\mathfrak{t}$ satisfies the more stringent genericity conditions
\begin{equation}\label{eecond2}
(\mathfrak{s}_y\mathfrak{t}_{y;\cc})\vert_{Q^\vee}\not=(\mathfrak{s}_{y^\prime}\mathfrak{t}_{y^\prime;\cc})\vert_{Q^\vee}\quad\hbox{ when }\,\, y,y^\prime\in 
\mathcal{O}\,\, \hbox{ and }\,\, y\not=y^\prime
\end{equation}
then
\begin{equation}\label{Eyc}
E_{y;\cc}^{\Lambda,\Lambda^\prime}(x;\mathfrak{t})=E_y^{\omega(J)}(x;(\omega\mathfrak{t})\vert_{Q^\vee}),
\end{equation}
where $\omega\in\Omega_\Lambda$ is such that $\mathcal{O}=\mathcal{O}_{\omega\cc}$.
\end{enumerate}
\end{theorem}
\begin{proof}
(1)\&(2). It suffices to prove that 
\begin{equation}\label{Eyctr}
\pi_{\cc,\mathfrak{t}}^{\Lambda,\Lambda^\prime}(Y^\mu)x^y=(\mathfrak{s}_y\mathfrak{t}_{y;\cc})^{-\mu}x^y+\textup{l.o.t.}\qquad\quad (\mu\in\Lambda^\prime).
\end{equation}
By Proposition \ref{ordercompatiblerep},
Theorem \ref{glueprop}(2), 
\eqref{O3} and Proposition \ref{trianY} this is correct for $\mu\in Q^\vee$.
A straightforward extension of the proof of Proposition \ref{trianY} shows that it holds true for all $\mu\in\Lambda^\prime$.\\
(3) The extra assumption on $\mathfrak{t}$ implies that $(\omega\mathfrak{t})\vert_{Q^\vee}\in T_{\omega(J)}^\prime$, hence $E_y^{\omega(J)}\bigl(x;(\omega\mathfrak{t})\vert_{Q^\vee}\bigr)$ is well defined. The result is now an immediate consequence of Theorem \ref{glueprop}(2) and Theorem \ref{Edef}(1).
\end{proof}
As a special case of Theorem \ref{propEVext0}(3) we have for $\cc\in C^J$ and $\mathfrak{t}\in T_J^\prime$,
\[
E_y^J(x;\mathfrak{t})=E_{y;\cc}^{Q^\vee,Q^\vee}(x;\mathfrak{t})\qquad\quad (y\in\mathcal{O}_\cc).
\]

The index $\cc$ is to indicate that $E_{y;\cc}^{\Lambda,\Lambda^\prime}(x;\mathfrak{t})$ depends on the choice of the representative of $\mathcal{O}_{\Lambda,\cc}$ inside $\overline{C}_+$.
The quasi-polynomials for different choices of representatives are related as follows.
\begin{corollary}
Let $\cc\in\overline{C}_+$ and choose a $W$-orbit $\mathcal{O}$ in $\mathcal{O}_{\Lambda,\cc}$.
 Fix $\mathfrak{t}\in {}^{\Lambda}T_{\Lambda^\prime}^\cc$ satisfying \eqref{eecond}. Then $\omega\mathfrak{t}\in T_{\Lambda^\prime}$
lies in $\omega\mathfrak{t}\in{}^{\Lambda}T_{\Lambda^\prime}^{\omega\cc}$ and satisfies \eqref{eecond}, and
\begin{equation}\label{indepCplus}
E_{y;\omega\cc}^{\Lambda,\Lambda^\prime}(x;\omega\mathfrak{t})=E_{y;\cc}^{\Lambda,\Lambda^\prime}(x;\mathfrak{t})\qquad \forall\, y\in\mathcal{O}.
\end{equation}
\end{corollary}
\begin{proof}
This follows immediately from the characterisation of $E_{y;\cc}^{\Lambda,\Lambda^\prime}(x;\mathfrak{t})$ as given in Theorem \ref{propEVext0}(1).
\end{proof}
The following is an extension of Theorem \ref{propEVext0}(3).
\begin{corollary}\label{latticedependenceE}
Let $(\Lambda_1,\Lambda_1^\prime), (\Lambda_2,\Lambda_2^\prime)\in (\mathcal{L}^{\times 2})_e$ such that $(\Lambda_1,\Lambda_1^\prime)\leq (\Lambda_2,\Lambda_2^\prime)$. Let $\cc\in\overline{C}_+$ and choose a $W$-orbit $\mathcal{O}$ in $\mathcal{O}_{\Lambda_1,\cc}$. Fix  
$\mathfrak{t}\in {}^{\Lambda_2}T_{\Lambda_2^\prime}^\cc$ such that \eqref{eecond} holds true for $\Lambda^\prime=\Lambda_1^\prime$.
Then for all $y\in\mathcal{O}$,
\begin{equation}\label{relpolLambda}
E_{y;\cc}^{\Lambda_1,\Lambda_1^\prime}(x;\mathfrak{t}\vert_{\Lambda_1^\prime})=E_{y;\cc}^{\Lambda_2,\Lambda_2^\prime}(x;\mathfrak{t}).
\end{equation}
In particular, the quasi-polynomial $E_{y;\cc}^{\Lambda_1,\Lambda_1^\prime}(x;\mathfrak{t}\vert_{\Lambda_1^\prime})$ \textup{(}$y\in\mathcal{O}$\textup{)} satisfies the extended eigenvalue equations
\begin{equation}\label{evYYext}
\pi_{\cc,\mathfrak{t}}^{\Lambda_2,\Lambda_2^\prime}(Y^\mu)E_{y;\cc}^{\Lambda_1,\Lambda_1^\prime}(\cdot;\mathfrak{t}\vert_{\Lambda_1^\prime})=
(\mathfrak{s}_y\mathfrak{t}_{y;\cc})^{-\mu}E_{y;\cc}^{\Lambda_1,\Lambda_1^\prime}(\cdot;\mathfrak{t}\vert_{\Lambda_1^\prime})\qquad\quad\forall\, \mu\in\Lambda_2^\prime.
\end{equation}
\end{corollary}
\begin{proof}
Note that \eqref{eecond} also also holds true for $\Lambda^\prime=\Lambda_2^\prime$ since $\Lambda_1^\prime\subseteq\Lambda_2^\prime$. Hence both sides of
\eqref{relpolLambda} are well defined. The result now follows from Theorem \ref{propEVext0}(1) and Proposition \ref{ordercompatiblerep}.
\end{proof}

\begin{remark}\label{zetatranslaterem}
Let $\zeta\in\Lambda\cap E_{\textup{co}}$ and impose the assumptions of Theorem \ref{propEVext0}(1).
Then \eqref{eecond} also holds true for
$\tau(\zeta)\mathcal{O}$ since $\mathfrak{s}_{y+\zeta}=\mathfrak{s}_y$, and 
\begin{equation*}
E_{y+\zeta;\cc}^{\Lambda,\Lambda^\prime}(x;\mathfrak{t})=
x^\zeta E_{y;\cc}^{\Lambda,\Lambda^\prime}(x;\mathfrak{t})\qquad\quad
(\zeta\in\Lambda\cap E_{\textup{co}})
\end{equation*}
for all $y\in\mathcal{O}$ 
(compare with Lemma \ref{remE1}(2) and its proof).
\end{remark}

Under appropriate generic conditions on $\mathfrak{t}$
one can now transport the results on
the quasi-polynomial eigenfunctions $E_y^J(x;\mathfrak{t})$ from Section \ref{genSection} to $E_{y;\cc}^{\Lambda,\Lambda^\prime}(x;\mathfrak{t})$ 
using Theorem \ref{propEVext0}(3). We finish this subsection by giving an explicit example.

For an element $w\in W_\Lambda$ define $d_w\in\mathcal{P}$ and $k_w(y)$ by \eqref{dwt} and \eqref{kawy}, respectively. 

\begin{proposition}\label{propEVext}
Let $\cc\in C^J$ and choose a $W$-orbit $\mathcal{O}$ in $\mathcal{O}_{\Lambda,\cc}$. Let $\omega\in\Omega_\Lambda$
such that $\mathcal{O}=\mathcal{O}_{\omega\cc}$.
Fix $\mathfrak{t}\in {}^{\Lambda}T_{\Lambda^\prime}^\cc$ satisfying \eqref{eecond2} and suppose that $\mathfrak{t}$ is $J$-regular. Then 

\begin{enumerate}
\item For $y\in\mathcal{O}$
we have
\[
E_{y;\cc}^{\Lambda,\Lambda^\prime}(x;\mathfrak{t})=
d_{w_y\omega}(\mathfrak{s}_J\mathfrak{t})^{-1}k_{w_y\omega}(\cc)^{-1}
\pi_{\cc,\mathfrak{t}}^{\Lambda,\Lambda^\prime}(S_{w_y\omega}^Y)x^\cc.
\]
\item Assume that \eqref{eecond2} also holds true for the $W$-orbit $\omega^{-1}\mathcal{O}=\mathcal{O}_\cc$ in $\mathcal{O}_{\Lambda,\cc}$. Then
\[
E_{\omega y;\cc}^{\Lambda,\Lambda^\prime}(\cdot;\mathfrak{t})=k_\omega(y)^{-1}\pi_{\cc,\mathfrak{t}}^{\Lambda,\Lambda^\prime}(S_\omega^Y)E_y^J(\cdot;\mathfrak{t}\vert_{Q^\vee})\qquad\quad (y\in\mathcal{O}_\cc).
\]
\end{enumerate}
\end{proposition}
\begin{proof}
(1) Fix $y\in\mathcal{O}_{\omega\cc}$. Then Corollary \ref{corOmega}, \eqref{tttddd} and Proposition \ref{ordercompatiblerep} give
\[
\pi_{\cc,\mathfrak{t}}^{\Lambda,\Lambda^\prime}(S_{w_y\omega}^Y)x^\cc=
\kappa_{D\omega}(\cc)\pi_{\omega\cc,\omega\mathfrak{t}}(S_{w_y}^Y)x^{\omega\cc}.
\]
Applying Theorem \ref{prop:I_and_E} we get
\[
\pi_{\cc,\mathfrak{t}}^{\Lambda,\Lambda^\prime}(S_{w_y\omega}^Y)x^\cc=
k_{w_y}(\omega\cc)\kappa_{D\omega}(\cc)d_{w_y}(\mathfrak{s}_{\omega(J)}\omega\mathfrak{t})
E_y^{\omega(J)}(x;(\omega\mathfrak{t})\vert_{Q^\vee}),
\]
so by \eqref{Eyc} it suffices to show that
\begin{equation}\label{t6d6}
\begin{split}
d_{w_y}(\mathfrak{s}_{\omega(J)}\omega\mathfrak{t})&=d_{w_y\omega}(\mathfrak{s}_\cc\mathfrak{t}),\\
k_{w_y}(\omega\cc)\kappa_{D\omega}(\cc)&=k_{w_y\omega}(\cc).
\end{split}
\end{equation}
For the first equation of \eqref{t6d6} note that by Lemma \ref{Omegas} we have
\[
d_{w_y}(\mathfrak{s}_{\omega(J)}\omega\mathfrak{t})=d_{w_y}(\omega(\mathfrak{s}_{J}\mathfrak{t}))=
(\omega^{-1}d_{w_y})(\mathfrak{s}_{J}\mathfrak{t}).
\]
But $\Pi(w_y\omega)=\omega^{-1}\Pi(w_y)$ since $\omega\in\Omega_\Lambda$, hence
$\omega^{-1}d_{w_y}=d_{w_y\omega}$ in $\mathcal{P}$. 
For the second equation it suffices to show that 
\begin{equation}\label{t7d7}
k_\omega(y)=\kappa_{D\omega}(y).
\end{equation}
Write $\omega=w_{\zeta,\Lambda}$ ($\zeta\in\Lambda_{\textup{min}}$), then we compute for $y\in E$,
\begin{equation*}
\begin{split}
k(w_{\zeta,\Lambda}y)&=\prod_{\alpha\in\Phi_0^+}k_\alpha^{\eta(\alpha(v_{\zeta,\Lambda}^{-1}y+\zeta))/2}\\
&=\prod_{\alpha\in\Pi(v_{\zeta,\Lambda})}k_\alpha^{\eta(\alpha(v_{\zeta,\Lambda}^{-1}y)+1)/2}
\prod_{\beta\in\Phi_0^+\setminus\Pi(v_{\zeta,\Lambda})}k_\beta^{\eta(\beta(v_{\zeta,\Lambda}^{-1}y))/2}\\
&=\prod_{\alpha\in\Pi(v_{\zeta,\Lambda})}k_\alpha^{-\eta(-\alpha(v_{\zeta,\Lambda}^{-1}y))/2}
\prod_{\beta\in\Phi_0^+\setminus\Pi(v_{\zeta,\Lambda}^{-1})}k_\beta^{\eta(\beta(y))/2}\\
&=\prod_{\alpha\in\Pi(v_{\zeta,\Lambda}^{-1})}k_\alpha^{-\eta(\alpha(y))/2}
\prod_{\beta\in\Phi_0^+\setminus\Pi(v_{\zeta,\Lambda}^{-1})}k_\beta^{\eta(\beta(y))/2}=
k(y)\kappa_{v_{\zeta,\Lambda}^{-1}}(y),
\end{split}
\end{equation*}
where we have used \eqref{zetalength} in the second equality,
\eqref{etaaffine} and 
\[
v_{\zeta,\Lambda}(\Phi_0^+\setminus\Pi(v_{\zeta,\Lambda}))=\Phi_0^+\setminus\Pi(v_{\zeta,\Lambda}^{-1})
\]
in the third equality, and $v_{\zeta,\Lambda}(\Pi(v_{\zeta,\Lambda}))=-\Pi(v_{\zeta,\Lambda}^{-1})$ in the fourth equality.  Now 
\eqref{t7d7} follows, since $Dw_{\zeta,\Lambda}=v_{\zeta,\Lambda}^{-1}$.\\
(2) Let $\omega\in\Omega_\Lambda$ and $y\in\mathcal{O}_\cc$. Then 
$\omega w_{y}\omega^{-1}=w_{\omega y}$, Proposition \ref{prop:I_and_E} and part (1) of the proposition give
\begin{equation*}
\begin{split}
\pi_{\cc,\mathfrak{t}}^{\Lambda,\Lambda^\prime}(S_\omega^Y)E_y^J(\cdot;\mathfrak{t}\vert_{Q^\vee})&=
d_{w_y}(\mathfrak{s}_J\mathfrak{t})^{-1}
k_{w_y}(\cc)^{-1}\pi_{\cc,\mathfrak{t}}^{\Lambda,\Lambda^\prime}(S_{w_{\omega y}\omega}^Y)x^\cc\\
&=\frac{d_{\omega w_y}(\mathfrak{s}_J\mathfrak{t})k_{\omega w_y}(\cc)}
{d_{w_y}(\mathfrak{s}_J\mathfrak{t})k_{w_y}(\cc)}
E_{\omega y;\cc}^{\Lambda,\Lambda^\prime}(\cdot;\mathfrak{t}).
\end{split}
\end{equation*}
Note that $d_{\omega w}=d_w$ ($\omega\in\Omega_\Lambda$, $w\in W_\Lambda$) in $\mathcal{P}$ since
$\Pi(\omega w)=\Pi(w)$. The result then follows from the observation that
\[
\frac{k_{\omega w_y}(\cc)}{k_{w_y}(\cc)}=\frac{k(\omega w_y\cc)}{k(w_y\cc)}=
\frac{k(\omega y)}{k(y)}=k_\omega(y).
\]
\end{proof}

\subsection{The theory for the $\textup{GL}_{r+1}$ root datum}\label{S73}
 In this subsection we make the results of the previous two subsections explicit for the root datum associated to $\textup{GL}_{r+1}$.
 
Let $\{\epsilon_i\}_{i=1}^{r+1}$ be the standard orthonormal basis of $\mathbb{R}^{r+1}$, and fix a positive integer $\ell\in\mathbb{Z}_{>0}$.
Then 
\[
\Phi_0:=\bigl\{(\epsilon_i-\epsilon_j)/\ell\,\,\,\, | \,\,\,\, 1\leq i\not=j\leq r+1\bigr\}\subset\mathbb{R}^{r+1}
\]
is a root system of type $A_r$, with roots having squared norm $2/\ell^2$ (this particular normalisation of the basis will be convenient for the metaplectic theory, see Sections  \ref{MetaplecticSection} and \ref{MfinalSection}). 
Take $\alpha_i:=(\epsilon_i-\epsilon_{i+1})/\ell$ \textup{(}$1\leq i\leq r$\textup{)} as the simple roots of $\Phi_0$.
Then $\varphi=(\epsilon_1-\epsilon_{r+1})/\ell$, and hence
\[
\alpha_0=\bigl((\epsilon_{r+1}-\epsilon_1)/\ell,1\bigr).
\]
The alcove $\overline{C}_+$ consists of
the vectors $y=(y_1,\ldots,y_{r+1})\in\mathbb{R}^{r+1}$ satisfying
\begin{equation}\label{ineq}
y_{r+1}\leq y_r\leq\cdots \leq y_1\leq y_{r+1}+\ell.
\end{equation}
The face $C^J$ consists of the vectors $y\in\overline{C}_+$ such that
\begin{enumerate}
\item $y_{i+1}=y_{i}$ if $i\in J_0$, 
\item $y_1=y_{r+1}+\ell$ if $0\in J$, 
\item the remaining inequalities in \eqref{ineq} are strict. 
\end{enumerate}

We now take $(\Lambda, \Lambda^\prime):=(\ell\mathbb{Z}^{r+1},\ell\mathbb{Z}^{r+1})\in (\mathcal{L}^{\times 2})_1$, with
$\ell\mathbb{Z}^{r+1}$ the lattice generated by the basis $\{\ell\epsilon_i\}_{i=1}^{r+1}$ of $\mathbb{R}^{r+1}$.
The abelian group 
$\Omega_{\ell\mathbb{Z}^{r+1}}$ is a free group of rank $1$ generated by 
\[
u:=w_{\ell\epsilon_1,\ell\mathbb{Z}^{r+1}}=\tau(\ell\epsilon_1)s_1\cdots s_r.
\]
Then $v_{\ell\epsilon_1,\ell\mathbb{Z}^{r+1}}=s_r\cdots s_2s_1$ and
 $u\alpha_j=\alpha_{j+1}$ ($0\leq j\leq r$) by \eqref{linearaction}, where the indices are read modulo $r+1$ \textup{(}this will be done throughout this example\textup{)}.

Note that $(\mathbb{R}^{r+1})_{\textup{co}}=
\mathbb{R}\varpi$ and $(\mathbb{R}^{r+1})_{\textup{co}}\cap\ell\mathbb{Z}^{r+1}=
\mathbb{Z}\varpi$ with
\[
\varpi:=\ell(\epsilon_1+\cdots+\epsilon_{r+1}).
\]
Furthermore, $I_{\ell\mathbb{Z}^{r+1}}=[1,r]$ (see \eqref{Ilambda}).
A complete set of representatives of the
$(\mathbb{R}^{r+1})_{\textup{co}}\cap\ell\mathbb{Z}^{r+1}$-orbits inside $(\ell\mathbb{Z}^{r+1})_{\textup{min}}$ is 
$\{0\}\cup\{\varpi_{i,\ell\mathbb{Z}^{r+1}}^\vee\}_{i=1}^r$ 
with 
\[
\varpi_{i,\ell\mathbb{Z}^{r+1}}^\vee:=\ell(\epsilon_1+\cdots+\epsilon_i)\qquad (1\leq i\leq r).
\]
For $s,i\in\mathbb{Z}$ we have
\[
u^{s(r+1)+i}=\tau(s\varpi)u^i
\]
and
\[
u^i=w_{\varpi_{i,\ell\mathbb{Z}^{r+1}}^\vee}\quad (1\leq i\leq r)
\]
(see, e.g., \cite[\S 9.3.5]{StAB}). 
\begin{lemma}
We have
\begin{equation}\label{stabOmega}
\Omega_{\ell\mathbb{Z}^{r+1},\cc}=\{1\}\qquad \forall\, \cc\in \overline{C}_+.
\end{equation}
In particular, ${}^{\ell\mathbb{Z}^{r+1}}T_{\ell\mathbb{Z}^{r+1}}^\cc=T_{\ell\mathbb{Z}^{r+1},J}$ for $\cc\in C^J$ and 
\[
\mathcal{O}_{\ell\mathbb{Z}^{r+1},\cc}=\bigsqcup_{(s,i)\in\mathbb{Z}\times [0,r]}\mathcal{O}_{u^i\cc+s\varpi}.
\]
\end{lemma}
\begin{proof}
If $u^s\in\Omega_{\ell\mathbb{Z}^{r+1},\cc}$ then 
$\cc=u^{s(r+1)}\cc=\cc+s\varpi$, hence $s=0$. 
\end{proof}

For $\cc\in C^J$ and $\mathfrak{t}\in T_{\ell\mathbb{Z}^{r+1},J}$ the $A_r$-type quasi-polynomial representation $\pi_{\cc,\mathfrak{t}}: \mathbb{H}\rightarrow\textup{End}(\mathcal{P}^{(\cc)})$ can thus be extended to a representation
\[
\widetilde{\pi}_{\cc,\mathfrak{t}}:=\pi_{\cc,\mathfrak{t}}^{\ell\mathbb{Z}^{r+1},\ell\mathbb{Z}^{r+1}}:\widetilde{\mathbb{H}}\rightarrow\textup{End}(\mathcal{P}_{\ell\mathbb{Z}^{r+1}}^{(\cc)})
\]
of the double affine Hecke algebra
\begin{equation}\label{tildeHH}
\widetilde{\mathbb{H}}:=\mathbb{H}_{\ell\mathbb{Z}^{r+1},\ell\mathbb{Z}^{r+1}}
\end{equation}
of type $\textup{GL}_{r+1}$.  
Note that $\widetilde{\mathbb{H}}$ depends on a single multiplicity parameter $k\in\mathbf{F}^\times$.
Algebraic generators of $\widetilde{\mathbb{H}}$ are $x^\lambda$ ($\lambda\in\ell\mathbb{Z}^{r+1}$),
$T_i$ ($1\leq i\leq r$) and $u^{\pm 1}$ (note that $T_0=uT_ru^{-1}$). The corresponding commuting elements $Y^\mu\in\widetilde{\mathbb{H}}$ for 
$\mu=\ell\sum_{i=1}^{r+1}m_i\epsilon_i\in\ell\mathbb{Z}^{r+1}$ ($m_i\in\mathbb{Z}$)
are explicitly given by 
\[
Y^\mu=(Y^{\ell\epsilon_1})^{m_1}\cdots \bigr(Y^{\ell\epsilon_{r+1}}\bigr)^{m_{r+1}}
\]
with
\begin{equation}\label{Yi}
Y^{\ell\epsilon_i}:=T_{i-1}^{-1}\cdots T_2^{-1}T_1^{-1}uT_r\cdots T_{i+1}T_i
\qquad (1\leq i\leq r+1).
\end{equation}
For $i=r+1$ this should be read as $Y^{\ell\epsilon_{r+1}}=T_r^{-1}\cdots T_2^{-1}T_1^{-1}u$.

For $\cc\in C^J$ and $\mathfrak{t}\in T_{\ell\mathbb{Z}^{r+1},J}$ the extended quasi-polynomial
representation
\[
\widetilde{\pi}_{\cc,\mathfrak{t}}: \widetilde{\mathbb{H}}\rightarrow\textup{End}(\mathcal{P}_{\ell\mathbb{Z}^{r+1}}^{(\cc)})
\]
is then characterised by the formulas
\begin{equation*}
\begin{split}
\widetilde{\pi}_{\cc,\mathfrak{t}}(T_i)x^y&=k_i^{\chi_{\mathbb{Z}}(\alpha_i(y))}x^{s_iy}+(k_i-k_i^{-1})\nabla_i(x^y),\\
\widetilde{\pi}_{\cc,\mathfrak{t}}(u)x^y&=\mathfrak{t}_{y;\cc}^{-\ell\epsilon_{r+1}}x^{s_1\cdots s_ry},\\
\widetilde{\pi}_{\cc,\mathfrak{t}}(x^\lambda)x^y&=x^{y+\lambda}
\end{split}
\end{equation*}
for $i\in [1,r]$, $\lambda\in\ell\mathbb{Z}^{r+1}$ and $y\in\mathcal{O}_{\ell\mathbb{Z}^{r+1},\cc}$.
For the corresponding monic quasi-polynomial eigenfunctions 
\[
\widetilde{E}_{y;\cc}(x;\mathfrak{t}):=E_{y;\cc}^{\ell\mathbb{Z}^{r+1},\ell\mathbb{Z}^{r+1}}(x;\mathfrak{t})
\]
of the $\widetilde{\pi}_{\cc,\mathfrak{t}}(Y^\mu)$ ($\mu\in\ell\mathbb{Z}^{r+1}$) 
we then have the following result.
\begin{corollary}
Let $\cc\in C^J$ and  $\mathfrak{t}\in T_{\ell\mathbb{Z}^{r+1},J}$. Choose a $W$-orbit $\mathcal{O}$ in $\mathcal{O}_{\ell\mathbb{Z}^{r+1},\cc}$ and assume that the generic conditions \eqref{eecond} hold true for $\Lambda^\prime=\ell\mathbb{Z}^{r+1}$. 
\begin{enumerate}
\item  We have
 \begin{equation}\label{translE}
\widetilde{E}_{y+s\varpi;\cc}(x;\mathfrak{t})=x^{s\varpi}\widetilde{E}_{y;\cc}(x;\mathfrak{t})
\end{equation}
for $y\in\mathcal{O}$ and $s\in\mathbb{Z}$.
\item If $\mathcal{O}=\mathcal{O}_{u^i\cc}$ \textup{(}$i\in\mathbb{Z}$\textup{)} and $\mathfrak{t}$ satisfies the genericity conditions \eqref{eecond2}, then
\begin{equation}\label{Eirel}
\widetilde{E}_{y;\cc}(x;\mathfrak{t})=E_y^{u^i(J)}(x;(u^i\mathfrak{t})\vert_{Q^\vee})\qquad\quad \forall\,y\in\mathcal{O}_{u^i\cc}.
\end{equation}
\end{enumerate}
\end{corollary}
\begin{proof}
(1) This is a special case of Remark \ref{zetatranslaterem}.\\
(2) This is a special case of Theorem \ref{propEVext0}(3). 
\end{proof}
\begin{remark}
Writing $i=s(r+1)+j$ in \eqref{Eirel} with $s\in\mathbb{Z}$ and $j\in [1,r]$, we have $u^i(J)=u^j(J)$ and $(u^i\mathfrak{t})\vert_{Q^\vee}=(q^{s\varpi}u^j\mathfrak{t})\vert_{Q^\vee}=(u^j\mathfrak{t})\vert_{Q^\vee}$, and so
\[
\widetilde{E}_{y;\cc}(x;\mathfrak{t})=E_y^{u^j(J)}(x;(u^j\mathfrak{t})\vert_{Q^\vee})\qquad \forall\, y\in\mathcal{O}_{u^i\cc}.
\]
This is in agreement with \eqref{translE} in view of Lemma \ref{remE1}(2).
\end{remark}
We finish this subsection by giving explicit descriptions of $T_{\ell\mathbb{Z}^{r+1},J}$ and $T_{\ell\mathbb{Z}^{r+1},J}^{\textup{red}}$.

Identifying $T_{\ell\mathbb{Z}^{r+1}}$ with $(\mathbf{F}^\times)^{r+1}$ by
\[
T_{\ell\mathbb{Z}^{r+1}}\overset{\sim}{\longrightarrow} (\mathbf{F}^\times)^{r+1},\qquad \mathfrak{t}\mapsto (\mathfrak{t}_1,\ldots,\mathfrak{t}_{r+1})
\]
with $\mathfrak{t}_i:=\mathfrak{t}^{\ell\epsilon_i}$ \textup{(}$1\leq i\leq r+1$\textup{)}, we have $\mathfrak{t}^{\alpha_0^\vee}=q^{\ell^2}\mathfrak{t}_{r+1}\mathfrak{t}_1^{-1}$ and 
$\mathfrak{t}^{\alpha_i^\vee}=\mathfrak{t}_i\mathfrak{t}_{i+1}^{-1}$ ($1\leq i\leq r$). Furthermore, $u$ then
acts on $(\mathbf{F}^\times)^{r+1}$ by
\[
u\mathfrak{t}=(q^{\ell^2}\mathfrak{t}_{r+1},\mathfrak{t}_1,\ldots,\mathfrak{t}_r).
\]
Combined with \eqref{stabOmega} we then conclude that for $\cc\in C^J$,
\[
{}^{\ell\mathbb{Z}^{r+1}}T_{\ell\mathbb{Z}^{r+1}}^\cc=T_{\ell\mathbb{Z}^{r+1},J}
\simeq\{\mathfrak{t}\in (\mathbf{F}^\times)^{r+1}\,\, | \,\, \mathfrak{t}_j=\mathfrak{t}_{j+1}\,\, (j\in J_0)\, \hbox{ and } \mathfrak{t}_1=q^{\ell^2}\mathfrak{t}_{r+1}\, 
\hbox{ if }\, 0\in J\}.
\]

An explicit description of $T_{\ell\mathbb{Z}^{r+1},J}^{\textup{red}}$ is as follows. 
 Let $J\subsetneq [0,r]$, which we regard as subset of $[1,r+1]$ by identifying $0\equiv r+1$. 
Then $\mathfrak{t}\in T_{\ell\mathbb{Z}^{r+1}}$ lies in $T_{\ell\mathbb{Z}^{r+1},J}^{\textup{red}}$ if and only if
\[
\mathfrak{t}_i=\mathfrak{t}_{i+1}=\cdots=\mathfrak{t}_j=\mathfrak{t}_{j+1}
\]
for all subintervals $[i,j]$ of $[1,r+1]$ such that $i-1,j+1\not\in J$ (here the numbers are read modulo $r+1$ again). In particular,
$T_{\ell\mathbb{Z}^{r+1},J}^{\textup{red}}$ is a subtorus of $T_{\ell\mathbb{Z}^{r+1}}$ of dimension $\#([0,r]\setminus J)$, cf. Remark \ref{parametersREM} in case of adjoint root datum.

In the case $J=[1,r]$ corresponding to the usual Macdonald polynomials, one thus has
\[
T_{\ell\mathbb{Z}^{r+1},[1,r]}=T_{\ell\mathbb{Z}^{r+1},[1,r]}^{\textup{red}}=\{z^\varpi\,\, | \,\, z\in\mathbf{F}^\times\}.
\]

\section{Uniform quasi-polynomial representations}\label{UnifMainSection}

In this section we re-parametrise the quasi-polynomial representation in a way that allows us to naturally glue them
together to form a uniform quasi-polynomial representation on $\mathbf{F}[E]$. The part of the re-parametrisation which does not involve twisting will be given by a map
\[
\widehat{\mathcal{G}}\rightarrow\{E\rightarrow T_{\Lambda^\prime}\},\qquad \widehat{\mathbf{g}}\mapsto \{y\mapsto\mathfrak{t}_y(\widehat{\mathbf{g}})\}
\]
satisfying $\mathfrak{t}_\cc(\widehat{\mathbf{g}})\in \widetilde{T}_{\Lambda^\prime,\mathbf{J}(\cc)}$ (see \eqref{widetildeTJ})
for all $\cc\in\overline{C}_+$. The resulting uniform quasi-polynomial representation of $\mathbb{H}_{Q^\vee,\Lambda^\prime}$
will be a gauged version of 
\[
\bigoplus_{\cc\in\overline{C}_+}\pi_{\cc,\mathfrak{t}_\cc(\widehat{\mathbf{g}})}^{Q^\vee,\Lambda^\prime}: \mathbb{H}_{Q^\vee,\Lambda^\prime}\rightarrow\textup{End}(\mathbf{F}[E]),
\]
which now depends on $q$, $\mathbf{k}$ and on the parametrising set $\widehat{\mathcal{G}}$ of the different choices of uniformizations of the induction parameters of the quasi-polynomial representations.

Throughout this section we fix $e\in\mathbb{Z}_{>0}$ and a pair $(\Lambda,\Lambda^\prime)\in (\mathcal{L}^{\times 2})_e$.
We assume throughout this section that $\mathbf{F}$ contains a $\textup{lcm}(e,2h)^{\textup{th}}$ root $q^{\frac{1}{\textup{lcm}(e,2h)}}$
of $q$, which we fix once and for all. 

We often do not record the dependence on $\Lambda^\prime$ when dealing with multiplicative characters $t$ of $\Lambda^\prime$ defined as products of co-character values $z^\alpha$ ($z\in\mathbf{F}^\times$, $\alpha\in\Phi_0$) since its interpretation as multiplicative character on $\Lambda^\prime$ is unambiguous (for an example of such a multiplicative character, see
\eqref{fraksprime}). If it is important to specify the lattice 
then we write $t\vert_{\Lambda^\prime}$.

\subsection{The $g$-parameters}\label{S81}
The re-parametrisation of the quasi-polynomial representation is based on factorisations of multiplicative characters in terms of co-characters. 
An important example is the following factorisation of the multiplicative character $q^{\textup{pr}_{E^\prime}(\lambda)}\in T_{\Lambda^\prime}$ for $\lambda\in\Lambda$.
\begin{lemma}\label{qtranslates}
Let $\lambda\in\Lambda$, then as multiplicative character of $\Lambda^\prime$ we have
\[
q^{\textup{pr}_{E^\prime}(\lambda)}\vert_{\Lambda^\prime}=\prod_{\alpha\in\Phi_0^+}\bigl(q_\alpha^{\alpha(\lambda)/h}\bigr)^\alpha\vert_{\Lambda^\prime}.
\]
\end{lemma}
\begin{proof}
Let $\lambda\in\Lambda$. Then $\textup{pr}_{E^\prime}(\lambda)\in\textup{pr}_{E^\prime}(\Lambda)\in\mathcal{L}$, hence $q^{\textup{pr}_{E^\prime}(\lambda)}$ is well defined as multiplicative character of $\Lambda^\prime$. For $\mu\in\Lambda^\prime$ we have, 
\begin{equation*}
\begin{split}
q^{\langle\textup{pr}_{E^\prime}(\lambda),\mu\rangle}&=q^{\frac{1}{h}\sum_{\alpha\in\Phi_0^+}\alpha(\lambda)\langle\alpha^\vee,\mu\rangle}\\
&=\prod_{\alpha\in\Phi_0^+}\bigl(q_\alpha^{\alpha(\lambda)/h}\bigr)^{\alpha(\mu)},
\end{split}
\end{equation*}
where we used Proposition \ref{latticecompatibilityprop} in the first equality and the fact that $\alpha(\lambda)$ and $\alpha(\mu)$ are integers in the second equality.
\end{proof}
For general multiplicative characters $\mathfrak{t}\in T_{\Lambda^\prime}$ we have the following result.
Recall that 
\[
\Lambda^\prime\cap E^\prime\subseteq\frac{1}{h}Q^\vee
\]
by Proposition \ref{latticecompatibilityprop}.
\begin{lemma}\label{preppar}
Let $\mathfrak{t}\in T_{\Lambda^\prime}$ and suppose that there exists a $\widetilde{\mathfrak{t}}\in T_{\frac{1}{h}Q^\vee}$ such that
\[
\widetilde{\mathfrak{t}}\vert_{\Lambda^\prime\cap E^\prime}=\mathfrak{t}\vert_{\Lambda^\prime\cap E^\prime}.
\] 
For $\alpha\in\Phi_0$ denote by
$\kappa_\alpha(\widetilde{\mathfrak{t}})\in\mathbf{F}^\times$ the value of $\widetilde{\mathfrak{t}}$ at $\alpha^\vee/h$. 
\begin{enumerate}
\item We have
\begin{equation}\label{splitt}
\mathfrak{t}\vert_{\Lambda^\prime\cap E^\prime}=\prod_{\alpha\in\Phi_0^+}\kappa_\alpha(\widetilde{\mathfrak{t}})^\alpha\vert_{\Lambda^\prime\cap E^\prime}
\end{equation}
as multiplicative characters of $\Lambda^\prime\cap E^\prime\in\mathcal{L}$.
\item We have
\[
\kappa_{-\alpha}(\widetilde{\mathfrak{t}})=\kappa_\alpha(\widetilde{\mathfrak{t}})^{-1},\qquad \kappa_{v\alpha}(\widetilde{\mathfrak{t}})=\kappa_\alpha(v^{-1}\widetilde{\mathfrak{t}}),\qquad
\kappa_\alpha(q^\lambda\widetilde{\mathfrak{t}})=q_\alpha^{\alpha(\lambda)/h}\kappa_\alpha(\widetilde{\mathfrak{t}})
\] 
for $v\in W_0$ and $\lambda\in\Lambda$,
where $q^\lambda$ is viewed here as the multiplicative character of $\frac{1}{h}Q^\vee$ mapping $\mu$ to $q^{\langle\lambda,\mu\rangle}$ for $\mu\in\frac{1}{h}Q^\vee$.
\end{enumerate}
\end{lemma}
\begin{proof}
(1) Let $\lambda\in\Lambda^\prime\cap E^\prime$. By Proposition \ref{latticecompatibilityprop} we have 
\[
\mathfrak{t}^\lambda=\widetilde{\mathfrak{t}}^{\,\textup{pr}_{E^\prime}(\lambda)}=\widetilde{\mathfrak{t}}^{\,\sum_{\alpha\in\Phi_0^+}\frac{\alpha(\lambda)}{h}\alpha^\vee}.
\]
Since $\alpha(\lambda)\in\mathbb{Z}$, we conclude that
\[
\mathfrak{t}^\lambda=\prod_{\alpha\in\Phi_0^+}\bigl(\widetilde{\mathfrak{t}}^{\,\alpha^\vee/h}\bigr)^{\alpha(\lambda)}=
\prod_{\alpha\in\Phi_0^+}\kappa_\alpha(\widetilde{\mathfrak{t}})^{\alpha(\lambda)}.
\]
(2) This follows by a direct check.
\end{proof}
Let $\mathcal{G}^{\textup{amb}}$ be the group of tuples $\mathbf{f}:=(f_\alpha)_{\alpha\in\Phi_0}$ of functions $f_\alpha: \mathbb{R}\rightarrow\mathbf{F}^\times$
(the group operation $\mathbf{f}\cdot\mathbf{f}^\prime$ is component-wise pointwise multiplication). 
For fixed $y\in E$ and $\mathbf{f}\in\mathcal{G}^{\textup{amb}}$, consider the multiplicative character
\begin{equation}\label{typrime}
\mathfrak{t}_y(\mathbf{f}):=\prod_{\alpha\in\Phi_0^+}f_\alpha(\alpha(y))^{\alpha}.
\end{equation}
Then $\mathbf{f}\mapsto \mathfrak{t}_y(\mathbf{f})\vert_{\Lambda^\prime}$ defines a group homomorphism $\mathcal{G}^{\textup{amb}}\rightarrow T_{\Lambda^\prime}$.

\begin{remark}\label{examplesunifsurj}\hfill
Suppose that $q^{\frac{1}{\textup{lcm}(e,2h)}}$ is part of an injective group homomorphism $\mathbb{R}\hookrightarrow\mathbf{F}^\times$, $d\mapsto q^{\,d}$.
\begin{enumerate}
\item For $y\in E$ write $q^y\vert_{\Lambda^\prime}$ for the multiplicative character of $\Lambda^\prime$, defined by the formula $\mu\mapsto q^{\langle y,\mu\rangle}$ for $\mu\in\Lambda^\prime$. Then we have for $\cc\in C^J$,
\[
q^\cc\vert_{\Lambda^\prime}\in T_{\Lambda^\prime,J},\qquad\quad q^{\textup{pr}_{E^\prime}(\cc)}\vert_{\Lambda^\prime}\in
\widetilde{T}_{\Lambda^\prime,J}.
\]
\item The map $C^J\cap E^\prime\rightarrow T_J$, $\cc\mapsto q^\cc\vert_{Q^\vee}$ is an injective map with image the set of characters $t\in T_J$ taking values in
the subgroup $\{q^{d}\}_{d\in\mathbb{R}}$ of $\mathbf{F}^\times$. Indeed, for such $t$ its pre-image $\textup{log}_q(t)\in C^J\cap E^\prime$ is the unique vector such that $t^{\alpha_i^\vee}=q^{\langle\alpha_i^\vee,\textup{log}_q(t)\rangle}$ for 
$i\in [1,r]$. 
By (the proof of) Lemma \ref{preppar} we then also have
\[
t=\mathfrak{t}_{\textup{log}_q(t)}(\widehat{\mathbf{p}})\vert_{Q^\vee}
\]
in $T$, with $\widehat{\mathbf{p}}=(\widehat{p}_\alpha)_{\alpha\in\Phi_0}\in\mathcal{G}^{\textup{amb}}$ defined by 
\begin{equation}\label{hatpspecial}
\widehat{p}_\alpha(d):=q^{2d/h\|\alpha\|^2}\qquad\quad (\alpha\in\Phi_0, d\in\mathbb{R}).
\end{equation}
In fact, for this choice of $\widehat{\mathbf{p}}\in\mathcal{G}^{\textup{amb}}$ we have
\begin{equation}\label{qyversion}
\mathfrak{t}_y(\widehat{\mathbf{p}})\vert_{\Lambda^\prime}=q^{\textup{pr}_{E^\prime}(y)}\vert_{\Lambda^\prime}
\qquad (y\in E),
\end{equation}
which immediately follows from \eqref{relationinvform}.
\end{enumerate}
\end{remark}

We will now consider a subset $\widehat{\mathcal{G}}\subset \mathcal{G}^{\textup{amb}}$ such that for all $\widehat{\mathbf{g}}\in\widehat{\mathcal{G}}$, the map
\[
E\rightarrow T_{\Lambda^\prime}, \qquad y\mapsto \mathfrak{t}_y(\widehat{\mathbf{g}})\vert_{\Lambda^\prime}
\] 
is $W$-equivariant. The conditions are motivated by Lemma \ref{preppar}(2) and Remark \ref{examplesunifsurj}(2).
\begin{definition}\label{gpar}
Denote by $\widehat{\mathcal{G}}$ the tuples $\widehat{\mathbf{g}}=(\widehat{g}_\alpha)_{\alpha\in\Phi_0}\in\mathcal{G}^{\textup{amb}}$ satisfying
\begin{enumerate}
\item $\widehat{g}_{v\alpha}(d+\ell)=q_\alpha^{\ell/h}\widehat{g}_\alpha(d)$,
\item $\widehat{g}_\alpha(-d)=\widehat{g}_\alpha(d)^{-1}$,
\item $\widehat{g}_\alpha(0)=1$
\end{enumerate}
for $v\in W_0$, $\alpha\in\Phi_0$, $\ell\in\mathbb{Z}$ and $d\in\mathbb{R}$.
\end{definition}
The restriction of $\widehat{\mathbf{g}}_\alpha$ to $\frac{1}{2}\mathbb{Z}$ is determined by the sign $\epsilon\in\{\pm 1\}$ such that 
\[
\widehat{g}_\alpha\bigl(\frac{1}{2}\bigr)=\epsilon q_{\alpha}^{\frac{1}{2h}}.
\]
Furthermore, the invariance properties from Definition \ref{gpar}(1)\&(2) may be reformulated as an equivariance property for suitable actions of the infinite dihedral group on $\mathbb{R}$ and $\mathbf{F}^\times$. 
The restriction of $\widehat{g}_\alpha$ to $\mathbb{R}\setminus\frac{1}{2}\mathbb{Z}$ is then determined by $\widehat{g}_\alpha\vert_{(0,1/2)}$, which may be chosen arbitrarily.
\begin{example}
$\widehat{\mathbf{p}}\in\mathcal{G}^{\textup{amb}}$ defined in Remark \ref{examplesunifsurj}(2) lies in $\widehat{\mathcal{G}}$.
\end{example}
\begin{remark}\label{basepoint}
The subset $\widehat{\mathcal{G}}$ is a $\mathcal{G}$-coset in $\mathcal{G}^{\textup{amb}}$, with $\mathcal{G}$ the subgroup of $\mathcal{G}^{\textup{amb}}$ 
consisting of tuples $\mathbf{g}=(g_\alpha)_{\alpha\in\Phi_0}$ satisfying
\begin{enumerate}
\item $g_{v\alpha}(d+\ell)=g_\alpha(d)$,
\item $g_\alpha(-d)=g_\alpha(d)^{-1}$,
\item $g_\alpha(0)=1$
\end{enumerate}
for $v\in W_0$, $\alpha\in\Phi_0$, $\ell\in\mathbb{Z}$ and $d\in\mathbb{R}$.
\end{remark}

\begin{lemma}\label{eqqlemma}
Let 
$\widehat{\mathbf{g}}\in\widehat{\mathcal{G}}$. Then 
\begin{equation}\label{eqq}
\mathfrak{t}_\lambda(\widehat{\mathbf{g}})\vert_{\Lambda^\prime}=q^{\textup{pr}_{E^\prime}(\lambda)}\vert_{\Lambda^\prime}
\end{equation}
for all $\lambda\in\Lambda$.
\end{lemma}
\begin{proof}
For $\lambda\in\Lambda$ and $\alpha\in\Phi_0^+$ we have $\alpha(\lambda)\in\mathbb{Z}$, hence 
\[
\widehat{g}_\alpha(\alpha(\lambda))=q_\alpha^{\alpha(\lambda)/h}
\]
by Definition \ref{gpar}(1)\&(3). The result now follows from
Lemma \ref{qtranslates}.
\end{proof}

\begin{lemma}\label{ginvlemma}
Let $\widehat{\mathbf{g}}\in\widehat{\mathcal{G}}$, $\mathbf{g}\in\mathcal{G}$ and $y\in E$. Then
\begin{enumerate}
\item $\mathfrak{t}_y(\widehat{\mathbf{g}})\vert_{\Lambda^\prime}\in T_{\Lambda^\prime}$ does not depend on the choice $\Phi_0^+$ of positive roots,
\item We have $\mathfrak{t}_{wy}(\widehat{\mathbf{g}})\vert_{\Lambda^\prime}=(w\mathfrak{t}_y(\widehat{\mathbf{g}}))\vert_{\Lambda^\prime}$ for all $w\in W_{P^\vee}$.
\item We have
$\mathfrak{t}_{wy}(\widehat{\mathbf{g}})\vert_{P^\vee}=(w\mathfrak{t}_y(\widehat{\mathbf{g}}))\vert_{P^\vee}$
 for all $w\in W_\Lambda$.
 \item We have $\mathfrak{t}_{wy}(\mathbf{g})=(Dw)\mathfrak{t}_y(\mathbf{g})$ for all $w\in W_\Lambda$.
\end{enumerate}
\end{lemma}
\begin{proof}
(1) This follows from the fact that
$\widehat{g}_{-\alpha}(-d)=\widehat{g}_{-\alpha}(d)^{-1}=\widehat{g}_\alpha(d)^{-1}$.\\
(2)\&(3) For $v\in W_0$ we have
\[
\mathfrak{t}_{vy}(\widehat{\mathbf{g}})=\prod_{\alpha\in\Phi_0^+}\widehat{g}_{v\alpha}(\alpha(y))^{v\alpha}=
\prod_{\alpha\in\Phi_0^+}\widehat{g}_{\alpha}(\alpha(y))^{v\alpha}=v\mathfrak{t}_y(\widehat{\mathbf{g}})
\]
in $T_{\Lambda^\prime}$,
where we have used (1) in the first equality. For $\lambda\in\Lambda$ we have in $T_{\Lambda^\prime}$, 
\begin{equation}\label{projtrans}
\mathfrak{t}_{y+\lambda}(\widehat{\mathbf{g}})=\prod_{\alpha\in\Phi_0^+}\widehat{g}_\alpha(\alpha(y)+\alpha(\lambda))^\alpha=\prod_{\alpha\in\Phi_0^+}q_\alpha^{\alpha(\lambda)\alpha/h}\widehat{g}_\alpha(\alpha(y))^\alpha=q^{\textup{pr}_{E^\prime}(\lambda)}\mathfrak{t}_y(\widehat{\mathbf{g}}),
\end{equation}
where the last equality is due to Lemma \ref{qtranslates}. When $\lambda\in P^\vee$ we conclude from \eqref{projtrans} that
$\mathfrak{t}_{y+\lambda}(\widehat{\mathbf{g}})=q^\lambda\mathfrak{t}_y(\widehat{\mathbf{g}})=\tau(\lambda)\mathfrak{t}_y(\widehat{\mathbf{g}})$ in $T_{\Lambda^\prime}$,
which proves (2). For (3) note that for $\lambda\in\Lambda$,
\[
q^{\textup{pr}_{E^\prime}(\lambda)}\vert_{P^\vee}=q^\lambda\vert_{P^\vee}
\]
since $P^\vee\subset E^\prime$. Hence \eqref{projtrans} also implies that $\mathfrak{t}_{y+\lambda}(\widehat{\mathbf{g}})\vert_{P^\vee}=(\tau(\lambda)\mathfrak{t}_y(\widehat{\mathbf{g}}))\vert_{P^\vee}$ for all $\lambda\in\Lambda$. This proves (3).\\
(4) This is immediate for $w=\tau(\lambda)$ with $\lambda\in\Lambda$ since $g_\alpha$ is $1$-periodic. For $w=v\in W_0$ it follows as in the proof of (2).
\end{proof}
Denote $\widetilde{T}_{\Lambda,J}^{\textup{red}}:=T_{\Lambda,J}^{\textup{red}}\cap\widetilde{T}_\Lambda$, with $\widetilde{T}_\Lambda$ given by \eqref{widetildeTJ}.
\begin{corollary}\label{ginvcor}
Let $\cc\in C^J$, $\widehat{\mathbf{g}}\in\widehat{\mathcal{G}}$ and $\mathbf{g}\in\mathcal{G}$. We then have $\mathfrak{t}_\cc(\widehat{\mathbf{g}})\vert_{\Lambda^\prime}\in \widetilde{T}_{\Lambda^\prime,J}$ and $\mathfrak{t}_\cc(\mathbf{g})\vert_{\Lambda^\prime}\in \widetilde{T}_{\Lambda^\prime,J}^{\textup{red}}$.
\end{corollary}
\begin{proof}
Let $\cc\in C^J$ and $j\in J$. We have $s_j\mathfrak{t}_{\cc}(\widehat{\mathbf{g}})=\mathfrak{t}_{s_j\cc}(\widehat{\mathbf{g}})=\mathfrak{t}_{\cc}(\widehat{\mathbf{g}})$ in $T_{P^\vee}$ by Lemma \ref{ginvlemma}(2).  
By \eqref{actiononT} we conclude that $\mathfrak{t}_{\cc}(\widehat{\mathbf{g}})^{\alpha_j^\vee}=1$. Hence $\mathfrak{t}_\cc(\widehat{\mathbf{g}})\in T_{\Lambda^\prime,J}$.
Clearly $\mathfrak{t}_\cc(\widehat{\mathbf{g}})\vert_{\Lambda^\prime\cap E_{\textup{co}}}\equiv 1$, hence $\mathfrak{t}_\cc(\widehat{\mathbf{g}})\in\widetilde{T}_{\Lambda^\prime,J}$. The statement about $\mathfrak{t}_\cc(\mathbf{g})$ follows in the same way, now using Lemma \ref{ginvlemma}(3).
\end{proof}

We will now add a $T_{\Lambda^\prime,[1,r]}$-factor to $\mathfrak{t}_y(\widehat{\mathbf{g}})$ (cf. Lemma \ref{resLambdaQ}). For the associated extended quasi-polynomial representations of $\mathbb{H}_{Q^\vee,\Lambda^\prime}$, this corresponds to twisting by an algebra automorphism $\Xi_{\mathfrak{t}}$ ($\mathfrak{t}\in T_{\Lambda^\prime,[1,r]}\simeq\textup{Hom}(\Omega_{\Lambda^\prime},\mathbf{F}^\times)$), see (the proof of) Proposition \ref{twistactionprop}.

\begin{definition}\hfill
\begin{enumerate}
\item Denote by $\mathcal{C}_{\Lambda,\Lambda^\prime}$ the space of functions $\mathfrak{c}: E_{\textup{co}}\rightarrow T_{\Lambda^\prime,[1,r]}$ such that 
\begin{equation}\label{ctrans}
\mathfrak{c}(y+\lambda)=q^\lambda\mathfrak{c}(y)\qquad \forall\, y\in E_{\textup{co}},\, \forall\, \lambda\in\textup{pr}_{E_{\textup{co}}}(\Lambda).
\end{equation}
\item For $\mathfrak{c}\in\mathcal{C}_{\Lambda,\Lambda^\prime}$, $\widehat{\mathbf{g}}\in\widehat{\mathcal{G}}$ and $y\in E$ set
\[
\mathfrak{t}_y(\widehat{\mathbf{g}},\mathfrak{c}):=\mathfrak{t}_y(\widehat{\mathbf{g}})\mathfrak{c}\bigl(\textup{pr}_{E_{\textup{co}}}(y)\bigr)\in T_{\Lambda^\prime}
\]
\textup{(}we suppress the dependence on the two lattices $\Lambda,\Lambda^\prime$\textup{)}.
\end{enumerate}
\end{definition}
Since $\mathfrak{c}(y)\vert_{Q^\vee}=1_T$ for $\mathfrak{c}\in\mathcal{C}_{\Lambda,\Lambda^\prime}$ and $y\in E_{\textup{co}}$, the set $\mathcal{C}_{\Lambda,\Lambda^\prime}$ identifies with the set of functions $E_{\textup{co}}/\textup{pr}_{E_{\textup{co}}}(\Lambda)\rightarrow \textup{Hom}(\Lambda^\prime/Q^\vee,\mathbf{F}^\times)\simeq\textup{Hom}(\Omega_{\Lambda^\prime},\mathbf{F}^\times)$.
In particular, $\mathcal{C}_{\Lambda,\Lambda^\prime}$ is nonempty. Note that
\begin{equation}\label{cextractLambda}
\mathfrak{c}(\textup{pr}_{E_{\textup{co}}}(\lambda))=q^{\textup{pr}_{E_{\textup{co}}}(\lambda)}\mathfrak{c}(0)\qquad\quad (\lambda\in\Lambda)
\end{equation}
as multiplicative characters of $\Lambda^\prime$.
Furthermore, if $(\Lambda_i,\Lambda_i^\prime)\in (\mathcal{L}^{\times 2})_e$ and 
$(\Lambda_1,\Lambda_1^\prime)\leq (\Lambda_2,\Lambda_2^\prime)$ then the assignment $\mathfrak{c}\mapsto \mathfrak{c}(\cdot)\vert_{\Lambda_1^\prime}$ defines a map
$\mathcal{C}_{\Lambda_2,\Lambda_2^\prime}\rightarrow\mathcal{C}_{\Lambda_1,\Lambda_1^\prime}$.

For $\mu\in\Lambda^\prime$ we have
\begin{equation}\label{decomptt}
\mathfrak{t}_y(\widehat{\mathbf{g}},\mathfrak{c})^{\mu}=\mathfrak{t}_y(\widehat{\mathbf{g}})^{\textup{pr}_{E^\prime}(\mu)}\mathfrak{c}(\textup{pr}_{E_{\textup{co}}}(\cc_y))^\mu.
\end{equation}
Special cases of this formula are
\[
\mathfrak{t}_y(\widehat{\mathbf{g}},\mathfrak{c})^\mu=\mathfrak{t}_y(\widehat{\mathbf{g}})^\mu\quad (\mu\in Q^\vee),\qquad
\mathfrak{t}_y(\widehat{\mathbf{g}},\mathfrak{c})^\nu=\mathfrak{c}(\textup{pr}_{E_{\textup{co}}}(\cc_y))^\nu\quad (\nu\in \Lambda^\prime\cap E_{\textup{co}}).
\]
\begin{example}\label{exampleqbaseExt}
Consider the setup of Remark \ref{examplesunifsurj}.
Then
\[
\mathfrak{c}(y):=q^y\in T_{\Lambda^\prime}\qquad (y\in E_{\textup{co}})
\]
lies $\mathcal{C}_{\Lambda,\Lambda^\prime}$. Taking $\mathbf{p}=(\widehat{p}_\alpha)_{\alpha\in\Phi_0}\in\widehat{\mathcal{G}}$ with $\widehat{p}_\alpha(d)$ defined by \eqref{hatpspecial}, we then have for any $y\in E$,
\begin{equation}\label{qyversionExt}
\mathfrak{t}_y(\widehat{\mathbf{p}},\mathfrak{c})=\mathfrak{t}_y(\widehat{\mathbf{p}})q^{\textup{pr}_{E_{\textup{co}}}(y)}=q^y
\end{equation}
in $T_{\Lambda^\prime}$. Here we used \eqref{qyversion}
for the second equality.
\end{example}
Lemma \ref{ginvlemma} and Corollary \ref{ginvcor} generalize as follows.
\begin{lemma}\label{ginvlemmaext}
Let $\widehat{\mathbf{g}}\in\widehat{\mathcal{G}}$, $\mathfrak{c}\in\mathcal{C}_{\Lambda,\Lambda^\prime}$ and $y\in E$. 
\begin{enumerate}
\item For $w\in W_\Lambda$ we have
\[
\mathfrak{t}_{wy}(\widehat{\mathbf{g}},\mathfrak{c})=w\mathfrak{t}_y(\widehat{\mathbf{g}},\mathfrak{c})
\]
in $T_{\Lambda^\prime}$.
\item $\mathfrak{t}_\cc(\widehat{\mathbf{g}},\mathfrak{c})\in {}^{\Lambda}T_{\Lambda^\prime}^\cc$ for $\cc\in\overline{C}_+$. 
In particular, for $\cc\in C^J$ we have $\mathfrak{t}_{\cc}(\widehat{\mathbf{p}},\mathfrak{c})\in T_{\Lambda^\prime,J}$.
\item For $y=\lambda\in\Lambda$ we have 
\[
\mathfrak{t}_\lambda(\widehat{\mathbf{g}},\mathfrak{c})=q^\lambda\mathfrak{c}(0)
\]
as multiplicative characters of $\Lambda^\prime$.
\end{enumerate}
\end{lemma}
\begin{proof}
(1) By Lemma \ref{ginvlemma}(2) and the fact that $\mathfrak{c}$ takes values in $T_{\Lambda^\prime}^{W_0}$, we have 
$\mathfrak{t}_{vy}(\widehat{\mathbf{g}},\mathfrak{c})=v\mathfrak{t}_y(\widehat{\mathbf{g}},\mathfrak{c})$ in $T_{\Lambda^\prime}$ for $v\in W_0$. 
Let $\lambda\in\Lambda$. By \eqref{projtrans} we have
\[
\mathfrak{t}_{y+\lambda}(\widehat{\mathbf{g}})=q^{\textup{pr}_{E^\prime}(\lambda)}\mathfrak{t}_y(\widehat{\mathbf{g}})
\]
in $T_{\Lambda^\prime}$. By \eqref{ctrans} we have
\[
\mathfrak{c}(\textup{pr}_{E_{\textup{co}}}(y+\lambda))=q^{\textup{pr}_{E_{\textup{co}}}(\lambda)}\mathfrak{c}(\textup{pr}_{E_{\textup{co}}}(y))
\]
in $T_{\Lambda^\prime}$. Combining both formulas, we conclude that
\[
\mathfrak{t}_{y+\lambda}(\widehat{\mathbf{g}},\mathfrak{c})=q^{\lambda}\mathfrak{t}_y(\widehat{\mathbf{g}},\mathfrak{c})
\]
in $T_{\Lambda^\prime}$, which completes the proof.\\
(2) Suppose now that $\cc\in C^J$. Then  $\mathfrak{t}_\cc(\widehat{\mathbf{g}},\mathfrak{c})\in T_{\Lambda^\prime,J}$ by Corollary \ref{ginvcor} and by the fact that $\mathfrak{c}$ takes values in  $T_{\Lambda^\prime,[1,r]}$. Furthermore, for $\omega\in\Omega_{\Lambda,\cc}$,
\[
\omega\mathfrak{t}_\cc(\widehat{\mathbf{g}},\mathfrak{c})=\mathfrak{t}_{\omega\cc}(\widehat{\mathbf{g}},\mathfrak{c})=\mathfrak{t}_\cc(\widehat{\mathbf{g}},\mathfrak{c})
\]
by part (1) of the lemma. Hence $\mathfrak{t}_\cc(\widehat{\mathbf{g}},\mathfrak{c})\in {}^{\Lambda}T_{\Lambda^\prime}^\cc$.\\
(3) For $\lambda\in\Lambda$ we have in $T_{\Lambda^\prime}$,
\[
\mathfrak{t}_\lambda(\widehat{\mathbf{g}},\mathfrak{c})=\mathfrak{t}_\lambda(\widehat{\mathbf{g}})\mathfrak{c}(\textup{pr}_{E_{\textup{co}}}(\lambda))=
q^{\textup{pr}_{E^\prime}(\lambda)}q^{\textup{pr}_{E_{\textup{co}}}(\lambda)}\mathfrak{c}(0)=q^\lambda\mathfrak{c}(0),
\]
where we used Lemma \ref{eqqlemma} and \eqref{decomptt}
for the second equality.
\end{proof}

\begin{remark}\label{linkttpar}
Let $\widehat{\mathbf{g}}\in\widehat{\mathcal{G}}$, $\mathfrak{c}\in\mathcal{C}_{\Lambda,\Lambda^\prime}$ and $\cc\in\overline{C}_+$. Applying Definition \ref{deftyc} to the multiplicative character $\mathfrak{t}_{\cc}(\widehat{\mathbf{g}},\mathfrak{c})\in {}^{\Lambda}T_{\Lambda^\prime}^\cc$  gives rise to the multiplicative characters 
\[\mathfrak{t}_{\cc}(\widehat{\mathbf{g}},\mathfrak{c})_{w\cc;\cc}=w\mathfrak{t}_{\cc}(\widehat{\mathbf{g}},\mathfrak{c})\in T_{\Lambda^\prime}\qquad\quad (w\in W_{\Lambda}).
\]
It then follows from Lemma \ref{ginvlemmaext} that
\begin{equation}\label{relspectrum}
\mathfrak{t}_{\cc}(\widehat{\mathbf{g}},\mathfrak{c})_{w\cc;\cc}=\mathfrak{t}_{w\cc}(\widehat{\mathbf{g}},\mathfrak{c})\in T_{\Lambda^\prime}\qquad\quad (w\in W_\Lambda)
\end{equation}
in $T_{\Lambda^\prime}$. 
\end{remark}
The following extension of 
Lemma \ref{actiondeformext} is a direct consequence of Lemma \ref{ginvlemmaext}(2).
\begin{lemma}\label{refactionext}
Let $\widehat{\mathbf{g}}\in\widehat{\mathcal{G}}$ and $\mathfrak{c}\in\mathcal{C}_{\Lambda,\Lambda^\prime}$. The formulas
\[
w\cdot_{\widehat{\mathbf{g}},\mathfrak{c}}x^y:=w_{\mathfrak{t}_{\cc_y}(\widehat{\mathbf{g}},\mathfrak{c})}x^y\qquad \quad (w\in W_{\Lambda^\prime},\, y\in E)
\]
define a linear left $W_{\Lambda^\prime}$-action on $\mathbf{F}[E]$. 
\end{lemma}
Concretely, by Remark \ref{linkttpar} we have the formulas
\begin{equation*}
\begin{split}
v\cdot_{\widehat{\mathbf{g}},\mathfrak{c}}x^y&:=x^{vy}\qquad\qquad\quad\,\,\, (v\in W_0),\\
\tau(\mu)\cdot_{\widehat{\mathbf{g}}, \mathfrak{c}}x^y&:=\mathfrak{t}_y(\widehat{\mathbf{g}},\mathfrak{c})^{-\mu}x^y\qquad (\mu\in\Lambda^\prime),
\end{split}
\end{equation*}
for $y\in E$. Furthermore, 
\begin{equation}
\label{formpexact2}
s_0\cdot_{\widehat{\mathbf{g}},\mathfrak{c}}x^y=\mathfrak{t}_y(\widehat{\mathbf{g}},\mathfrak{c})^{\varphi^\vee}x^{s_\varphi y},\qquad
(w_{\zeta,\Lambda^\prime})\cdot_{\widehat{\mathbf{g}},\mathfrak{c}}x^y=\mathfrak{t}_{v_{\zeta,\Lambda^\prime}^{-1}y}(\widehat{\mathbf{g}},\mathfrak{c})^{-\zeta}x^{v_{\zeta,\Lambda^\prime}^{-1}y}
\end{equation}
for $\zeta\in\Lambda_{\textup{min}}^\prime$,
which follows from Remark \ref{linkttpar} and \eqref{OmegaPol1}. Since $\mathfrak{c}$ takes values in $T_{\Lambda^\prime,[1,r]}$ we have that
\[
w\cdot_{\widehat{\mathbf{g}},\mathfrak{c}}x^y=\Theta_{\mathfrak{c}(\textup{pr}_{E_{\textup{co}}}(y))}(w)\bigl(w_{\mathfrak{t}_{\cc_y}(\widehat{\mathbf{g}})}(x^y)\bigr)
\]
for $w\in W_{\Lambda^\prime}$ and $y\in E$ by \eqref{twistactionclassical}. In particular, the $\mathfrak{c}$-dependence of the $W_{\Lambda^\prime}$-action $\cdot_{\widehat{\mathbf{g}},\mathfrak{c}}$ on $\mathbf{F}[E]$ (and later on, in Theorem \ref{gthm}, the $\mathfrak{c}$-dependence of the corresponding extended double affine Hecke algebra representation), corresponds to twisting the representations by multiplicative characters of $\Omega_{\Lambda^\prime}$, cf. Proposition \ref{twistactionprop}.

\begin{example}\label{oneparexample}
Consider the setup of Example \ref{exampleqbaseExt}.
Then we have
\begin{equation}\label{expliA}
\tau(\mu)\cdot_{\widehat{\mathbf{p}},\mathfrak{c}} x^y=q^{-\langle\mu,y\rangle}x^y\qquad (\mu\in\Lambda^\prime,\, y\in E),
\end{equation}
so in this case $\cdot_{\widehat{\mathbf{p}},\mathfrak{c}}$ is the standard $W_{\Lambda^\prime}$-action on $\mathbf{F}[E]$ by $q$-translations and reflections. 
\end{example}

\subsection{The uniform quasi-polynomial representation}\label{unifSection}

For $\widehat{\mathbf{g}}\in\widehat{\mathcal{G}}$, $\mathfrak{c}\in\mathcal{C}_{\Lambda,\Lambda^\prime}$ and $\cc\in\overline{C}_+$ we have the extended quasi-polynomial representation 
\[
\pi_{\cc,\mathfrak{t}_\cc(\widehat{\mathbf{g}},\mathfrak{c})}^{\Lambda,\Lambda^\prime}: \mathbb{H}_{\Lambda,\Lambda^\prime}\rightarrow\textup{End}(\mathcal{P}^{(\cc)}_\Lambda)
\] 
from Theorem \ref{glueprop}, 
since $\mathfrak{t}_\cc(\widehat{\mathbf{g}},\mathfrak{c})\in {}^{\Lambda}T_{\Lambda^\prime}^\cc$ by Lemma \ref{ginvlemmaext}(2).
Their direct sum turns $\mathbf{F}[E]$ into an $\mathbb{H}_{\Lambda,\Lambda^\prime}$-module, with an explicit $(\widehat{\mathbf{g}},\mathfrak{c})$-dependent $\mathbb{H}_{\Lambda,\Lambda^\prime}$-action in terms of truncated Demazure-Lusztig operators. 

To connect it to metaplectic representation theory, in particular to the $\textup{GL}_{r+1}$-type double affine Hecke algebra representation from \cite[\S 5]{SSV}, 
it is convenient to fix a representative $\widehat{\mathbf{p}}$ of $\widehat{\mathcal{G}}$ and 
twist the $\mathbb{H}_{\Lambda,\Lambda^\prime}$-actions $\pi_{\cc,\mathfrak{t}_\cc(\widehat{\mathbf{p}}\cdot\mathbf{g},\mathfrak{c})}^{\Lambda,\Lambda^\prime}$ 
on $\mathcal{P}_\Lambda^{(\cc)}$ by a linear automorphism of $\mathcal{P}_\Lambda^{(\cc)}$ that separates the parameters $\mathbf{g}\in\mathcal{G}$ from $\widehat{\mathbf{p}}\cdot\mathbf{g}\in \widehat{\mathcal{G}}$
(the explicit connection to the metaplectic representation theory will be the subject of Section \ref{MfinalSection}). 
The linear automorphism, depending on $\mathbf{g}\in\mathcal{G}$, is defined as follows.

\begin{lemma}\label{ISOlemma}
Let $\cc\in\overline{C}_+$ and $\mathbf{g}\in\mathcal{G}$.
There exists a unique $\mathcal{P}_\Lambda$-linear automorphism $\Gamma_{\Lambda,\mathbf{g}}^{(\cc)}$ of $\mathcal{P}_\Lambda^{(\cc)}$ satisfying
\begin{equation}\label{explw}
\Gamma_{\Lambda,\mathbf{g}}^{(\cc)}(x^{w\cc})=\Bigl(\prod_{\alpha\in\Pi(Dw)}g_\alpha(\alpha(\cc))\Bigr)x^{w\cc}\qquad\quad (w\in W_\Lambda).
\end{equation}
\end{lemma}
\begin{proof}
Let $\mathbf{F}_{\textup{cl}}$ be the algebraic closure of $\mathbf{F}$. Write $\mathcal{G}_{\textup{cl}}$ for the space of $\mathbf{g}$-parameters $\mathcal{G}$ over the field $\mathbf{F}_{\textup{cl}}$ (see Remark \ref{basepoint}). For $\alpha\in\Phi_0$ choose $\widetilde{\mathbf{g}}=(\widetilde{g}_\alpha(d))_{\alpha\in\Phi_0}\in\mathcal{G}_{\textup{cl}}$
satisfying $\widetilde{g}_\alpha(d)^2=g_\alpha(d)$ for $\alpha\in\Phi_0$ and $d\in\mathbb{R}$.
Define $\gamma_{\widetilde{\mathbf{g}}}: E\rightarrow\mathbf{F}_{\textup{cl}}^\times$ by 
\begin{equation}\label{gammag}
\gamma_{\widetilde{\mathbf{g}}}(y):=\prod_{\alpha\in\Phi_0^+}\widetilde{g}_\alpha(\alpha(y)).
\end{equation}
Then 
\begin{equation}\label{cocycleg}
\begin{split}
\gamma_{\widetilde{\mathbf{g}}}(y+\mu)&=\gamma_{\widetilde{\mathbf{g}}}(y)\qquad\qquad\qquad\qquad\quad\quad\,\,\, (\mu\in\Lambda),\\
\gamma_{\widetilde{\mathbf{g}}}(vy)&=\Bigl(\prod_{\alpha\in\Pi(v)}g_\alpha(\alpha(y))^{-1}\Bigr)\gamma_{\widetilde{\mathbf{g}}}(y)\qquad (v\in W_0).
\end{split}
\end{equation}
Let $\Gamma_{\Lambda,\mathbf{g}}^{(\cc)}$ be the $\mathbf{F}_{\textup{cl}}$-linear automorphism of $\bigoplus_{y\in\mathcal{O}_{\Lambda,\cc}}\mathbf{F}_{\textup{cl}}x^y$ defined by
\begin{equation}\label{ascoboundary}
\Gamma_{\Lambda,\mathbf{g}}^{(\cc)}(x^y):=\frac{\gamma_{\widetilde{\mathbf{g}}}(\cc)}{\gamma_{\widetilde{\mathbf{g}}}(y)}x^y\qquad\quad (y\in\mathcal{O}_{\Lambda,\cc}).
\end{equation}
Then 
\[
\Gamma_{\Lambda,\mathbf{g}}^{(\cc)}(x^{w\cc})=\Bigl(\prod_{\alpha\in\Pi(Dw)}g_\alpha(\alpha(\cc))\Bigr)x^{w\cc}\qquad\quad (w\in W_\Lambda)
\]
by \eqref{cocycleg}. Hence $\Gamma_{\Lambda,\mathbf{g}}^{(\cc)}$ does not depend on the choice of $\widetilde{\mathbf{g}}$, and it restricts to a $\mathbf{F}$-linear automorphism of
$\mathcal{P}_\Lambda^{(\cc)}$ satisfying \eqref{explw}.
Note that $\Gamma_{\Lambda,\mathbf{g}}^{(\cc)}$ is $\mathcal{P}_\Lambda$-linear due to $\Lambda$-translation invariance of $\gamma_{\widetilde{\mathbf{g}}}$.
\end{proof}
\begin{remark}\label{O4rem}
If $\cc^\prime\in\mathcal{O}_{\Lambda,\cc}\cap\overline{C}_+$ then $\cc^\prime=\omega\cc$ for some $\omega\in\Omega_{\Lambda}$ and 
\[
\Gamma_{\Lambda,\mathbf{g}}^{(\cc)}=\Bigl(\prod_{\alpha\in \Pi(D\omega)}g_\alpha(\alpha(\cc))\Bigr)\Gamma_{\Lambda,\mathbf{g}}^{(\omega\cc)}
\]
in view of \eqref{ascoboundary} and \eqref{cocycleg}.
\end{remark}

The $W_{\Lambda^\prime}$-action $\cdot_{\widehat{\mathbf{g}},\mathfrak{c}}$ on $\mathbf{F}[E]$ from Lemma \ref{refactionext} preserves $\mathcal{P}_{\Lambda}^{(\cc)}$.
The following lemma shows that $\Gamma_{\Lambda,\mathbf{g}}^{(\cc)}$ intertwines the 
$W_{\Lambda^\prime}$-action $\cdot_{\widehat{\mathbf{p}},\mathfrak{c}}$ on $\mathcal{P}_\Lambda^{(\cc)}$ with a twisted version of the $W_{\Lambda^\prime}$-action $\cdot_{\widehat{\mathbf{g}},\mathfrak{c}}$ on $\mathcal{P}_\Lambda^{(\cc)}$, where $\widehat{\mathbf{g}}:=\widehat{\mathbf{p}}\cdot\mathbf{g}$.
\begin{lemma}\label{gtransaction}
Let $\cc\in\overline{C}_+$, $\widehat{\mathbf{p}}\in\widehat{\mathcal{G}}$, $\mathbf{g}\in\mathcal{G}$ and $\mathfrak{c}\in\mathcal{C}_{\Lambda,\Lambda^\prime}$. Set $\widehat{\mathbf{g}}:=\widehat{\mathbf{p}}\cdot\mathbf{g}\in\widehat{\mathcal{G}}$. Then
\[
\Gamma_{\Lambda,\mathbf{g}}^{(\cc)}\bigl(w\cdot_{\widehat{\mathbf{p}},\mathfrak{c}}\Gamma_{\Lambda,\mathbf{g}}^{(\cc)\,-1}(x^y)\bigr)=
\Bigl(\prod_{a\in\Pi(w)}g_{Da}(a(y))\Bigr)w\cdot_{\widehat{\mathbf{g}},\mathfrak{c}}x^y
\]
for $w\in W_{\Lambda^\prime}$ and $y\in\mathcal{O}_{\Lambda,\cc}$.
\end{lemma}
\begin{proof}
We use the notations introduced in the proof of Lemma \ref{ISOlemma}.
Fix $\cc\in\overline{C}_+$. 
By a standard argument using \eqref{rootsdescription} and Remark \ref{basepoint}(1) it suffices to prove that
\begin{equation}\label{todogtrans}
\Gamma_{\Lambda,\mathbf{g}}^{(\cc)}\bigl(s_j\cdot_{\widehat{\mathbf{p}},\mathfrak{c}}\Gamma_{\Lambda,\mathbf{g}}^{(\cc)\,-1}(x^y)\bigr)=
g_{D\alpha_j}(\alpha_j(y))\,s_{j}\cdot_{\widehat{\mathbf{g}},\mathfrak{c}}x^y
\end{equation}
for $j\in [0,r]$ and $y\in\mathcal{O}_{\Lambda,\cc}$ and 
\begin{equation}\label{todogtrans2}
\Gamma_{\Lambda,\mathbf{g}}^{(\cc)}\bigl(\omega\cdot_{\widehat{\mathbf{p}},\mathfrak{c}}\Gamma_{\Lambda,\mathbf{g}}^{(\cc)\,-1}(x^y)\bigr)=\omega\cdot_{\widehat{\mathbf{g}},\mathfrak{c}}x^y
\end{equation}
for $\omega\in\Omega_{\Lambda^\prime}$ and $y\in\mathcal{O}_{\Lambda,\cc}$.

For $i\in [1,r]$ and $y\in\mathcal{O}_{\Lambda,\cc}$ we have
\[
\Gamma_{\Lambda,\mathbf{g}}^{(\cc)}\bigl(s_i\cdot_{\widehat{\mathbf{p}},\mathfrak{c}}\Gamma_{\Lambda,\mathbf{g}}^{(\cc)\,-1}(x^y)\bigr)=\frac{\gamma_{\widetilde{\mathbf{g}}}(y)}{\gamma_{\widetilde{\mathbf{g}}}(s_iy)}x^{s_iy}=
g_{\alpha_i}(\alpha_i(y))\,s_{i}\cdot_{\widehat{\mathbf{g}},\mathfrak{c}}x^y,
\]
which is \eqref{todogtrans} for $j=i\in [1,r]$. By \eqref{formpexact2} we have for $y\in\mathcal{O}_{\Lambda,\cc}$,
\begin{equation*}
\Gamma_{\Lambda,\mathbf{g}}^{(\cc)}\bigl(s_0\cdot_{\widehat{\mathbf{p}},\mathfrak{c}}\Gamma_{\Lambda,\mathbf{g}}^{(\cc)\,-1}(x^y)\bigr)=\frac{\gamma_{\widetilde{\mathbf{g}}}(y)}{\gamma_{\widetilde{\mathbf{g}}}(s_\varphi y)}
\mathfrak{t}_y(\widehat{\mathbf{p}},\mathfrak{c})^{\varphi^\vee}x^{s_\varphi y}.
\end{equation*}
Since $\mathfrak{t}_y(\widehat{\mathbf{g}},\mathfrak{c})=\mathfrak{t}_y(\widehat{\mathbf{p}},\mathfrak{c})\mathfrak{t}_y(\mathbf{g})$, formula \eqref{todogtrans} for $j=0$ then follows from the fact that
\[
\mathfrak{t}_y(\mathbf{g})^{\varphi^\vee}=g_{\varphi}(\varphi(y))\prod_{\alpha\in\Pi(s_\varphi)}g_\alpha(\alpha(y))=\frac{\gamma_{\widetilde{\mathbf{g}}}(y)}{g_{D\alpha_0}(\alpha_0(y))\gamma_{\widetilde{\mathbf{g}}}(s_\varphi y)}
\]
by \eqref{alpha0}, \eqref{thetalength} and \eqref{cocycleg}. 

It remains to prove \eqref{todogtrans2}. By \eqref{formpexact2} we have for $y\in E$ and $\zeta\in\Lambda^\prime_{\textup{min}}$,
\[
w_{\zeta,\Lambda^\prime}\cdot_{\widehat{\mathbf{g}},\mathfrak{c}}x^y=\mathfrak{t}_{v_{\zeta,\Lambda^\prime}^{-1}y}(\mathbf{g})^{-\zeta}(w_{\zeta,\Lambda^\prime}\cdot_{\widehat{\mathbf{p}},\mathfrak{c}}x^y).
\]
Formula \eqref{todogtrans2} will thus follow from the equality
\begin{equation}\label{todoOmega}
\frac{\gamma_{\widetilde{\mathbf{g}}}(y)}{\gamma_{\widetilde{\mathbf{g}}}(v_{\zeta,\Lambda^\prime}^{-1}y)}=\mathfrak{t}_{v_{\zeta,\Lambda^\prime}^{\,-1}y}(\mathbf{g})^{-\zeta}
\qquad (\zeta\in\Lambda^\prime_{\textup{min}}).
\end{equation}
This in turn follows from a straightforward computation using \eqref{zetalength}, \eqref{cocycleg} and Remark \ref{basepoint}.
\end{proof}

We now have the following uniform variant of the quasi-polynomial representation $\pi_{\cc,\mathfrak{t}}^{\Lambda,\Lambda^\prime}$ of $\mathbb{H}_{\Lambda,\Lambda^\prime}$ from Theorem \ref{glueprop}.
\begin{theorem}\label{gthm}
Let $\mathbf{g}\in\mathcal{G}$, $\widehat{\mathbf{p}}\in\widehat{\mathcal{G}}$ and $\mathfrak{c}\in\mathcal{C}_{\Lambda,\Lambda^\prime}$. 
\begin{enumerate}
\item The formulas 
\begin{equation}\label{gthmform}
\begin{split}
\pi_{\mathbf{g},\widehat{\mathbf{p}},\mathfrak{c}}^{\Lambda,\Lambda^\prime}(T_j)x^y&:=k_j^{\chi_{\mathbb{Z}}(\alpha_j(y))}g_{D\alpha_j}(\alpha_j(y))^{-1}s_j\cdot_{\widehat{\mathbf{p}},\mathfrak{c}} x^y+(k_j-k_j^{-1})\nabla_j(x^y),\\
\pi_{\mathbf{g},\widehat{\mathbf{p}},\mathfrak{c}}^{\Lambda,\Lambda^\prime}(x^\lambda)x^y&:=x^{y+\lambda},\\
\pi_{\mathbf{g},\widehat{\mathbf{p}},\mathfrak{c}}^{\Lambda,\Lambda^\prime}(\omega)x^y&:=\omega\cdot_{\widehat{\mathbf{p}},\mathfrak{c}}x^y
\end{split}
\end{equation}
for $j\in [0,r]$, $\lambda\in\Lambda$, $\omega\in\Omega_{\Lambda^\prime}$ and $y\in E$ define a representation 
\[
\pi_{\mathbf{g},\widehat{\mathbf{p}},\mathfrak{c}}^{\Lambda,\Lambda^\prime}: \mathbb{H}_{\Lambda,\Lambda^\prime}\rightarrow\textup{End}(\mathbf{F}[E]).
\]  
\item Let $\cc\in\overline{C}_+$. Then
$\mathcal{P}_\Lambda^{(\cc)}$ 
is a $\mathbb{H}_{\Lambda,\Lambda^\prime}$-submodule of 
$(\mathbf{F}[E],\pi_{\mathbf{g},\widehat{\mathbf{p}},\mathfrak{c}}^{\Lambda,\Lambda^\prime})$, and the $\mathcal{P}_\Lambda$-linear automorphism $\Gamma_{\Lambda,\mathbf{g}}^{(\cc)}$ of $\mathcal{P}_\Lambda^{(\cc)}$ realises an isomorphism
\[
\Gamma_{\Lambda,\mathbf{g}}^{(\cc)}: \bigl(\mathcal{P}_\Lambda^{(\cc)},\pi_{\mathbf{g},\widehat{\mathbf{p}},\mathfrak{c}}^{\Lambda,\Lambda^\prime}(\cdot)\vert_{\mathcal{P}_\Lambda^{(\cc)}}\bigr)\overset{\sim}{\longrightarrow}
\mathcal{P}_{\Lambda,\mathfrak{t}_\cc(\widehat{\mathbf{g}},\mathfrak{c})}^{(\cc)}
\]
of $\mathbb{H}_{\Lambda,\Lambda^\prime}$-modules,
where $\widehat{\mathbf{g}}:=\widehat{\mathbf{p}}\cdot\mathbf{g}\in\widehat{\mathcal{G}}$.
\end{enumerate}
\end{theorem}
\begin{proof}
Fix $\cc\in\overline{C}_+$. The linear operators on $\mathbf{F}[E]$ defined by \eqref{gthmform} preserve $\mathcal{P}_\Lambda^{(\cc)}$, and
$\mathfrak{t}_{\cc}(\widehat{\mathbf{g}},\mathfrak{c})\in {}^{\Lambda}T_{\Lambda^\prime}^\cc$ by Corollary \ref{ginvcor}.
In view of Theorem \ref{glueprop}
 it thus suffices to show that 
\begin{equation}\label{todo10}
\Gamma_{\Lambda,\mathbf{g}}^{(\cc)}\bigl(\pi_{\mathbf{g},\widehat{\mathbf{p}},\mathfrak{c}}^{\Lambda,\Lambda^\prime}(h)\Gamma_{\Lambda,\mathbf{g}}^{(\cc)\,-1}(x^y)\bigr)=
\pi_{\cc,\mathfrak{t}_\cc(\widehat{\mathbf{g}},\mathfrak{c})}^{\Lambda,\Lambda^\prime}(h)x^y\qquad\quad (y\in\mathcal{O}_{\Lambda,\cc})
\end{equation}
for $h=T_j$ ($0\leq j\leq r$), $h=x^\lambda$ ($\lambda\in\Lambda$) and $h=\omega$ ($\omega\in\Omega_{\Lambda^\prime}$).

For $h=x^\lambda$ ($\lambda\in\Lambda$) this is trivial. For $h=T_j$ 
note that $\Gamma_{\Lambda,\mathbf{g}}^{(\cc)}$ commutes with the truncated divided difference operator $\nabla_j$ since $\Gamma_{\Lambda,\mathbf{g}}^{(\cc)}$ is $\mathcal{P}$-linear. Hence in this case \eqref{todo10} 
 follows from Lemma \ref{gtransaction}.

Finally, for $h=\omega$ ($\omega\in\Omega_{\Lambda^\prime}$) formula \eqref{todo10} follows from Lemma \ref{refactionext} and \eqref{todogtrans2}.
\end{proof}
\begin{remark}
Recall that both the definition of the extended quasi-polynomial representation $\pi_{\cc,\mathfrak{t}_\cc(\widehat{\mathbf{g}},\mathfrak{c})}^{\Lambda,\Lambda^\prime}$ of $\mathbb{H}_{\Lambda,\Lambda^\prime}$
and of $\Gamma_{\Lambda,\mathbf{g}}^{(\cc)}\in\textup{End}_{\mathcal{P}_\Lambda}(\mathcal{P}_\Lambda^{(\cc)})$ depend on the choice of a representative $\cc$ of the $W_\Lambda$-orbit $\mathcal{O}_{\Lambda,\cc}$ inside $\overline{C}_+$ (see Remark \ref{O3rem} and Remark \ref{O4rem}). This is no longer the case for
$\bigl(\mathcal{P}_\Lambda^{(\cc)},\pi_{\mathbf{g},\widehat{\mathbf{p}},\mathfrak{c}}^{\Lambda,\Lambda^\prime}(\cdot)\vert_{\mathcal{P}_\Lambda^{(\cc)}}\bigr)$, but now it depends on a choice of a base-point $\widehat{\mathbf{p}}\in\widehat{\mathcal{G}}$. Note that for all $\cc\in\overline{C}_+$
\[
\pi_{\mathbf{g},\widehat{\mathbf{p}},\mathfrak{c}}^{\Lambda,\Lambda^\prime}(\cdot)\vert_{\mathcal{P}_\Lambda^{(\cc)}}\simeq\pi_{\mathbf{g}^\prime,\widehat{\mathbf{p}}^\prime,\mathfrak{c}}^{\Lambda,\Lambda^\prime}(\cdot)\vert_{\mathcal{P}_\Lambda^{(\cc)}}
\]
when $\widehat{\mathbf{p}}\cdot\mathbf{g}=\widehat{\mathbf{p}}^\prime\cdot\mathbf{g}^\prime$, in view of Theorem \ref{gthm}(3).
\end{remark}

The uniform quasi-polynomial representations for different root data are related as follows. 
\begin{lemma}\label{ordercompatiblerep2}
Let $(\Lambda_i,\Lambda_i^\prime)\in (\mathcal{L}^{\times 2})_e$ such that $(\Lambda_1,\Lambda_1^\prime)\leq (\Lambda_2,\Lambda_2^\prime)$. Then
\[
\pi_{\mathbf{g},\widehat{\mathbf{p}},\mathfrak{c}}^{\Lambda_2,\Lambda_2^\prime}\vert_{\mathbb{H}_{\Lambda_1,\Lambda_1^\prime}}=
\pi_{\mathbf{g},\widehat{\mathbf{p}},\mathfrak{c}(\cdot)\vert_{\Lambda_1^\prime}}^{\Lambda_1,\Lambda_1^\prime}.
\]
\end{lemma}
\begin{proof}
This is immediate (cf. Proposition \ref{ordercompatiblerep}).
\end{proof}

\subsection{Uniform quasi-polynomial eigenfunctions}\label{ExtUnifSection}

 In the following theorem we introduce the uniform versions of the monic quasi-polynomial eigenfunctions $E_{y;\cc}^{\Lambda,\Lambda^\prime}(x;\mathfrak{t})$ from Theorem \ref{propEVext0}. 
 
\begin{proposition}\label{propgE}
Let $\mathcal{O}\subset E$ be a $W$-orbit,  $\widehat{\mathbf{p}}\in\widehat{\mathcal{G}}$, $\mathfrak{c}\in\mathcal{C}_{\Lambda,\Lambda^\prime}$ and
$\mathbf{g}\in\mathcal{G}$ such that 
\begin{equation}\label{ggencond}
\mathfrak{s}_y\mathfrak{t}_y(\widehat{\mathbf{g}},\mathfrak{c})\not=\mathfrak{s}_{y^\prime}\mathfrak{t}_{y^\prime}(\widehat{\mathbf{g}},\mathfrak{c})\,\,
\textup{ in }\, T_{\Lambda^\prime}\, \textup{ when }\, y,y^\prime\in \mathcal{O}\, \textup{ and } y\not=y^\prime,
\end{equation}
where $\widehat{\mathbf{g}}:=\widehat{\mathbf{p}}\cdot\mathbf{g}\in\widehat{\mathcal{G}}$.
For each $y\in \mathcal{O}$, the following holds true.
\begin{enumerate}
\item There exists a unique joint eigenfunction $\mathcal{E}_y^{\Lambda,\Lambda^\prime}(x;\mathbf{g},\widehat{\mathbf{p}},\mathfrak{c})\in\mathbf{F}[E]$ of
the commuting operators $\pi_{\mathbf{g},\widehat{\mathbf{p}},\mathfrak{c}}^{\Lambda,\Lambda^\prime}(Y^\mu)$ \textup{(}$\mu\in\Lambda^\prime$\textup{)}
satisfying
\begin{equation}\label{lot}
\mathcal{E}_y^{\Lambda,\Lambda^\prime}(x;\mathbf{g},\widehat{\mathbf{p}},\mathfrak{c})=
x^y+\textup{l.o.t.}
\end{equation}
\item We have
\begin{equation}\label{ee}
\pi_{\mathbf{g},\widehat{\mathbf{p}},\mathfrak{c}}^{\Lambda,\Lambda^\prime}(Y^\mu)\mathcal{E}_y^{\Lambda,\Lambda^\prime}(\cdot;\mathbf{g},\widehat{\mathbf{p}},\mathfrak{c})=(\mathfrak{s}_y\mathfrak{t}_y(\widehat{\mathbf{g}},\mathfrak{c}))^{-\mu}
\mathcal{E}_y^{\Lambda,\Lambda^\prime}(\cdot;\mathbf{g},\widehat{\mathbf{p}},\mathfrak{c})\qquad\forall\, \mu\in\Lambda^{\prime}.
\end{equation}
\item We have
\begin{equation}\label{relEun}
\Gamma_{\Lambda,\mathbf{g}}^{(\cc_y)}\Bigl(\mathcal{E}_y^{\Lambda,\Lambda^\prime}(x;\mathbf{g},\widehat{\mathbf{p}},\mathfrak{c})\Bigr)=
\Bigl(\prod_{\alpha\in\Pi(v_y^{-1})}g_\alpha(\alpha(c_y))\Bigr)
E_{y;\cc_y}^{\Lambda,\Lambda^\prime}(x;\mathfrak{t}_{\cc_y}(\widehat{\mathbf{g}},\mathfrak{c})).
\end{equation}
\end{enumerate}
\end{proposition}
\begin{proof}
Write $\cc=\cc_y\in\overline{C}_+$. Note that \eqref{ggencond} implies the genericity condition \eqref{eecond} for $\mathfrak{t}=\mathfrak{t}_\cc(\widehat{\mathbf{g}},\mathfrak{c})\in
{}^{\Lambda}T_{\Lambda^\prime}^\cc$, in view of \eqref{relspectrum}. Hence the right hand side of \eqref{relEun} is well defined. The result now follows immediately from Theorem \ref{gthm}(2),
Theorem \ref{propEVext0}(1)\&(2) and Lemma \ref{ISOlemma}.
\end{proof}
Note that
\[
\mathcal{E}_y^{\Lambda,\Lambda^\prime}(x;\mathbf{g},\widehat{\mathbf{p}},\mathfrak{c})=x^y\,\,\textup{ when }\,\, y\in\overline{C}_+
\]
by \eqref{lot}.
In particular, this holds true when $y\in E_{\textup{co}}=C^{[1,r]}$. 

The explicit formula for the $Y$-weight of $\mathcal{E}_y^{\Lambda,\Lambda^\prime}(x;\mathbf{g},\widehat{\mathbf{p}},\mathfrak{c})\in\mathbf{F}[E]$ is
\begin{equation}\label{tyg}
\mathfrak{s}_y\mathfrak{t}_y(\widehat{\mathbf{g}},\mathfrak{c})=\mathfrak{c}(\textup{pr}_{E_{\textup{co}}}(y))\prod_{\alpha\in\Phi_0^+}\bigl(k_\alpha^{\eta(\alpha(y))}\widehat{g}_\alpha(\alpha(y))\bigr)^\alpha,
\end{equation}
where $\widehat{\mathbf{g}}=\widehat{\mathbf{p}}\cdot\mathbf{g}$. 
Note that \eqref{ee} in particular implies that
\begin{equation}\label{centeraction}
\pi^{\Lambda,\Lambda^\prime}_{\mathbf{g},\widehat{\mathbf{p}},\mathfrak{c}}(Y^\zeta)\mathcal{E}_y^{\Lambda,\Lambda^\prime}(\cdot;\mathbf{g},\widehat{\mathbf{p}},\mathfrak{c})=\mathfrak{c}(\textup{pr}_{E_{\textup{co}}}(\cc_y))^{-\zeta}\mathcal{E}_y^{\Lambda,\Lambda^\prime}(\cdot;\mathbf{g},\widehat{\mathbf{p}},\mathfrak{c})
\qquad\quad(\zeta\in\Lambda^\prime\cap E_{\textup{co}})
\end{equation}
for all $y\in\mathcal{O}$.
\begin{proposition}\label{ordercompatiblerep3}
Let $(\Lambda_i,\Lambda_i^\prime)\in (\mathcal{L}^{\times 2})_e$ such that $(\Lambda_1,\Lambda_1^\prime)\leq (\Lambda_2,\Lambda_2^\prime)$. 
Let $\mathcal{O}\subset E$ be a $W$-orbit,  $\widehat{\mathbf{p}}\in\widehat{\mathcal{G}}$, $\mathfrak{c}\in\mathcal{C}_{\Lambda_2,\Lambda^\prime_2}$ and
$\mathbf{g}\in\mathcal{G}$ such that 
\begin{equation}\label{ggencond2}
\mathfrak{s}_y\mathfrak{t}_y(\widehat{\mathbf{g}},\mathfrak{c}(\cdot)\vert_{\Lambda_1^\prime})\not=\mathfrak{s}_{y^\prime}\mathfrak{t}_{y^\prime}(\widehat{\mathbf{g}},
\mathfrak{c}(\cdot)\vert_{\Lambda_1^\prime})\,\,
\textup{ in }\, T_{\Lambda_1^\prime}\, \textup{ when }\, y,y^\prime\in \mathcal{O}\, \textup{ and } y\not=y^\prime,
\end{equation}
where $\widehat{\mathbf{g}}:=\widehat{\mathbf{p}}\cdot\mathbf{g}\in\widehat{\mathcal{G}}$. Then
\begin{equation}\label{relEE}
\mathcal{E}_y^{\Lambda_2,\Lambda_2^\prime}(x;\mathbf{g},\widehat{\mathbf{p}},\mathfrak{c})=\mathcal{E}_y^{\Lambda_1,\Lambda_1^\prime}(x;\mathbf{g},\widehat{\mathbf{p}},\mathfrak{c}(\cdot)\vert_{\Lambda_1^\prime})
\qquad \forall\, y\in\mathcal{O}.
\end{equation}
\end{proposition}
\begin{proof}
Note that \eqref{ggencond2} implies that \eqref{ggencond} holds true for $\Lambda^\prime=\Lambda_2^\prime$. Hence both sides of \eqref{relEE} are well defined.  Furthermore, both sides of \eqref{relEE} are joint eigenfunctions of $\pi_{\mathbf{g},\widehat{\mathbf{p}},\mathfrak{c}}^{\Lambda,\Lambda^\prime}(Y^\mu)$ \textup{(}$\mu\in\Lambda_1^\prime$\textup{)}
by Lemma \ref{ordercompatiblerep2}. The result now follows from Proposition \ref{propgE}(1) applied to the pair of lattices $(\Lambda,\Lambda^\prime)=(\Lambda_1,\Lambda_1^\prime)$.
\end{proof}

The relation between $
\mathcal{E}_y^{\Lambda,\Lambda^\prime}(x;\mathbf{g},\widehat{\mathbf{p}},\mathfrak{c})$ and the quasi-polynomial eigenfunctions relative to the adjoint root datum is as follows.
\begin{corollary}\label{relEuncor}
Let $\mathcal{O}\subset E$ be a $W$-orbit, $\widehat{\mathbf{p}}\in\widehat{\mathcal{G}}$, $\mathfrak{c}\in\mathcal{C}_{\Lambda,\Lambda^\prime}$ and $\mathbf{g}\in\mathcal{G}$. Set $\widehat{\mathbf{g}}:=\widehat{\mathbf{p}}\cdot\mathbf{g}\in\widehat{\mathcal{G}}$ and suppose that 
the genericity conditions \eqref{ggencond2} hold true for $\Lambda_1^\prime=Q^\vee$.
For $y\in\mathcal{O}$ we then have
\begin{equation}\label{relEun2}
\Gamma_{Q^\vee,\mathbf{g}}^{(\cc_y)}\bigl(\mathcal{E}_y^{\Lambda,\Lambda^\prime}(x;\mathbf{g},\widehat{\mathbf{p}},\mathfrak{c})\bigr)=
\Bigl(\prod_{\alpha\in\Pi(v_y^{-1})}g_\alpha(\alpha(c_y))\Bigr)E_y^{\mathbf{J}(\cc_y)}(x;\mathfrak{t}_{\cc_y}(\widehat{\mathbf{g}})\vert_{Q^\vee}).
\end{equation}
Under these assumptions, 
$\mathcal{E}_y^{\Lambda,\Lambda^\prime}(x;\mathbf{g},\widehat{\mathbf{p}},\mathfrak{c})$ does not depend on
$\mathfrak{c}\in\mathcal{C}_{\Lambda,\Lambda^\prime}$.
\end{corollary}
\begin{proof}
We have $\Gamma_{\Lambda,\mathbf{g}}^{(\cc)}\vert_{\mathcal{P}^{(\cc)}}=\Gamma_{Q^\vee,\mathbf{g}}^{(\cc)}$ ($\cc\in\overline{C}_+$) and 
$\mathcal{E}_y^{\Lambda,\Lambda^\prime}(x;\mathbf{g},\widehat{\mathbf{p}},\mathfrak{c})\in\mathcal{P}^{(\cc_y)}$, hence the left hand side of \eqref{relEun2}
is well defined.

Since $(Q^\vee,Q^\vee)\leq (\Lambda,\Lambda^\prime)$, formula \eqref{relEun2} then follows from Proposition \ref{ordercompatiblerep3}, \eqref{relEun} and the fact that
\[
\mathcal{E}_{y;\cc_y}^{Q^\vee,Q^\vee}(x;\mathfrak{t}_{\cc_y}(\widehat{\mathbf{g}},\mathfrak{c}(\cdot)\vert_{Q^\vee}))=E_y^{\mathbf{J}(\cc_y)}(x;\mathfrak{t}_{\cc_y}(\widehat{\mathbf{g}})).
\]
Here we have also used that $\mathfrak{c}(\cdot)\vert_{Q^\vee}\equiv 1_T$.
\end{proof}

Changing the perspective by extending the eigenvalue equations from adjoint root datum to extended root data,
leads to the following main result of this subsection.
\begin{theorem}\label{unifcalE}
Let $\mathcal{O}\subset E$ be a $W$-orbit, $\widehat{\mathbf{p}}\in\widehat{\mathcal{G}}$ and $\mathbf{g}\in\mathcal{G}$. Suppose that
\begin{equation}\label{ggencond3}
\mathfrak{s}_y\mathfrak{t}_y(\widehat{\mathbf{g}})\not=\mathfrak{s}_{y^\prime}\mathfrak{t}_{y^\prime}(\widehat{\mathbf{g}})\,\,
\textup{ in }\, T\, \textup{ when }\, y,y^\prime\in \mathcal{O}\, \textup{ and } y\not=y^\prime,
\end{equation}
where $\widehat{\mathbf{g}}:=\widehat{\mathbf{p}}\cdot\mathbf{g}\in\widehat{\mathcal{G}}$. Then
\[
\mathcal{E}_y(x;\mathbf{g},\widehat{\mathbf{p}}):=\mathcal{E}_y^{Q^\vee,Q^\vee}(x;\mathbf{g},\widehat{\mathbf{p}},1_T)\qquad (y\in\mathcal{O})
\]
satisfies for all $(\Lambda,\Lambda^\prime)\in (\mathcal{L}^{\times 2})_e$ and all $\mathfrak{c}\in\mathcal{C}_{\Lambda,\Lambda^\prime}$, 
\begin{equation}\label{ee2}
\pi_{\mathbf{g},\widehat{\mathbf{p}},\mathfrak{c}}^{\Lambda,\Lambda^\prime}(Y^\mu)\mathcal{E}_y(\cdot;\mathbf{g},\widehat{\mathbf{p}})=(\mathfrak{s}_y\mathfrak{t}_y(\widehat{\mathbf{g}},\mathfrak{c}))^{-\mu}
\mathcal{E}_y(\cdot;\mathbf{g},\widehat{\mathbf{p}})\qquad\forall\, \mu\in\Lambda^{\prime}.
\end{equation}
Furthermore, 
\begin{equation}\label{ee3}
\Gamma_{Q^\vee,\mathbf{g}}^{(\cc_y)}\bigl(\mathcal{E}_y(x;\mathbf{g},\widehat{\mathbf{p}})\bigr)=
\Bigl(\prod_{\alpha\in\Pi(v_y^{-1})}g_\alpha(\alpha(c_y))\Bigr)E_y^{\mathbf{J}(\cc_y)}(x;\mathfrak{t}_{\cc_y}(\widehat{\mathbf{g}})\vert_{Q^\vee}).
\end{equation}
\end{theorem}
\begin{proof}
Fix $(\Lambda,\Lambda^\prime)\in (\mathcal{L}^{\times 2})_e$ and $\mathfrak{c}\in\mathcal{C}_{\Lambda,\Lambda^\prime}$.
By \eqref{decomptt} the parameter conditions \eqref{ggencond3} imply \eqref{ggencond}. Furthermore $\mathfrak{c}(\cdot)\vert_{Q^\vee}=1_T$, 
hence \eqref{relEE} shows that
\[
\mathcal{E}_y^{\Lambda,\Lambda^\prime}(x;\mathbf{g},\widehat{\mathbf{p}},\mathfrak{c})=\mathcal{E}_y(x;\mathbf{g},\widehat{\mathbf{p}})
\]
for all $y\in\mathcal{O}$. Then \eqref{ee2} follows from \eqref{ee}, and \eqref{ee3} from \eqref{relEun2}.
\end{proof}
\begin{remark}
Under special conditions on the parameters $\mathbf{g}$ and $\widehat{\mathbf{p}}$ the quasi-poly\-no\-mials $\mathcal{E}_y^{\Lambda,\Lambda^\prime}(x;\mathbf{g},\widehat{\mathbf{p}},\mathfrak{c})$ ($y\in\mathcal{O}$) may depend on $\mathfrak{c}\in\mathcal{C}_{\Lambda,\Lambda^\prime}$. For instance, this is the case 
when $(\Lambda,\Lambda^\prime)=(Q^\vee,P^\vee)$ 
with $Q^\vee\subsetneq P^\vee$ if the $Y$-spectrum satisfies \eqref{ggencond} but not \eqref{ggencond3}.
\end{remark}

Further results for $\mathcal{E}_y(x;\mathbf{g},\widehat{\mathbf{p}})$ and $\mathcal{E}_y^{\Lambda,\Lambda^\prime}(x;\mathbf{g},\widehat{\mathbf{p}},\mathfrak{c})$, such as orthogonality and symmetrisation, can be directly obtained from the corresponding results for $E_{y;\cc}^{\Lambda,\Lambda^\prime}(x;\mathfrak{t})$ and $E_y^J(x;\mathfrak{t})$ using \eqref{relEun} and \eqref{relEun2}.

By Lemma \ref{ginvlemmaext}(1), the $W_{\Lambda^\prime}$-action $\cdot_{\widehat{\mathbf{p}},\mathfrak{c}}$ on $\mathbf{F}[E]$ extends to a $W_{\Lambda^\prime}\ltimes\mathcal{P}_\Lambda$-action, with $\mathcal{P}_\Lambda$
acting by multiplication operators (cf. Lemma \ref{actiondeform}). This extends to a $W_{\Lambda^\prime}\ltimes\mathcal{Q}_\Lambda$-action on 
$\mathbf{F}_{\mathcal{Q}_\Lambda}[E]:=\mathcal{Q}_\Lambda\otimes_{\mathcal{P}_\Lambda}\mathbf{F}[E]$,
cf. Subsection \ref{aWaSection}.
\begin{theorem}\label{CGaffineextended}
For $\widehat{\mathbf{p}}\in\widehat{\mathcal{G}}$, $\mathbf{g}\in\mathcal{G}$ and $\mathfrak{c}\in\mathcal{C}$, the formulas
\begin{equation*}
\begin{split}
\sigma_{\mathbf{g},\widehat{\mathbf{p}},\mathfrak{c}}^{\Lambda,\Lambda^\prime}(s_j)(x^y)&:=
\frac{k_j^{\chi_{\mathbb{Z}}(\alpha_j(y))}(x^{\alpha_j^\vee}-1)}{g_{D\alpha_j}(\alpha_j(y))(k_jx^{\alpha_j^\vee}-k_j^{-1})}
s_j\cdot_{\widehat{\mathbf{p}},\mathfrak{c}}x^y+\Bigl(\frac{k_j-k_j^{-1}}{k_jx^{\alpha_j^\vee}-k_j^{-1}}\Bigr)x^{y-\lfloor D\alpha_j(y)\rfloor\alpha_j^\vee},\\
\sigma_{\mathbf{g},\widehat{\mathbf{p}},\mathfrak{c}}^{\Lambda,\Lambda^\prime}(\omega)(x^y)&:=\omega\cdot_{\widehat{\mathbf{p}},\mathfrak{c}}x^y,
\qquad\qquad\sigma_{\mathbf{g},\widehat{\mathbf{p}},\mathfrak{c}}^{\Lambda,\Lambda^\prime}(f)x^y:=fx^y
\end{split}
\end{equation*}
for $y\in E$, $0\leq j\leq r$, $\omega\in\Omega_{\Lambda^\prime}$ and $f\in\mathcal{Q}_\Lambda$
defines a representation 
\[
\sigma_{\mathbf{g},\widehat{\mathbf{p}},\mathfrak{c}}^{\Lambda,\Lambda^\prime}:W_{\Lambda^\prime}\ltimes\mathcal{Q}_\Lambda\rightarrow\textup{End}_{\mathbf{F}}(\mathbf{F}_{\mathcal{Q}_\Lambda}[E]).
\]
\end{theorem}
\begin{proof}
One can derive the result from Lemma \ref{aCGrem} by similar arguments as in the proof of Theorem \ref{gthm},
using the extension of $\Gamma_{\Lambda,\mathbf{g}}^{(\cc)}$ to a $\mathcal{Q}_\Lambda$-linear automorphism of $\mathcal{Q}_\Lambda\otimes_{\mathcal{P}_\Lambda}\mathcal{P}_\Lambda^{(\cc)}$. 
Alternatively, one can repeat the proof of Theorem \ref{aCG},
now using the $\mathbb{H}_{\Lambda,\Lambda^\prime}$-representation $\pi_{\mathbf{g},\widehat{\mathbf{p}},\mathfrak{c}}^{\Lambda,\Lambda^\prime}$.
\end{proof}

\section{Metaplectic representations and metaplectic polynomials}\label{MetaplecticSection}

In this subsection we generalise the $\textup{GL}_{r+1}$-type metaplectic double affine Hecke algebra representation and the associated metaplectic polynomials from 
\cite[\S 5]{SSV} to arbitrary root systems using the uniform quasi-polynomial representation
for extended double affine Hecke algebras (Theorem \ref{gthm}) and the corresponding uniform quasi-polynomials 
(Theorem \ref{unifcalE}). This section provides full proofs for the results announced in \cite[\S 5]{SSV}. 

The metaplectic double affine Hecke algebra action is an action on spaces $\mathcal{P}_\Lambda$ of polynomials on a torus $T_\Lambda$, with $\Lambda\in\mathcal{L}$. It depends on a metaplectic datum, which is a pair $(n,\mathbf{Q})$ with $n\in\mathbb{Z}_{>0}$ and $\mathbf{Q}: \Lambda \rightarrow\mathbb{Q}$ a non-zero $W_0$-invariant quadratic form that restricts to an integer-valued quadratic form on $Q^\vee$.
The metaplectic datum leads to a finite root system $\Phi_0^m$, the so-called metaplectic root system, which is isomorphic to either $\Phi_0$ or $\Phi_0^\vee$ (see Subsection \ref{S91}). Its associated co-root lattice $Q^{m\vee}$ is contained in $Q^\vee$. 

The metaplectic representation will be realised as a sub-representation of the uniform quasi-polynomial representation
relative to the metaplectic root system $\Phi_0^m$, 
with the parameters $(\mathbf{g},\widehat{\mathbf{p}},\mathfrak{c})$ fixed in such a way that the twisted $W_0\ltimes Q^{m\vee}$-action $\cdot_{\widehat{\mathbf{p}},\mathfrak{c}}$ on $\mathcal{P}_\Lambda$ reduces to the standard action by $q$-dilation and reflection operators. 

The associated uniform quasi-polynomials $\mathcal{E}_y(x;\mathbf{g},\widehat{\mathbf{p}})\in\mathcal{P}_\Lambda$ ($y\in\Lambda$) will provide the metaplectic polynomials. 

The $\mathbf{g}$-dependence reduces to a dependence on so-called metaplectic parameters $\underline{h}\in\mathcal{M}_{(n,\mathbf{Q})}$
which, in the context of representation theory of metaplectic covers of reductive groups over a non-archimedean local field, are certain Gauss sums
(see, e.g., \cite{CGP,CG}). 

{}In the remainder of the section we fix a metaplectic datum $(n,\mathbf{Q})$, and we assume  
that $\mathbf{Q}$ takes positive values on co-roots. We write $\kappa_\ell\in\mathbb{Z}_{>0}$ for the natural number
\begin{equation}\label{kappal}
\kappa_\ell:=\mathbf{Q}(\alpha^\vee)\qquad\quad (\alpha\in{}^\ell\Phi_0)
\end{equation}
for $\ell\in\{\textup{sh},\textup{lg}\}$.

\subsection{The metaplectic parameters}\label{S91}
Write $\mathbf{B}: \Lambda \times \Lambda \rightarrow
\mathbb{Q}$ for the $W_0$-invariant symmetric bilinear pairing
\[
\mathbf{B}(\lambda,\mu):=\mathbf{Q}(\lambda+\mu)-\mathbf{Q}(\lambda
)-\mathbf{Q}(\mu)\qquad (\lambda,\mu\in \Lambda).
\]
Then 
\begin{equation}\label{eqn:BQ}
\mathbf{B}(\lambda,\alpha^\vee)=\mathbf{Q}(\alpha^\vee)\alpha(\lambda)\qquad (\lambda\in \Lambda,\, \alpha\in \Phi_0)
\end{equation}
and
\begin{equation}\label{def:m}
m(\alpha):=\frac{n}{\textup{gcd}(n,\mathbf{Q}(\alpha^\vee))}=\frac{\textup{lcm}%
(n,\mathbf{Q}(\alpha^\vee))}{\mathbf{Q}(\alpha^\vee)} \qquad (\alpha\in\Phi_0)
\end{equation}
defines a $W_0$-invariant $\mathbb{Z}_{>0}$-valued function on $\Phi_0$. The associated metaplectic root system is defined by
\[
\Phi_0^m:=\left\{\alpha^m\,\, | \,\,
\alpha\in\Phi_0\right\}\,\,\textup{ with }\,\, \alpha^m:=m(\alpha)^{-1}\alpha.
\]
Then $\Phi_0^m\simeq\Phi_0$ if $m$ is constant, and $\Phi_0^m\simeq\Phi_0^\vee$ otherwise (see \cite{SSV}). In particular, the Weyl group of $\Phi_0^m$ is still $W_0$. 
The fixed basis $\{\alpha_1,\ldots,\alpha_r\}$ of $\Phi_0$ determines a basis $\{\alpha_1^m,\ldots,\alpha_r^m\}$ of $\Phi_0^m$.
Note that
$\Phi_0^{m\vee}=\{\alpha^{m\vee}\}_{\alpha\in\Phi_0}$, with $\alpha^{m\vee}=m(\alpha)\alpha^\vee$. Write 
\[
Q^{m\vee}:=\mathbb{Z}\Phi_0^{m\vee}\subseteq Q^\vee
\]
for the coroot lattice of $\Phi_0^m$.

Let $\Phi^m:=\Phi_0^m\times\mathbb{Z}$ be the affine root system of $\Phi_0^m$, with corresponding extended basis 
$\{\alpha_0^m,\alpha_1^m,\ldots,\alpha_r^m\}$. Then $\alpha_0^m=(-\vartheta^m,1)$ with 
$\vartheta^m\in\Phi_0^{m+}$
the highest root in $\Phi_0^{m+}$. For the underlying root $\vartheta\in\Phi_0^+$ this means that
\begin{equation*}
\vartheta=
\begin{cases}
\varphi\quad &\hbox{ if } m \hbox{ is constant},\\
\theta\quad &\hbox{ otherwise},
\end{cases}
\end{equation*}
where 
$\theta\in\Phi_0^+$ is the highest short root.

Consider the affine Weyl group $W^m=W_0\ltimes Q^{m\vee}$ corresponding to the affine root system $\Phi^m$. The associated simple reflections are denoted by 
\[
s_j^m:=s_{\alpha_j^m}\in W^m\qquad\quad (j\in [0,r]).
\]
Then $s_i^m=s_i$ for $i\in [1,r]$, and
$s_0^m=\tau(\vartheta^{m\vee})s_\vartheta$. 
The fundamental alcove is now given by
\[
C_+^m:=\{y\in E \,\,\, | \,\,\, \alpha_i(y)>0\,\,\, (1\leq i\leq r)\,\,\, \& \,\,\, \vartheta(y)<m(\vartheta)\}.
\]
Its faces are denoted by $C^{mJ}$ for $J\subsetneq [0,r]$.

Let $\mathbf{k}: \Phi^m\rightarrow\mathbf{F}^\times$ be a multiplicity function on $\Phi^m$. Its values at simple roots are denoted by $k_j:=k_{\alpha_j^m}$ for $j\in [0,r]$.
Write $\mathcal{G}^m$ for the $g$-parameter space $\mathcal{G}$ with respect to the root system $\Phi_0^m$ (see Remark \ref{basepoint}). 

\begin{definition}\label{hdefinition}
We write $\mathcal{M}=\mathcal{M}_{(n,\mathbf{Q})}$ for the set of tuples 
\[
\underline{h}:=(h_s(\alpha))_{s\in\mathbf{Q}(\alpha^\vee)\mathbb{Z},\,\alpha\in\Phi_0}
\]
with $h_s(\alpha)\in\mathbf{F}^\times$ satisfying 
\begin{equation}\label{gmetparremark}
\begin{split}
h_{s+\textup{lcm}(n,\mathbf{Q}(\alpha^\vee))}(v\alpha)&=h_s(\alpha)\qquad\quad (v\in W_0,\, s\in\mathbf{Q}(\alpha^\vee)\mathbb{Z}),\\
h_0(\alpha)&=-1,\\
h_s(\alpha)h_{-s}(\alpha)&=k_{\alpha^m}^{-2}\qquad\qquad (s\in\mathbf{Q}(\alpha^\vee)\mathbb{Z}\setminus\textup{lcm}(n,\mathbf{Q}(\alpha^\vee))\mathbb{Z})
\end{split}
\end{equation}
for all $\alpha\in\Phi_0$. 
\end{definition}
We call $\mathcal{M}_{(n,\mathbf{Q})}$ the set of metaplectic parameters relative to the metaplectic datum $(n,\mathbf{Q})$. Note that for $\underline{h}\in\mathcal{M}_{(n,\mathbf{Q})}$,
\begin{equation}\label{half}
h_{\textup{lcm}(n,\mathbf{Q}(\alpha^\vee))/2}(\alpha)\in\{\pm k_{\alpha^m}^{-1}\}\quad \hbox{ if }\,\, m(\alpha)\in 2\mathbb{Z}
\end{equation}
in view of \eqref{def:m}.

\begin{lemma}\label{metgpar}
We have a surjective map 
\[
\mathcal{G}^m\twoheadrightarrow\mathcal{M}_{(n,\mathbf{Q})},\qquad \mathbf{g}
\mapsto\underline{h}^{\mathbf{g}}
=(h_s^{\mathbf{g}}(\alpha))_{s\in\mathbf{Q}(\alpha^\vee)\mathbb{Z},\,\alpha\in\Phi_0}
\]
with $h_s^{\mathbf{g}}(\alpha)\in\mathbf{F}^\times$ defined by 
\begin{equation}\label{eq:metg}
h_s^{\mathbf{g}}(\alpha):=-k_{\alpha^m}^{-\chi_{\mathbb{Z}\setminus\textup{lcm}(n,\mathbf{Q}(\alpha^\vee))\mathbb{Z}}(s)}
g_{\alpha^m}(s/\textup{lcm}(n,\mathbf{Q}(\alpha^\vee))).
\end{equation}
\end{lemma}
\begin{proof}
This is a straightforward check, which is left to the reader.
\end{proof}

\begin{remark}
For the metaplectic representation of the affine Hecke algebra defined in \cite[Thm. 3.7]{SSV},
parameters $h_s(\alpha)$ ($s\in \mathbb{Z}$, $\alpha\in\Phi_0$) are used which satisfy
\begin{equation}\label{gmetparremarkn}
\begin{split}
h_{s+n}(v\alpha)&=h_s(\alpha)\qquad\quad (v\in W_0, s\in\mathbb{Z}),\\
h_0(\alpha)&=-1,\\
h_s(\alpha)h_{-s}(\alpha)&=k_{\alpha^m}^{-2}\qquad\qquad\,\,\,\, (s\in\mathbb{Z}\setminus n\mathbb{Z})
\end{split}
\end{equation}
for all $\alpha\in\Phi_0$ (see \cite[Def. 3.5]{SSV}), although the metaplectic representation only depends on the parameters
$h_s(\alpha)$ with $s\in\mathbf{Q}(\alpha^\vee)\mathbb{Z}$ and $\alpha\in\Phi_0$, which naturally form a tuple of metaplectic parameters in $\mathcal{M}_{(n,\mathbf{Q})}$.
Conversely, note that a tuple $\underline{h}=(h_s(\alpha))_{s\in\mathbf{Q}(\alpha^\vee)\mathbb{Z},\alpha\in\Phi_0}\in\mathcal{M}_{(n,\mathbf{Q})}$ of metaplectic parameters
can be extended to a set of parameters $(h_s(\alpha))_{s\in\mathbb{Z},\alpha\in\Phi_0}$ in such a way that \eqref{gmetparremarkn} holds true. For instance, one can take the extension
\begin{equation*}
\begin{split}
h_{\textup{gcd}(n,\mathbf{Q}(\alpha^\vee))s}(\alpha)&:=h_{\mathbf{Q}(\alpha^\vee)s\ell_\alpha}(\alpha)\qquad\,\,\,\,\, (s\in\mathbb{Z}),\\
h_s(\alpha)&:=k_{\alpha^m}^{-1}\qquad\qquad\qquad\,\,\, (s\in\mathbb{Z}\setminus\textup{gcd}(n,\mathbf{Q}(\alpha^\vee))\mathbb{Z})
\end{split}
\end{equation*}
for $\alpha\in\Phi_0$, where $\ell_\alpha\in\mathbb{Z}$ is such that $\textup{gcd}(n,\mathbf{Q}(\alpha^\vee))-\ell_\alpha\mathbf{Q}(\alpha^\vee)\in n\mathbb{Z}$.
\end{remark}

\subsection{The metaplectic basic representation}\label{msec}
In \cite[Thm. 5.4]{SSV} we introduced the metaplectic basic representation of the $\textup{GL}_{r+1}$ double affine Hecke algebra. 
In this subsection we introduce the metaplectic basic representation for extended double affine Hecke algebras for all types.

The metaplectic representation depends on $q$, $\mathbf{k}$ and on metaplectic parameters $\underline{h}\in\mathcal{M}_{(n,\mathbf{Q})}$. To keep the notations manageable, we will suppress the dependence on $\underline{h}\in\mathcal{M}_{(n,\mathbf{Q})}$ from the notations.

Let $\mathbb{H}^m:=\mathbb{H}(\mathbf{k},q)$ be the double affine Hecke algebra associated to the root system $\Phi_0^m$  (see Definition \ref{defDAHA}), which we call the {\it metaplectic double affine Hecke algebra}. 
We denote by $\mathcal{L}^m$ the set of finitely generated abelian subgroups $\Lambda^m\subset E$ such that 
\begin{equation}\label{latticecondm}
Q^{m\vee}\subseteq\Lambda^m\quad \&\quad \alpha(\Lambda^m)\subseteq m(\alpha)\mathbb{Z}\quad \forall\, \alpha\in\Phi_0.
\end{equation}
Note that it is the set $\mathcal{L}$ from Definition \ref{latticeconddef}, with the role of $\Phi_0$ taken over by $\Phi_0^m$. 

We have the following supplement to Lemma \ref{metgpar}.
\begin{lemma}\label{metgparext}
Let $\Lambda\in\mathcal{L}$ and $\mathbf{g}\in\mathcal{G}^m$. 
For $\alpha\in\Phi_0$ and $\lambda\in\Lambda$ we have
\begin{equation}\label{relgng2}
g_{\alpha^m}(\alpha^m(\lambda))=-k_{\alpha^m}^{1-\chi_{\mathbb{Z}}(\alpha(\lambda)/m(\alpha))}h_{\mathbf{Q}(\alpha^\vee)\alpha(\lambda)}^{\mathbf{g}}(\alpha).
\end{equation}
\end{lemma}
\begin{proof}
By \eqref{def:m} we have 
\[
\alpha^m(\lambda)= \frac{\mathbf{Q}(\alpha^\vee)\alpha(\lambda)}{\textup{lcm}(n,\mathbf{Q}(\alpha^\vee))},
\]
and $\alpha(\lambda)\in\mathbb{Z}$ since $\lambda\in\Lambda$. Hence we can apply \eqref{eq:metg} to obtain
\[
g_{\alpha^m}(\alpha^m(\lambda))=-k_{\alpha^m}^{\chi_{\mathbb{Z}\setminus\textup{lcm}(n,\mathbf{Q}(\alpha^\vee))\mathbb{Z}}(\mathbf{Q}(\alpha^\vee)\alpha(\lambda))}
h_{\mathbf{Q}(\alpha^\vee)\alpha(\lambda)}^{\mathbf{g}}(\alpha).
\]
The result now follows from the identity
\[
\chi_{\mathbb{Z}\setminus\textup{lcm}(n,\mathbf{Q}(\alpha^\vee))\mathbb{Z}}(\mathbf{Q}(\alpha^\vee)\alpha(\lambda))=1-
\chi_{\mathbb{Z}}(\alpha(\lambda)/m(\alpha)),
\]
which follows from an elementary computation using \eqref{def:m}.
\end{proof}

Let $\Lambda^{m\prime}\in\mathcal{L}^m$ such that $\langle\Lambda,\Lambda^{m\prime}\rangle\subseteq\mathbb{Z}$. One can for instance take $\Lambda^{m\prime}=Q^{m\vee}$, since the inclusions $Q^{m\vee}\subseteq Q^\vee\subseteq Q$ imply that 
\[
\langle \Lambda,Q^{m\vee}\rangle\subseteq\langle P^\vee,Q^{m\vee}\rangle\subseteq\mathbb{Z}.
\]
The metaplectic extended affine Weyl group $W_{\Lambda^{m\prime}}=W_0\ltimes\Lambda^{m\prime}$ acts on $\mathcal{P}_\Lambda$ by $q$-dilations and reflections,
\[
v(x^\lambda):=x^{v\lambda} \quad (v\in W_0),\qquad\quad \tau(\mu)(x^{\lambda}):=q^{-\langle\lambda,\mu\rangle}x^\lambda\quad (\mu\in \Lambda^{m\prime})
\]
for all $\lambda\in\Lambda$ (cf.  Subsection \ref{ap} and Subsection \ref{S71}).
Note that $s_0^m(x^\lambda)=q_\vartheta^{m(\vartheta)\vartheta(\lambda)}x^{s_\vartheta\lambda}$.

Let $r_\ell(s)\in\{0,\ldots,\ell-1\}$ be the remainder of $s\in\mathbb{Z}$ modulo $\ell\in\mathbb{Z}_{>0}$. Define linear maps $\nabla_j^m: \mathcal{P}_\Lambda\rightarrow\mathcal{P}_\Lambda$ for $j\in [0,r]$ by 
\begin{equation*}
\begin{split}
\nabla_0^m(x^\lambda)&:=\Bigl(\frac{1-(q_\vartheta^{m(\vartheta)}x^{-\vartheta^\vee})^{(r_{m(\vartheta)}(-\vartheta(\lambda))+\vartheta(\lambda))}}{1-(q_\vartheta^{m(\vartheta)}x^{-\vartheta^\vee})^{m(\vartheta)}}\Bigr)x^\lambda,\\
\nabla_i^m(x^\lambda)&:=\Bigl(\frac{1-x^{(r_{m(\alpha_i)}(\alpha_i(\lambda))-\alpha_i(\lambda))\alpha_i^\vee}}{1-x^{m(\alpha_i)\alpha_i^\vee}}\Bigr)x^\lambda
\end{split}
\end{equation*}
for $\lambda\in\Lambda$ and $i\in [1,r]$.
We now have the following action 
of the metaplectic double affine Hecke algebra on the space $\mathcal{P}_\Lambda$ of polynomials on $T_\Lambda$.
\begin{theorem}\label{met_basic_rep}
Fix $\Lambda\in\mathcal{L}$ and $\underline{h}\in\mathcal{M}_{(n,\mathbf{Q})}$.

For $\Lambda^m,\Lambda^{m\prime}\in \mathcal{L}^m$ such that $\Lambda^m\subseteq\Lambda$ and 
$\langle\Lambda,\Lambda^{m\prime}\rangle\subseteq\mathbb{Z}$, the formulas
\begin{equation}\label{pim}
\begin{split}
\pi^{m}_\Lambda(T_0)x^\lambda&:=-k_0h_{\mathbf{B}(\lambda,\vartheta^\vee)}(\vartheta)q_{\vartheta}^{m(\vartheta)\vartheta(\lambda)}x^{s_\vartheta\lambda}+(k_0-k_0^{-1})\nabla_0^m(x^\lambda),\\
\pi^{m}_\Lambda(T_i)x^\lambda&:=-k_ih_{-\mathbf{B}(\lambda,\alpha_i^\vee)}(\alpha_i)x^{s_i\lambda}+(k_i-k_i^{-1})\nabla_i^m(x^\lambda),\\
\pi^{m}_\Lambda(x^\mu)x^\lambda&:=x^{\lambda+\mu},\\
\pi^{m}_\Lambda(\omega)x^\lambda&:=\omega(x^\lambda)
\end{split}
\end{equation}
for $\lambda\in \Lambda$, $i\in [1,r]$, $\mu\in\Lambda^m$ and $\omega\in\Omega_{\Lambda^{m\prime}}$ define a representation
\[
\pi^m_\Lambda:
\mathbb{H}_{\Lambda^m,\Lambda^{m\prime}}\rightarrow\textup{End}(\mathcal{P}_\Lambda),
\]
where $\mathbb{H}_{\Lambda^m,\Lambda^{m\prime}}$ is the extension of the metaplectic double affine Hecke algebra $\mathbb{H}^m$ by the lattices
$(\Lambda^m,\Lambda^{m\prime})$.
\end{theorem}
\begin{proof}
Without loss of generality we may assume that $\mathbf{F}$ contains a $(2h)^{\textup{th}}$-root $q^{\frac{1}{2h}}$ of $q$. 
Write $\widehat{\mathcal{G}}^m$ for the set $\widehat{\mathcal{G}}$ of $\widehat{\mathbf{g}}$-parameters relative to the root system $\Phi_0^m$ (Definition \ref{gpar}). For $\widehat{\mathbf{p}}\in\widehat{\mathcal{G}}^m$, $\mathbf{g}\in\mathcal{G}^m$ and $\mathfrak{c}\in\mathcal{C}_{\Lambda^m,\Lambda^{m\prime}}$ the uniform quasi-polynomial representation 
$\pi_{\mathbf{g},\widehat{\mathbf{p}},\mathfrak{c}}^{\Lambda^m,\Lambda^{\prime m}}: \mathbb{H}_{\Lambda^m,\Lambda^{m\prime}}\rightarrow\textup{End}(\mathbf{F}[E])$ relative to the finite root system $\Phi_0^m$ (see Theorem \ref{gthm})
contains $\mathcal{P}_\Lambda$ as a subrepresentation, because $\Lambda^m\subseteq\Lambda$.
Denote 
\[
\pi_{\mathbf{g},\widehat{\mathbf{p}},\mathfrak{c};\Lambda}:=\pi_{\mathbf{g},\widehat{\mathbf{p}},\mathfrak{c}}^{\Lambda^m,\Lambda^{m\prime}}(\cdot)\vert_{\mathcal{P}_\Lambda}: \mathbb{H}_{\Lambda^m,\Lambda^{m\prime}}\rightarrow\textup{End}(\mathcal{P}_\Lambda)
\]
for its representation map. Its defining formulas are
\begin{equation}\label{metaplecticform}
\begin{split}
\pi_{\mathbf{g},\widehat{\mathbf{p}},\mathfrak{c};\Lambda}(T_j)x^\lambda&:=k_j^{\chi_{\mathbb{Z}}(\alpha_j^m(\lambda))}g_{D\alpha_j^m}(\alpha_j^m(\lambda))^{-1}s_j^m\cdot_{\widehat{\mathbf{p}},\mathfrak{c}} x^\lambda+(k_j-k_j^{-1})\nabla_j(x^\lambda),\\
\pi_{\mathbf{g},\widehat{\mathbf{p}},\mathfrak{c};\Lambda}(x^\mu)x^\lambda&:=x^{\lambda+\mu},\\
\pi_{\mathbf{g},\widehat{\mathbf{p}},\mathfrak{c};\Lambda}(\omega)x^\lambda&:=\omega\cdot_{\widehat{\mathbf{p}},\mathfrak{c}}x^\lambda
\end{split}
\end{equation} 
for $\lambda\in\Lambda$, $j\in [0,r]$, $\mu\in\Lambda^m$ and $\omega\in\Omega_{\Lambda^{m\prime}}$, with $\nabla_j$ the truncated divided difference operator \eqref{nablaj} relative to the root system $\Phi_0^m$.

Fix $\widehat{\mathbf{p}}\in\widehat{\mathcal{G}}^m$ and $\mathfrak{c}\in\mathcal{C}_{\Lambda^m,\Lambda^{\prime m}}$ such that
\begin{equation}\label{specialvalue}
\widehat{p}_{\alpha^m}(z):=q_{\alpha}^{m(\alpha)^2z/h}\quad (z\in n^{-1}\mathbb{Z}),\qquad\quad
\mathfrak{c}(\lambda):=q^\lambda\quad (\lambda\in\textup{pr}_{E_{\textup{co}}}(\Lambda)).
\end{equation}
For this specific choice we have
\begin{equation}\label{qydecomp}
q^\lambda=\mathfrak{t}_\lambda(\widehat{\mathbf{p}},\mathfrak{c})\qquad (\lambda\in\Lambda)
\end{equation} 
in $T_{\Lambda^{\prime m}}$ (compare with Example \ref{exampleqbaseExt}). In particular,
$w\cdot_{\widehat{\mathbf{p}},\mathfrak{c}}(x^\lambda)=w(x^\lambda)$ for $w\in W_{\Lambda^{m\prime}}$ and $\lambda\in\Lambda$.
The defining formulas \eqref{metaplecticform} for the corresponding
representation $\pi_{\mathbf{g};\Lambda}:=\pi_{\mathbf{g},\widehat{\mathbf{p}},\mathfrak{c};\Lambda}$ of $\mathbb{H}_{\Lambda^m,\Lambda^{m\prime}}$ reduce to
\begin{equation}\label{thisisit}
\begin{split}
\pi_{\mathbf{g};\Lambda}(T_j)x^\lambda&=k_j^{\chi_{\mathbb{Z}}(\alpha_j^m(\lambda))}g_{D\alpha_j^m}(\alpha_j^m(\lambda))^{-1}
s_j^m(x^\lambda)+(k_j-k_j^{-1})\nabla_j(x^\lambda),\\
\pi_{\mathbf{g};\Lambda}(x^\mu)x^\lambda&=x^{\lambda+\mu},\\
\pi_{\mathbf{g};\Lambda}(\omega)x^\lambda&=\omega(x^\lambda)
\end{split}
\end{equation}
for $j\in [0,r]$, $\omega\in\Omega_{\Lambda^{m\prime}}$, $\mu\in\Lambda^m$ and $\lambda\in\Lambda$. 
We now show that \eqref{thisisit} matches with the defining formulas for $\pi^{m}_\Lambda$ under the correspondence \eqref{eq:metg} between $\mathbf{g}$-parameters and metaplectic parameters.

Let $\mathbf{g}\in\mathcal{G}^m$ and set $\underline{h}:=\underline{h}^{\mathbf{g}}\in\mathcal{M}_{(n,\mathbf{Q})}$ in the defining formulas \eqref{pim} for $\pi^m_\Lambda$. 
It is then clear that
$\pi_{\mathbf{g};\Lambda}(p)=\pi^{m}_\Lambda(p)$ for $p\in\mathcal{P}_{\Lambda^m}$, and $\pi_{\mathbf{g};\Lambda}(\omega)=\pi^{m}_\Lambda(\omega)$ for $\omega\in\Omega_{\Lambda^{m\prime}}$. 
So it suffices to show that
\begin{equation}\label{todo100}
\pi_{\mathbf{g};\Lambda}(T_j)=\pi^{m}_\Lambda(T_j)\qquad\quad (j\in [0,r]).
\end{equation}
To show this, first note that
\[
\nabla_j(\cdot)\vert_{\mathcal{P}_\Lambda}=\nabla_j^m
\]
as endomorphisms of $\mathcal{P}_\Lambda$,
since $r_\ell(s)=s-\ell\lfloor\frac{s}{\ell}\rfloor$ for $s\in\mathbb{Z}$ and $\ell\in\mathbb{Z}_{>0}$ and
\[
\mathbf{B}(\lambda,\alpha^\vee)=\textup{lcm}(n,\mathbf{Q}(\alpha^\vee))\alpha^m(\lambda)
\] 
for $\lambda\in\Lambda$ and $\alpha\in\Phi_0$ by \eqref{eqn:BQ} and \eqref{def:m} (compare with the proof of Lemma \ref{metgparext}). Then \eqref{todo100} follows from 
\begin{equation}\label{relgng}
-k_{\alpha^m}h_{-\mathbf{B}(\lambda,\alpha^\vee)}^{\mathbf{g}}(\alpha)=k_{\alpha^m}^{\chi_{\mathbb{Z}}(\alpha^m(\lambda))}g_{\alpha^m}(\alpha^m(\lambda))^{-1}
\qquad (\alpha\in\Phi_0,\,\lambda\in\Lambda),
\end{equation}
which in turn is a consequence of \eqref{relgng2} and \eqref{gmetparremark}.
\end{proof}

\begin{definition}
We write $\mathcal{P}_\Lambda^{m}$ for $\mathcal{P}_\Lambda$ endowed with the $\mathbb{H}_{\Lambda^m,\Lambda^{m\prime}}$-actions $\pi^{m}_\Lambda$ from Theorem \ref{met_basic_rep}. We call $\pi^m_\Lambda$
the metaplectic basic representation.
\end{definition}

\begin{remark}\label{compLL}\hfill
\begin{enumerate}
\item
The condition $\langle\Lambda,\Lambda^{m\prime}\rangle\subseteq \mathbb{Z}$ in Theorem \ref{met_basic_rep} may be relaxed to 
$\langle\Lambda,\Lambda^{m\vee}\rangle\subseteq e^{-1}\mathbb{Z}$ when $q$ has an $e^{th}$ root in $\mathbf{F}$.
\item Suppose that $\Lambda_1,\Lambda_2\in\mathcal{L}$ with $\Lambda_1\subseteq\Lambda_2$ and that $\Lambda^m,\Lambda^{m\prime}\in\mathcal{L}^m$
such that $\Lambda^m\subseteq\Lambda_1$ and $\langle\Lambda_2,\Lambda^{m\prime}\rangle\subseteq\mathbb{Z}$. Then
\[
\pi_{\Lambda_2}^m(\cdot)\vert_{\mathcal{P}_{\Lambda_1}}=\pi_{\Lambda_1}^m
\]
as representations of $\mathbb{H}_{\Lambda^m,\Lambda^{m\prime}}$.
\end{enumerate}
\end{remark}

\begin{corollary}\label{Macdonaldmcase}
Let $\Lambda\in\mathcal{L}$ and $\underline{h}\in\mathcal{M}_{(n,\mathbf{Q})}$. 
Let $\Lambda^m,\Lambda^{m\prime}\in \mathcal{L}^m$ such that $\Lambda^m\subseteq\Lambda$ and 
$\langle\Lambda,\Lambda^{m\prime}\rangle\subseteq\mathbb{Z}$. 
The inclusion map
$\mathcal{P}_{\Lambda^m}\hookrightarrow\mathcal{P}_{\Lambda}$ is a morphism
\[
\mathcal{P}_{\Lambda^m,1_{T_{\Lambda^{m\prime}}}}^{(0)}\hookrightarrow \mathcal{P}_{\Lambda}^m
\]
of $\mathbb{H}_{\Lambda^m,\Lambda^{m\prime}}$-modules, 
 i.e., the $\mathbb{H}_{\Lambda^m,\Lambda^{m\prime}}$-submodule
$(\mathcal{P}_{\Lambda^m},\pi_\Lambda^m(\cdot)\vert_{\mathcal{P}_{\Lambda^m}})$ of $\mathcal{P}_\Lambda^m$
is Cherednik's polynomial representation of $\mathbb{H}_{\Lambda^m,\Lambda^{m\prime}}$.
\end{corollary}
\begin{proof}
Note that $\nabla_j^m(\cdot)\vert_{\mathcal{P}_{\Lambda^m}}$
is the standard divided difference operator associated to the simple root $\alpha_j^m\in\Phi^m$. Furthermore,
for $\lambda\in\Lambda^m$ and $\alpha\in\Phi_0$ we have
\[
\mathbf{B}(\lambda,\alpha^\vee)=\textup{lcm}(n,\mathbf{Q}(\alpha^\vee))\alpha^m(\lambda)\in\textup{lcm}(n,\mathbf{Q}(\alpha^\vee))\mathbb{Z}
\]
by \eqref{eqn:BQ} and \eqref{def:m}, hence $h_{\mathbf{B}(\lambda,\alpha^\vee)}(\alpha)=-1$ (see Definition \ref{hdefinition}).
Hence the defining formulas \eqref{pim} for $\pi_\Lambda^m$, restricted to $\mathcal{P}_{\Lambda^m}$,
reduce to the defining formulas for
$\pi_{0,1_{T_{\Lambda^{m\prime}}}}^{\Lambda^m,\Lambda^{m\prime}}$.
\end{proof}

\begin{remark}\label{examplesSSV} \hfill
\begin{enumerate}
\item[(1)] The restriction of $\pi^{m}_\Lambda$ to the affine Hecke algebra generated by $T_1, \dots, T_r$ and $x^{\nu}$ for $\nu \in Q^{m\vee}$ coincides with the metaplectic representation of the affine Hecke algebra from \cite[Thm. 3.7]{SSV} (be aware that the finite root system in \cite{SSV}, denoted by $\Phi$, is $\Phi_0^\vee$ in the present paper).
\item[(2)] We show in this remark that for the $\textup{GL}_{r+1}$ root datum, Theorem \ref{met_basic_rep} reduces to \cite[Thm. 5.4]{SSV}.
Recall the notations from Subsection \ref{S73}. Let $\Phi_0=\{\epsilon_i-\epsilon_j\}_{1\leq i\not=\leq j\leq r+1}$ be the root system of type $A_r$ with $\ell=1$ (i.e., the roots have squared length equal to $2$). As simple roots we take $\alpha_i:=\epsilon_i-\epsilon_{i+1}$ for $i=1,\ldots,r$.
Let $(n,\mathbf{Q})$ be an associated metaplectic datum. Write
$\kappa=\kappa_{\textup{lg}}\in\mathbb{Z}_{>0}$ for the value of $\mathbf{Q}$ at co-roots. The resulting multiplicity function $m$ is identically equal to $n/\textup{gcd}(n,\kappa)$, hence the metaplectic root system $\Phi_0^m=m^{-1}\Phi_0$ is the root system of type $A_{r+1}$ from Subsection \ref{S73} with $\ell=m$ (i.e., the squared length of the roots is equal to $2/m^2$). 

Take $\Lambda=\bigoplus_{i=1}^{r+1}\mathbb{Z}\epsilon_i$, which lies in $\mathcal{L}$, and 
\[
(\Lambda^m,\Lambda^{m\prime})=(m\mathbb{Z}^{r+1},m\mathbb{Z}^{r+1})\in (\mathcal{L}^m)^{\times 2}.
\]
Then 
\[
\widetilde{\mathbb{H}}^m=\mathbb{H}_{m\mathbb{Z}^{r+1},m\mathbb{Z}^{r+1}}
\]
is the $\textup{GL}_{r+1}$-type double affine Hecke algebra of type $\textup{GL}_{r+1}$ relative to the $A_r$ root system $\Phi_0^m$.
The resulting representation
\[
\pi^{m}_{\mathbb{Z}^{r+1}}: \widetilde{\mathbb{H}}^m\rightarrow\textup{End}(\mathcal{P}_{\mathbb{Z}^{r+1}})
\]
is the metaplectic basic representation $\widehat{\pi}^{(n,\kappa)}$ from \cite[Thm. 5.4]{SSV}.
\end{enumerate}
\end{remark}

The proof of Theorem \ref{met_basic_rep} implies that the metaplectic basic representation $\pi^m_\Lambda$ is a sub-representation of a uniform quasi-polynomial representation, hence it is isomorphic to a direct sum of extended quasi-polynomial representations (see Theorem \ref{gthm}(2)). 
We now describe this isomorphism in more detail.

Let $\Lambda\in\mathcal{L}$. The set
\[
X^m(\Lambda):=\Lambda\cap\overline{C}_+^m
\]
parametrizes the $W^m$-orbits in $\Lambda$, 
\begin{equation}\label{decomppLambda}
\Lambda=\bigsqcup_{\cc\in X^m(\Lambda)}\mathcal{O}_\cc.
\end{equation}
Note that $X^m(\Lambda)=\Lambda_{\textup{min}}$ when $\Lambda\in\mathcal{L}^m$. Consider the associated decomposition
\begin{equation}\label{deccP}
\mathcal{P}_\Lambda=\bigoplus_{\cc\in X^m(\Lambda)}\mathcal{P}^{(\cc)}
\end{equation}
as $\mathcal{P}_{Q^{m\vee}}$-modules. In fact, \eqref{deccP} is a decomposition as $\mathbb{W}^m$-modules, with $\mathbb{W}^m$ the double affine Weyl group associated to the metaplectic root system $\Phi_0^m$, see Proposition \ref{dWprop}.
For $\mathbf{g}\in\mathcal{G}^m$, consider the $\mathcal{P}_{Q^{m\vee}}$-linear automorphism
\[
\bigoplus_{\cc\in X^m(\Lambda)}\Gamma^{(\cc)}_{Q^{m\vee},\mathbf{g}}
\] 
of $\mathcal{P}_\Lambda$. By Lemma \ref{metgparext} and Lemma \ref{metgpar} we can express it entirely in terms of the 
metaplectic parameters $\underline{h}^{\mathbf{g}}$ associated to $\mathbf{g}$, leading to the following lemma.

In the following lemma the elements $\cc_y\in\overline{C}_+^m$ and $v_y\in W_0$ \textup{(}see Definition \ref{wydef}\textup{)} are relative to the metaplectic root system $\Phi_0^m$.
\begin{lemma}\label{relationGamma}
For fixed $\underline{h}=(h_s(\alpha))_{s,\alpha}\in\mathcal{M}_{(n,\mathbf{Q})}$
the formulas
\[
\Gamma^m(x^{\lambda}):=\Bigl(\prod_{\alpha\in\Pi(v_\lambda^{-1})}\bigl(-k_{\alpha^m}^{1-\chi_{\mathbb{Z}}(\alpha(\cc_\lambda)/m(\alpha))}h_{\mathbf{Q}(\alpha^\vee)\alpha(\cc_\lambda)}(\alpha)\bigr)
\Bigr)x^{\lambda}\qquad\quad (\lambda\in\Lambda)
\]
define a $\mathcal{P}_{Q^{m\vee}}$-linear
automorphism $\Gamma^m=\Gamma^m_{\underline{h}}$ of $\mathcal{P}_\Lambda$ satisfying
\[
\Gamma^m_{\underline{h}^{\mathbf{g}}}=\bigoplus_{\cc\in X^m(\Lambda)}\Gamma_{Q^{m\vee},\mathbf{g}}^{(\cc)}\qquad\quad (\mathbf{g}\in\mathcal{G}^m).
\]
\end{lemma}
\begin{proof}
This follows from Lemma \ref{ISOlemma}, Lemma \ref{metgpar} and Lemma \ref{metgparext}.
\end{proof}
For $\underline{h}\in\mathcal{M}_{(n,\mathbf{Q})}$ and $\lambda\in\Lambda$ write
\begin{equation}\label{tgenmet}
\mathfrak{t}_\lambda(\underline{h}):=\prod_{\alpha\in\Phi_0^+}\Bigl(-k_{\alpha^m}^{1-\chi_{\mathbb{Z}}(\alpha(\lambda)/m(\alpha))}h_{\mathbf{Q}(\alpha^\vee)\alpha(\lambda)}(\alpha)\Bigr)^{\frac{\alpha}{m(\alpha)}},
\end{equation}
viewed as multiplicative character of $\Lambda^{m\prime}$ for any $\Lambda^{m\prime}\in\mathcal{L}^m$. Note that 
\begin{equation}\label{relrel}
\mathfrak{t}_\lambda(\underline{h}^{\mathbf{g}})=
\mathfrak{t}_\lambda(\mathbf{g})\qquad\quad(\mathbf{g}\in\mathcal{G}^m),
\end{equation} 
where we use the definition \eqref{typrime} of $\mathfrak{t}_\lambda(\mathbf{f})$ relative to the
metaplectic root system $\Phi_0^m$. By Lemma \ref{ginvlemmaext}(2) and Lemma \ref{metgpar} we conclude that for $\cc\in X^m(\Lambda)$ and $\underline{h}\in\mathcal{M}_{(n,\mathbf{Q})}$,
\begin{equation}\label{rr2}
\mathfrak{t}_\cc(\underline{h})\in T_{\mathbf{J}(\cc),\Lambda^{m\prime}}^{\textup{red}}\,\,\hbox{ and }\,\, q^\cc, q^\cc\mathfrak{t}_\cc(\underline{h})\in {}^{\Lambda^m}T_{\Lambda^{m\prime}}^\cc.
\end{equation}

\begin{corollary}\label{corSSVlink}
Let $\Lambda\in\mathcal{L}$ and $\underline{h}\in\mathcal{M}_{(n,\mathbf{Q})}$.
\begin{enumerate}
\item 
The $\mathcal{P}_{Q^{m\vee}}$-linear automorphism $\Gamma^m$ of $\mathcal{P}_\Lambda$ defines an isomorphism
 \[
\mathcal{P}^{m}_\Lambda\overset{\sim}{\longrightarrow} \bigoplus_{\cc\in X^m(\Lambda)}\mathcal{P}_{q^\cc\mathfrak{t}_{\cc}(\underline{h})}^{(\cc)}
\]
of $\mathbb{H}^m$-modules.
\item For $\Lambda^m,\Lambda^{m\prime}\in \mathcal{L}^m$ such that $\Lambda^m\subseteq\Lambda$ and 
$\langle\Lambda,\Lambda^{m\prime}\rangle\subseteq\mathbb{Z}$ we have
\begin{equation}\label{isoEXTm}
\mathcal{P}^{m}_\Lambda\simeq \bigoplus_{\cc\in X^m(\Lambda,\Lambda^m)}\mathcal{P}_{\Lambda^m,q^\cc\mathfrak{t}_\cc(\underline{h})}^{(\cc)}
\end{equation}
as $\mathbb{H}_{\Lambda^m,\Lambda^{m\prime}}$-modules, where $X^m(\Lambda,\Lambda^m)$ is a complete set of representatives of the $\Omega_{\Lambda^{m}}$-orbits in $X^m(\Lambda)$. 
\end{enumerate}
\end{corollary}
\begin{proof}
Without loss of generality we may assume that $\mathbf{F}$ contains a $(2h)^{th}$ root of $q$.

Let $\mathbf{g}\in\mathcal{G}^m$ such that $\underline{h}=\underline{h}^{\mathbf{g}}$.
Choose $\widehat{\mathbf{p}}\in\widehat{\mathcal{G}}^m$ and $\mathfrak{c}\in\mathcal{C}_{\Lambda^m,\Lambda^{m\prime}}$ satisfying \eqref{specialvalue}. Then
\begin{equation}\label{isostart}
\mathcal{P}_\Lambda^m=(\mathcal{P}_\Lambda,\pi^{\Lambda^m,\Lambda^{m\prime}}_{\mathbf{g},\widehat{\mathbf{p}},\mathfrak{c}}(\cdot)\vert_{\mathcal{P}_\Lambda})
\end{equation}
as $\mathbb{H}_{\Lambda^m,\Lambda^{m\prime}}$-modules, by the proof of Theorem \ref{met_basic_rep}. Then
Lemma \ref{relationGamma} and Theorem \ref{gthm}(2) show that
the automorphism $\Gamma^m$ of $\mathcal{P}_\Lambda$ is an isomorphism
\[
\mathcal{P}_\Lambda^m\overset{\sim}{\longrightarrow}\bigoplus_{\cc\in X^m(\Lambda)}\mathcal{P}^{(\cc)}_{\mathfrak{t}_\cc(\widehat{\mathbf{g}},\mathfrak{c})}
\]
of $\mathbb{H}^m$-modules, where $\widehat{\mathbf{g}}:=\widehat{\mathbf{p}}\cdot\mathbf{g}$. Then (1) follows from the fact that
\begin{equation}\label{rr1}
\mathfrak{t}_\lambda(\widehat{\mathbf{g}},\mathfrak{c})=q^\lambda\mathfrak{t}_\lambda(\mathbf{g})=q^\lambda\mathfrak{t}_\lambda(\underline{h}^{\mathbf{g}})\qquad\quad
(\lambda\in\Lambda)
\end{equation}
by \eqref{qydecomp} and \eqref{relrel}.

Part (2) follows in a similar manner, now using that $X^m(\Lambda,\Lambda^m)$ parametrizes the $W_{\Lambda^m}$-orbits in $\Lambda$, leading to
the decomposition
\[
\mathcal{P}_\Lambda=\bigoplus_{\cc\in X^m(\Lambda,\Lambda^m)}\mathcal{P}_{\Lambda^m}^{(\cc)}
\]
of $\mathcal{P}_\Lambda$ as $\mathcal{P}_{\Lambda^m}$-modules. The isomorphism 
\[
\mathcal{P}^{m}_\Lambda\overset{\sim}{\longrightarrow} \bigoplus_{\cc\in X^m(\Lambda,\Lambda^m)}\mathcal{P}_{\Lambda^m,q^\cc\mathfrak{t}_\cc(\underline{h})}^{(\cc)}
\]
of $\mathbb{H}_{\Lambda^m,\Lambda^{m\prime}}$-modules is then realised by
$\bigoplus_{\cc\in X^m(\Lambda,\Lambda^m)}\Gamma^{(\cc)}_{\Lambda^m,\mathbf{g}}$, where $\Gamma_{\Lambda^m,\mathbf{g}}^{(\cc)}$ is the $\mathcal{P}_{\Lambda^m}$-linear automorphism of $\mathcal{P}_{\Lambda^m}^{(\cc)}$ from Lemma \ref{ISOlemma} and $\mathbf{g}\in\mathcal{G}^m$ is such that $\underline{h}=\underline{h}^{\mathbf{g}}$.
\end{proof}

\begin{remark}
Remark \ref{O3rem} and the fact that $\Gamma^m$ restricts to the identity on $\mathcal{P}_{\Lambda^m}$
show that Corollary \ref{corSSVlink}(2) is consistent with Corollary \ref{Macdonaldmcase}.
\end{remark}

Corollary \ref{corSSVlink}(2) and Proposition \ref{twistactionprop} allow to add twist parameters to $\pi^m_\Lambda: \mathbb{H}_{\Lambda^m,\Lambda^{m\prime}}\rightarrow\textup{End}(\mathcal{P}_\Lambda)$, with the twist parameters encoding a choice of multiplicative character of $\Omega_{\Lambda^{m\prime}}$ for each
$W^m$-orbit in $\mathcal{O}_\Lambda$.
We do not give full details here, since the metaplectic polynomials defined in Subsection \ref{mpsec} do not depend on these twist parameters.

We now count the number of free $h_s(\alpha)$-parameters in the metaplectic basic representation $\pi^{m}_\Lambda: \mathbb{H}_{\Lambda^m,\Lambda^{m\prime}}\rightarrow
\textup{End}(\mathcal{P}_\Lambda)$.
Consider the subgroup
\[
Z_{\Lambda,\textup{lg}}^m:=\{\varphi(\lambda)\,\, \textup{mod}\,\,m(\varphi)\mathbb{Z}\,\,\, | \,\,\, \lambda\in\Lambda\}
\]
of $\mathbb{Z}/m(\varphi)\mathbb{Z}$ and the subgroup
\begin{equation*}
Z_{\Lambda,\textup{sh}}^m:=
\begin{cases}
\{\theta(\lambda)\,\, \textup{mod}\,\,m(\theta)\mathbb{Z}\,\,\, | \,\,\, \lambda\in\Lambda\}\qquad &\hbox{ if }\, \theta\not=\varphi,\\
\{\overline{0}\}\qquad &\hbox{ if }\, \theta=\varphi
\end{cases}
\end{equation*}
of $\mathbb{Z}/m(\theta)\mathbb{Z}$,
where $\overline{0}$ is the neutral element of 
$\mathbb{Z}/m(\theta)\mathbb{Z}$. Both subgroups are independent of the lattices $\Lambda^m$ and $\Lambda^{m\prime}$. 

For $s\in\mathbb{Z}_{>0}$ the $2$-group $\mathbb{Z}_2$ acts on $\mathbb{Z}/s\mathbb{Z}$ by group automorphisms via sign changes.
Any subgroup $Z$ of 
$\mathbb{Z}/s\mathbb{Z}$ is $\mathbb{Z}_2$-stable, and its subgroup $Z^{\textup{inv}}$ of $\mathbb{Z}_2$-fixed elements is of order $1$ if $[\mathbb{Z}/s\mathbb{Z}:Z]$
is odd and of order $2$ otherwise.
This in particular applies to the subgroups $Z_{\Lambda,\textup{lg}}^m\subseteq\mathbb{Z}/m(\varphi)\mathbb{Z}$ and
$Z_{\Lambda,\textup{sh}}^m\subseteq\mathbb{Z}/m(\theta)\mathbb{Z}$. 

Write
\[
Z_{\Lambda,\ell}^{m,\textup{reg}}:=Z_{\Lambda,\ell}^m\setminus Z_{\Lambda,\ell}^{m,\textup{inv}}\subseteq Z_{\Lambda,\ell}^m.
\]
Then $Z_{\Lambda,\textup{sh}}^{m,\textup{reg}}=\emptyset$ if $\Phi_0=\Phi_0^{\textup{lg}}$ is simply laced. Let
\[
\widetilde{Z}_{\Lambda,\ell}^{m,\textup{reg}}\subseteq Z_{\Lambda,\ell}^{m,\textup{reg}}
\]
be a complete set of representatives of the $\mathbb{Z}_2$-orbits in $Z_\ell^{m,\textup{reg}}$.
\begin{proposition}\label{pargcount}
The metaplectic basic representation $\pi^m_\Lambda: \mathbb{H}_{\Lambda^m,\Lambda^{m\prime}}\rightarrow\textup{End}(\mathcal{P}_\Lambda)$ depends on
\[
\sum_{\ell\in\{\textup{sh},\textup{lg}\}}\#\widetilde{Z}_{\Lambda,\ell}^{m,\textup{reg}}
\]
free parameters $\{h_{\kappa_\ell z}(\varphi_\ell) \,\, | \,\, 
z\in \widetilde{Z}_{\Lambda,\ell}^{m,\textup{reg}}, \ell\in\{\textup{sh},\textup{lg}\}\}$
and on 
\[
\sum_{\ell\in\{\textup{sh},\textup{lg}\}}\Bigl(\#Z_{\Lambda,\ell}^{m,\textup{inv}}-1\Bigr)
\]
choices of signs $\{ h_{\kappa_\ell z}(\varphi_\ell) \,\, | \,\,
z\in Z_{\Lambda,\ell}^{m,\textup{inv}}\setminus\{\overline{0}\}, \ell\in\{\textup{sh},\textup{lg}\}\}$.
\end{proposition}
\begin{proof}
This follows from Definition \ref{hdefinition}, \eqref{half}, 
Theorem \ref{met_basic_rep}, \eqref{eqn:BQ} and \eqref{def:m}.
\end{proof}
Note that the maximum number of signs in $\pi_\Lambda^m$ is $2$.
\begin{remark}\label{check}
By Proposition \ref{pargcount} the metaplectic representation
$\pi_{\Lambda}^{m}$ can be defined over $\mathbf{F}=\mathbb{Q}\bigl(q,\mathbf{k}^{\frac{1}{2}},\check{h}\bigr)$, where $\check{h}$
stands for the parameters $h_{\kappa_\ell z}(\varphi_\ell)$ ($z\in \widetilde{Z}_{\Lambda,\ell}^{m,\textup{reg}}$, $\ell\in\{\textup{sh},\textup{lg}\}$),
viewed as indeterminates. 
\end{remark}

\begin{remark}
Consider the $\textup{GL}_{r+1}$ metaplectic basic representation $\pi_{\mathbb{Z}^{r+1}}^m: \widetilde{\mathbb{H}}\rightarrow\textup{End}(\mathcal{P}_{\mathbb{Z}^{r+1}})$ (see Remark \ref{examplesSSV}(2)).
In this case the multiplicity function $\mathbf{k}$ is a single scalar $k\in\mathbf{F}^\times$, $Z_{\textup{sh}}^m=\{\overline{0}\}$ and  
\[
Z_{\textup{lg}}^m=\mathbb{Z}/m\mathbb{Z},
\]
hence $Z_{\textup{lg}}^{m,\textup{reg}}=Z_{\textup{lg}}^m\setminus\{\overline{0}\}$ if $m$ is odd and $Z_{\textup{lg}}^{m,\textup{reg}}=Z_{\textup{lg}}^m\setminus\{\overline{0},
\overline{m/2}\}$
if $m$ is even. As representatives for the $\mathbb{Z}_2$-orbits in $Z_\ell^{m,\textup{reg}}$ we can take
\begin{equation*}
\widetilde{Z}_{\textup{lg}}^m=\{\overline{\ell}\,\, | \,\, \ell\in\mathbb{Z}:\,\, 1\leq \ell<\tfrac{m}{2}\}\subset \mathbb{Z}/m\mathbb{Z},
\end{equation*}
which is empty if $m\leq 2$. We conclude that the representation $\pi_{\mathbb{Z}^{r+1}}^{m}$ does not have free metaplectic parameters if $m\leq 2$, and it has
$\lfloor \frac{m-1}{2}\rfloor$ free metaplectic parameters if $m>2$. Furthermore, $\pi_{\mathbb{Z}^{r+1}}^m$ depends on a sign $\epsilon\in\{\pm 1\}$ when $m$ is even. This is in accordance with \cite[\S 5]{SSV}.
\end{remark}
\subsection{A metaplectic affine Weyl group action on rational functions}\label{CGsubsection}
Denote by $\mathcal{Q}^m$ the quotient field of $\mathcal{P}_{Q^{m\vee}}$.
For a finitely generated abelian subgroup $L\subset E$ let $\mathbf{F}_{\mathcal{Q}^m}[L]$ be the $\mathcal{Q}^m$-subalgebra of $\mathbf{F}_{\mathcal{Q}^m}[E]$ generated by $x^\lambda$ ($\lambda\in L$).
In view of \eqref{relgng} we can now formulate the following metaplectic version of Theorem \ref{CGaffineextended}. 
\begin{corollary}\label{WhitCGcor}
For $\underline{h}\in\mathcal{M}_{(n,\mathbf{Q})}$ the formulas
\begin{equation*}
\begin{split}
\sigma^{m}_\Lambda(s_0^m)(x^\lambda)&:=-k_0h_{\mathbf{B}(\lambda,\vartheta^\vee)}(\vartheta)
\Bigl(\frac{x^{\alpha_0^{m\vee}}-1}{k_0x^{\alpha_0^{m\vee}}-k_0^{-1}}\Bigr)q_\vartheta^{m(\vartheta)\vartheta(\lambda)}x^{s_\vartheta\lambda}\\&\qquad\qquad\qquad\qquad
+\Bigl(\frac{k_0-k_0^{-1}}{k_0x^{\alpha_0^{m\vee}}-k_0^{-1}}\Bigr)
x^{\lambda-\lfloor -\vartheta^m(\lambda)\rfloor \alpha_0^{m\vee}},\\
\sigma^{m}_\Lambda(s_i^m)(x^\lambda)&:=-k_ih_{-\mathbf{B}(\lambda,\alpha_i^\vee)}(\alpha_i)
\Bigl(\frac{x^{\alpha_i^{m\vee}}-1}{k_ix^{\alpha_i^{m\vee}}-k_i^{-1}}\Bigr)x^{s_i\lambda}+
\Bigl(\frac{k_i-k_i^{-1}}{k_ix^{\alpha_i^{m\vee}}-k_i^{-1}}\Bigr)
x^{\lambda-\lfloor \alpha_i^m(\lambda)\rfloor \alpha_i^{m\vee}},\\
\sigma_\Lambda^{m}(\omega)(x^\lambda)&:=\omega(x^\lambda),\qquad \sigma_\Lambda^{m}(f)x^\lambda:=fx^\lambda
\end{split}
\end{equation*}
for $\lambda\in\Lambda$, $1\leq i\leq r$, $\omega\in\Omega_{\Lambda^{m\prime}}$ and $f\in\mathbf{F}_{\mathcal{Q}^m}[\Lambda^m]$
define a representation 
\[
\sigma^m_\Lambda: W_{\Lambda^{m\prime}}\ltimes\mathbf{F}_{\mathcal{Q}^m}[\Lambda^m]\rightarrow
\textup{End}_{\mathbf{F}}(\mathbf{F}_{\mathcal{Q}^m}[\Lambda]).
\]
\end{corollary}
Note that $\sigma_\Lambda^m=\pi_\Lambda^{m,X-\textup{loc}}\circ\ss$, cf. Remark \ref{linkCGlocalrem}.

The restriction of $\sigma_\Lambda^{m}$ to $W_0\ltimes\mathbf{F}_{\mathcal{Q}^m}[\Lambda^m]$ reduces to \cite[Prop. 3.19(iii)]{SSV} (with the role of the root system replaced by the dual root system). 
Let $\rho^\vee$ and $\rho^{m\vee}$ be the half-sum of positive co-roots for $\Phi_0$ and $\Phi_0^m$, respectively.  Then the conjugated $W_0$-action 
\begin{equation}\label{sigmaCG}
\sigma^{\textup{CG}}_\Lambda(v)f:=x^{\rho^\vee-\rho^{m\vee}}\sigma_{\Lambda+P^\vee}^m(v)\bigl(x^{-\rho^\vee+\rho^{m\vee}}f\bigr)\qquad (v\in W_0,\, f\in\mathbf{F}_{\mathcal{Q}^m}[\Lambda])
\end{equation}
reduces to the Chinta-Gunnells' \cite{CG07,CG} action when the multiplicity functions $\mathbf{k}$ and $h_s$ are constant,
see \cite[Thm. 3.21 \& Cor. 3.22]{SSV} (the right hand side of \eqref{sigmaCG} is well defined, since $\mathbf{Q}$ has a unique extension to a $W_0$-invariant rational quadratic form on $\Lambda+P^\vee$). This follows from the explicit formula
\begin{equation}\label{sigmaCGi}
\begin{split}
\sigma^{\textup{CG}}_\Lambda(s_i)(fx^\lambda)&=k_i^2h_{\mathbf{Q}(\alpha_i^\vee)-\mathbf{B}(\lambda,\alpha_i^\vee)}(\alpha_i)
\Bigl(\frac{1-x^{-\alpha_i^{m\vee}}}{1-k_i^2x^{\alpha_i^{m\vee}}}\Bigr)(s_if)x^{\alpha_i^\vee+s_i\lambda}\\
&+\Bigl(\frac{1-k_i^2}{1-k_i^2x^{\alpha_i^{m\vee}}}\Bigr)(s_if)x^{\lambda+\lfloor -(\lambda,\alpha_i^{m\vee})\rfloor\alpha_i^{m\vee}}
\end{split}
\end{equation}
for $i\in [1,r]$, $f\in\mathbf{F}_{\mathcal{Q}^m}[\Lambda^m]$ and $\lambda\in\Lambda$, which easily follows from Corollary \ref{WhitCGcor} (see \cite[Cor. 3.22]{SSV}).
Hence
\begin{equation}\label{metaTi}
\begin{split}
\mathcal{T}_i^m(f):=&-f+\Bigl(\frac{1-k_i^2x^{\alpha_i^{m\vee}}}{1-x^{\alpha_i^{m\vee}}}\Bigr)\bigl(f-x^{\alpha_i^{m\vee}}\sigma_\Lambda^{\textup{CG}}(s_i)(f)\bigr)\\
=&-f+\Bigl(\frac{1-k_i^2x^{\alpha_i^{m\vee}}}{1-x^{\alpha_i^{m\vee}}}\Bigr)\bigl(f-x^{\rho^\vee}\sigma_\Lambda^m(s_i)(fx^{-\rho^\vee})\bigr)\\
=&-k_ix^{\rho^\vee}\pi_{\Lambda}^{m,X-\textup{loc}}(T_i^{-1})(fx^{-\rho^\vee})
\end{split}
\end{equation}
for $i\in [1,r]$ and $f\in\mathbf{F}_{\mathcal{Q}^m}[\Lambda]$
are the metaplectic Demazure-Lusztig operators from \cite[(11)]{CGP} when the multiplicity functions $\mathbf{k}$ and $h_s$ are constant (see Remark \ref{DLform} for the third equality in \eqref{metaTi}). By the third equality in \eqref{metaTi} the ($W_0,\{s_1,\ldots,s_r\})$-braid relations of the metaplectic Demazure-Lusztig operators (see, e.g., \cite[Prop. 7]{CGP}) are a consequence of Theorem \ref{met_basic_rep}. 
In fact, Corollary \ref{WhitCGcor} and Theorem \ref{met_basic_rep} extends the Chinta-Gunnells $W_0$-action and the $H_0$-action by metaplectic Demazure-Lusztig operators to actions of the affine Weyl group $W$ and the affine Hecke algebra $H$, respectively. Including the action of $\mathcal{P}$ by multiplication operators, they promote to  actions
of the double affine Weyl group $\mathbb{W}$ and the double affine Hecke algebra $\mathbb{H}$, respectively.

\begin{remark}
Recently the Chinta-Gunnells $W_0$-action \eqref{sigmaCG} and the associated $H_0$-action by metaplectic Demazure-Lusztig operators have been generalised to arbitrary Coxeter groups in \cite[\S 3.3]{PP2}.
In case of the affine Weyl group $(W,\{s_0,\ldots,s_r\})$, we expect that the resulting representations of $W$ and $H$ (see \cite[Prop. 3.3.4]{PP2})
are isomorphic to $\sigma_\Lambda^m\vert_W$ and $\pi_\Lambda^m\vert_{H}$ with appropriate choices of $\mathbf{k}$ and $\underline{h}$. 
\end{remark}

\subsection{Metaplectic polynomials}\label{mpsec}
In this subsection we use the quasi-poly\-no\-mial eigenfunctions $E_y^J(x;\mathfrak{t})$ 
from Theorem \ref{Edef} to construct the root system generalisations of the metaplectic polynomials from \cite[Thm. 5.7]{SSV}. 
 
Let $\Lambda\in\mathcal{L}$. The torus elements $\mathfrak{s}_\lambda$ ($\lambda\in\Lambda$) relative to the metaplectic root system $\Phi_0^m$ are given by
\begin{equation}\label{sm}
\mathfrak{s}_\lambda:=\prod_{\alpha\in\Phi_0^+}k_{\alpha^m}^{\eta(\alpha^m(\lambda))\alpha^m}=
\prod_{\alpha\in\Phi_0^+}\Bigl(k_{\alpha^m}^{\chi_{\mathbb{Z}_{>0}}(\alpha(\lambda)/m(\alpha))-\chi_{\mathbb{Z}_{\leq 0}}(\alpha(\lambda)/m(\alpha))}\Bigr)^{\frac{\alpha}{m(\alpha)}},
\end{equation}
see \eqref{ffrak}. We will view $\mathfrak{s}_\lambda$ and $q^\lambda\mathfrak{s}_\lambda\mathfrak{t}_\lambda(\underline{h})$ as multiplicative character of $\Lambda^{m\prime}$ for any lattice $\Lambda^{m\prime}\in\mathcal{L}^m$. Note that by Proposition \ref{facecor}, Lemma \ref{ginvlemma}(4) and \eqref{relrel} 
we have for $\cc\in X^m(\Lambda)$,
\begin{equation}\label{rr3}
w(q^\cc\mathfrak{s}_\cc\mathfrak{t}_\cc(\underline{h}))=
q^{w\cc}\mathfrak{s}_{w\cc}\mathfrak{t}_{w\cc}(\underline{h})\qquad\quad (w\in W^m).
\end{equation}

\begin{theorem}\label{mptheorem}
Let $\Lambda\in\mathcal{L}$ and let $\mathcal{O}$ be a $W^m$-orbit in $\Lambda$. Choose metaplectic parameters $\underline{h}\in\mathcal{M}_{(n,\mathbf{Q})}$ such that in $T_{Q^{m\vee}}$,
\begin{equation}\label{genmet}
q^\lambda\mathfrak{s}_\lambda\mathfrak{t}_\lambda(\underline{h})\not=q^{\lambda^\prime}\mathfrak{s}_{\lambda^\prime}\mathfrak{t}_{\lambda^\prime}(\underline{h})\qquad\hbox{ if }\,\,\lambda,\lambda^\prime\in\mathcal{O}\,\, 
\hbox{ and }\,\, \lambda\not=\lambda^\prime.
\end{equation}
For all $\lambda\in\mathcal{O}$ we then have:
\begin{enumerate}
\item There exists a unique simultaneous eigenfunction $E_\lambda^m(x)\in\mathcal{P}_\Lambda$ of the commuting operators $\pi^m_\Lambda(Y^\mu)$ \textup{(}$\mu\in Q^{m\vee}$\textup{)} satisfying 
\[
E^m_\lambda(x)=x^\lambda+\textup{l.o.t.}
\]
\item For all $\Lambda^{m\prime}\in\mathcal{L}^m$ satisfying $\langle \Lambda,\Lambda^{m\prime}\rangle\subseteq\mathbb{Z}$ we have
\[
\pi_\Lambda^m(Y^\mu)E^m_\lambda=\bigl(q^\lambda\mathfrak{s}_\lambda\mathfrak{t}_\lambda(\underline{h})\bigr)^{-\mu}E^m_\lambda\qquad \forall\, \mu\in\Lambda^{m\prime}.
\]
\item We have
\[
\Gamma^m\bigl(E^m_\lambda(x)\bigr)=\Bigl(\prod_{\alpha\in\Pi(v_\lambda^{-1})}\bigl(-k_{\alpha^m}^{1-\chi_{\mathbb{Z}}(\alpha(\cc_\lambda)/m(\alpha))}
h_{\mathbf{Q}(\alpha^\vee)\alpha(\cc_\lambda)}(\alpha)\bigr)
\Bigr)E_\lambda^{\mathbf{J}(\cc_\lambda)}(x;
q^{\cc_\lambda}\mathfrak{t}_{\cc_\lambda}(\underline{h})\vert_{Q^{m\vee}}),
\]
where $\cc_\lambda\in\overline{C}_+^m$ and $v_\lambda\in W_0$ \textup{(}see Definition \ref{wydef}\textup{)} are relative to the metaplectic root system $\Phi_0^m$.
\end{enumerate}
\end{theorem}
\begin{proof}
Note that under the assumption on $\Lambda^{m\prime}\in\mathcal{L}^m$ as stated in part (2), the metaplectic representation $\pi_\Lambda^m: \mathbb{H}_{Q^{m\vee},\Lambda^{m\prime}}\rightarrow\textup{End}(\mathcal{P}_\Lambda)$ is well defined since $Q^{m\vee}\subseteq\Lambda$.\\
(1) This follows from Corollary \ref{corSSVlink}(1) and Theorem \ref{Edef}.\\
(2) This is a direct consequence of part (1), Theorem \ref{propEVext0}(2)\&(3) and \eqref{rr3}.\\
(3) This follows from Corollary \ref{corSSVlink}(1). The normalisation factor can be derived from Lemma \ref{relationGamma}.
\end{proof}
\begin{remark}
Suppose that $\mathbf{F}$ contains a $(2h)^{th}$ root $q^{\frac{1}{2h}}$ of $q$.
Let $\mathbf{g}\in\mathcal{G}^m$ such that $\underline{h}=\underline{h}^{\mathbf{g}}$, and choose $\widehat{\mathbf{p}}\in\widehat{\mathcal{G}}^m$ satisfying \eqref{specialvalue}.
Then
\[
E^m_\lambda(x)=\mathcal{E}_\lambda(x;\mathbf{g},\widehat{\mathbf{p}}),
\]
which follows from the proof of Corollary \ref{corSSVlink}.
\end{remark}

We call $E_\lambda^m\in\mathcal{P}_\Lambda$ the {\it monic metaplectic polynomial} of degree $\lambda\in\Lambda$ 
relative to the metaplectic datum $(n,\mathbf{Q})$ and the choice of lattice $\Lambda\in\mathcal{L}$, which we suppress from the notations. 

It follows from Theorem \ref{mptheorem} that the metaplectic basic representation $\pi^m: \mathbb{H}\rightarrow\textup{End}(\mathcal{P}_\Lambda)$ is $Y$-semisimple if \eqref{genmet} holds true for all $W$-orbits $\mathcal{O}=\mathcal{O}_\cc$ ($\cc\in X^m(\Lambda)$) in $\Lambda$. Its $Y$-spectrum then consists of the multiplicative characters
\begin{equation}\label{tnew}
q^{\lambda}\mathfrak{s}_\lambda\mathfrak{t}_\lambda(\underline{h})=q^\lambda\prod_{\alpha\in\Phi_0^+}\Bigl(
-k_{\alpha^m}^{1-2\chi_{\mathbb{Z}_{\leq 0}}(\alpha(\lambda)/m(\alpha))}h_{\mathbf{Q}(\alpha^\vee)\alpha(\lambda)}(\alpha)\Bigr)^{\frac{\alpha}{m(\alpha)}}
\qquad\quad (\lambda\in\Lambda)
\end{equation}
(the explicit formula follows from \eqref{sm} and \eqref{tgenmet}). Note that by \eqref{rr2}
we have for $\cc\in X^m(\Lambda)$,
\[
q^\cc\mathfrak{s}_\cc\mathfrak{t}_\cc(\underline{h})\in L_{\mathbf{J}(\cc)},\qquad\quad
q^\cc,q^\cc\mathfrak{t}_\cc(\underline{h})\in T_{\mathbf{J}(\cc)},\qquad\quad
\mathfrak{t}_\cc(\underline{h})\in T_{\mathbf{J}(\cc)}^{\textup{red}}.
\]
\begin{remark}
To compare \eqref{tnew} with the spectrum of the $\textup{GL}_{r+1}$-type metaplectic polynomials in \cite[\S 5.4]{SSV}, note that
\[
q^{\lambda}\mathfrak{s}_\lambda\mathfrak{t}_\lambda(\underline{h})=q^\lambda\prod_{\alpha\in\Phi_0^+}\Bigl(\sigma_{\mathbf{Q}(\alpha^\vee)\alpha(\lambda)}^{(\underline{h})}(\alpha)\Bigr)^{-\frac{\alpha}{m(\alpha)}}
\]
with $\sigma_s^{(\underline{h})}(\alpha)$ defined by
\begin{equation*}
\sigma_s^{(\underline{h})}(\alpha):=
\begin{cases}
k_{\alpha^m}^{-1}\qquad &\hbox{if }\,\, s\in\textup{lcm}(n,\mathbf{Q}(\alpha^\vee))\mathbb{Z}_{>0},\\
-k_{\alpha^m}h_{-s}(\alpha)\qquad &\hbox{if }\,\, s\in\mathbf{Q}(\alpha^\vee)\mathbb{Z}\setminus\textup{lcm}(n,\mathbf{Q}(\alpha^\vee))\mathbb{Z}_{>0}.
\end{cases}
\end{equation*}
This follows by a direct computation using \eqref{def:m} and \eqref{gmetparremark}.
\end{remark}

We conclude this section by considering the type $A_r$ metaplectic polynomials in a bit more detail.
We use the notations from Remark \ref{examplesSSV}(2). We thus have the explicit realisation of the $A_r$-type root system $\Phi_0$ in $\mathbb{R}^{r+1}$ as $\{\epsilon_i-\epsilon_j\}_{i\not=j}$ with $\{\epsilon_i\}_{i=1}^{r+1}$ the standard orthonormal basis of $\mathbb{R}^{r+1}$, and we have a metaplectic datum $(n,\mathbf{Q})$ giving rise to the constant metaplectic multiplicity function $m=n/\textup{gcd}(n,\kappa)\in\mathbb{Z}_{>0}$.

Take $\Lambda=\mathbb{Z}^{r+1}\in\mathcal{L}$, so that 
\[
\mathcal{P}_{\mathbb{Z}^{r+1}}=\mathbf{F}[x_1^{\pm 1},\ldots,x_{r+1}^{\pm 1}]
\]
with $x_i:=x^{\epsilon_i}$ for $i\in [1,r+1]$. The resulting $A_{r}$-type monic metaplectic polynomials 
\[
E_\lambda^m(x)\in\mathcal{P}_{\mathbb{Z}^{r+1}}=\mathbf{F}[x_1^{\pm 1},\ldots,x_{r+1}^{\pm 1}]
\]
of degree $\lambda\in\mathbb{Z}^{r+1}$ then satisfies 
\begin{equation}\label{evGL}
\pi_{\mathbb{Z}^{r+1}}^m(Y^\mu)E_\lambda^m=\bigl(q^\lambda\mathfrak{s}_\lambda\mathfrak{t}_\lambda(\underline{h})\bigr)^{-\mu}E_\lambda^m\qquad\quad\forall\, \mu\in m\mathbb{Z}^{r+1},
\end{equation}
with $\pi_{\mathbb{Z}^{r+1}}^m: \widehat{\mathbb{H}}=\mathbb{H}_{m\mathbb{Z}^{r+1},m\mathbb{Z}^{r+1}}\rightarrow\textup{End}(\mathcal{P}_{\mathbb{Z}^{r+1}})$
the $\textup{GL}_{r+1}$ metaplectic basic representation. The Coxeter-type expressions for the $Y$-elements are determined by the formulas
\[
Y^{m\epsilon_i}=T_{i-1}^{-1}\cdots T_2^{-1}T_1^{-1}uT_r\cdots T_{i+1}T_i\in\widetilde{\mathbb{H}}\qquad\quad (1\leq i\leq r+1),
\]
which follows from \eqref{Yi} with $\ell=m$. 
\begin{remark}
When all the free parameters are formal,  i.e., when $\mathbf{F}=\mathbb{Q}(q,\mathbf{k}^{\frac{1}{2}},\check{h})$ (see Remark \ref{check}), then the generic condition
\eqref{genmet} is valid for all $W$-orbits $\mathcal{O}$ in $\mathbb{Z}^{r+1}$, and hence the commuting operators $\pi_{\mathbb{Z}^{r+1}}^m(Y^\mu)$ ($\mu\in m\mathbb{Z}^{r+1}$) are simultaneously diagonalisable. By inspection of the spectrum (see \eqref{tnew}), their common eigenspaces are $1$-dimensional. In this situation $E_\lambda^{m}(x)\in\mathcal{P}_{\mathbb{Z}^{r+1}}$ can be characterised as the unique Laurent polynomial $\sum_{\lambda^\prime\in\mathbb{Z}^{r+1}}d_{\lambda^\prime} x^{\lambda^\prime}$ satisfying the eigenvalue equations \eqref{evGL} and satisfying $d_\lambda=1$. This proves \cite[Thm. 5.7]{SSV}, and shows that $E_\lambda^m(x)$ is the $\textup{GL}_{r+1}$ metaplectic polynomial $E_\lambda^{(m)}(x)$ from \cite[\S 5.4]{SSV}.
\end{remark}

\section{Limits to metaplectic Iwahori-Whittaker functions}\label{MfinalSection}

In this section we show that the Whittaker limits of the metaplectic polynomials $E_\lambda^m(x)$ and their anti-symmetric versions produce
metaplectic Iwahori-Whittaker and metaplectic spherical functions, respectively. 
The key point is to relate formula \eqref{linkmW2} for the Whittaker limit $\overline{E}_y^J(x)$ of $E_y^J(x;\mathfrak{t})$ with the expression of the metaplectic Iwahori-Whittaker function in terms of metaplectic Demazure-Lusztig operators obtained by Patnaik and Puskas in \cite[(6.48)]{PP}. 

We freely use the notations from the previous section and from Subsection \ref{rationalsection}. In particular, we assume that the ground field $\mathbf{F}$ is $\mathbf{K}(q^{\frac{1}{2h}})$ with $\mathbf{K}$ of characteristic zero, 
and that the multiplicity functions $\mathbf{k}$ and $h_s$ take value in $\mathbf{K}^\times$. 
We write $\overline{\mathcal{M}}_{(n,\mathbf{Q})}$ for the subset of parameters $\underline{h}\in\mathcal{M}_{(n,\mathbf{Q})}$ with multiplicity functions $h_s$ taking values in $\mathbf{K}^\times$. We furthermore 
assume that $\Lambda\in\mathcal{L}$ contains $P^\vee$. 

In this section we work in the metaplectic context. We thus have a fixed metaplectic datum $(n,\mathbf{Q})$, and the quasi-polynomials $E_y^J(x;\mathfrak{t})$ and $E_y^{J,-}(x;\mathfrak{t})$ and their Whittaker limits $\overline{E}_y^J(x)$ and $\overline{E}_y^{J,-}(x)$ occurring in this section will be relative to the metaplectic affine root system $\Phi^m$ and a choice of a multiplicity function $\mathbf{k}$ on $\Phi^m$. Recall that the Whittaker limit boils down to specialising coefficients at $q^{-\frac{1}{2h}}=0$ (see Subsection \ref{rationalsection}). 

\begin{remark}\label{specialisationcond}
The link to representation metaplectic Whittaker functions for metaplectic covers of reductive groups requires 
the field $\mathbf{K}$ to be a non-archimedean local field containing the $n^{\textup{th}}$ roots of unity, and with the cardinality of its residue field equal to $1$ modulo $2n$ (here $n$ is the positive integer from the metaplectic datum $(n,\mathbf{Q})$). In this case the multiplicity function $\mathbf{k}$ takes on the constant value $k$ with $k^{-2}$ equal to the cardinality of the residue field. The multiplicity functions $h_s$ of the metaplectic paramaters $\underline{h}\in\overline{\mathcal{M}}_{(n,\mathbf{Q})}$ are also constant and are equal to certain Gauss sums. For further details, see \cite{CGP,PP}.
\end{remark}

Let $\overline{H}_0^m$ be the finite Hecke algebra over $\mathbf{K}$ viewed as subalgebra of $\mathbb{H}^m$ in the natural way.  
Consider the right $\overline{H}_0^m$-action on $\mathbf{K}_{\mathcal{Q}^m}[\Lambda]$, defined by 
\begin{equation*}
(fx^\lambda)\blacktriangleleft^m T_v:=\pi_\Lambda^{m,X-\textup{loc}}(T_{v^{-1}})(fx^\lambda)
\end{equation*}
for $v\in W_0$, $f\in\mathcal{Q}^m$ and $\lambda\in\Lambda$. This is the metaplectic analog of the $\overline{H}_0$-action defined in Lemma \ref{Wlemma}. In fact, by
Corollary \ref{corSSVlink}(1) we have
\begin{equation}\label{trianglecomp}
\Gamma^m\bigl(fx^\lambda\blacktriangleleft^mT_v\bigr)=\Gamma^m(fx^\lambda)\blacktriangleleft T_v\qquad\quad (v\in W_0),
\end{equation}
where $\blacktriangleleft$ now denotes the right $\overline{H}_0^m$-action \eqref{actionWhittaker} on $\mathbf{K}_{\mathcal{Q}^m}[\Lambda]$.

Define $\gamma_\mu=\gamma_\mu(\underline{h})\in\mathbf{K}^\times$ for $\mu\in\Lambda$ by 
\[
\gamma_\mu:=\prod_{\alpha\in\Pi(v_\mu^{-1})}\Bigl(-k_{\alpha^m}^{1-\chi_{\mathbb{Z}}(\alpha(\cc_\mu)/m(\alpha))}h_{\mathbf{Q}(\alpha^\vee)\alpha(\cc_\mu)}(\alpha)\Bigr),
\]
where $\cc_\lambda\in\overline{C}_+^m$ and $v_\lambda\in W_0$ \textup{(}see Definition \ref{wydef}\textup{)} are defined relative to the metaplectic root system $\Phi_0^m$. Then $\Gamma^m(x^\mu)=\gamma_\mu x^\mu$ for all $\mu\in\Lambda$, see Lemma \ref{relationGamma}. We furthermore write
\[
\kappa_v^m(y):=\prod_{\alpha\in\Phi_0^m}k_{\alpha^m}^{-\eta(\alpha^m(y))}\qquad\qquad (v\in W_0,\, y\in E),
\]
cf. \eqref{kvy}. 
\begin{lemma}\label{oookkk}
For $\lambda\in\Lambda$ and $\underline{h}\in\overline{\mathcal{M}}_{(n,\mathbf{Q})}$ the Whittaker limit $\overline{E}_\lambda^m(x)$ of the metaplectic polynomial
$E_\lambda^m(x)\in\mathcal{P}_\Lambda$ is well defined. Then
\begin{equation}\label{GeGeee}
\Gamma^m(\overline{E}_\lambda^m(x))=\gamma_\lambda\,\overline{E}_\lambda^{\mathbf{J}(\cc_\lambda)}(x)\qquad\qquad \forall\,\lambda\in\Lambda
\end{equation}
and
\begin{equation}\label{mmstart}
\begin{split}
\overline{E}_\lambda^m(x)&=x^\lambda\qquad\qquad\qquad\qquad\qquad\quad\, (\lambda\in\Lambda\cap\overline{E}_-),\\
\overline{E}_{v\lambda}^m(x)&=
\frac{\gamma_{v\lambda}}{\kappa_v^m(\lambda)\gamma_\lambda}\,
x^\lambda\blacktriangleleft^m T_v^{-1}\qquad\quad\,\,\, (v\in W_0,\,\lambda\in\Lambda\cap E_-).
\end{split}
\end{equation}
\end{lemma}
\begin{proof}
By Theorem \ref{mptheorem} we have 
\begin{equation}\label{GGGGG}
\Gamma^m(E_\lambda^m(x))=\gamma_\lambda\,E_\lambda^{\mathbf{J}(\cc_\lambda)}\bigl(x;q^{\cc_\lambda}\mathfrak{t}_{\cc_\lambda}(\underline{h})\vert_{Q^{m\vee}}\bigr)
\end{equation}
for $\lambda\in\Lambda$. Since $\gamma_\mu$ does not depend on $q^{\frac{1}{2h}}$, it follows from \eqref{GGGGG} and Proposition \ref{prop:qlimit}
that the Whittaker limit $\overline{E}_\lambda^m(x)$ of $E_\lambda^m(x)$ exists, and that \eqref{GeGeee} holds true.

By \eqref{fff} we have $\overline{E}_\lambda^m(x)=x^\lambda$ for $\lambda\in\Lambda\cap\overline{E}_-$. For $\lambda\in\Lambda\cap E_-$ and $v\in W_0$ we have
by \eqref{linkmW2} and \eqref{trianglecomp},
\begin{equation*}
\begin{split}
\Gamma^m\bigl(\overline{E}_{v\lambda}^m(x)\bigr)&=\frac{\gamma_{v\lambda}}{\kappa_v^m(\lambda)}\,x^\lambda\blacktriangleleft T_v^{-1}\\
&=\frac{\gamma_{v\lambda}}{\kappa_v^m(\lambda)\gamma_\lambda}\,
\Gamma^m(x^\lambda)\blacktriangleleft T_v^{-1}\\
&=\frac{\gamma_{v\lambda}}{\kappa_v^m(\lambda)\gamma_\lambda}\,\Gamma^m\bigl(x^\lambda\blacktriangleleft^mT_v^{-1}\bigr),
\end{split}
\end{equation*}
which concludes the proof of the lemma.
\end{proof}
For $\lambda\in\Lambda$ define the antisymmetric variant of the metaplectic polynomial $E_\lambda^m(x)$ by
\[
E_\lambda^{m,-}(x):=\pi_\Lambda^m(\mathbf{1}_-)E_\lambda^m(x),
\]
with $\mathbf{1}_-\in\overline{H}_0^m$ given by \eqref{basicminus0}.
\begin{corollary}
For $\underline{h}\in\overline{\mathcal{M}}_{(n,\mathbf{Q})}$ and $\lambda\in\Lambda$ the Whittaker limit $\overline{E}_\lambda^{m,-}(x)$ of 
$E_\lambda^{m,-}(x)$ is well defined. Furthermore, 
\begin{equation}\label{GeGeGe}
\Gamma^m\bigl(\overline{E}_\lambda^{m,-}(x)\bigr)=\gamma_\lambda\,\overline{E}_\lambda^{\mathbf{J}(\cc_\lambda),-}(x)\qquad\forall\,\lambda\in\Lambda
\end{equation}
and for $\lambda\in\Lambda\cap E_-$,
\begin{equation}\label{mmmstart}
\overline{E}_\lambda^{m,-}(x)=\kappa_{w_0}^m(0)^{-2}\sum_{v\in W_0}(-1)^{\ell(v)}\kappa_v^m(0)
\frac{\kappa_v^m(\lambda)\gamma_\lambda}{\gamma_{v\lambda}}\,\overline{E}_{v\lambda}^{m}(x).
\end{equation}
\end{corollary}
\begin{proof}
By Corollary \ref{corSSVlink}(1) we have
\begin{equation*}
\begin{split}
\Gamma^m\bigl(E_\lambda^{m,-}(x)\bigr)&=\gamma_\lambda\,\pi_{\cc_\lambda,q^{\cc_\lambda}\mathfrak{t}_{\cc_\lambda}(\underline{h})}(\mathbf{1}_-)
E_\lambda^{\mathbf{J}(\cc_\lambda)}\bigl(x;q^{\cc_\lambda}\mathfrak{t}_\lambda(\underline{h})\vert_{Q^{m\vee}}\bigr)\\
&=\gamma_\lambda\,E_\lambda^{\mathbf{J}(\cc_\lambda),-}\bigl(x;q^{\cc_\lambda}\mathfrak{t}_\lambda(\underline{h})\vert_{Q^{m\vee}}\bigr).
\end{split}
\end{equation*}
Hence the Whittaker limit $\overline{E}_\lambda^{m,-}(x)$ of $E_\lambda^{m,-}(x)$ exists, and \eqref{GeGeGe} holds true.
For $\lambda\in\Lambda\cap E_-$ we have by \eqref{GeGeGe}, Corollary \ref{linkmWcor} and \eqref{GeGeee},
\begin{equation*}
\begin{split}
\Gamma^m\bigl(\overline{E}_\lambda^{m,-}(x)\bigr)&=\kappa_{w_0}^m(0)^{-2}\sum_{v\in W_0}(-1)^{\ell(v)}\kappa_v^m(0)\kappa_v^m(\lambda)
\gamma_\lambda\,\overline{E}_{v\lambda}^{\mathbf{J}(\cc_\lambda)}(x)\\
&=\kappa_{w_0}^m(0)^{-2}\sum_{v\in W_0}(-1)^{\ell(v)}\kappa_v^m(0)
\frac{\kappa_v^m(\lambda)\gamma_\lambda}{\gamma_{v\lambda}}\,\overline{E}_{v\lambda}^{m}(x),
\end{split}
\end{equation*}
which completes the proof.
\end{proof}

To relate \eqref{mmstart} to formula \cite[(6.48)]{PP} of the metaplectic Iwahori-Whittaker function, we need to compare the $W_0$-action $\sigma_\Lambda^{\textup{CG}}$ and the metaplectic Demazure-Lusztig operator $\mathcal{T}_i^m$ 
with \cite[(4.7)]{PP} and \cite[(4.10)]{PP}.
The operator $w_{\alpha_i}\ast$ in \cite[(4.7)]{PP} corresponds to $I\circ\sigma_\Lambda^{\textup{CG}}(s_i)\circ I$,
where $I$ is the $\mathbf{F}$-algebra automorphism of $\mathbf{F}_{\mathcal{Q}^m}[\Lambda]$ mapping $x^\lambda$ to $x^{-\lambda}$ for all $\lambda\in\Lambda$.
In this identification the parameters $(n(\alpha_i^\vee),v)$ in \cite[(4.10)]{PP} correspond to $(m(\alpha_i),\mathbf{k}^2)$.
Similarly, the operator $\mathbf{T}_{\alpha_i}$ in \cite[(4.10)]{PP} corresponds to 
\begin{equation}\label{metaTitilde}
\widetilde{\mathcal{T}}_i^m:=
I\circ \mathcal{T}_i^m\circ I,
\end{equation}
with $\mathcal{T}_i^m$ the metaplectic Demazure-Lusztig operator \eqref{metaTi}. 

The $\widetilde{\mathcal{T}}_i^m$ ($i\in [1,r]$) satisfy the braid relations by the last formula of \eqref{metaTi} (this was originally observed in \cite[Prop. 7]{CGP}), and hence we can define
for $v\in W_0$ the linear operator
\[
\widetilde{\mathcal{T}}_v^m:=\widetilde{\mathcal{T}}_{i_1}^m\cdots\widetilde{\mathcal{T}}_{i_\ell}^m
\]
acting on $\mathbf{K}_{\mathcal{Q}^m}[\Lambda]$, where $v=s_{i_1}\cdots s_{i_\ell}$ is a reduced expression of $v\in W_0$.

\begin{definition}
For $\lambda\in\Lambda\cap\overline{E}_+$ and $v\in W_0$ set
\begin{equation*}
\mathcal{W}_{v,\lambda}:=\widetilde{\mathcal{T}}_v^m\bigl(x^\lambda\bigr),\qquad\quad
\mathcal{W}_\lambda:=\sum_{v\in W_0}\mathcal{W}_{v,\lambda}
\end{equation*}
in $\mathcal{P}_\Lambda$.
\end{definition}
The expressions \cite[Cor. 5.4]{PP} and \cite[(5.8)]{PP} of the metaplectic Iwahori-Whittaker function and the spherical Whittaker function in terms of metaplectic Demazure-Lusztig operators show that in the realm of representation theory of metaplectic covers of reductive groups (see Remark \ref{specialisationcond}), the polynomial $\mathcal{W}_{v,\lambda}$ is a metaplectic Iwahori-Whittaker function and $\mathcal{W}_\lambda$ is a metaplectic spherical Whittaker function 
(the formula expressing the metaplectic spherical Whittaker function in terms of metaplectic Demazure-Lusztig operators was originally obtained in \cite[Thm. 16]{CGP}). 

In the following theorem we express $\mathcal{W}_{v,\lambda}$ and $\mathcal{W}_\lambda$ in terms of the Whittaker limits $\overline{E}_\lambda^m(x)$ and
$\overline{E}_\lambda^{m,-}(x)$ of the metaplectic polynomials $E_\lambda^m(x)$ and $E_\lambda^{m,-}(x)$. 
\begin{theorem}\label{mWthm}
Let $\underline{h}\in\overline{\mathcal{M}}_{(n,\mathbf{Q})}$ and $\lambda\in\Lambda\cap\overline{E}_+$.
\begin{enumerate}
\item
For $v\in W_0$ we have
\[
\mathcal{W}_{v,\lambda}=(-1)^{\ell(v)}\kappa_v^m(0)\frac{\kappa_v^m(-\lambda-\rho^\vee)\gamma_{-\lambda-\rho^\vee}}{\gamma_{-v(\lambda+\rho^\vee)}}\,\,
x^{-\rho^\vee}
I\bigl(\overline{E}^{m}_{-v(\lambda+\rho^\vee)}(x)\bigr).
\]
\item We have
\begin{equation}\label{preFinal2}
\mathcal{W}_\lambda=\kappa_{w_0}^m(0)^2\,x^{-\rho^\vee}I\bigl(\overline{E}_{-\lambda-\rho^\vee}^{m,-}(x)\bigr).
\end{equation}
\end{enumerate}
\end{theorem}
\begin{proof}
(1) Substituting \eqref{metaTitilde} and the last equality of \eqref{metaTi} we have
\begin{equation}\label{Wstepfirst}
\begin{split}
\mathcal{W}_{v,\lambda}=\widetilde{\mathcal{T}}_v^m\bigl(x^\lambda\bigr)&=(-1)^{\ell(v)}\kappa_v^m(0)\,x^{-\rho^\vee}
I\Bigl(\pi_\Lambda^{m,X-\textup{loc}}\bigl(T_{v^{-1}}^{-1}\bigr)(x^{-\lambda-\rho^\vee})\Bigr)\\
&=(-1)^{\ell(v)}\kappa_v^m(0)\,x^{-\rho^\vee}I\bigl(x^{-\lambda-\rho^\vee}\blacktriangleleft^m T_v^{-1}\bigr).
\end{split}
\end{equation}
Note that $-\lambda-\rho^\vee\in\Lambda\cap E_-$, so the result follows from \eqref{mmstart}.\\
(2) This follows directly by combining part (1) and \eqref{mmmstart}.
\end{proof}
\begin{remark}
Formulas \eqref{GeGeee} and \eqref{GeGeGe} now also allow to express $\mathcal{W}_{v,\lambda}$ and $\mathcal{W}_\lambda$
in terms of $\overline{E}_y^J(x)$ and
$\overline{E}_y^{J,-}(x)$.
\end{remark}
Finally, we shortly explain how McNamara's \cite[Thm. 15.2]{McN} Casselman-Shalika type formula for the metaplectic spherical Whittaker function can be derived from \eqref{CStype}. First note that $\Gamma^m$ has a unique extension to a $\mathcal{Q}^m$-linear automorphism of $\mathbf{K}_{\mathcal{Q}^m}[\Lambda]$.
Define a linear $W_0$-action on $\mathbf{K}_{\mathcal{Q}^m}[\Lambda]$ by 
\[
v\blacktriangleright^m (fx^\lambda):=\sigma_\Lambda^m(v)(fx^\lambda)
\]
for $v\in W_0$, $f\in\mathcal{Q}^m$ and $\lambda\in\Lambda$. Then Corollary \ref{WhitCGcor}, \eqref{linkCGlocal} and Corollary \ref{corSSVlink}(1) imply that
\begin{equation}\label{compppp}
\Gamma^m\bigl(v\blacktriangleright^m (fx^\lambda)\bigr)=v\blacktriangleright \Gamma^m(fx^\lambda)\qquad\quad (v\in W_0),
\end{equation}
where $v\blacktriangleright (fx^\lambda)=\sigma_{\cc_\lambda,q^{\cc_\lambda}\mathfrak{t}_\lambda(\underline{h})}(v)(fx^\lambda)$ (see \eqref{startsigma}).
\begin{proposition}
Let $\lambda\in\Lambda\cap E_-$. Then 
\begin{equation}\label{Final1}
\overline{E}_\lambda^{m,-}(x)=\Bigl(\prod_{\alpha\in\Phi_0^+}\frac{1-k_{\alpha^m}^{-2}x^{-\alpha^{m\vee}}}{1-x^{-\alpha^{m\vee}}}\Bigr)
\sum_{v\in W_0}(-1)^{\ell(v)}v\blacktriangleright^m x^\lambda
\end{equation}
and 
\begin{equation}\label{Final2}
I(\mathcal{W}_\lambda)=\kappa_{w_0}^m(0)^2\Bigl(\prod_{\alpha\in\Phi_0^+}\frac{1-k_{\alpha^m}^{-2}x^{-\alpha^{m\vee}}}{1-x^{-\alpha^{m\vee}}}\Bigr)
\sum_{v\in W_0}(-1)^{\ell(v)}\bigl(\prod_{\alpha\in\Pi(v^{-1})}x^{\alpha^{m\vee}}\bigr)
\sigma_\Lambda^{\textup{CG}}(v)\bigl(x^{-\lambda}\bigr).
\end{equation}
\end{proposition}
\begin{proof}
The proof of \eqref{Final1} follows by combining \eqref{GeGeGe}, \eqref{CStype} and \eqref{compppp}. Substituting 
\eqref{Final1} into \eqref{preFinal2} and using that
\[
x^{\rho^\vee}\,\bigl(v\blacktriangleright^mx^{-\lambda-\rho^\vee}\bigr)=x^{\rho^{m\vee}}\sigma_\Lambda^{CG}(v)\bigl(x^{-\lambda-\rho^{m\vee}}\bigr)
\]
by \eqref{sigmaCG}, we get
\[
I(\mathcal{W}_\lambda)=\kappa_{w_0}^m(0)^2\Bigl(\prod_{\alpha\in\Phi_0^+}\frac{1-k_{\alpha^m}^{-2}x^{-\alpha^{m\vee}}}{1-x^{-\alpha^{m\vee}}}\Bigr)
\sum_{v\in W_0}(-1)^{\ell(v)}x^{\rho^{m\vee}}
\sigma_\Lambda^{\textup{CG}}(v)\bigl(x^{-\lambda-\rho^{m\vee}}\bigr).
\]
The result now follows from the fact that
\[
\rho^{m\vee}-v\rho^{m\vee}=\sum_{\alpha\in\Pi(v^{-1})}\alpha^{m\vee}.
\]
\end{proof}
In the context of representation theory of metaplectic reductive groups over non-archimedean local fields (see Remark \ref{specialisationcond}), formula \eqref{Final2} reduces to McNamara's
\cite[Thm. 15.2]{McN} Casselman-Shalika type formula for the metaplectic spherical Whittaker function, see also \cite[Thm. 4.9]{SSV}.

\section{Index of symbols}

In the first column of the following table we list spaces, groups and sets that will frequently occur in the paper. 
In the second column we list the notations that we try to use throughout the paper for their elements.
\vspace{.2cm}\\
\begin{center}
\begin{tabular}{ |c|c|}
\hline
$\mathbb{R}$ & $d$\\
\hline
$E$ & $y$\\
\hline
$E^*$ & $\xi$\\
\hline
$\overline{C}_+$ & $\cc$\\
\hline
$\Phi_0$ & $\alpha,\beta$\\
\hline
$Q^\vee$ & $\mu,\nu$\\
\hline
$\mathcal{L}$ & $\Lambda$\\
\hline
$\Lambda$ & $\lambda$\\
\hline
$\Lambda_{\textup{min}}^\prime$ & $\zeta$\\
\hline
$W_0$ & $v$\\
\hline
$\Phi$ & $a,b$\\
\hline
$W$ & $u,w$\\
\hline
$\mathbb{W}$ & $g$\\
\hline
$\Omega_{\Lambda^\prime}$ & $\omega$\\
\hline
$\{0,\ldots,r\}$ & $j$\\
\hline
$\{1,\ldots,r\}$ & $i$\\
\hline
$\mathbf{F}^\times$ & $z$\\
\hline
\end{tabular}
\end{center}

\noindent
\S \ref{ConvSection} {\bf Conventions}\\

\noindent
$\mathbf{F},q,\textup{Hom}(G,H),T_\Lambda,R[A],\mathcal{P}_\Lambda,x^\lambda,t^\lambda,[\ell,m]$.\\

\noindent
\S \ref{subsect_root} {\bf Root systems and Weyl groups}\\

\noindent
$E,\langle\cdot,\cdot\rangle, \|\cdot\|, E_{\textup{co}}, E^{\textup{reg}}, E_{\pm},y_{\pm},E^\prime, E_{\pm}^\prime,\textup{pr}_{E^\prime},\textup{pr}_{E_{\textup{co}}},m,
\Phi_0,\Phi_0^{\pm},\Delta_0=\{\alpha_i\}_{i=1}^r,\varphi, Q,$\\
$\Phi_0^\vee,\Delta_0^\vee=\{\alpha_i^\vee\}_{i=1}^r,W_0,w_0,s_\alpha,s_i,Q^\vee,P^\vee,P^{\vee,\pm},\varpi_i^\vee$.\\

\noindent
\S  \ref{S22} {\bf Affine root systems and affine Weyl groups}\\

\noindent
$E^{a,\textup{reg}},C_{\pm}, \mathcal{O}_y,E^*\oplus\mathbb{R}\overset{D}{\twoheadrightarrow} E^*,\Phi,\Phi^{\pm},\Delta=\{\alpha_j\}_{j=0}^r,W,W_y, s_a,s_0,\tau(\mu),W\overset{D}{\twoheadrightarrow} W_0$.\\

\noindent
\S \ref{cosetsection} {\bf Length function and parabolic subgroups}\\

\noindent
$\ell(w),\Pi(w),\cc_y,w_y,\mu_y,v_y,\mathbf{J}(\cc),\Phi_J,\Phi_J^{\pm},W_J, W^J$.\\

\noindent
\S \ref{S24} {\bf The Coxeter complex of $W$}\\

\noindent
$\Sigma(\Phi),C^J,C_J,\mathcal{F}_\Sigma(E,X)$.\\

\noindent
\S \ref{S25} {\bf The double affine Weyl group}\\

\noindent
$\Phi^\vee,a^\vee,K,\widehat{Q}^\vee,\mathbb{W},\mathcal{L},\Lambda,\widehat{\Lambda},j_\Lambda,\mathbf{q}$.\\

\noindent
\S \ref{ap} {\bf Algebras of $q$-difference reflection operators}\\

\noindent 
$\mathbf{F},q_\alpha,T_\Lambda,T,t^\lambda,\varpi,z^\xi,z^\mu,\mathcal{P}_\Lambda,\mathcal{P},\mathcal{Q}_\Lambda,\mathcal{Q},
W\ltimes\mathcal{P},W\ltimes\mathcal{Q}$.\\

\noindent
\S \ref{DAHA} {\bf The double affine Hecke algebra}\\

\noindent 
${}^{\textup{sh}}\Phi_0,{}^{\textup{lg}}\Phi_0,{}^{\textup{sh}}k,{}^{\textup{lg}}k,k,\mathbf{k},k_a,k_j,\chi_{\pm},\chi,
\mathbb{H}=\mathbb{H}(\mathbf{k},q),H=H(\mathbf{k}),H_0=H_0(\mathbf{k})$,\\
$\mathcal{P}_Y,T_w,Y^\mu,U_0,p(Y^{\pm 1}),\delta$.\\

\noindent
\S \ref{intSection} {\bf Intertwiners}\\

\noindent
$S_j^X,S_w^X,S_j^Y,S_w^Y,\ss,\mathbb{H}^{X-\textup{loc}},\mathbb{H}^{Y-\textup{loc}},\widetilde{S}_j^X,\widetilde{S}_w^X,\widetilde{S}_j^Y,\widetilde{S}_w^Y$.\\

\noindent
\S \ref{IndPar} {\bf Induction parameters}\\

\noindent
$T_{\Lambda,J},T_{\Lambda,J}^{\textup{red}},T_J,T_J^{\textup{red}},T_\Lambda^G,\widetilde{T}_\Lambda, \widetilde{T}_{\Lambda,J},L_{\Lambda,J},L_J,L^J,1_{T_\Lambda},W_J^{\textup{red}},J_0,\Phi_0^J,\Phi_{0,J_0},\Phi_0^{+,J}$,\\
$m(\varphi),h,\eta_J(\alpha),n_i(\alpha),\mathfrak{s}_J$.\\

\noindent
\S \ref{SectionMt} {\bf The $\mathbb{H}$-module $\mathbb{M}_t^J$}\\

\noindent
$\mathbb{H}_J^X,\mathbb{H}_J^Y,\chi_{J,t}^X,\chi_{J,t},H_J,\mathbf{F}_{J,t},1_{J,t},\mathbb{M}_t^J,m_t^J,J(t),L^J,\mathbf{F}_{J,t}^X,{}^X\mathbb{M}_t^J,\pi,t_{\textup{sph}},\mathcal{M}(T),\mathcal{E},m_{w;t}^J,\mu_w,v_w$.\\

\noindent
\S \ref{qpsection} {\bf Spaces of quasi-polynomials}\\

\noindent
$R,R[E],x^y,x^{(y,m)},\mathcal{P}^{(\cc)},\mathfrak{t}_y,w_{\mathfrak{t}},T_{\mathbb{R}}^\vee,\mathcal{P}_\gamma,[\cc]$.\\

\noindent
\S \ref{truncsection} {\bf Truncated divided difference operators}\\

\noindent
$\nabla_a,\nabla_j$.\\

\noindent
\S \ref{gbrsectionStatement} {\bf The quasi-polynomial realisation of $\mathbb{M}_t^J$}\\

\noindent
$\chi_B,\eta,\kappa_v(y),\pi_{\cc,\mathfrak{t}},\mathcal{P}_{\mathfrak{t}}^{(\cc)},\phi_{\cc,\mathfrak{t}}$.\\

\noindent
\S \ref{FaceSection} {\bf The face dependence of the quasi-polynomial representation}\\

\noindent
$e_{w,w^\prime;\mathfrak{t}}^{J,h}, \textup{pr}_{\cc,\cc^\prime}^{\mathfrak{t}^\prime}$.\\

\noindent
\S \ref{aWaSection} {\bf Affine Weyl group actions on quasi-rational functions}\\

\noindent
$\mathbf{F}_{\mathcal{Q}}[E],\mathcal{Q}_\gamma,\sigma_{\cc,\mathfrak{t}},\check{s}_j,\pi_{\cc,\mathfrak{t}}^{X-\textup{loc}},\varphi,\xi_y$.\\

\noindent
\S \ref{idsection} {\bf Properties of the base point $\mathfrak{s}_J$}\\

\noindent
$\mathfrak{s}_y,k(y),k_w(y)$.\\

\noindent
\S \ref{proof1section} {\bf The $\mathbb{H}$-action on $\mathcal{P}^{(\cc)}$}\\

\noindent
$\psi_{\cc,\mathfrak{t}},H_{i,m}$.\\

\noindent
\S \ref{POsection} {\bf The parabolic Bruhat order on $E$}\\

\noindent
$\leq_B,(E,\leq),\prec,\prec_\alpha,\textup{sgn}$.\\

\noindent
\S \ref{POsection2} {\bf Triangularity properties}\\

\noindent
$G_{\mathfrak{t}}(a)$.\\

\noindent
\S \ref{genSection} {\bf Genericity conditions for the multiplicity function}\\

\noindent
$T_J^\prime,\lambda_J,q^y,P,\varpi_i,I^{\textup{co}},\mathfrak{t}(\mathbf{z})$, $W_0^I$.\\

\noindent
\S \ref{monicSection} {\bf The monic $Y$-eigenbasis of $\mathcal{P}_{\mathfrak{t}}^{(\cc)}$}\\

\noindent
$m_y^J(x;\mathfrak{t}),d_{y,y^\prime;\mathfrak{t}}^J,d_{w,w^\prime;\mathfrak{t}}^J,\mathcal{S}(M),M[s],E_y^J(x;\mathfrak{t}),e_{y,y^\prime;\mathfrak{t}}^J,e_{w,w^\prime;\mathfrak{t}}^J$.\\

\noindent
\S \ref{CreationSection} {\bf Creation operators and irreducibility conditions}\\

\noindent
$d_w(x),\epsilon_{w,u},\epsilon_{w,u}^\prime,\mathcal{P}^{(\cc)}_{<w},n_w$.\\

\noindent
\S \ref{ClosureSection} {\bf Closure relations}\\

\noindent
$e_{w,u^\prime;\mathfrak{t}^\prime}^{J,J^\prime}$.\\

\noindent
\S \ref{S64} {\bf The normalised $Y$-eigenbasis and pseudo-duality}\\

\noindent
$r_w(x),P_y^J(x;\mathfrak{t}),t_y,\rho_w$.\\

\noindent
\S \ref{S65} {\bf (Anti)symmetrisation}\\

\noindent
$\mathbf{1}_{\pm},\mathcal{Q}_{\mathfrak{t}}^{(\cc),+},\mathcal{P}_{\mathfrak{t}}^{(\cc),+},E_y^{J,\pm}(x;\mathfrak{t}),P_y^{J,+}(x;\mathfrak{t}),\mathcal{O}_\cc^+,W_{0,y_{\pm}},W_0^{y_\pm},g_y,C_y^{\pm}(y^\prime)$.\\

\noindent
\S \ref{unitaritySection} {\bf Pseudo-unitarity and orthogonality relations}\\

\noindent
$\ast,\mathbf{F}_u^\times,\mathbf{F}_r,T_u,T_r,N_w(x),(\cdot,\cdot)_{J,\mathfrak{t}},p^*,\textup{ct},\mathcal{W}$.\\

\noindent
\S \ref{rationalsection} {\bf The Whittaker limit}\\

\noindent
$\mathbf{K},\mathbf{F}_{\textup{reg}},\overline{\mathcal{P}}^{(\cc)},\overline{\mathcal{P}},\mathcal{P}_{\textup{reg}}^{(\cc)},\overline{\mathcal{Q}}^{(\cc)},\overline{\mathcal{Q}},\overline{d},\overline{f},\overline{H}_0,\overline{H},\overline{T},\overline{T}_J^{\textup{red}},\overline{E}_y^J(x),
\overline{E}_y^{J,\pm}(x),\blacktriangleleft,\blacktriangleright, s_\mu$.\\

\noindent
\S \ref{S70} {\bf Extended root datum and extended affine Weyl groups}\\

\noindent
$e,(\mathcal{L}^{\times 2})_e,(\Lambda_1,\Lambda_1^\prime)\leq (\Lambda_2,\Lambda_2^\prime),q^{\ell/e},q_\alpha^{\ell/e},W_{\Lambda},\Omega_{\Lambda},w_{\lambda,\Lambda},v_{\lambda,\Lambda},\Lambda_{\textup{min}},\zeta^\prime,\omega(j),I_{\Lambda^\prime},\varpi_{i,\Lambda^\prime}^\vee$.\\

\noindent 
\S \ref{S700} {\bf The extended double affine Hecke algebra}\\

\noindent
$\mathbb{H}_{\Lambda,\Lambda^\prime},H_{\Lambda^\prime},\Lambda^{\prime+}$.\\

\noindent
\S \ref{S71} {\bf The extended quasi-polynomial representation}\\

\noindent
$\mathcal{P}_\Lambda^{(\cc)},\pi_{\cc,\mathfrak{t}}^{\Lambda,\Lambda^\prime},\mathcal{P}_{\Lambda,\mathfrak{t}}^{(\cc)},\sigma_{\cc,\mathfrak{t}}^{Q^\vee,\Lambda^\prime},V_{\cc,\mathfrak{t}},V_{\cc,\mathfrak{t}}(\omega),\phi_\omega,W_{\Lambda,\cc},\Omega_{\Lambda,J},\Omega_{\Lambda,\cc},\Omega_{\Lambda}^{\cc},{}^\Lambda T_{\Lambda^\prime}^\cc,\mathfrak{t}_{y;\cc},w_{\mathfrak{t};\cc},\psi_{\cc,\mathfrak{t}}^{\Lambda,\Lambda^\prime}$.\\

\noindent
\S \ref{Tp} {\bf Twist parameters}\\

\noindent
$\Xi_{\mathfrak{t}},\Theta_{\mathfrak{t}}$.\\

\noindent
\S \ref{S72} {\bf The extended eigenvalue equations}\\

\noindent
$E_{y;\cc}^{\Lambda,\Lambda^\prime}(x;\mathfrak{t})$.\\

\noindent
\S \ref{S73} {\bf The theory for the $\textup{GL}_{r+1}$ root datum}\\

\noindent
$\epsilon_i,u,\varpi,\varpi_{i,\ell\mathbb{Z}^{r+1}}^\vee,\widetilde{\mathbb{H}},\widetilde{\pi}_{\cc,\mathfrak{t}},Y_i,\widetilde{E}_{y;\cc}(x;\mathfrak{t})$.\\

\noindent
\S \ref{S81} {\bf The $g$-parameters}\\

\noindent
$\widetilde{\mathfrak{t}}$, $\kappa_\alpha(\widetilde{\mathfrak{t}}),\mathcal{G}^{\textup{amb}},\mathbf{f},f_\alpha,\widehat{\mathcal{G}},q^y,\widehat{\mathbf{g}},\widehat{g}_\alpha,\mathcal{G},\mathbf{g},g_\alpha,\mathcal{C}_{\Lambda,\Lambda^\prime},\mathfrak{c},\mathfrak{t}_y(\widehat{\mathbf{g}},\mathfrak{c}),w\cdot_{\widehat{\mathbf{p}},\mathfrak{c}}x^y$.\\

\noindent
\S \ref{unifSection} {\bf The uniform quasi-polynomial representation}\\

\noindent
$\Gamma_{\Lambda,\mathbf{g}}^{(\cc)},\mathbf{F}_{\textup{cl}},\widetilde{\mathbf{g}},\gamma_{\widetilde{\mathbf{g}}},\pi_{\mathbf{g},\widehat{\mathbf{p}},\mathfrak{c}}^{\Lambda,\Lambda^\prime}$.\\

\noindent
\S \ref{ExtUnifSection} {\bf Uniform quasi-polynomial eigenfunctions}\\

\noindent
$\mathcal{E}_y^{\Lambda,\Lambda^\prime}(x;\mathbf{g},\widehat{\mathbf{p}},\mathfrak{c}),\mathcal{E}_y(x;\mathbf{g},\widehat{\mathbf{p}}),\mathbf{F}_{\mathcal{Q}_\Lambda}[E],\sigma_{\mathbf{g},\widehat{\mathbf{p}},\mathfrak{c}}^{\Lambda,\Lambda^\prime}$.\\

\noindent
\S \ref{MetaplecticSection} {\bf Metaplectic representations and metaplectic polynomials}\\

\noindent
$n,\mathbf{Q},\kappa_\ell$.\\

\noindent
\S \ref{S91} {\bf The metaplectic parameters}\\

\noindent
$\mathbf{B},m(\alpha),\Phi_0^m,\alpha^m,Q^{m\vee},\Phi^m,\vartheta,\theta,s_j^m,C_+^m,C^{mJ},\mathcal{G}^m,\mathcal{M}=\mathcal{M}_{(n,\mathbf{Q})}$,\\
$\underline{h},h_s(\alpha),\underline{h}^{\mathbf{g}}, h_s^{\mathbf{g}}(\alpha),\ell_\alpha$.\\

\noindent
\S \ref{msec} {\bf The metaplectic basic representation}\\

\noindent
$\mathbb{H}^m,\mathcal{L}^m,\Lambda^m,\Lambda^{m\prime},r_\ell(s),\nabla_j^m,\pi_\Lambda^{m},\mathbb{H}_{\Lambda^m,\Lambda^{m\prime}},\widehat{\mathcal{G}}^m,\pi_{\mathbf{g},\widehat{\mathbf{p}},\mathfrak{c};\Lambda},
\pi_{\mathbf{g};\Lambda},
\mathcal{P}_\Lambda^{m},X^m(\Lambda)$,\\
$\Gamma^m=\Gamma^m_{\underline{h}},\mathfrak{t}_\lambda(\underline{h}),X^m(\Lambda,\Lambda^m),
\kappa,\widetilde{\mathbb{H}},Z_{\Lambda,\ell}^m,\overline{0},Z_{\Lambda,\ell}^{m,\textup{inv}},
Z_{\Lambda,\ell}^{m,\textup{reg}},\widetilde{Z}_{\Lambda,\ell}^{m,\textup{reg}},\check{h}$.\\

\noindent
\S \ref{CGsubsection} {\bf A metaplectic affine Weyl group action on rational functions}\\

\noindent
$\mathcal{Q}^m,\mathbf{F}_{\mathcal{Q}^m}[L],\sigma_\Lambda^m,\rho^\vee,\rho^{m\vee},\sigma_\Lambda^{\textup{CG}},\mathcal{T}_i^m$.\\

\noindent
\S \ref{mpsec} {\bf The metaplectic polynomials}\\

\noindent
$E_\lambda^m,\sigma_s^{(\underline{h})}(\alpha)$.\\

\noindent
\S \ref{MfinalSection} {\bf Limits to metaplectic Iwahori-Whittaker functions}\\

\noindent
$\overline{\mathcal{M}}_{(n,\mathbf{Q})},\overline{H}_0^m,\blacktriangleleft^m,\gamma_\mu=\gamma_\mu(\underline{h}),\kappa_v^m(y),
\overline{E}_\lambda^m(x),E_\lambda^{m,-}(x),\overline{E}_\lambda^{m,-},\widetilde{\mathcal{T}}_i^m,\widetilde{\mathcal{T}}_v^m,\mathcal{W}_{v,\lambda},
\mathcal{W}_\lambda,\blacktriangleright^m$.



\end{document}